\documentclass[11pt]{article}

\usepackage[utf8]{inputenc} 
\usepackage[svgnames]{xcolor}
\usepackage[T1]{fontenc}
\usepackage{url}
\usepackage[cmex10]{amsmath} 
\usepackage{algorithm}
\usepackage{ifthen}
\usepackage{cite}
\usepackage{amsfonts}
\usepackage{algpseudocode}
\usepackage{enumitem}
\usepackage[margin=0.3in]{geometry} 
\usepackage{fullpage}

\usepackage{amssymb, amsthm}
\usepackage{bm}
\usepackage{wrapfig}
\usepackage{graphicx}
\usepackage{subfig}
\graphicspath{{figs/}}
\usepackage[numbers]{natbib}
\usepackage{listings}
\usepackage{mathrsfs} 
\usepackage{hyperref}
\usepackage{caption}
\usepackage{lmodern}
\usepackage{booktabs}       
\usepackage{amsfonts}       
\usepackage{nicefrac}       
\usepackage{microtype}      
\usepackage{accents}
\usepackage{varwidth}
\usepackage{setspace}

\numberwithin{equation}{section}

\definecolor{dgreen}{rgb}{0.0, 0.5, 0.0}

\newcommand{\xtruev}{\mathbf{x}_0}
\newcommand{\ev}{\mathbf{e}}
\newcommand{\e}{\epsilon}
\newcommand{\fv}{\mathbf{f}}
\newcommand{\xv}{\mathbf{x}}
\newcommand{\vv}{\mathbf{v}}
\newcommand{\rv}{\mathbf{r}}
\newcommand{\pv}{\mathbf{p}}
\newcommand{\uv}{\mathbf{u}}
\newcommand{\qv}{\mathbf{q}}
\newcommand{\yv}{\mathbf{y}}
\newcommand{\zv}{\mathbf{z}}
\newcommand{\wv}{\mathbf{w}}
\newcommand{\bv}{\mathbf{b}}
\newcommand{\cv}{\mathbf{c}}
\newcommand{\sv}{\mathbf{s}}
\newcommand{\Zv}{\mathbf{Z}}
\newcommand{\Am}{\mathbf{A}}
\newcommand{\Idm}{\mathbf{I}}
\newcommand{\Vm}{\mathbf{V}}
\newcommand{\Sm}{\mathbf{S}}
\newcommand{\Qm}{\mathbf{Q}}
\newcommand{\Um}{\mathbf{U}}

\newcommand{\Bm}{\mathbf{B}}
\newcommand{\Cm}{\mathbf{C}}
\newcommand{\Dm}{\mathbf{D}}

\newcommand{\Vv}{\mathbf{V}}

\newcommand{\Pm}{\mathbf{P}}
\newcommand{\Fm}{\mathbf{F}}
\newcommand{\Gm}{\mathbf{G}}
\newcommand{\Om}{\mathbf{O}}
\newcommand{\E}{\mathbb{E}}
\newcommand{\inn}{\mathrm{in}}
\newcommand{\out}{\mathrm{out}}
\newcommand{\betav}{\boldsymbol{\beta}}
\newcommand{\Delm}{\boldsymbol{\Delta}}

\newcommand{\covm}{\boldsymbol{\Sigma}}
\newcommand{\presm}{\boldsymbol{\Pi}}
\newcommand{\nuv}{\boldsymbol{\nu}}
\newcommand{\muv}{\boldsymbol{\mu}}
\newcommand{\ide}[1]{\accentset{*}{#1}}

\newcommand{\dv}{\mathbf{d}}
\newcommand{\RR}{\mathbb{R}}
\newcommand{\gram}{\mathbf{G}}

\newcommand{\be}{\begin{equation}}
	\newcommand{\ee}{\end{equation}}
\newcommand{\ben}{\begin{equation*}}
	\newcommand{\een}{\end{equation*}}
\newcommand{\abs}[1]{\lvert#1\rvert}
\newcommand{\norm}[1]{\lVert#1\rVert}

\newtheorem{lem}{Lemma}
\newtheorem{defi}{Definition}

\newtheorem{fact}{Fact}
\newtheorem{rem}{Remark}

\def\app#1#2{%
  \mathrel{%
    \setbox0=\hbox{$#1\sim$}%
    \setbox2=\hbox{%
      \rlap{\hbox{$#1\propto$}}%
      \lower1.1\ht0\box0%
    }%
    \raise0.25\ht2\box2%
  }%
}
\def\approxprop{\mathpalette\app\relax}

\renewcommand{\P}{\mathbb{P}}


\title{A Non-asymptotic Analysis of Generalized Approximate Message Passing Algorithms with Right Rotationally Invariant Designs} 
\author{%
Collin Cademartori
  \thanks{Columbia University, USA, Email: {\tt cac2301@columbia.edu}}
  \and Cynthia Rush
  \thanks{Columbia University, USA, Email: {\tt cgr2130@columbia.edu}. This work was supported in part by NSF CCF $\#1849883$. }
}

\begin{document}

\maketitle

%
\begin{abstract}
Approximate Message Passing (AMP) algorithms  are a class of iterative procedures for computationally-efficient estimation in high-dimensional inference and estimation tasks. Due to the presence of an `Onsager' correction term in its iterates, for  $N \times M$ design matrices $\Am$ with i.i.d.\ Gaussian entries, the asymptotic distribution of the estimate at any iteration of the algorithm can be exactly characterized in the large system limit as $M/N\to \delta\in (0,\infty)$ via a scalar recursion referred to as state evolution. In this paper, we show that appropriate functionals of the iterates, in fact, concentrate around their limiting values  predicted by these asymptotic distributions with rates exponentially fast in $N$ for a large class of AMP-style algorithms, including those that are used when high-dimensional generalized linear regression models are assumed to be the data-generating process, like the generalized AMP algorithm, or those that are used when the measurement matrix is assumed to be right rotationally invariant instead of i.i.d.\ Gaussian, like vector AMP and generalized vector AMP. In practice, these more general AMP algorithms have many applications, for example in  in communications or imaging, and this work provides the first study of finite sample behavior of such algorithms.
\end{abstract}

\tableofcontents

\newpage

\section{Introduction}
Approximate message passing (AMP) algorithms are a class of iterative methods for solving various high-dimensional statistical estimation and inference problems  \cite{AMP, AMP_GM, BayMont11, krz12, GAMP}.  In this paper we focus specifically on the problem of high-dimensional (generalized) linear regression. In particular, we consider estimating an unknown coefficient vector or signal $\xtruev$, assuming knowledge of both the output $\yv$ and matrix $\Am$, using models of the form
\begin{equation}
\begin{split}
\label{eqn:gen_reg}
\yv &\sim p\left(\yv\mid\Am \xtruev\right), \qquad \xtruev \stackrel{iid}{\sim} p(\xtruev),
\end{split}
\end{equation}
where $\Am\in \mathbb{R}^{M\times N}$ with $M < N$,  $p(\xtruev)$ is a prior distribution on the signal, and $p\left( \yv\mid \cdot \right)$ is some output distribution. We will give special consideration in the paper to the case where $p(\yv\mid \cdot)$ is a Gaussian distribution. Letting $f(\cdot)$ denote a normal density, this allows us to model the problem as
\begin{equation}
\begin{split}
\label{eqn:lin_reg}
\yv &= \Am \xtruev + \wv, \qquad \wv \stackrel{iid}{\sim} f(\wv), \qquad \xtruev \stackrel{iid}{\sim} p(\xtruev).
\end{split}
\end{equation}
 AMP-style algorithms can accommodate a range of estimation procedures for the models in \eqref{eqn:gen_reg}-\eqref{eqn:lin_reg}, including maximum a posteriori (MAP)  and minimum mean squared error (MMSE) estimation. See \cite{feng2022unifying} for a tutorial on AMP.

In its original form, AMP
is the following two-step iteration for recovering $\xtruev$ from knowledge of $\yv$ and $\Am$ under the linear model of \eqref{eqn:lin_reg} and the assumption that $\Am$ is i.i.d.\ Gaussian.  At iteration $k \geq 0$, the algorithm updates its estimate of the signal $\xtruev$ with estimates $\hat{\xv}_1, \hat{\xv}_2, ...$. Initializing with $\hat{\xv}_{-1} = \vv_{-1} = \mathbf{1}$, calculate for $k \geq 0$,
\begin{align}
 \rv_k &= \hat{\xv}_{k-1} + \Am^T \vv_{k-1}, \qquad \hat{\xv}_k = g_{k}(\rv_k),  \label{eq:AMP1} \\
\vv_k &= \yv-\Am\hat{\xv}_k + \frac{N}{M} \vv_{k-1} \mathrm{div}[ g_k(\rv_k)] , \label{eq:AMP2}
\end{align}
where $g_k: \mathbb{R} \rightarrow \mathbb{R}$, which acts elementwise on its vector input, is the so-called `denoiser' function, an appropriately-chosen Lipschitz function  depending on the estimation procedure to be performed, and $\mathrm{div}[g_k(\rv)] = \sum_{i=1}^N \frac{\partial}{\partial [\rv]_i} g_k([\rv]_i)$ is the divergence of the denoiser, measuring the sensitivity of $g_k$ at its input.  
The $\vv_k$ update in \eqref{eq:AMP2} can be interpreted as  a corrected residual: the usual residual $\yv-\Am\hat{\xv}_k$ with a so-called `Onsager correction' given by $\frac{N}{M} \vv_{k-1} \mathrm{div}[ g_k(\rv_k)]$.

With the presence of the correction term in the residual step,
it is possible under certain conditions to characterize the asymptotic distribution of $\rv_k$ in the large system limit where $N\to \infty$ and ${M}/{N}\to \delta\in (0,1)$ (see \cite{BayMont11, Bolt12}). In particular, for variances $\tau_{k}$ that can be characterized exactly by a scalar recursion referred to as state evolution, in this limit, the elements of the vectors $\rv_k$ behave like samples from a Gaussian distribution with mean vector $\xtruev$ and covariance matrix $\tau_k \mathbf{I}_N$, denoted henceforth as $N(\xtruev,\tau_k \mathbf{I}_N)$, where $\xtruev$ is the true signal and $\mathbf{I}_N$ is an $N \times N$ identity matrix. 

More formally, it is shown in \cite{BayMont11} that for all pseudo-Lipschitz (defined in Section~\ref{sec:main}) loss functions $\phi$, in the large system limit, empirical averages converge to deterministic limits, 
\begin{equation}\label{eq:emp_avg}
  \frac{1}{N}\sum_{i=1}^N \phi([\rv_k]_i, [\xtruev]_i) \rightarrow \mathbb{E}\{\phi(X_0+\tau_k Z, X_0)\},
\end{equation}
where $Z \sim N(0,1)$ is independent of $X_0 \sim p(\xtruev)$.  In \cite{AMP_FS}, this distributional convergence is refined with a finite sample analysis showing that these empirical averages 
exhibit concentration around their limits with rates exponential in $N$ up to $t = O\left(\frac{\log N}{\log  \log  N}\right)$ iterations. 

However, asymptotic guarantees for the AMP algorithm in \eqref{eq:AMP1}-\eqref{eq:AMP2} have been primarily studied in settings where  the elements of the design matrices $\Am$ are i.i.d.\ sub-Gaussian \cite{BayMont11, Bayati2015} (though, an active line of recent research studies \emph{universality} properties of AMP algorithms more generally; see Section~\ref{sec:related}) and the concentration results only when the elements are i.i.d.\ Gaussian. Furthermore, considering the linear models in \eqref{eqn:gen_reg}-\eqref{eqn:lin_reg} concentration results have only been given for AMP for the model in \eqref{eqn:lin_reg}, which is limited to the case of Gaussian outputs. In particular, there has been no concentration results developed for generalized versions of AMP, which can handle models such as \eqref{eqn:gen_reg}, or for AMP algorithms with measurement matrices that are not i.i.d.\ Gaussian, despite such algorithms playing an important role in high-dimensional statistics \cite{candes2020, sur2019modern, donoho2016high}, wireless communications \cite{hou2022sparse, biyik2017generalized}, and many other applications \cite{venkataramanan2022estimation, mondelli2022optimal, mondelli2021approximate, mondelli2021pca, barbier2019optimal, zhang2022precise}. 

In this work, we will extend the finite sample analysis for AMP for the linear model in \eqref{eqn:lin_reg} given in \cite{AMP_FS} for a suite of AMP-style algorithms that can be used when it is  assumed that \textbf{(i)} the data follows the generalized linear model in \eqref{eqn:gen_reg}, \textbf{(ii)} the random design matrices $\Am$ comes from a class of right-rotationally invariant matrices (defined below), and \textbf{(iii)} both are assumed. In particular, generalized approximate message passing (GAMP) \cite{GAMP}, is an AMP-style algorithm used to perform estimation when $\Am$ is assumed to be i.i.d.\ Gaussian but the data come from the generalized linear regression model in \eqref{eqn:gen_reg}. Moreover, vector approximate message passing (VAMP) and its generalized extension GVAMP are AMP-style algorithms, recently introduced in \cite{VAMP} and \cite{VAMP_general}, designed to perform estimation for the models \eqref{eqn:lin_reg} and \eqref{eqn:gen_reg}, respectively, but under a much larger class of random design matrices $\Am$.
In particular, this is the class of  \emph{right orthogonally invariant} $\Am$, meaning that $\Am \mathbf{V}$ has the same distribution as $\Am$ for any $N \times N$ orthogonal matrix $\mathbf{V}$.  As we will see, this assumption represents a significant relaxation of the i.i.d.\ sub-Gaussian condition.

This work refines the existing GAMP and VAMP \cite{VAMP, GAMP, fletcher2018inference}  asymptotics by providing a finite sample analysis  of GVAMP -- which works under the assumption of right rotationally invariant design matrices and generalized liner models \eqref{eqn:gen_reg}  -- and by extension to GAMP and VAMP as well. In particular, the non-asymptotic result for VAMP shown in this paper is given under only slightly stronger conditions than are used in the proof of its large system limit asymptotics in \cite{VAMP}. In more detail, in \cite{VAMP}, it was shown that empirical averages of pseudo-Lipschitz functions of the VAMP estimates $\frac{1}{N}\sum_{i=1}^N\phi\left([ \hat{\xv}_{1k}]_i,[\xv_0]_i \right)$ converge in the large system limit (i.e.\ as $N\to\infty$) to their expected values under limiting distributions that can be characterized exactly. Here, we use concentration tools to show that this convergence in fact occurs exponentially fast in $N$ and demonstrate how these results can be extended to the GVAMP and GAMP algorithms as well. 

To summarize, this work generalizes the previous finite sample analyses for the AMP algorithm simultaneously in two main ways:
\begin{itemize}
\item Our work covers GAMP and GVAMP algorithms  for generalized linear problems in the form of \ref{eqn:gen_reg}, with output distributions constrained only by a certain log-concavity requirement. This encompasses many common data models not covered by the Gaussian case, including, for example, logistic and Poisson regression.
\item Our analysis applies to a much less restrictive distributional assumption over the measurement matrix $\Am$ than previous finite sample analyses. In particular, our measurement matrix model samples the singular vector matrices uniformly from their orthogonal groups and allows the singular values to be marginally drawn from any arbitrary distribution with bounded support. This setup is substantially more general than the case of i.i.d.\ Gaussian matrices, and can in particular accommodate arbitrarily poor conditioning of the measurement matrix.
\end{itemize}

We emphasize that while refining earlier asymptotic results to provide non-asymptotic guarantees is useful to clarify the effect of the iteration number on the problem dimension and the sped of convergence to the state evolution predictions, it also has other application-specific implications; e.g.\ the finite sample analyses of \cite{AMP_FS} were used to establish the error exponent for sparse regression codes with AMP decoding for the additive white Gaussian noise channel \cite {rush2017error} and the results in this work will be similarly useful in establishing error exponent for sparse regression codes with GAMP or VAMP decoding for more general channel models considered in \cite{hou2022sparse, biyik2017generalized}.

Structurally, the proof we provide is based on an approach used for the finite sample analysis of AMP in \cite{AMP_FS}; though, there are critical differences in the required analysis for GVAMP.   In short, we develop a suite of concentration of measure tools to prove that, as the algorithm runs, the output retains concentration around its expected values.  The idea is that if the algorithm concentrates through iteration $k$, then we can prove that it will concentrate at iteration $k+1$, with only a slightly degraded rate for the concentration.  The differences between our proof and that given in \cite{AMP_FS} are detailed in Section~\ref{sec:discuss}. 
At a high level, when studying generalized versions of AMP algorithms, one must employ the output vector $\mathbf{y}$ in the denoising functions, which is not the case in standard AMP. This causes added dependencies within the algorithm, as now the measurement matrix $\mathbf{A}$ is no longer independent of the denoioser functions. 
On the other hand, studying the algorithm under rotationally invariant assumptions on $\mathbf{A}$ (as opposed to i.i.d.\ Gaussian assumptions), also complicates the analyses by adding additional dependencies that need to be handled carefully in the concentration arguments (see point \textbf{(1)} in Section~\ref{sec:discuss}).

\subsection{Related Work} \label{sec:related}

Recent work by Li and Wei \cite{li2022non} extends the finite sample analysis of \cite{AMP_FS} to a related AMP algorithm for spiked matrix estimation.  This work shows rates exponential in $N$ for up to $O\left(\frac{N}{\texttt{poly} \log N}\right)$ iterations by using novel proof methods, improving on the $O\left(\frac{\log N}{\log  \log  N}\right)$ iteration guarantees found in this paper and in \cite{AMP_FS}. An intriguing open question is whether their proof technique can be extended to AMP for linear and generalized linear regression, as is studied here.

As we mentioned above, AMP algorithms have been demonstrated both empirically and theoretically to perform well over a wide range of problems and to converge faster than the simple iterative soft thresholding method. However, there are now many variants on the original AMP algorithm. Remaining within the realm of Gaussian design matrices, for example, works have extended the AMP analyses to non-separable denoising functions that do not act componentwise on their vector arguments and can therefore take advantage of correlation between entries of the unknown signal \cite{berthier2020state, ma2019analysis}, to matrices with independent entries and a blockwise variance structure \cite{javanmard2013state, donoho2013information}, or to incorporate additional side information on the signal in order to accommodate streaming estimation procedures \cite{liu2022rigorous}. As mentioned, GAMP is designed to accommodate general separable output distributions; similarly for GVAMP. In other words, GAMP extends AMP to solve problems like \eqref{eqn:gen_reg}. MLAMP \cite{MLAMP}  takes this idea further, accommodating multi-layer problems where separable nonlinear channels are stacked between multiple linear transformations. 

In addition to VAMP, there have been other approaches to extending AMP to work under more general assumptions on the measurement matrix beyond i.i.d.\ Gaussian entries. These include the orthogonal approximate message passing algorithm (OAMP) \cite{OAMP}, extensions to matrices with i.i.d.\ sub-Gaussian entries \cite{bayati2015universality, chen2021universality}, semi-random matrices \cite{dudeja2022universality}, and other generalizations of AMP for rotationally invariant matrices  \cite{fan2022approximate, wang2022universality, opper2016theory}.
The OAMP algorithm is designed, like VAMP, to accommodate a larger class of design matrices $\Am$. In particular, OAMP assumes that $\Am$ has distribution which is orthogonally invariant. If $\Am$ has singular value decomposition $\mathbf{U} \mathbf{S} \mathbf{V}^T$, then this requires that $\mathbf{U}$ and $\mathbf{V}$ are uniformly distributed and $\mathbf{U}$, $\mathbf{V}$, and $\mathbf{S}$ are mutually independent.  Extending finite sample analyses to more general classes of AMP algorithms like those discussed above is an exciting avenue for future work.

The types of generalized versions of AMP we study here, namely GAMP and GVAMP, which are designed for data models such as \eqref{eqn:gen_reg}, and VAMP, which allows to extend beyond i.i.d.\ Gaussian measurement matrices, have found many applications in the literature to date. These algorithms play an important role in high-dimensional statistics \cite{candes2020, sur2019modern, donoho2016high}, wireless communications \cite{hou2022sparse, biyik2017generalized}, and many other applications \cite{venkataramanan2022estimation, mondelli2022optimal, mondelli2021approximate, mondelli2021pca, barbier2019optimal, zhang2022precise}. In particular, the finite sample analysis for VAMP and GAMP presented here can be used to find the error rates for the capacity-achieving sparse regression coding schemes introduced in~\cite{hou2022sparse, biyik2017generalized}, which use VAMP and GAMP as decoders. 
Moreover, AMP algorithms are often used to establish computational limits of various high-dimensional statistics problems, as they are conjectured to be optimal amongst polynomial-time algorithms in some problem settings. In particular, finite sample analyses like that presented in this work, allow one to study such computational limits in the context of changing sparsity,  where the fraction of nonzeros amongst the signal elements approaches zero as the sample size grows (as opposed to the fraction of nonzeros
being fixed in the limit). For example, in
a very simple Bernoulli-Rademacher signal model of rank-one matrix estimation from noisy observations, the work in \cite{macris2020all} considers both information theoretic and computational limits, with the computational limits found via a finite sample analysis of the associated AMP algorithm. We believe the analyses in this work could be used to study computational limits for changing sparsity regimes in more general problem settings.

As we shall see below, the i.i.d.\ (sub-)Gaussian assumption is stronger than is needed to characterize the asymptotic behavior of (G)VAMP. And if one restricts attention to VAMP, the orthogonally invariant condition can be relaxed to the less stringent right orthogonally invariance condition. Furthermore, (G)VAMP has been demonstrated empirically to converge faster than GAMP across a range of design matrices (see \cite[Section VI]{VAMP}).

We finally mention that part of this work was presented at ISIT 2020 \cite{cademartori2020exponentially}. In particular, \cite{cademartori2020exponentially} provides the statement of the result for VAMP but does not include proof details, nor does it include the extension of these finite sample analyses to generalized versions of VAMP or AMP, namely to GVAMP or GAMP. These extensions are significant contributions beyond the original VAMP result, as they provide finite sample analyses for AMP algorithms for the generalized regression model of \eqref{eqn:gen_reg}, whereas the VAMP result applied only to \eqref{eqn:lin_reg}. Applications of the generalized versions of VAMP and AMP were discussed above.

\subsection{Notation}
Throughout, we use lowercase boldface letters like $\vv$ to denote vectors and uppercase boldface letters like $\Vm$ to denote matrices. We use $\Vm^T$ to denote the transpose of $\Vm$. For matrices $\Vm$, we use $\Vm_{ij}$ to denote the element in the $i^{\mathrm{th}}$ row and $j^{\mathrm{th}}$ column. Similarly, $\Vm_{\cdot j}$ is the $j^{\mathrm{th}}$ columns of the matrix and $\Vm_{i\cdot}$ is the $i^{\mathrm{th}}$ row. For vectors $\vv$, we use $[\vv]_i$ to denote the $i^{\mathrm{th}}$ element of $\vv$.

For general vectors $\xv$ and $\yv$, $p(\xv)$  denotes the probability density function of $\xv$ and $p(\xv | \yv)$ is the conditional probability density function of $\xv$ given $\yv$. When $p(\xv)$ is a (potentially multivariate) normal density, we will denote it by $f\left( \xv\mid \boldsymbol{\mu},\boldsymbol{\Sigma} \right)$ where $\boldsymbol{\mu}$ is the mean vector and $\boldsymbol{\Sigma}$ the covariance matrix, and we let $N(\mu,\sigma^2)$ denote a Gaussian random variable with mean $\mu$ and variance $\sigma^2$.

\subsection{Outline}
In Section \ref{sec:VAMP}, we introduce the VAMP and GVAMP iterations and provide intuition for their estimation schemes. In section \ref{sec:main}, we state our main concentration result, Theorem \eqref{thm:main}, and introduce the state evolution recursion which exactly characterizes the distributions of the limiting variables. In section \ref{sec:general}, we state and prove a concentration result for a more general abstract recursion and then prove Theorem \eqref{thm:main} by showing how to recover the (G)VAMP estimates from this abstract recursion. In section \ref{sec:discuss}, we conclude with a  discussion of future work.

 \section{Vector AMP and its Generalized Variant} \label{sec:VAMP}
In this section we discuss the VAMP algorithm and its generalized version, GVAMP. We focus on these two algorithms primarily, as the way that GAMP generalizes AMP is analogous, and our main analysis studies the most general of all these algorithms, GVAMP (from which one can derive results for GAMP and VAMP as well).
The VAMP iteration is similar to that of AMP in \eqref{eq:AMP1}-\eqref{eq:AMP2}, and, like AMP, can be derived as a quadratic approximation to a belief propagation algorithm associated to a particular factor graph. The reader is referred to \cite{VAMP} for such a derivation. The GVAMP iteration extends the core ideas of the VAMP algorithm  to handle arbitrary output distributions. Formally, GVAMP operates like two VAMP iterations glued together at the ends. Here, we provide some self-contained intuition for VAMP and GVAMP.

Algorithm $\ref{alg:VAMP}$ presents VAMP. We define a function $g_2:\mathbb{R}^{N} \times \mathbb{R}_{+} \rightarrow  \mathbb{R}^N$ as
\begin{equation}\label{eqn:LMMSE}
g_2(\rv_{2k},\gamma_{2k}) = (\gamma_w \Am^T \Am+\gamma_{2k} \mathbf{I})^{-1} (\gamma_w\Am^T \yv+\gamma_{2k} \rv_{2k}),
\end{equation}
where $\gamma_w = 1/\tau_w$ and $\tau_w < \infty$ is the elementwise variance of the noise $\wv$. We assume that the noise variance  $\tau_w$  is known to simplify the analysis, though, practically VAMP can be run without such knowledge \cite{EM-VAMP}.  The function $g_2$ in \eqref{eqn:LMMSE} has divergence with respect to the components of its first argument $\rv_{2k}$ given by
\[
  \mathrm{div}\left[g_2(\rv_{2k},\gamma_{2k})\right] = \sum_{i=1}^N \frac{\partial}{ \partial [\rv_{2k}]_i} g_2([\rv_{2k}]_i,\gamma_{2k})  = \gamma_{2k}\mathrm{Tr}\left((\gamma_w\Am^T \Am+\gamma_{2k} \mathbf{I})^{-1}\right),
\]
where $\mathrm{Tr}(\cdot)$ is the trace operator, i.e.\ it sum the diagonal elements of its argument.
Analogous to denoisers $g_k$ in AMP, the function $g_1:  \mathbb{R}^{N} \times \mathbb{R}_{+} \rightarrow  \mathbb{R}^N$ taking input $(\rv_{1k},\gamma_{1k})$ 
must be specified by the user.  In VAMP, however, we characterize the dependence of the denoiser on the iteration number through a parameter, $\gamma_{1k}>0$, and the function $g_1$ itself does not change across iterations.  
\begin{algorithm}
\caption{VAMP\label{alg:VAMP}}
\begin{algorithmic}[1]
\Require{Number of iterations $K$, design matrix $\Am\in\mathbb{R}^{M\times N}$, observed $\yv\in\mathbb{R}^M$, and denoiser $g_1(\cdot,\gamma_{1k})$. (Note: $\Am$ and $\yv$ enter the algorithm through the definition of $g_2$ in \eqref{eqn:LMMSE}.)}
\State Initialize $\rv_{10}$ and $\gamma_{10}\geq 0$.
\For{$k\gets 0,\ldots,K$}
\State $\hat{\xv}_{1k} \gets g_1(\rv_{1k},\gamma_{1k}),$
\hspace{35mm}  $\alpha_{1k}\gets \mathrm{div}\left[g_1(\rv_{1k},\gamma_{1k})\right],$
\State $\eta_{1k} \gets \gamma_{1k}/\alpha_{1k},$
\hspace{41mm}  $\gamma_{2k} \gets \eta_{1k} - \gamma_{1k},$
\State $\rv_{2k} \gets \left(\eta_{1k}\hat{\xv}_{1k} - \gamma_{1k} \rv_{1k}\right)/\gamma_{2k},$
\State
\State $\hat{\xv}_{2k} \gets g_2(\rv_{2k},\gamma_{2k}),$
\hspace{35mm}  $\alpha_{2k}\gets \mathrm{div}\left[g_2(\rv_{2k},\gamma_{2k})\right],$
\State $\eta_{2k} \gets \gamma_{2k}/\alpha_{2k},$ 
\hspace{41mm}  $\gamma_{1(k+1)} \gets \eta_{2k} - \gamma_{2k},$
\State $\rv_{1(k+1)} \gets \left(\eta_{2k}\hat{\xv}_{2k} - \gamma_{2k} \rv_{2k}\right)/\gamma_{1k}.$
\EndFor
\end{algorithmic}
\end{algorithm} 
The choice of $g_1$, and its relationship to $g_2$, can be better understood by carefully examining VAMP in the regression setting of~\eqref{eqn:lin_reg}
with, specifically, Gaussian noise $\wv\sim N\left(\mathbf{0},\tau_w \mathbf{I}_M \right)$.
As described above, we can view VAMP as trying to estimate some summary (e.g.\ the mode or mean) of the posterior distribution of the signal given the output.  Recall that the posterior distribution is proportional to the product of the data likelihood and the signal prior, so in our case is given by
\begin{equation}
\label{eq:posterior}
p(\xv_0 \mid \yv) \propto f\left(\yv\mid \Am \xv_0,\gamma_w^{-1} \mathbf{I} \right)\cdot\prod_{i=1}^Np([\xv_{0}]_i),
\end{equation}
where $f(\yv \mid \boldsymbol{\mu}, \mathbf{\Sigma})$ is a multivariate normal density evaluated at $\yv$ having mean $\boldsymbol{\mu}$ and covariance $\mathbf{\Sigma}$, so that $ f\left(\yv\mid \Am \xv_0,\gamma_w^{-1} \mathbf{I} \right)$ is the data likelihood in the case of Gaussian noise. In \eqref{eqn:lin_reg}, we assume that $\xtruev$ has i.i.d.\ elements, thus $\prod_{i=1}^Np([\xv_{0}]_i)$ is the signal prior distribution.  

In general, the posterior in \eqref{eq:posterior} is difficult to calculate; therefore, so is computing any summary statistic from it.  For this reason, at each iteration, VAMP replaces the task of computing a posterior summary with two easier ones. First, observe that $f(\yv \mid \Am \xv_0,\gamma_w^{-1} \mathbf{I})$, as a function of $[\xv_{0}]_i$, is proportional to a normal density $f([\xv_{0}]_i\mid [\rv_1]_{i}, [{\boldsymbol \tau}_1]_{i})$, for some mean $[\rv_1]_{i}$ and variance $[{\boldsymbol \tau}_1]_{i}$. 

Making the further simplifying assumption that $[{\boldsymbol \tau}_1]_{i}$ are equal across $i$, this yields the approximate posterior
\begin{equation}\label{eqn:approx_1}
p(\xv_0 \mid \yv) \approxprop f(\xv_0 \mid \rv_1,\tau_1 \mathbf{I})\cdot\prod_{i=1}^N p([\xv_{0}]_i).
\end{equation}
If we instead approximate the priors $p([\xv_{0}]_i)$ by independent normal distributions with means $[\rv_2]_{i}$ and equal variances $\tau_2$, then we get the approximate posterior 
\begin{equation}\label{eqn:approx_2}
p(\xv_0 \mid \yv) \approxprop f\left(\yv \mid \Am \xv_0,\gamma_w^{-1}\mathbf{I}\right) \cdot f\left(\xv_0 \mid \rv_{2},\tau_{2} \mathbf{I}\right).
\end{equation}

In each iteration, VAMP uses both approximate posteriors -- \eqref{eqn:approx_1} and \eqref{eqn:approx_2} -- to update the estimate of $\xv_0$. Estimating $\xv_0$ by posterior \eqref{eqn:approx_1} requires specifying $\rv_1$ and $\tau_1$, which control the approximate likelihood of the data. Since approximation \eqref{eqn:approx_2} uses the true likelihood, VAMP uses the estimate from \eqref{eqn:approx_2} to update $\rv_1$ and $\tau_1$ in \eqref{eqn:approx_1}. Similarly, the estimates from \eqref{eqn:approx_1} (which use the true prior) are used to update parameters $\rv_2$ and $\tau_2$ for the approximate prior in \eqref{eqn:approx_2}. By iterating these steps, VAMP uses past estimates to improve both approximations, and then uses the improved approximations to further improve our estimates.

In this context, $g_1$ and $g_2$ perform estimation of the two approximate models. The definitions of these functions depend on both the signal prior and desired posterior summary. In $\eqref{eqn:approx_2}$, the prior is modeled as Gaussian regardless of the true prior. Since the resulting approximate posterior is again Gaussian, and the MMSE and MAP estimates are identical, \eqref{eqn:LMMSE} is the natural choice for $g_2,$ as it calculates the mean of \eqref{eqn:approx_2}. However,  \eqref{eqn:approx_1} depends on the choice of prior $p(\xv_0)$ and the desired summary, so $g_1$ must be specified by the user accordingly.

Next we turn to the GVAMP iteration. For GVAMP, the user must supply two denoising functions, $g_{x1}$ and $g_{z1}$, because now both the prior and the output distributions are user-specified. Indeed, the function $g_{x1}$ is meant to combine an estimate $\rv_{2k}$ of $\xv_0$ with the prior information $p(\xv_0)$ to produce a new estimate of $\xv_0$. Likewise, the function $g_{z1}$ is meant to combine an estimate $\pv_{2k}$ of $\zv_0$ with the specified output distribution $p(y\mid z)$ to produce an updated estimate of $\zv_0$.

We define two other denoising functions in terms of the singular value decomposition $\Am=\Um \Sm \Vm^T$ with $\Um\in\mathbb{R}^{M\times M}$, $\Vm\in\mathbb{R}^{N\times N}$ orthogonal, and $\Sm\in\mathbb{R}^{M\times N}$ the rectangular diagonal matrix of singular values. Let $g_{x2}: \mathbb{R}^N \times \mathbb{R}^M  \times \mathbb{R}_{+}  \times \mathbb{R}_{+} \rightarrow  \mathbb{R}^N$ and $g_{z2}: \mathbb{R}^N \times \mathbb{R}^M  \times \mathbb{R}_{+}  \times \mathbb{R}_{+} \rightarrow  \mathbb{R}^M$ be defined as
\begin{align}
  g_{x2}\left( \rv_{2k},\pv_{2k},\gamma_{2k},\tau_{2k}\right) &= \Vm\Dm_k\left( \tau_{2k}\Sm^T\Um^T\pv_{2k} + \gamma_{2k}\Vm^T \rv_{2k}\right),\\
  g_{z2}\left( \rv_{2k},\pv_{2k},\gamma_{2k},\tau_{2k} \right) &= \Um\Sm\Dm_k\left( \tau_{2k}\Sm^T\Um^T\pv_{2k} + \gamma_{2k}\Vm^T \rv_{2k}\right).
\end{align} 
where
\[
\Dm_k = \mathrm{diag}(\mathbf{d}_k), \text{ with } [\mathbf{d}_k]_{n} = \frac{1}{\tau_{2k}s_{n}^2 + \gamma_{2k}} \text{ and } s_n = \begin{cases}[\Sm]_{nn}, & n\leq M, \\ 0, & M<n\leq N.\end{cases}
\]
The divergences of these functions with respect to the components of $\rv_{2k}$ and $\pv_{2k}$, respectively, are
\begin{align}
  \mathrm{div}[g_{x2}\left( \rv_{2k},\pv_{2k},\gamma_{2k},\tau_{2k}\right)] &= \sum_{n=1}^N \frac{\gamma_{2k}}{\tau_{2k}s^2_n + \gamma_{2k}}, \qquad  \mathrm{div}[g_{z2}\left( \rv_{2k},\pv_{2k},\gamma_{2k},\tau_{2k} \right)] = \sum_{n=1}^N \frac{\tau_{2k}s^2_n}{\tau_{2k}s^2_n + \gamma_{2k}}.\label{eqn:gx2_def}
\end{align}
The above formulas follow because $\mathrm{Tr}(\mathbf{A}\mathbf{B}) = \mathrm{Tr}(\mathbf{B}\mathbf{A})$ when both products are defined. As for VAMP, the iteration dependence of the denoising is entirely controlled by parameters $\gamma_{2k}$ and $\tau_{2k}$. 

With this setup, we specify the GVAMP algorithm  in Algorithm~\ref{alg:GVAMP}.
\begin{algorithm}
\caption{GVAMP\label{alg:GVAMP}}
\begin{algorithmic}[1]
\Require{Number of iterations $K$, design matrix $\Am\in\mathbb{R}^{M\times N}$, observed $\yv\in\mathbb{R}^M$, and denoisers $g_{x1}(\cdot,\gamma_{1k})$ and $g_{z1}(\cdot,\tau_{1k})$}. (Note: $\Am$ and $\yv$ enter the algorithm through $g_{x2}$ and $g_{z2}$ in \eqref{eqn:gx2_def}.)
\State Initialize $\rv_{11}$, $\pv_{11}$, $\gamma_{11}\geq 0$, and $\tau_{11}\geq 0$.
\For{$k\gets 1,\ldots,K$}
\State $\hat{\xv}_{1k} \gets g_{x1}(\rv_{1k},\gamma_{1k}),$ 
\hspace{34mm}  $\alpha_{1k}\gets \mathrm{div}\left[g_{x1}(\rv_{1k},\gamma_{1k})\right],$
\State $\rv_{2k} \gets \left(\hat{\xv}_{1k} - \alpha_{1k} \rv_{1k}\right)/(1-\alpha_{1k}),$ 
\hspace{13mm}   $\gamma_{2k} \gets \gamma_{1k}\left( \frac{1}{\alpha_{1k}}-1 \right), $
\State
\State $\hat{\zv}_{1k} \gets g_{z1}(\pv_{1k},\tau_{1k}),$  
 \hspace{35mm}    $\beta_{1k}\gets \mathrm{div}\left[g_{z1}(\pv_{1k},\tau_{1k})\right],$
\State  $\pv_{2k} \gets \left(\hat{\zv}_{1k} - \beta_{1k} \pv_{1k}\right)/(1-\beta_{1k}),$
 \hspace{13mm}   $\tau_{2k} \gets \tau_{1k}\left( \frac{1}{\beta_{1k}}-1 \right),$
 \State
\State $\hat{\xv}_{2k} \gets g_{x2}(\rv_{2k},\pv_{2k},\gamma_{2k},\tau_{2k}),$
 \hspace{20mm}   $\alpha_{2k}\gets \mathrm{div}\left[g_{x2}(\rv_{2k},\pv_{2k},\gamma_{2k},\tau_{2k})\right],$
\State $\rv_{1(k+1)} \gets \left(\hat{\xv}_{2k} - \alpha_{2k} \rv_{2k}\right)/(1-\alpha_{2k}),$
 \hspace{7mm}  $\gamma_{1(k+1)} \gets \gamma_{2k}\left( \frac{1}{\alpha_{2k}}-1 \right),$
\State
\State $\hat{\zv}_{2k} \gets g_{z2}(\rv_{2k},\pv_{2k},\gamma_{2k},\tau_{2k}),$
 \hspace{21mm}  $\beta_{2k}\gets \mathrm{div}\left[g_{z2}(\rv_{2k},\pv_{2k},\gamma_{2k},\tau_{2k})\right],$
\State $\pv_{1(k+1)} \gets \left(\hat{\zv}_{2k} - \beta_{2k} \pv_{2k}\right)/(1-\beta_{2k}),$
 \hspace{7mm}  $\tau_{1(k+1)} \gets \tau_{2k}\left( \frac{1}{\beta_{2k}}-1 \right).$
\EndFor
\end{algorithmic}
\end{algorithm}
As in the case of VAMP, we can understand the GVAMP iteration as solving a series of simpler surrogate estimation problems that are combined between steps to yield an estimate of $\xv_0$ that accounts for all of the information in the full model. Each of these surrogate models incorporates the estimates from one or more of the other surrogate models by relating them to the parameters they are estimating through additive Gaussian noise. In other words, each of the surrogate models assume that $\rv_{uk}=\xv_0 + \vv$ and $\pv_{uk}=\zv_0 + \wv$ for $u \in \{1,2\}$ where $\vv$ and $\wv$ are i.i.d.\ mean $\mathbf{0}$ Gaussian vectors with variances estimated by the model parameters $\gamma_{uk}$ and $\tau_{uk}$.

The first surrogate model incorporates the prior with an estimate $\rv_{1k}$ of $\xv_0$, having density
\begin{equation}
  \label{eq:approx_gvamp_1}
  p_1(\rv_{1k}\mid\gamma_{1k}) = \prod_{i=1}^Np([\xv_0]_i) \, f([\rv_{1k}]_i\mid [\xv_0]_i,1/\gamma_{1k}),
\end{equation}
where $f(\cdot\mid\mu,\sigma^2)$ is the normal density with mean $\mu$ and variance $\sigma^2$. We can view $g_{x1}$ as performing inference for the parameter $\xv_0$ of this model. 
The second surrogate model combines estimates $\rv_{2k}$ of $\xv_0$ and $\pv_{2k}$ of $\zv_0$ with the information from the full model encoded by the linear constraint $\zv_0 = \Am\xv_0$. Its density is given as follows.
\begin{equation}
  \label{eq:approx_gvamp_2}
  p_2(\rv_{2k},\pv_{2k}\mid\gamma_{2k},\tau_{2k}) = \prod_{i=1}^Nf([\xv_0]_i\mid [\rv_{2k}]_i,1/\gamma_{2k})\prod_{j=1}^Mf([\zv_0]_j\mid [\Am\xv_0]_j,1/\gamma_w) \, f([\pv_{2k}]_j\mid [\zv_0]_j,1/\tau_{2k}).
\end{equation}
We can view $g_{x2}$ and $g_{z2}$ as performing inference for the parameters $\xv_0$ and $\zv_0$ of this model respectively in the limit as $\gamma_w\to \infty$ (i.e. as the linear constraint is made exact).
The final surrogate model combines an estimate $\pv_{1k}$ of $\zv_0$ with the output distribution and has density
\begin{equation}
  \label{eq:approx_gvamp_3}
  p_3(\pv_{1k}\mid\tau_{1k}) = \prod_{j=1}^Mf([\zv_0]_j\mid [\pv_{1k}]_j,1/\tau_{1k}) \, p([\yv]_j\mid [\zv_0]_j).
\end{equation}
The function $g_{z1}$ can be seen as performing inference for the parameter $\zv_0$ of this model. The steps between the denoising steps in the GVAMP algorithm can viewed analogously to those in VAMP. In particular, they update the running estimates of the parameters $\xv_0$ and $\zv_0$ by combining the denoiser outputs with the previous estimates in a convex combination, with the relative weighting determined by the sensitivity (i.e.\ the divergence) of the denoiser at the previously denoised input.

\section{Main Result} \label{sec:main}
Before stating our main concentration result for GVAMP, from which results for VAMP and GAMP follow, we first revisit the notion of empirical convergence, we provide some more details on the matrix assumptions required for the VAMP algorithm to perform well, and we spend some time discussing the assumptions under which our main results follow.
 First, we define the class of functions that act as test functions for assessing convergence and concentration.
\begin{defi}
  A function $\phi:\mathbb{R}^J \rightarrow \mathbb{R}$ is pseudo-Lipschitz
of order $2$, denoted $\phi\in\mathrm{PL}(2)$, if for vectors $\vv, \vv' \in \mathbb{R}^J$ and some constant $L\geq 0$, it satisfies 
$\left|\phi(\vv)-\phi(\vv')\right|\leq L\|\vv-\vv'\|\left(1+\|\vv\|+\|\vv'\|\right).$
\end{defi}
With this concept, we can then define empirical convergence and concentration of vector sequences to $L^2$ random variables/distributions in the following way.
\begin{defi}[Empirical convergence]
We will say that a sequence of random vectors $\{\vv_n\}_{n=0}^\infty\subset\mathbb{R}^J$ converges empirically (with $2$nd moment) to $\mathbf{V}\in\mathbb{R}^J$ if $\mathbb{E}\{ [\mathbf{V}_j]^2\}<\infty$ for all $1\leq j\leq J$ and if
$\frac{1}{N}\sum_{n=0}^N\phi(\vv_n) \to \mathbb{E}\{\phi(\mathbf{V})\},$
almost surely for any $\phi\in\mathrm{PL}(2):\mathbb{R}^J\to\mathbb{R}$.
\end{defi}

\begin{defi}[Exponentially fast concentration]
  We will say that a random vector sequence $\{\vv_n\}_{n=0}^\infty\subset\mathbb{R}^J$ concentrates exponentially fast (in $N$) on $\mathbf{V}\in\mathbb{R}^J$ if $\mathbb{E}\{ [\mathbf{V}_j]^2\}<\infty$ for all $1\leq j\leq J$, and if for all $\epsilon>0$ and all $N\geq 0$,
\[
\P\Big(\Big| \frac{1}{N}\sum_{n=0}^N\phi(\vv_n) - \mathbb{E}\phi(\Vv) \Big| > \epsilon \Big) \leq K\exp(-kN\epsilon^2),
\]
for any $\phi\in\mathrm{PL}(2):\mathbb{R}^J\to\mathbb{R}$, where $K,k >0$ are universal constants independent of $\phi$ and $N$. 
\label{def:concentration}
\end{defi}

If, for any finite $N\geq 0$, $\vv_1,\ldots,\vv_N$ are drawn i.i.d.\ from a subgaussian distribution, then \cite[Lemma B.4]{AMP_FS} shows that $\{\vv_n\}$ concentrates on $\vv_1$. Thus, this definition roughly requires the sequence $\{\vv_n\}$ to behave asymptotically like an i.i.d.\ sample from a subgaussian distribution.
In what follows, we simply write that a vector sequence converges empirically to or concentrates exponentially fast on some random variable, without explicitly stating  that the random variable has a finite $2$nd moment.

The following definitions allow us to characterize those matrices $\Am$ to which our results apply.

\begin{defi}[Haar Distribution] \label{def:Haar}
A matrix $\Vm\in\mathbb{R}^{N\times N}$ is Haar distributed on the class of $N\times N$ orthogonal matrices if $\Vm_0\Vm \stackrel{d}{=} \Vm$ for any non-random orthogonal matrix $\Vm_0 \in \mathbb{R}^{N \times N}$.
\end{defi}

\begin{defi}[Orthogonal Invariance]  \label{def:orthog_inv}
For a matrix $\Am \in \mathbb{R}^{M \times N}$, let $\sv$ be its vector of singular values and let $\Am = \Um \Sm \Vm^T$ be its singular  value decomposition, where $ \Um\in\mathbb{R}^{M\times M}$ and $\Vm\in\mathbb{R}^{N\times N}$ are orthogonal matrices and $\Sm\in\mathbb{R}^{M\times N}$ is the rectangular, diagonal matrix with $\Sm_{ii}=\sv_{i}$ for $1\leq i\leq M$.  $\Am$ is \textbf{orthogonally invariant} if  $\Vm$ and $\Um$ are independent and \textbf{Haar distributed} on the groups of $N\times N$ and $M\times M$ orthogonal matrices, respectively. This implies that the distributions of $\Um_0\Am \Vm_0$ for any fixed orthogonal $\Um_0\in\mathbb{R}^{M\times M}$ and $\Vm_0\in\mathbb{R}^{N\times N}$ and $\Am$ are identical.
\end{defi}

\subsection{Assumptions for the Main Results}
Now we state the assumptions under which we provide our main result for the GVAMP algorithm. We can remove some of these assumptions for the finite sample analysis of VAMP (see Remark~\ref{rem:VAMP}).

\textbf{Assumption 0.} The initial estimates are generated as $\rv_{10}=\Vm\rv^{\mathrm{init}}$ and $\pv_{10}=\Um\pv^{\mathrm{init}}$ for vectors $\rv^{\mathrm{init}}$ and $\pv^{\mathrm{init}}$ with i.i.d. subgaussian entries and independent of the matrices $\Vm$, $\Um$. The output vector $\yv$ can be expressed as $\yv = h\left( \zv_0,\wv \right)$ where $\zv_0=\Am\xv_0$ is the transformed input, $\wv$ is an independent ``disturbance'' vector with i.i.d. subgaussian components, and $h$ is a measurable function. We note that this function is also implicitly constrained by the Lipschitz assumption on $g_{z1}$ below since this denoiser leverages the likelihood $f(y=h(z,w)\mid \cdot)$ for estimating $\zv_0$.

Next, the truth $\xv_0$ is also independent of these other quantities and has i.i.d. subgaussian components. These conditions imply that $\xv_0$, $\wv$, and the initial estimates $\rv_{10}$ and $\pv_{10}$ jointly concentrate on a random vector $(X_0, W,R_{10}, P_{10})$ by Lemma \ref{lem:orth_concentration}. Our main technical result Lemma \ref{lem:main_general} will also imply that $\zv_0$ concentrate on a Gaussian random variable $Z_0$.

The entries of the singular value vector $\sv\in\mathbb{R}_+^{M}$ are sampled i.i.d. from a distribution supported in a bounded interval $[0,s_{\max}]$, where the upper bound $s_{\max}$ is independent of $N$.
Furthermore, $\sv$ is independent of the singular vector matrices $\Um$ and $\Vm$. The concentration assumption for $\xv_0$ is satisfied whenever the components of $\xv_0$ are drawn independently from a subgaussian distribution.

Furthermore, the initial precision estimates $\gamma_{10}$ and $\tau_{10}$ are positive for all $N$ and converge to some $\overline{\gamma}_{10},\overline{\tau}_{10}>0$, respectively. Finally, we assume a fixed sparsity level $\delta = \frac{M}{N}$ independent of $N$.

\textbf{Assumption 1.} The design matrix $\Am \in \mathbb{R}^{M \times N}$ is orthogonally invariant. If $\Am$ has singular value decomposition $\Um\Sm\Vm^T$, then this is equivalent to $\Vm$ and $\Um$ being independently Haar distributed on the group of $N\times N$ and $M\times M$ orthogonal matrices, denoted $\textsf{O}(N)$ and $\textsf{O}(M)$, respectively. The Haar property means that $\Vm'\Vm\stackrel{d}{=}\Vm$ and $\Um'\Um\stackrel{d}{=}\Um$ for any other $\Vm'\in \textsf{O}(N)$ and $\Um'\in \textsf{O}(M)$. Definitions of right orthogonal invariance and the Haar distribution along with some useful properties of these kinds of random elements are given in Definition~\ref{def:Haar}.

To obtain concentration for VAMP, it suffices to assume that $\Am$ is just right orthogonally invariant. This only makes the Haar assumption on $\Vm$, and $\Um$ may be an arbitrary orthogonal matrix. (In particular, no independence needs to be assumed between $\Vm$ and $\Um$.)

\textbf{Assumption 2.}
The log prior $\log p(x)$ and the log likelihood $\log p(y\mid z)$ (as a function of $z$) are concave and $\beta$-smooth. The latter property just requires that, for some $\beta>0$,
\begin{equation}\label{eq:beta_smooth}
  -\frac{\partial^2}{\partial x^2}\log p(x) < \beta, \qquad -\frac{\partial^2}{\partial z^2}\log p(y\mid z) < \beta.
\end{equation}

\textbf{Assumption 3.} The estimating functions $g_{x1}$, $g_{z1}$, $g_{x2}$, and $g_{z2}$ are separable\footnote{A function $g:\mathbb{R}^J\times\mathbb{R}\to\mathbb{R}^J$ is \emph{separable} if for $\vv \in \mathbb{R}^J$ and $z\in\mathbb{R}$, there exist a function $\widetilde{g}:\mathbb{R}^2\to\mathbb{R}$ for which $[g(\vv,z)]_{j} = \widetilde{g}([\vv]_{j},z)$ for all $j \in [J]$.}, and both $g_{x1}$ and $g_{z1}$ and their derivatives are uniformly Lipschitz\footnote{A function $\phi(\wv,c)$ is uniformly Lipschitz at $c_0$ if there is an open neighborhood $U$ of $c_0$ and a constant $L>0$ such that $\phi(\cdot,c)$ is $L-$Lipschitz for any fixed $c \in U$ and  $\left|\phi(\wv,c_1) - \phi(\wv,c_2)\right|\leq L\left(1+\|\wv\|\right)\left|c_1-c_2\right|$ for all $c_1,c_2\in U$.} at $\overline{\gamma}_{1k}$ for all $k\geq 0$. When $g_{x1}$ and $g_{z1}$ are the MAP or MMSE estimators for the models \eqref{eqn:approx_1}, they are separable. Henceforth, we will slightly abuse notation by writing
these functions as taking both vector and scalar input. For VAMP, these same assumptions apply to the estimating functions $g_1$ and $g_2$.

\textbf{Assumption 4.} The $\alpha_{ik}$ and $\beta_{ik}$ are truncated to lie in some interval $[t_{\min},t_{\max}]\subset (0,1)$. The $\gamma_{ik}$ and $\tau_{ik}$ are also clipped so that they lie in some intervals $[\gamma_{\min},\gamma_{\max}]$ with $0<\gamma_{\min}<\gamma_{\max}<\infty$ and $[\tau_{\min},\tau_{\max}]$ with $0<\tau_{\min}<\tau_{\max}<\infty$ respectively.
Intuitively, this prevents the $\rv$ and $\pv$ updates from blowing up and thus assists with convergence. From the point of view of concentration, the truncation allows us to control the exponential rate of convergence. Truncating the $\alpha_{ik}$ and $\beta_{ik}$ is not necessary if the second derivatives in \eqref{eq:beta_smooth} are also lower bounded by some $\alpha>0$ (in which case the densities are also $\alpha$-strongly log-concave).

\textbf{Assumption 5.} We terminate the algorithm according to stopping criteria, given explicitly in the next section. Essentially they imply that we stop if the MSE of our current estimate is sufficiently small or if there is a sufficiently small change in the estimate between successive iterations.

\begin{rem}
\label{rem:VAMP}
  For VAMP, some simplifications of the above assumptions are possible. Specifically, in Assumption 1, we only need $\Vm$ to be Haar distributed ($\Um$ can be treated as deterministic), and in Assumption 2, the latter inequality in \eqref{eq:beta_smooth} is automatic (since the output is normal by definition).
\end{rem}

\begin{rem}
We note that these assumptions do exclude our analysis from applying to (G)VAMP using a denoiser $g_1$ having a derivative that is piecewise constant and, thus, not uniformly Lipschitz, like the soft-thresholding function. However, if the derivative  is bounded (like in the case of the soft-thresholding function), the analysis will still apply to appropriately constructed smooth approximations to the function that would provide qualitatively similar inference.
\end{rem}

\begin{rem}
  The requirement of \textbf{Assumption 0} that $(\xv_0,\rv_{10},\pv_{10})$ jointly  concentrate exponentially fast holds if $\rv^{\mathrm{init}}$ and $\pv^{\mathrm{init}}$ are i.i.d. sampled from a subgaussian distribution and are independent of each other and of all other quantities. This follows by Lemmas \ref{lem:PLsubgaussconc} and \ref{lem:orth_concentration} as well as the assumed concentration of $\xv_0$ and the assumed independence of $\xv_0$, $\Um$, and $\Vm$.
\end{rem}

\begin{rem}
The assumption of fixed sparsity level $\delta$ can be relaxed to allow for sequences $\delta_N=\frac{M}{N}$, bounded away from $0$ and converging to some limit $\delta_N\to\delta > 0$ as $N\to\infty$. However, this requires introducing an $N$ dependence to the state evolution, adding notational and conceptual complexity. 
\end{rem}

\begin{rem}
If we make the additional assumption that $\P(S=0)=0$ (i.e.\ that $\Am$ has full rank with probability one), then can express the SVD of $\Am$ as $\Am = \Um\Sm_{M}\Vm^T_M$ where $\Sm_M$ and $\Vm_M$ are the matrices obtained from $\Sm$ and $\Vm$ by retaining only the first $M$ columns, so that $\Sm\in\mathbb{R}^{M\times M}$ and $\Vm\in\mathbb{R}^{N\times M}$. This ``economy'' decomposition is unique up to permutations of the singular values/vectors with probability one, so we have that $p(\Am) \propto p(\Um,\Vm_M,\sv) = p(\Um)p(\Vm_M)p(\sv)$.
It follows from symmetry properties of the Haar measure that the distribution of $\Vm_M$ is uniform over $N\times M$ matrices with orthonormal columns and, for any singular value bound $s_{\max} < \infty$, we are free within our assumptions to take $[\sv]_i\stackrel{i.i.d.}{\sim}\mathsf{uniform}\left( [0,s_{\max}] \right)$ (as all bounded variables exhibit subgaussian tail concentration). With this choice, the right-hand side above becomes constant; thus, our assumptions are broad enough to apply to the uniform distribution over the class of full rank matrices $\Am$ with singular values drawn from $(0,s_{\max}]^M$.
\end{rem}

For VAMP, our assumptions are only slightly more restrictive than those used for the asymptotic analyses in \cite{VAMP}, with the main differences being that we assume (i) exponentially fast concentration rather than convergence of input quantities and (ii)  clipping and truncating of scalar iterates in Assumption 4 (though,   \cite{VAMP} also suggest clipping $\gamma$ as a practical matter of algorithm stability).

\subsection{Main Results}
Rangan \emph{et al.} \cite{VAMP} show empirical convergence of the vector sequence $\left\lbrace\left([\hat{\xv}_{1k}]_{i}, [\rv_{1k}]_{i}, [\xv_0]_{i}\right)\right\rbrace_{i=1}^N$
to $(\hat{X}_{1k},R_{1k},X_0)$ for the case of the VAMP iterates where $R_{1k} =X_0 + \sqrt{\tau_{1k}}Z$ with $\hat{X}_{1k} = g_1(R_{1k},\overline{\gamma}_{1k})$ and  $Z\sim N(0,1)$ independent of $X_0 \sim p(\xv_0)$.
The constants $\tau_{1k}$ and $\overline{\gamma}_{1k}$ that describe the limiting variable $\hat{X}_{1k}$ can be characterized exactly by the state evolution equations for VAMP, which we turn our attention to in Section~\ref{sec:SE} after we present our main result, Theorem~\ref{thm:main}, characterizing the asymptotic rate of this empirical convergence for VAMP and GVAMP. 

\newtheorem{theorem}{Theorem}
\begin{theorem}
\textbf{[VAMP and GVAMP Concentration.]} Under \textbf{Assumption 0} - \textbf{Assumption 5}  given above, for any $\phi\in\mathrm{PL}(2):\mathbb{R}^2\to\mathbb{R}$, any $k\geq 0$, and any $\epsilon \in (0,1)$,
\begin{eqnarray}
\mathbb{P}\Big(\Big|\frac{1}{N}\sum_{i=1}^N \phi([\hat{\xv}_{1k}]_{i}, [\xv_0]_{i}) - \mathbb{E}\{\phi(\hat{X}_{1k},X_0)\}\Big|\geq\epsilon\Big)\leq C' C_{k}e^{-c' c_{k}N\epsilon^2},
\label{eq:thm1_result}
\end{eqnarray}
where $C_k =  C^{2k} (k!)^{16}$ and $c_k =  \frac{1}{c^{2k} (k!)^{22}}$  with $C, C'$ and $c, c'$ being universal constants not depending on $\epsilon$ or $N$ and $\hat{X}_{1k}= g_1(R_{1k},\overline{\gamma}_{1k})$ where $R_{1k} =X_0 + \sigma_{1k}Z$ with $Z\sim N(0,1)$ independent of $X_0$. The constants $\sigma_{1k}$ and $\overline{\gamma}_{1k}$  are defined in Eq.\ \eqref{eq:SE} in Section~\ref{sec:SE}. In \eqref{eq:thm1_result}, the vector $\hat{\xv}_{1k}$ is that given in either Algorithm~\ref{alg:GVAMP} or Algorithm~\ref{alg:VAMP}.
\label{thm:main}
\end{theorem}

\begin{rem}
We emphasize that Theorem~\ref{thm:main} is true for both GVAMP (Algorithm~\ref{alg:GVAMP}) and VAMP (Algorithm~\ref{alg:VAMP}), while the assumptions can be slightly simplified for VAMP as discussed in Remark~\ref{rem:VAMP}. This is because moving from proving VAMP concentration to proving GVAMP concentration requires the addition of an 'output' channel (whereas as the 'input' channel is analyzed in both cases). The analysis of the 'output' channel is symmetric to that of the input' channel. There is a technical difficulty in that GVAMP requires analyzing denoising functions that depend on the vector $\zv_0 = \Am\xv_0$, whereas a direct analysis of the VAMP algorithm would not need to consider this, but our theoretical analysis of the GVAMP algorithm below still implies the equivalent result for the (simplified) VAMP algorithm.
\end{rem}

The proof of Theorem~\ref{thm:main} is given in Section~\ref{sec:proof_thm1}. The key to proving Theorem 1 comes in Lemma \ref{lem:joint_dists}, which establishes that the GVAMP iterates are jointly equal in distribution to the iterates of another recursion defined entirely in terms of functions of i.i.d.\ Gaussian variables. The proof strategy for our main technical concentration lemma (Lemma \ref{lem:main_general_long}, which implies Theorem \ref{thm:main}) then boils down to using the theory concentration for transformations of Gaussians to inductively establish the desired results for this distributionally equivalent recursion.

We finally mention that the proof of Theorem~\ref{thm:main} implies a similar concentration result for GAMP as well. This is because, conceptually, we have seen that upgrading the finite sample analysis of an AMP-style algorithm (VAMP) to its generalized version (GVAMP) only requires repeating the same analysis for a symmetric 'output' channel and including a few more assumptions to deal with dependencies between the measurement matrix in the algorithm and the dependence of the denoiser on the output. Hence, GAMP concentration follows from the proof in~\cite{AMP_FS}.

To save space and to avoid introducing further notation, we only state an informal theorem of this result here, and we refer to the statement of the GAMP algorithm in \cite[Section 4 Equation (55)]{feng2022unifying} and its corresponding state evolution in \cite[Section 4 Equation (57)-(58)]{feng2022unifying}. 

\begin{theorem}
\textbf{[Informal GAMP Concentration.]} Considering the GAMP algorithm in \cite[Section 4 Equation (55)]{feng2022unifying} and its state evolution in \cite[Section 4 Equation (57)-(58)]{feng2022unifying}, the asymptotic equivalence of its iterates and its state evolution  (see for example for an asymptotic statement in  \cite[Section 4 Theorem 4.2]{feng2022unifying}) can be upgraded to exponentially fast concentration with rates of $C C_{k}e^{-c c_{k}N\epsilon^2}$ where $C_k =  C^{2k} (k!)^{\kappa}$ and $c_k =  \frac{1}{c^{2k} (k!)^{\kappa'}}$  with $\kappa, \kappa'$ and $c, c'$ being universal constants not depending on $\epsilon$ or $N$. While we do not specify $\kappa, \kappa'$ explicitly in this informal statement, these will be small constant values  (like $\kappa=16$ and $\kappa'=22$ from Theorem~\ref{thm:main}).
\label{thm:GAMP}
\end{theorem}

The assumptions for this result are slightly different than for GVAMP. In particular, all assumptions about the matrix $\Am$ above are replaced by the assumption that $\Am$ has i.i.d. $N(0,1)$ entries (the same condition under which the large system analysis of GAMP was originally given). Furthermore, the initialization may now be randomly generated without the additional multiplication by $\Vm$ and $\Um$ (as in \textbf{Assumption 0}). The remaining assumptions are analogous to those given above, and these represent natural extensions of the assumptions used for original large system analysis of GAMP. In particular, the Lipschitz condition on the denoisers must be extended to the uniformly Lipschitz condition in order to control the rate of change in the $\tau$ parameters, the initial data and disturbance vector $\wv$ must now be drawn i.i.d. from subgaussian distributions in order to ensure concentration at the appropriate rate, and stronger control is needed over the denoiser derivatives in order to control the probability of algorithm divergence.

\subsection{State Evolution} \label{sec:SE}
As mentioned previously, state evolution equations describe the dynamics of AMP-style algorithms. Generally, these are recursive relationships that characterize the effective signal-to-noise ratios at various stages of the algorithm.
To specify the GVAMP state evolution, which provides the constants in the expectation in Theorem~\ref{thm:main} we must define sensitivity and error functions for each of the approximate models \eqref{eq:approx_gvamp_1}, \eqref{eq:approx_gvamp_2}, and \eqref{eq:approx_gvamp_3}. Following \cite{VAMP}, the sensitivity functions are
\begin{align}
  A_{x1}(\gamma_1,\sigma^2_{1}) = \mathbb{E}\{g_{x1}'(R_1,\gamma_1)\}, \quad &\text{ and } \quad A_{x2}(\gamma_2,\tau_2) = \lim_{N\to\infty} \frac{\gamma_2}{N}\mathrm{Tr}\left[\left( \tau_2\Sm^T\Sm+\gamma_2 \right)^{-1}\right], \label{eq:A2}\\
  A_{z1}(\tau_1,\rho^2_1) = \mathbb{E}\{g_{z1}'(P_1,\tau_1)\}, \quad &\text{ and } \quad A_{z2}(\gamma_2,\tau_2) = \lim_{N\to\infty} \frac{\tau_{2}}{M}\mathrm{Tr}\left[(\Sm^T\Sm)\left( \tau_2\Sm^T\Sm+\gamma_2 \right)^{-1}\right], \label{eq:A3}
\end{align}
where $R_1\sim N(X_0,\sigma^2_{1})$ and $P_1\sim N(Z_0,\rho^2_{1})$ and the derivatives of $g_{x1}$ and  $g_{z1}$ are taken with respect to their first input.  Next, the error functions are defined as
\begin{align*}
  \mathcal{E}_{x1}(\gamma_1,\sigma^2_{1}) = \mathbb{E}[\left(g_{x1}(R_1,\gamma_1)-X_0\right)^2], \, &\text{ and } \, \mathcal{E}_{x2}(\gamma_2,\tau_2,\sigma^2_2,\rho^2_2) = \lim_{N\to\infty} \frac{1}{N}\mathbb{E}[\left\|g_{x2}(\rv_2,\pv_2,\gamma_2,\tau_2)-\xv_0\right\|^2 ],\\
  \mathcal{E}_{z1}(\tau_1,\rho^2_{1}) = \mathbb{E}[\left(g_{z1}(P_1,\tau_1)-Z_0\right)^2], \, &\text{ and } \, \mathcal{E}_{z2}(\gamma_2,\tau_2,\sigma^2_2,\rho^2_2) = \lim_{N\to\infty} \frac{1}{N}\mathbb{E}[\left\|g_{z2}(\rv_2,\pv_2,\gamma_2,\tau_2)-\zv_0\right\|^2 ],
\end{align*}
where $\rv_2\sim N(\xv_0,\sigma^2_2\mathbf{I})$ and $\pv_2\sim N(\zv_0,\rho^2_2\mathbf{I})$, and where $\zv_0 = \Am\xv_0$.

In terms of these functions, we define the scalar evolution equations for GVAMP.  This sequence of equations determines the constants used in expectation in Theorem~\ref{thm:main}.  Initializations $\overline{\gamma}_{10}$ and $\overline{\tau}_{10}$ are defined in \textbf{Assumption 0} as limits of the precision inputs to the GVAMP algorithm, and we initialize the  $\sigma$ and $\rho$ terms as $\sigma^2_{10} = \mathbb{E}\left\lbrace \left( R_{10}-X_0 \right)^2\right\rbrace$  and  $\rho^2_{10} = \mathbb{E}\left\lbrace \left( P_{10}-Z_0 \right)^2\right\rbrace$ where the random variables $X_0, Z_0, R_{10},$ and $P_{10}$ are defined in \textbf{Assumption 0}. Then for $k \geq 0$,
\begin{equation}
  \label{eq:SE}
  \begin{split}
    \overline{\alpha}_{1k} &= A_{x1}(\overline{\gamma}_{1k},\sigma^2_{1k}),\\
    \sigma^2_{2k} &= \frac{ \mathcal{E}_{x1}(\overline{\gamma}_{1k},\sigma^2_{1k}) - \overline{\alpha}_{1k}^2\sigma^2_{1k}}{(1-\overline{\alpha}_{1k})^2},\\
    \overline{\gamma}_{2k} &= \overline{\gamma}_{1k}\Big( \frac{1}{\overline{\alpha}_{1k}}-1 \Big),\\
    \overline{\alpha}_{2k} &= A_{x2}(\overline{\gamma}_{2k},\overline{\tau}_{2k}),\\
    \sigma^2_{1(k+1)} &= \frac{\mathcal{E}_{x2}(\overline{\gamma}_{2k},\overline{\tau}_{2k},\rho^2_{2k},\sigma^2_{2k}) - \overline{\alpha}_{2k}^2\sigma^2_{2k}}{(1-\overline{\alpha}_{2k})^2},\\
     \overline{\gamma}_{1(k+1)} &= \overline{\gamma}_{2k}\Big( \frac{1}{\overline{\alpha}_{2k}}-1\Big),
  \end{split}
  \hspace{0.5cm}
  \begin{split}
    \overline{\beta}_{1k} &= A_{z1}(\overline{\tau}_{1k},\sigma^2_{2k}),\\
    \rho^2_{2k} &= \frac{\mathcal{E}_{z1}(\overline{\tau}_{1k},\rho^2_{1k}) - \overline{\beta}_{1k}^2\rho^2_{1k} }{(1-\overline{\beta}_{1k})^2},\\
    \overline{\tau}_{2k} &= \overline{\tau}_{1k}\Big( \frac{1}{\overline{\beta}_{1k}}-1 \Big),\\
    \overline{\beta}_{2k} &= A_{z2}(\overline{\gamma}_{2k},\overline{\tau}_{2k}),\\
    \rho^2_{1(k+1)} &= \frac{ \mathcal{E}_{z2}(\overline{\gamma}_{2k},\overline{\tau}_{2k},\rho^2_{2k},\sigma^2_{2k}) - \overline{\beta}_{2k}^2\rho^2_{2k} }{(1-\overline{\beta}_{2k})^2},\\
    \overline{\tau}_{1(k+1)} &= \overline{\tau}_{2k}\Big( \frac{1}{\overline{\beta}_{2k}}-1 \Big).
  \end{split}
\end{equation}
This state evolution closely resembles the state evolution for VAMP derived in \cite{VAMP}, with additional terms to account for the added complexity of the output distribution. In short, the VAMP state evolution is structurally the same recursion with only half the terms.

\section{A General Concentration Result} \label{sec:general}

Theorem \ref{thm:main} is a consequence of a concentration result for a more general iteration, Algorithm~\ref{alg:general_gvamp}, which we introduce now. Let $f_p^{\inn}(p_1,p_2,w_p^{(1)},w_p^{(2)},\tau,\gamma):\mathbb{R}^{2d}\times\mathbb{R}^2\times\mathbb{R}^2\to\mathbb{R}^{d}$ be a function of $2d+2$ arguments and two parameters, and let $f_p^{\out}$, $f_q^{\inn}$, and $f_q^{\out}$ be defined similarly. Furthermore let $\fv_p^{\inn}(\pv_1,\pv_2,\wv^{\inn}_p,\tau,\gamma):\mathbb{R}^{2d\times N}\times\mathbb{R}^{2\times N}\times \mathbb{R}^2\to\mathbb{R}^{N\times d}$ be given by applying $f_p^{\inn}$ separably over the $N$ rows of its arguments. 

We note that if $M< N$, then $\pv^\inn_k$ and $\pv^\out_k$ (defined in Algorithm~\ref{alg:general_gvamp} below) have different numbers of columns. When this occurs, we must take care in defining the notion of separability. In general, we will encounter situations in which we want to extend a function $\phi(x,y,w):\mathbb{R}^d\times\mathbb{R}^d\times\mathbb{R}^2\to\mathbb{R}^d\in PL(2)$ of $2d+2$ arguments $x,y$ and $w$ separably over vectorized inputs $\xv\in\mathbb{R}^{d\times N}$, $\yv\in\mathbb{R}^{d\times M}$ and $\wv^{\inn}\in\mathbb{R}^{2\times N}$ or $\wv^{\out}\in\mathbb{R}^{2\times M}$ for which $M\leq N$.
When $\wv\in\mathbb{R}^{2\times M}$, define the separable extension by simply truncating $\xv$ to an $d\times M$ matrix, and the resulting extended function takes values in $\mathbb{R}^M$. When $\wv\in\mathbb{R}^{2\times N}$, it will suffice for our purposes to consider only those $\phi$ for which $\phi([\xv]_i,[\yv]_i,[\wv]_i)$ is independent of $y$ for all $i> M$. In this case, the separable extension of $\phi$ over these vectorized arguments clearly makes sense. By convention, unless otherwise noted, we take $[\yv]_i=\mathbf{0}$ for $i > M$ so that expressions like $\phi([\xv]_i,[\yv]_i,[\wv]_i)$ are well-defined for all $i\leq N$. In particular, we use these conventions to define separable extensions of the functions $\fv_p^{\inn/\out}(\cdot,\cdot,\cdot,\tau,\gamma)$ and $\fv_q^{\inn/\out}(\cdot,\cdot,\cdot,\tau,\gamma)$ over vectorized arguments/parameters in the following sections.

Now, letting $\Vm$ and $\Um$ be orthogonal matrices, we can state Algorithm \ref{alg:general_gvamp} which defines a recursion in terms of these quantities.
\begin{algorithm}[ht]
  \setstretch{1.3}
\caption{General Recursion for GVAMP\label{alg:general_gvamp}}
\begin{algorithmic}[1]
  \Require{Orthogonal matrices $\Vm \in\mathbb{R}^{N \times N}$ and $\Um\in\mathbb{R}^{M\times M}$; separable denoisers $\fv^{\inn}_p$, $\fv_p^{\out}$, $\fv^{\inn}_q$, and $\fv_q^{\out}$ (defined over their argument according to the conventions given above);
    parameter update functions $\Gamma_{p/q}^{\inn/\out}:\mathbb{R}^d\to\mathbb{R}^d$; and disturbance vectors $\wv^{\inn}_p, \wv^{\inn}_q\in\mathbb{R}^{L\times N}$ and $\wv^{\out}_p, \wv^{\out}_q\in\mathbb{R}^{2\times M}$.}
   \State Initialize $\uv_0^{\inn}\in\mathbb{R}^{N\times d}$, $\uv_0^{\out}\in\mathbb{R}^{M\times d}$, $\gamma_{pk}^{\inn}\in\mathbb{R}^d$, and $\gamma_{pk}^{\out}\in\mathbb{R}^d$.
\For{$k\gets 0,\ldots,K$} 
\State $\pv^{\inn}_{k}\gets \Vm \uv_k^{\inn}$, \hspace{55mm}  $\pv_{k}^{\out}\gets\Um\uv_{k}^{\out}$,   \label{eq:alg_gvamp_p_step}
\State $\alpha_{pk}^\inn \gets \mathrm{div}\left[ \fv_p^\inn\left( \pv_k^\inn,\pv_k^\out,\wv^\inn_p,\gamma^\out_{pk},\gamma^\inn_{pk} \right) \right]$, \quad  \quad $\alpha_{pk}^\out \gets \mathrm{div}\left[ \fv_p^\out\left( \pv_k^\inn,\pv_k^\out,\wv^\out_p,\gamma^\out_{pk},\gamma_{pk} \right) \right]$,
\State \begin{varwidth}[t]{\linewidth}
  $[\vv_k^{\inn}]_j\gets \frac{1}{1-[\alpha_{pk}^\inn]_j}\left[ \fv_p^{\inn}\left( \pv_k^{\inn},\pv_k^{\out},\wv^\inn_p,\gamma^\out_{pk},\gamma^\inn_{pk} \right)_j-[\alpha_{pk}^\inn]_j[\pv_k^{\inn}]_j \right]$,
  \par \qquad  $[\vv_k^{\out}]_j\gets \frac{1}{1-[\alpha_{pk}^\out]_j}\left[ \fv_p^{\out}\left( \pv_k^{\inn},\pv_k^{\out},\wv^\out_p,\gamma^\out_{pk},\gamma^\inn_{pk} \right)_j-[\alpha_{pk}^\out]_j[\pv_k^{\out}]_j \right]$.   \label{eq:alg_gvamp_v_step}
  \end{varwidth}
  \State
  \State $\gamma^\inn_{qk} \gets \Gamma^\inn_{q}(\gamma^\inn_{pk},\alpha_{pk}^\inn)$,  \hspace{42mm}  $\gamma^\out_{qk}\gets \Gamma_q^\out(\gamma^\out_{pk},\alpha_{pk}^\out)$,  \label{eq:alg_gvamp_gamma_step}
  \State $\qv_k^{\inn}\gets \Vm^T\vv_k^{\inn}$,  \hspace{51mm}   $\qv_k^{\out}\gets \Um^T\vv_k^{\out}$, \par   \label{eq:alg_gvamp_q_step}
 \State $\alpha_{qk}^\inn \gets \mathrm{div}\left[ \fv_q^\inn\left( \qv_k^\inn,\qv_k^\out,\wv^\inn_q,\gamma^\out_{qk},\gamma^\inn_{qk} \right) \right]$, \quad \quad   $\alpha_{qk}^\out \gets \mathrm{div}\left[ \fv_q^\out\left( \qv_k^\inn,\qv_k^\out,\wv^\out_q,\gamma^\out_{qk},\gamma^\inn_{qk} \right) \right]$,
\State \begin{varwidth}[t]{\linewidth}
  $[\uv_{k+1}^{\inn}]_j\gets \frac{1}{1-[\alpha_{qk}^\inn]_j}\left[ \fv_q^{\inn}\left( \qv_k^{\inn},\qv_k^{\out},\wv^\inn_q,\gamma^\out_{qk},\gamma^\inn_{qk} \right)_j-[\alpha_{qk}^\out]_j[\qv_k^{\inn}]_j \right]$, \par
\qquad   $[\uv_{k+1}^{\out}]_j\gets \frac{1}{1-[\alpha_{qk}^\out]_j}\left[ \fv_q^{\out}\left( \qv_k^{\inn},\qv_k^{\out},\wv^\out_q,\gamma^\out_{qk},\gamma^\inn_{qk} \right)_j-[\alpha_{qk}^\out]_j[\qv_k^{\out}]_j \right]$.
\end{varwidth}
\State $\gamma^\inn_{p(k+1)} \gets \Gamma^\inn_{p}(\gamma^\inn_{qk},\alpha_{qk}^\inn)$,  \hspace{35mm}   $\gamma^\out_{p(k+1)}\gets \Gamma_p^\out(\gamma^\out_{qk},\alpha_{qk}^\out)$,

\EndFor
\end{algorithmic}
\end{algorithm}
Likewise, the VAMP iterates can be generated from a simpler general recursion given in Algorithm~\ref{alg:general_vamp}. The VAMP algorithm is not strictly a special case of GVAMP; similarly, Algorithm \ref{alg:general_vamp} is not a special case of Algorithm~\ref{alg:general_gvamp}. However, Algorithm \ref{alg:general_vamp} can be seen as a subset of Algorithm~\ref{alg:general_gvamp} by discarding all the ``output'' steps in Algorithm \ref{alg:general_gvamp} and making the denoisers $f_p^\inn$ and $f_q^\inn$ depend only on the input variables. The resulting ``input only'' algorithm exactly recovers Algorithm \ref{alg:general_vamp}.
\begin{algorithm}
\caption{General Recursion for VAMP\label{alg:general_vamp}}
\begin{algorithmic}[1]
\Require{Orthogonal matrix $\Vm \in\mathbb{R}^{N \times N}$, separable (w.r.t.\ the first two arguments) denoisers 
\[f_p: \mathbb{R}^{d\times N} \times  \mathbb{R}^{d\times N} \times  \mathbb{R} \rightarrow  \mathbb{R}^{2\times N} \text{ and } f_q: \mathbb{R}^{d\times N} \times  \mathbb{R}^{2\times N} \times  \mathbb{R} \rightarrow  \mathbb{R}^N,\] 
divergence functions $C_i: \mathbb{R}\rightarrow \mathbb{R}$ for $i\in\{1,2\}$, parameter update functions $\Gamma_i: \mathbb{R}^2\rightarrow \mathbb{R}$ for $i\in\{1,2\}$, and disturbance vectors $\wv^p \in\mathbb{R}^{N}$ and $\wv^q \in\mathbb{R}^{2N}$.}
   \State Initialize $\uv_0 \in\mathbb{R}^{N}$.
\For{$k\gets 0,\ldots,K$} 
\State $\pv_k\gets \Vm \uv_k,$  \label{eq:alg_vamp_p_step}
\hspace{71.5mm} $\alpha_{1k} \gets \mathrm{div} [f_p(\pv_k,\wv^p,\gamma_{1k})],$
\State $\vv_k\gets C_1(\alpha_{1k})\left[f_p(\pv_k,\wv^p,\gamma_{1k})-\alpha_{1k}\pv_{k}\right],$
\hspace{23mm}  $\gamma_{2k}\gets \Gamma_1\left(\gamma_{1k},\alpha_{1k}\right),$
\State
\State $\qv_k \gets \Vm^T\vv_k,$  \label{eq:alg_vamp_q_step}
\hspace{69.5mm}  $\alpha_{2k} \gets \mathrm{div}[f_q(\qv_k,\wv^q,\gamma_{2k})],$
\State $\uv_{k+1} \gets C_2(\alpha_{2k})\left[f_q(\qv_k,\wv^q,\gamma_{2k})-\alpha_{2k} \qv_k\right],$
\hspace{20mm}  $\gamma_{1(k+1)} \gets \Gamma_2\left(\gamma_{2k},\alpha_{2k}\right).$
\EndFor
\end{algorithmic}
\end{algorithm}

As a result, all of the arguments regarding concentration of the iterates in Algorithm \ref{alg:general_gvamp} can be specialized to arguments for concentration of the iterates of Algorithm \ref{alg:general_vamp}. Henceforth, unless specified otherwise, all arguments and discussion will focus on Algorithm \ref{alg:general_gvamp}. To recover the relevant statements for Algorithm \ref{alg:general_vamp}, it suffices to drop the ``$\inn$'' and ``$\out$'' superscripts (and, when both occur, eliminate dependence on the ``$\out$'' variable).

We note that the divergence terms $\boldsymbol {\alpha} \in \mathbb{R}^d$ in Algorithm \ref{alg:general_gvamp} are computed component-wise, so e.g.
\[
  [\alpha^{\inn}_{pk}]_j = \sum_{i=1}^n \frac{\partial}{\partial p^{\inn}_j}\left[f_p^{\inn}\left( p_1^{\inn},\ldots,p_d^{\inn},p^{\out}_1,\ldots,p^{\out}_d,w_p,\gamma^{\out},\gamma^{\inn} \right)\right]_j\;\Bigg\vert_{\left([\pv^{\inn}_k]_i,[\pv^{\out}_k]_i,[\wv^{\inn}_p]_i,\gamma^{\inn}_{pk},\gamma^{\out}_pk\right)},
\]
for each $j=1,\ldots,d$. 

We establish a correspondence between the general algorithm, Algorithm~\ref{alg:general_gvamp}, and GVAMP in Algorithm~\ref{alg:GVAMP}, through the following translation. We take the orthogonal matrices $\Um, \Vm$ in Algorithm \ref{alg:general_gvamp} to be the left and right singular vector matrices of $\Am$, respectively. We make the following definitions:
\begin{equation}\label{eq:general_identity}
  \begin{split}
     \pv_k^{\inn} = \left[\rv_{1k}\; \boldsymbol{0} \right], \quad \qquad \pv_k^{\out}=\left[\pv_{1k}\; \zv_0,\right],   &\quad \qquad \vv_k^{\inn}= \left[\rv_{2k}\;, \xv_0\right], \quad  \qquad \vv_k^{\out} = \left[\pv_{2k},\; \boldsymbol{0}\right],\\
      \qv_k^{\inn} = \left[\Vm^T \rv_{2k},\; \Vm^T\xv_0\right], \quad \qv_k^{\out} = \left[\Um^T\pv_{2k},\; \boldsymbol{0}\right],  &\quad \qquad\uv_{k}^{\inn} = \left[\Vm^T \rv_{1k},\; \boldsymbol{0}\right],\quad \uv_{k}^{\out} = \left[\Um^T\pv_{1k},\; \Um^T\zv_0\right].
  \end{split}
\end{equation}
Furthermore, we make the following function definitions:
\begin{equation}
  \begin{split}
  \label{eq:translation_2}
  f_p^{\inn}\left( p^{\inn},p^{\out},w_p,\tau,\gamma \right) &= \left[g_{x1}\left([p^{\inn}]_1,\gamma \right),\; [w_p]_1 \right],\\
  f_p^{\out}\left( p^{\inn},p^{\out},w_p,\tau,\gamma \right) &= \left[g_{z1}\left([p^{\out}]_1,\tau,h\left([p^{\out}]_2,[w_p]_1\right)\right),\; 0\right],\\
  f_q^{\inn}\left( q^{\inn},q^{\out},w_q,\tau,\gamma \right) &= \left[\left( [\tau]_1 [w_q]_1 [q^{\out}]_1 + [\gamma]_1 [q^{\inn}]_1 \right)/\left( [\tau]_1[w_q]_1^2+[\gamma]_1\right)\;,0\right],\\
  f_q^{\out}\left( q^{\inn},q^{\out},w_q,\tau,\gamma \right) &= \left[[w_q]_1\left( [\tau]_1 [w_q]_1 [q^{\out}]_1 + [\gamma]_1 [q^{\inn}]_1 \right)/\left( [\tau]_1[w_q]_1^2+[\gamma]_1 \right),\; [w_q]_1[q^{\inn}]_2\right].
  \end{split}
\end{equation}
Finally we define disturbance vectors and constants:
\begin{align}
\wv^\inn_p = \xv_0, \qquad \wv^\out_p = \wv, \qquad &\wv^\inn_q = [(\sv\;\boldsymbol{0})\; \uv_0^{\inn}], \qquad \wv^\out_q = [\sv\; \uv_0^{\out}], \label{eq:disturbance_trans} \\
\alpha^\inn_{pk} = \left(\alpha_{1k},0\right), \qquad \alpha^\inn_{qk} = \left(\alpha_{2k},0\right), \qquad &\alpha^\out_{pk} = \left(\beta_{1k},0\right), \qquad \alpha^\out_{qk} = \left(\beta_{2k},0\right), \label{eq:alpha_trans} \\
\gamma^\inn_{pk}=(\gamma_{1k},1), \qquad \gamma^\inn_{qk} = (\gamma_{2k},1), \qquad &\gamma^\out_{pk} = (\tau_{1k},1), \qquad \gamma^\out_{qk} = (\tau_{2k},1). \label{eq:gamma_trans}
\end{align}
In \eqref{eq:disturbance_trans}, $(\sv\;\boldsymbol{0})$ is the $N$-vector obtained by padding the positive singular values of $\Am$ with $0$'s, and $[\sv\; \uv_0^{\out}]$ is the $M\times 2$ matrix with columns $\sv$ and $\uv_0^{\out}$. Keeping with our matrix notation, $[\wv^{\out}_p]_i$ will denote column $i$ of this matrix, i.e.\ the vector $([\zv_0]_i,[\yv]_i)$. In all cases, where we perform truncation of the scalar quantities in GVAMP, we assume the same truncation here. We claim that these quantities satisfy the recursion in Algorithm \ref{alg:general_gvamp}. We start by establishing this for the scalar terms inductively. For the $\gamma$ terms, if equality has been established for all scalar terms up to the current term, then equality follows for the current term immediately if we define the $\Gamma$ functions as follows:
\[
 \Gamma^\inn_p(\gamma,\alpha) = \Gamma^\out_p(\gamma,\alpha) = \Gamma^\inn_q(\gamma,\alpha) = \Gamma^\out_q(\gamma,\alpha) =  \left(\frac{\gamma_1(1-\alpha_1)}{\alpha_1},1\right).
\]
For the $\alpha$ terms, we first observe that the equality follows immediately for $\alpha^\inn_{pk}$ and $\alpha^\out_{pk}$ given the definitions of $f^\inn_p$ and $f^\out_p$ respectively (assuming equality for all prior scalar terms). It follows directly from the definitions of $f^\inn_q$ and $f^\out_q$ that
\[
[\alpha^\inn_{qk}]_1 = \frac{1}{N}\sum_{i=1}^N \frac{[\gamma_{qk}^\inn]_1}{[\gamma_{qk}^\out]_1[\wv^\inn_q]^2_{i1} + \gamma_{qk}^\inn}_1, \qquad [\alpha_{qk}^\out]_1 = \frac{1}{M} \sum_{i=1}^M\frac{[\wv^\out_q]_{i1}^2[\gamma^\out_{qk}]_1}{[\gamma^\out_{qk}]_1[\wv^\out_q]_{i1}^2 + [\gamma^\inn_{qk}]_1}.
\]
Substituting $(\sv\;\boldsymbol{0})$ for $\wv^\inn_q$ and $\sv$ for $\wv^\out_q$, it follows that these are exactly the divergences of $g_{x2}$ and $g_{z2}$ derived in \cite{VAMP_general}, and hence are equal to $\alpha_{2k}$ and $\beta_{2k}$, respectively.

Observe that lines $1$ and $4$ in Algorithm~\ref{alg:general_gvamp} follow trivially from the definitions in \eqref{eq:general_identity} and that
\begin{align*}
  [\vv_k^{\inn}]_1 = \rv_{2k}&\stackrel{(a)}{=} \frac{ g_{x1}\left( \rv_{1k},\gamma_{1k} \right)-\alpha_{1k}\rv_{1k} }{1-\alpha_{1k}} \stackrel{(b)}{=} \frac{g_{x1}\left( [\pv_k^\inn]_1,[\gamma^\inn_{pk}]_1 \right)-[\alpha^\inn_{pk}]_1\pv^\inn_{k}}{1-[\alpha^{\inn}_{pk}]_1},
\end{align*}
where $(a)$ follows from Algorithm~\ref{alg:GVAMP}.
The proof for $\vv_k^{\out}$ follows in the same way. Next,
\begin{align*}
[\uv_{k}^{\inn}]_1 = \Vm^T \rv_{1k}&\stackrel{(a)}{=} \Vm^T\left[ \Vm \Dm_k\left( \tau_{2k}\Sm^T\Um^T\pv_{2k}+ \gamma_{2k}\Vm^T\rv_{2k}\right)-\alpha_{2k}\rv_{2k} \right]/(1-\alpha_{2k})\\
                   &\stackrel{(b)}{=} \left[ \Dm_k\left( [\gamma^\out_{qk}]_1\Sm^T[\qv_k^{\out}]_1  + \gamma^\inn_{qk} [\qv_k^{\inn}]_1 \right)-[\alpha^\inn_{qk}]_1[\qv_k^{\inn}]_1 \right]/(1-[\alpha^\inn_{qk}]_1)\\
  &\stackrel{(e)}{=} \left[\left[ \fv_q^{\inn}\left( \qv_k^{\inn},\qv_k^{\out},\wv^\inn_q,\gamma^\inn_{qk},\gamma^\out_{qk} \right)-\alpha^\inn_{qk}\qv_k^{\inn} \right]/(1-\alpha^\inn_{qk})\right]_1,
\end{align*}
where $(a)$ follows from Algorithm~\ref{alg:GVAMP} and the definition (15) in \cite{VAMP_general}, $(b)$ follows from the definitions in \eqref{eq:general_identity},
and $(e)$ follows from definition \eqref{eq:translation_2}. The proof for $[\uv_k^{\out}]_1$ follows in the same way using the fact that
$g_{z2}\left( \rv_{2k},\pv_{2k},\tau_{2k},\gamma_{2k} \right) = \Um\Sm\Dm_k\left[ \tau_{2k}\Sm^T\Um^T\pv_{2k}+\gamma_{2k}\Vm^T\rv_{2k} \right].$

Relative to the large system behavior of the iterates, the primary innovation of (G)VAMP can be understood in terms of the fact that its iterates can be generated by this more general recursion.
Utilizing the left and right singular vectors separately allows for the effect of $\Am$ to be broken up within the general algorithm, isolating the effects of the (rectangular diagonal) singular values matrix $\Sm  \in\mathbb{R}^{M\times N}$ and of the orthogonal matrices $\Vm  \in\mathbb{R}^{N\times N}$ and $\Um\in\mathbb{R}^{M\times M}$ in various stages.

Notice that in Theorem \ref{thm:main}, we want to understand the concentration properties of the estimate $\hat{\xv}_{1k}$. But this is just a function of $\rv_{1k}$ and, thus, by the above identities \eqref{eq:general_identity} of $\pv^\inn_k$ and $\wv^\inn_p = \xv^0$. Therefore, once we have concentration for the iterates of the general recursion, we will be able to easily obtain concentration for $\hat{\xv}_{1k}$ in the proof of Theorem 1.

As this translation shows, the effect of the singular values
can be entirely subsumed within the denoiser $f_q$ through the disturbance vectors $\wv^{\inn/\out}_q$. Thus, we do not require any distributional assumptions about
the singular values beyond empirical convergence to some bounded random variable (\textbf{Assumption 0}). In the case of GVAMP, we require only distributional assumptions about $\Um$ and $\Vm$. Moreover, the matrix $\Um$ does not occur in Algorithm \ref{alg:general_vamp}, and it follows that the large system behavior of VAMP can be characterized using only distributional assumptions on $\Vm$.

\subsection{Conditions}
\label{sec:assumptions_general}
To prove our general concentration result, we require a number of conditions on the quantities in Algorithm~\ref{alg:general_gvamp}. Using the translations above, we show that these conditions are implied by \textbf{Assumptions 0}-\textbf{5} that we made about the GVAMP quantities in Section~\ref{sec:main}.

\textbf{Condition 0.} The initial sequences $\uv^{\inn/\out}_0$ are independent of each other and the matrices $\Um$, $\Vm$ and concentrate on random variables $U_0^\inn$ and $U_0^\out$, respectively. This will always occur if the initializations are random samples from subgaussian prior distributions, as is guaranteed by our translation \eqref{eq:general_identity} and \textbf{Assumption 0}.

Furthermore, the initial $\gamma^\inn_{p0}$ and $\gamma^\out_{p0}$ converge to $\overline{\gamma}^\inn_{p0}\geq 0$ and $\overline{\gamma}^\out_{p0}\geq 0$, respectively, as $N\to\infty$. 
From \eqref{eq:general_identity}, we can see that the conditions on the $\gamma^{\inn/\out}_{p0}$ are guaranteed by \textbf{Assumption 0} on the corresponding constants in the GVAMP algorithm.

\textbf{Condition 1.}
For any $\phi\in PL(2)$, assume that the components of the sequences $\wv^{\inn/\out}_p$ and $\wv^{\inn/\out}_q$ are sampled i.i.d. from subgaussian distributions.
Furthermore, we assume that the (components of the) corresponding limiting variables $W^{\inn/\out}_p$ and $W^{\inn/\out}_q$ are independent of $\Vm$ and $\Um$ (thus, independent of the limiting $P_k^{\inn/\out}$ and $Q_k^{\inn/\out}$) and independent of each other.
Under the translation in \eqref{eq:disturbance_trans}, this condition is implied by \textbf{Assumption 1} for $\wv^{\inn}_p$, $\wv^\out_p$, and $\wv^\out_q$. However, we need to verify the desired concentration for $\wv^\inn_q=(\sv\;\boldsymbol{0})$. Let $W^\inn_q$ be a random variable with mixture distribution $\delta \mu_S + (1-\delta)\mu_0$ where $\mu_S$ is the distribution of $S$, $\mu_0$ is the Dirac point measure at $0$, and $\delta = \lim_{N\to\infty}\frac{M}{N}$. Then for any $\phi\in PL(2)$, we have
\begin{align*}
  \frac{1}{N}\sum_{i=1}^N \phi([\wv^\inn_q]_i) = \delta_N\frac{1}{M}\sum_{i=1}^M\phi(\sv_i) + (1-\delta_N)\phi(0)
\end{align*}
Now using our standard concentration results (Lemmas \ref{products} and \ref{sums}) as well as the assumed concentration of $\sv$ and $\delta_N$, we get that the above concentrates around $\delta \mathbb{E}\phi(S) + (1-\delta)\phi(0)=\mathbb{E}\phi(W^\inn_q) $, as needed.

\textbf{Condition 2.} Same as \textbf{Assumption 2}.

\textbf{Condition 3.} The denoisers $f^{\inn/\out}_p$ and $f^{\inn/\out}_q$ are separable, and both these functions and their derivatives are either
\begin{enumerate}
\item uniformly Lipschitz at each $\left(\overline{\gamma}^{\inn}_{pk},\overline{\gamma}^{\out}_{pk}\right)$ and $\left(\overline{\gamma}^{\inn}_{qk},\overline{\gamma}^{\out}_{qk}\right)$, respectively, for $k\geq 1$, or
\item uniformly bounded conditionally Lipschitz (UBCL) at the same parameters for all $k\geq 1$. This condition requires the following (in the notation for the $p$ case, the $q$ case being entirely symmetric):
  \begin{enumerate}
  \item $f^{\inn/\out}_p$ is continuous in $(p^\inn,p^\out,w_p)$ for all parameters $(\gamma^\inn,\gamma^\out)$.
  \item In a neighborhood of $\left(\overline{\gamma}^{\inn}_{pk},\overline{\gamma}^{\out}_{pk}\right)$, $f^{\inn/\out}_p$ is Lipschitz in $(p^{\inn},p^{\out})$ for all $w_p$, with Lipschitz constant continuous in $w_p$.
  \item The domain of $w_p$ is compact.
  \item The Lipschitz uniformity over the parameters is satisfied for all inputs, i.e.
    \[
      \left| f^{\inn/\out}_p\left(p^{\inn},p^{\out},w_p,\gamma_1 \right)-f^{\inn/\out}_p\left(p^{\inn},p^{\out},w_p,\gamma_2 \right) \right|\leq L\left( 1 + \|\left(p^{\inn},p^{\out},w_p\right) \| \right)\left| \gamma_1-\gamma_2 \right|
    \]
    for all $\gamma_1,\gamma_2$ in a neighborhood of $\overline{\gamma}_{pk}$.
  \end{enumerate}
\end{enumerate}
It is easy to see that the $f^{\inn/\out}_q$ given in the translation \eqref{eq:translation_2} are separable by definition and either uniformly Lipschitz or UBCL. In particular, the only functions which are not simply uniformly Lipschitz are the $[f_q^{\inn/\out}]_2 = [w_q]^1q^{\inn/\out}$. However, it is immediately clear that these are UBCL since they are free of $\gamma_{qk}^{\inn/\out}$ and since $w_q^{\inn/\out}$ consists of the singular values which were assumed to lie in a compact interval $[0,S_{\max}]$ in \textbf{Assumption 0}. We note that the $f^{\inn/\out}_p$ are separable and uniformly Lipschitz when the $g_{x1}/g_{z1}$ are, which is guaranteed by \textbf{Assumption 3} above.

\textbf{Condition 4.}  The functions $C^{\inn/\out}_p, C^{\inn/\out}_q, \Gamma^{\inn/\out}_p, \Gamma^{\inn/\out}_q$ are bounded over their domains and Lipschitz continuous with Lipschitz constant independent of $N$ and $k$.
For the functions given in the translation \eqref{eq:translation_2}, this condition is satisfied as long as the domains of the functions are compact (and independent of $N$ and $k$). We note that clipping of the $\gamma_1$ and $\alpha_1$ scalars is guaranteed by \eqref{eq:alpha_trans} and \eqref{eq:gamma_trans} along with  \textbf{Assumption 4}.

\textbf{Condition 5.}  Following \cite{AMP_FS}, we define stopping criteria that determine when the algorithm has effectively converged. First, we stop the iteration if $\gamma^\out_{pk} < \epsilon_1$ or $\gamma^\out_{qk}<\epsilon_2$. In terms Algorithm~\ref{alg:VAMP}, this  is equivalent to stopping the algorithm when the variance of $|R_{1k}-X_0|$ or $|R_{2k}-X_0|$ is sufficiently small. Next, we stop if $\rho^{\inn/\out}_{pk} < \epsilon'_1$ or $\rho^{\inn/\out}_{qk}<\epsilon'_2$, where these $\rho$ quantities are defined in the next section. This condition essentially stops the algorithm if the difference in the (asymptotic distributions of) successive iterates is sufficiently small.

\textbf{Condition 6.} The interval $[t_{\min},t_{\max}]$ within which the $\alpha_1$ scalars are truncated includes the $[\overline{\alpha}^{\inn/\out}_{pk}]_1$ and $[\overline{\alpha}^{\inn/\out}_{qk}]_1$ (defined in the general recursion state evolution \eqref{eq:se_general}) for all iterations $k$ prior to termination.

This condition is satisfiable as long as these $[\overline{\alpha}]_1$ scalars do not converge to $1$ or $0$ as $k\to\infty$. But as long as the second derivatives in \eqref{eq:beta_smooth} are bounded away from zero over any compact subset of their domains, then for either the MAP or MMSE denoisers $g_{x1}$ and $g_{z1}$, the derivatives of the $f_p^{\inn/\out}$ and $f_q^{\inn/\out}$ will be bounded away from $1$ and $0$ over any compact subset of their domains (see e.g.\ Lemma 1 in \cite{RanganDerivatives}). Likewise, if the asymptotic variances $\tau^{\inn/\out}_{pk}$ and $\tau^{\inn/\out}_{qk}$ (also defined in \eqref{eq:se_general}) do not diverge or shrink to $0$ before termination, the $[\overline{\alpha}^{\inn/\out}_{pk}]_1$ and $[\overline{\alpha}^{\inn/\out}_{qk}]_1$ cannot converge to $1$ or $0$ as $k\to\infty$ because the $\gamma$ scalars are clipped to a compact interval.

\subsection{Notation} \label{sec:notation}
In the above translation between the general recursion and GVAMP, we used general iterates of dimension $d=2$ in order to track the GVAMP iterates, the initial data, and the transformed input $\zv_0$ throughout the algorithm. Henceforth, we will analyze the general recursion with $d=1$ in order to simplify notation and proof ideas. However, these arguments can all be extended to the case of general dimension $d$ (see Appendix \ref{app:matrix_case} for details on how the following results generalize to $d>1$). For $k \geq 0$, define matrices $\Um^{\inn/\out}_k \in \mathbb{R}^{N \times (k+1)}$ having columns $\uv^{\inn/\out}_i$ for $0\leq i\leq k$, and define $N \times (k+1)$ matrices $\Vm^{\inn/\out}_k$, $\Pm^{\inn/\out}_k$, and $\Qm^{\inn/\out}_k$ analogously.

Now, for $k \geq 1$,  define $  \Cm^{\inn/\out}_{pk}, \Cm^{\inn/\out}_{uk} \in \mathbb{R}^{N \times (2k+1)}$ to be the matrices containing, respectively, the $\pv^{\inn/\out}, \vv^{\inn/\out}$ iterates of the algorithm up to and including $\pv^{\inn/\out}_k$ and the  $\uv^{\inn/\out}, \qv^{\inn/\out}$ iterates up to and including  $\uv^{\inn/\out}_{k}$. Similarly, let $\Cm^{\inn/\out}_{vk},  \Cm^{\inn/\out}_{qk}  \in \mathbb{R}^{N \times 2k}$ be the matrices containing, respectively, the $\pv^{\inn/\out}, \vv^{\inn/\out}$ iterates   up to and including $\vv^{\inn/\out}_{k-1}$  and the $\uv^{\inn/\out}, \qv^{\inn/\out}$ iterates up to and including  $\qv^{\inn/\out}_{k-1}$. We use the initial conditions $\Cm^{\inn/\out}_{p0} = \Pm^{\inn/\out}_0 = \pv^{\inn/\out}_0, \Cm^{\inn/\out}_{u0} = \Um^{\inn/\out}_0 = \uv^{\inn/\out}_0,$ and $\Cm^{\inn/\out}_{v0} = \Cm^{\inn/\out}_{q0}= \emptyset$.
 In other words, these matrices are defined  as follows:
\begin{align}
  \label{eq:Cmat}
  \Cm^{\inn/\out}_{pk} = [\Pm^{\inn/\out}_k\;  \Vm^{\inn/\out}_{k-1}],\;& \quad \Cm^{\inn/\out}_{uk} = [\Um^{\inn/\out}_k\;\Qm^{\inn/\out}_{k-1}],\nonumber\\
  \Cm^{\inn/\out}_{vk} = [\Pm^{\inn/\out}_{k-1}\;\Vm^{\inn/\out}_{k-1}],\;& \quad \Cm^{\inn/\out}_{qk}= [ \Um^{\inn/\out}_{k-1}\;\Qm^{\inn/\out}_{k-1}],
\end{align}
where for two matrices $\mathbf{M}_1 \in \mathbb{R}^{N \times p_1}$ and $\mathbf{M}_2 \in \mathbb{R}^{N \times p_2}$, by the notation $\mathbf{M} =  [\mathbf{M}_1 \;  \mathbf{M}_2]$, we mean the $N \times (p_1 + p_2)$ matrix obtained from concatenating the columns of $\mathbf{M}_1$ and $\mathbf{M}_2$. 

Next, we define sigma-algebras generated by previous algorithm output at any iteration using the matrices defined in the previous paragraph. We use the initializations $\mathcal{P}^{\inn/\out}_{0} =\sigma\lbrace \wv_p^{\inn/\out},\Um^{\inn/\out}_{0}\rbrace$ and $\mathcal{Q}^{\inn/\out}_0 =\sigma\lbrace \wv_q^{\inn/\out}, \Um^{\inn/\out}_0,\Pm^{\inn/\out}_0,\Vm^{\inn/\out}_0\rbrace$, and for $k \geq 1$ define
\begin{equation}
\begin{split}
 \mathcal{P}^{\inn/\out}_{k} &=\sigma\lbrace \wv_p^{\inn/\out},\Um^{\inn/\out}_{k},\Pm^{\inn/\out}_{k-1},\Vm^{\inn/\out}_{k-1},\Qm^{\inn/\out}_{k-1}\rbrace,\\
\mathcal{Q}^{\inn/\out}_{k} &=\sigma\lbrace \wv_q^{\inn/\out},\Um^{\inn/\out}_{k},\Pm^{\inn/\out}_{k},\Vm^{\inn/\out}_{k},\Qm^{\inn/\out}_{k-1}\rbrace.
\label{eq:sigma_algebras}
\end{split}
\end{equation}
Notice that $ \mathcal{P}^{\inn/\out}_{k} $ includes the iterates of the algorithm up to and including $ \uv^{\inn/\out}_{k}$ (i.e.\ all output after completing the $(k-1)^{th}$ iteration of the algorithm) and  $\mathcal{Q}^{\inn/\out}_{k}$ includes the iterates  up to and including $ \qv^{\inn/\out}_{k}$ (i.e.\ all output through step \eqref{eq:alg_gvamp_v_step} in the $k^{th}$ iteration). We define the combined sigma-algebras $\mathcal{P}_k = \sigma\left\lbrace \mathcal{P}_k^{\inn},\mathcal{P}_k^{\out}\right\rbrace$ and $\mathcal{Q}_k = \sigma\left\lbrace \mathcal{Q}_k^{\inn},\mathcal{Q}_k^{\out}\right\rbrace$.

Finally, for $k \geq 0$, define the matrices $\Bm^\perp_{\Cm_{qk}^{\inn/\out}}, \Bm^\perp_{\Cm_{vk}^{\inn/\out}} \in\mathbb{R}^{N\times(N-2k)}$, and $\Bm^\perp_{\Cm_{uk}^{\inn/\out}}, \Bm^\perp_{\Cm_{pk}^{\inn/\out}} \in\mathbb{R}^{N\times (N-2k-1)}$ where, for example, $\Bm^\perp_{\Cm_{pk}^{\out}}$ has columns that form an orthonormal basis for $\mathrm{span}( \Cm^{\out}_{pk})^\perp$ where $\Cm^{\out}_{pk}$ is defined in \eqref{eq:Cmat}, and the other matrices are defined analogously.
Moreover, for $k \geq 1$, define  the matrix $\Bm_{\Cm^{\inn/\out}_{vk}}\in\mathbb{R}^{N\times 2k}$ to have as columns an orthonormal basis for $\mathrm{span}( \Cm^{\inn/\out}_{vk})$, and for $k\geq 0$, define $\Bm_{\Cm^{\inn/\out}_{uk}}\in\mathbb{R}^{N\times (2k+1)}$ as a matrix having as columns an orthonormal basis for $\mathrm{span}( \Cm^{\inn/\out}_{uk})$.
Then we let $\Om^{\inn/\out}_{p0} =\Bm^\perp_{\Cm_{v0}^{\inn/\out}} \in \mathbb{R}^{n \times n}$ be any (deterministic) orthogonal matrix, and for $k \geq 1$, we define the orthogonal matrices
\be
\Om^{\inn/\out}_{pk} = [\Bm^\perp_{\Cm_{vk}^{\inn/\out}}\;\Bm_{\Cm_{vk}^{\inn/\out}}]\in \mathbb{R}^{N\times N} \quad \text{ and } \quad \Om^{\inn/\out}_{q(k-1)} = [\Bm^\perp_{\Cm_{u(k-1)}^{\inn/\out}}\;\Bm_{\Cm_{u(k-1)}^{\inn/\out}}]\in \mathbb{R}^{N\times N}.
\label{eq:Up_Uq_matrices}
\ee
To wrap up the notation, we state a definition and lemma that will be useful throughout the proof.

\begin{defi}[Isotropic Invariance] \label{def:isotropic_inv}
A vector $\xv \in \mathbb{R}^N$ is \textbf{isotropically invariant} (or, spherically symmetric) if $\Vm_0 \xv \overset{d}{=} \xv$ for any orthogonal  matrix $\Vm_0 \in \mathbb{R}^{N \times N}$. Notice this implies that $\xv/\|\xv\|$ is uniformly distributed on the unit sphere.
\end{defi}

The following lemma about isotropically invariant vectors will be used throughout the proof.

\begin{lem}
\label{lem:isotropic_inv}
$\xv \overset{d}{=} \|\xv\| (\wv_0/\|\wv_0\|)$ where $\wv_0 \sim \mathcal{N}(\mathbf{0}, \mathbb{I}_{N \times N})$ for any isotropically invariant $\xv\in \mathbb{R}^N$.
\end{lem}
\begin{proof}
The result follows from two facts. First, $ \xv/\|\xv\|  \overset{d}{=}\wv_0/\|\wv_0\|$ since both $ \xv/\|\xv\| $ and $\wv_0/\|\wv_0\|$ are uniformly distributed on the unit sphere. Second, $\xv/\|\xv\|$ and $\|\xv\|$ are independent. To see this, notice that $\xv/\|\xv\|$ is uniformly-distributed on the sphere irrespective of the value of $\|\xv\|$; thus, they are independent. Therefore, $\xv= \|\xv\| (\xv/\|\xv\|)  \overset{d}{=} \|\xv\| (\wv_0/\|\wv_0\|)$, giving the desired result.
\end{proof}

\subsection{General Algorithm Concentration}  \label{sec:main_result_general}
Under the conditions given in Section~\ref{sec:assumptions_general}, we state our general concentration result for Algorithm~\ref{alg:general_gvamp} in Lemmas~\ref{lem:cond_dist} - \ref{lem:main_general}. 
To obtain our concentration results, we first introduce deviance terms $\Delm^{\inn/\out}_{pk}$ and $\Delm^{\inn/\out}_{qk}$ quantifying the discrepancy between the finite sample behavior of the iterates $(\pv^{\inn/\out}_1,...,\pv^{\inn/\out}_k)$ and $(\qv^{\inn/\out}_1,...,\qv^{\inn/\out}_k)$ and their limiting values.  This is done in Lemma~\ref{lem:cond_dist} by studying the distributions of the vectors conditional on the previous output of the algorithm, summarized by the sigma-algebras in \eqref{eq:sigma_algebras}.  Before we state and prove Lemma~\ref{lem:cond_dist}, we state a lemma, originally from \cite[Lemmas 4 and 5]{VAMP}, that characterizes the distribution of random Haar matrices, conditional on linear constraints.

\begin{lem}{\cite[Lemmas 4 and 5]{VAMP}}
Let $\Vm \in \mathbb{R}^{N \times N}$ be a random, Haar-distributed matrix (see Definition~\ref{def:Haar}).  Suppose that for deterministic matrices $\textbf{M}_1, \textbf{M}_2 \in \mathbb{R}^{N \times s}$, for some $1 \leq s \leq N$, we know that $\Vm$ satisfies $\textbf{M}_1 = \Vm \textbf{M}_2$. Then, if $\textbf{M}_1$ and $\textbf{M}_2$ are full column rank, 
\[\Vm \Big \lvert_{\{\textbf{M}_1 = \Vm \textbf{M}_2\}} \overset{d}{=} \textbf{M}_1 (\textbf{M}_1^T \textbf{M}_1)^{-1} \textbf{M}_2^T + \Bm^\perp_{\textbf{M}_1} \widetilde{\Vm}  [\Bm^\perp_{\textbf{M}_2}]^T,\]
where $\widetilde{\Vm}$ is Haar distributed, independent of $\Vm$, and  $ \Bm^\perp_{\textbf{M}_1},  \Bm^\perp_{\textbf{M}_2} \in \mathbb{R}^{N \times (N-s)}$ are any matrices whose columns are orthogonal bases for $Range(\mathbf{M}_1)^{\perp}$ and $Range(\mathbf{M}_2)^{\perp}$, respectively.
\label{lem:V_cond}
\end{lem}

To use the above lemma, we notice that conditioning on sigma-algebras $\mathcal{Q}^{\inn/\out}_k$ and $\mathcal{P}^{\inn/\out}_k$, defined in \eqref{eq:sigma_algebras}, is equivalent to conditioning on two linear constraints. As an illustrative example, consider conditioning on the sigma-algebra $\mathcal{Q}^{\inn}_k=\sigma\lbrace \Um^{\inn}_k,\Pm^{\inn}_k,\Vm^{\inn}_k,\Qm^{\inn}_{k-1}\rbrace$, which contains all of the input iterates up until just before step \eqref{eq:alg_gvamp_q_step}. First, notice that the action of $\Vm$  in Algorithm~\ref{alg:general_gvamp} is only in step \eqref{eq:alg_gvamp_p_step}, which reads $\pv^\inn_k\gets \Vm \uv^\inn_k$, and step \eqref{eq:alg_gvamp_q_step}, which reads $\qv^\inn_k \gets \Vm^T\vv^\inn_k$. Thus, conditioning on $\mathcal{Q}^{\inn}_k$ is equivalent to conditioning on the linear constraint $ [\Pm^{\inn}_k\;  \Vm^{\inn}_{k-1}] = \Vm [\Um^\inn_k\;\Qm^\inn_{k-1}]$ or, equivalently, on $\Cm_{pk}^\inn=\Vm\Cm^\inn_{uk}$ for $\Cm^\inn_{pk}$ and $\Cm_{uk}^\inn$ defined in \eqref{eq:Cmat}.
By a similar argument, we can see that conditioning on $\mathcal{P}^{\inn}_{k+1}=\sigma\lbrace \Um^{\inn}_{k+1},\Pm^{\inn}_{k},\Vm^{\inn}_{k},\Qm^{\inn}_{k}\rbrace$ is equivalent to conditioning on $ [\Pm^\inn_k\;  \Vm^\inn_{k}] = \Vm [\Um^\inn_k\;\Qm^\inn_{k}]$ or on $\Cm^\inn_{v(k+1)}=\Vm\Cm^\inn_{q(k+1)}$. Continuing this argument, conditioning on $\mathcal{Q}^{\out}_k$ or $\mathcal{P}^{\out}_{k+1}$ is equivalent to conditioning on $\Cm_{pk}^\out=\Um\Cm^\out_{uk}$ or  $\Cm^\out_{v(k+1)}=\Um\Cm^\out_{q(k+1)}$, respectively.

Now we give the conditional distribution lemma,  Lemma~\ref{lem:cond_dist}.  In the lemma and in what follows, we use the notation $\mathbf{0}_{k}$ to denotes a vector of zeros of length $k$. This lemma and the following, Lemma \ref{lem:joint_dists}, depend on vectors $\betav_{pk}^{\inn/\out},\betav_{qk}^{\inn/\out}\in\mathbb{R}^k$ and constants $\rho_{pk}^{\inn/\out},\rho_{qk}^{\inn/\out}\in\mathbb{R}_{>0}$, which we leave unspecified for the time being. The statements of these lemmas are valid for any choice of these parameters, and in section \ref{sec:limit_defs} we define them as appropriate limits of GVAMP quantities so as to ensure the desired concentration behavior.

\begin{lem}
If $[\Cm^{\inn}_{vk}]^T\Cm^{\inn}_{vk}$ has full rank for $0\leq k\leq K$, then we have for all such $k$ that
\begin{equation}
\pv^{\inn}_0 \lvert_{\mathcal{P}_0} \stackrel{d}{=} \sqrt{\rho^\inn_{p0}} \,\Om^\inn_{p0} \, \overline{\Zv}^\inn_{p0} + \Delm^\inn_{p0}, \qquad \text{ and  } \qquad \pv^\inn_k\lvert_{\mathcal{P}_k} \stackrel{d}{=} \sum_{\ell=0}^{k-1}[\betav^\inn_{pk}]_{\ell +1} \, \pv^\inn_{\ell} + \sqrt{\rho^\inn_{pk}} \, \Om^\inn_{pk} \, \overline{\Zv}^\inn_{pk} + \Delm^\inn_{pk},
\label{eq:p_conds}
\end{equation}
where $\Om^\inn_{pk}$ is defined in \eqref{eq:Up_Uq_matrices} while $\rho^\inn_{pk}$ and $\betav^\inn_{pk}$ are unspecified for the time being, with appropriate values given in what follows in \eqref{eq:rhos}, and 
\begin{align*}
\Delm^\inn_{p0} &= \left(\frac{\|\uv^\inn_0\|}{\|\Zv^\inn_{p0}\|}-\sqrt{\rho^\inn_{p0}}\right)\Bm^\perp_{\Cm^\inn_{v0}}\Zv^\inn_{p0},\\
\Delm^\inn_{pk} &=  \Cm^\inn_{vk}\left(([\Cm^\inn_{qk}]^T\Cm^\inn_{qk})^{-1}[\Cm_{qk}^\inn]^T\uv^\inn_k-\left[\begin{matrix}\betav^\inn_{pk}\\ \boldsymbol{0}_k\end{matrix}\right]\right) +\left[\frac{\|[\Bm^\perp_{\Cm_{qk}^\inn}]^T\uv^\inn_k\|}{\|\Zv^\inn_{pk}\|}-\sqrt{\rho^\inn_{pk}}\right]\Bm^\perp_{\Cm_{vk}^\inn}\Zv^\inn_{pk}- \sqrt{\rho^\inn_{pk}} \, \Bm_{\Cm_{vk}^\inn} \, \breve{\Zv}^\inn_{pk}.
\end{align*}
The matrices $\Cm^\inn_{vk}$ and $\Cm^\inn_{qk}$ are defined in \eqref{eq:Cmat}. In the above, $\overline{\Zv}^\inn_{pk}$ are length-$N$ vectors with independent, standard Gaussian entries that are independent across $0 \leq k \leq K$ and independent of the corresponding conditioning sigma-algebra $\mathcal{P}_k$. These are decomposed as $\overline{\Zv}^\inn_{pk}=[\Zv^\inn_{pk} \, \lvert \, \breve{\Zv}^\inn_{pk}]$ where $\Zv^\inn_{pk}$ is length $N-2k$ and $\breve{\Zv}^\inn_{pk}$ is length $2k$. In particular, this means $\Om^\inn_{pk}\overline{\Zv}^\inn_{pk}=\Bm^\perp_{\Cm_{vk}^\inn}\Zv^\inn_{pk}+\Bm_{\Cm^\inn_{vk}}\breve{\Zv}^\inn_{pk}$.

If $[\Cm^\inn_{pk}]^T\Cm^\inn_{pk}$ has full rank for $0\leq k\leq K$, we have for all such $k$ that
\be
\qv^\inn_0\lvert_{\mathcal{Q}_0} \stackrel{d}{=} \sqrt{\rho^\inn_{q0}} \, \Om^\inn_{q0} \, \overline{\Zv}^{\inn}_{q0}+\Delm^\inn_{q0}, \qquad \text{ and  } \qquad \qv^\inn_k\lvert_{\mathcal{Q}_k} \stackrel{d}{=} \sum_{\ell=0}^{k-1} [\betav^\inn_{qk}]_{\ell +1} \, \qv^\inn_{\ell} + \sqrt{\rho^\inn_{qk}} \, \Om^\inn_{qk} \, \overline{\Zv}^{\inn}_{qk} + \Delm^\inn_{qk},
\label{eq:q_conds}
\ee
where $\Om^\inn_{qk}$ is defined in \eqref{eq:Up_Uq_matrices}  while  $\rho^\inn_{qk}$ and $\betav^\inn_{qk}$  are unspecified for the time being, with appropriate values given in what follows in \eqref{eq:rhos}, and 
\begin{align}
\Delm^\inn_{q0} &= \frac{[\pv^\inn_0]^T\vv^\inn_0}{\|\pv^\inn_0\|^2}\uv^\inn_0 +\left[\frac{\|[\Bm^\perp_{\Cm_{p0}^\inn}]^T\vv^\inn_0\|}{\|\Zv^\inn_{q0}\|}-\sqrt{\rho^\inn_{q0}}\right]\Bm^\perp_{\Cm_{u0}^\inn} \, \Zv^\inn_{q0} - \sqrt{\rho^\inn_{q0}} \, \Bm_{\Cm_{u0}^\inn} \, \breve{\Zv}^\inn_{q0}, \label{eq:Deltaq0} \\
  \Delm^\inn_{qk} &= \Cm^\inn_{uk}\left(([\Cm^\inn_{pk}]^T\Cm^\inn_{pk})^{-1}[\Cm^\inn_{pk}]^T \vv^\inn_k - \left[\begin{matrix} \mathbf{0}_{k+1} \\ \betav^\inn_{qk}\end{matrix}\right]\right)+ \left[\frac{\|[\Bm^\perp_{\Cm_{pk}^\inn}]^T\vv^\inn_k\|}{\|\Zv^\inn_{qk}\|} -\sqrt{\rho^\inn_{qk}}\right]\Bm^\perp_{\Cm_{uk}^\inn}  \, \Zv^\inn_{qk} -  \sqrt{\rho^\inn_{qk}}  \, \Bm_{\Cm_{uk}^\inn} \, \breve{\Zv}^\inn_{qk}.
\label{eq:Deltaqk}
\end{align}
The matrices $\Cm^\inn_{uk}$ and $\Cm^\inn_{pk}$ are defined in \eqref{eq:Cmat}. In the above, $\overline{\Zv}^\inn_{qk} \in\mathbb{R}^N$ are length-$N$ vectors with independent, standard Gaussian entries that are independent across $0 \leq k \leq K$ and independent of the corresponding conditioning sigma-algebra $\mathcal{Q}_k$. As above, these are decomposed as $\overline{\Zv}^\inn_{qk}=[\Zv^\inn_{qk} \, \lvert \, \breve{\Zv}^\inn_{qk}]$ where $\Zv^\inn_{qk}$ is length $N-2k-1$ and $\breve{\Zv}^\inn_{qk}$ is length $2k+1$. Again we have that $\Om^\inn_{qk}\overline{\Zv}^\inn_{qk}=\Bm^\perp_{\Cm_{uk}^\inn}\Zv^\inn_{qk}+\Bm_{\Cm^\inn_{uk}}\breve{\Zv}^\inn_{qk}$. The same statement holds with all input variables replaced by their output equivalents. 

In both cases, the output statements for $ \pv^\out_k\lvert_{\mathcal{P}_k} $ and  $\qv^\out_k\lvert_{\mathcal{Q}_k}$ are analogous to those given in  \eqref{eq:p_conds} and \eqref{eq:q_conds} with $\out$ replacing $\inn$ everywhere.

Furthermore, we have that $\pv_k^\inn$ and $\pv_k^\out$ are conditionally independent given $\mathcal{P}_k$ and that $\qv_k^\inn$ and $\qv_k^\out$ are conditionally independent given $\mathcal{Q}_k$ for all $k\geq 0$.
\label{lem:cond_dist}
\end{lem}

\begin{proof} 
We prove the lemma for $\qv^\inn_0$ and $\qv^\inn_k$, as the base case for $\qv^\inn_0$ is slightly more involved than that of $\pv^\inn_0$. The proof for $\pv^\inn_0$ and $\pv^\inn_k$ is broadly similar, and the proof for the output variables exactly mirrors the proof for the input variables. 

We first prove the statement for $\qv^\inn_0$. Observe from the Algorithm~\ref{alg:general_gvamp} definition, $\qv^\inn_0 = \Vm^T\vv^\inn_0$ and that $\vv^\inn_0\in\mathcal{Q}_0$ since $\mathcal{Q}_0 =\sigma\left\lbrace\mathcal{Q}^\inn_0,\mathcal{Q}^\out_0\right\rbrace= \sigma\lbrace \wv_q^{\inn/\out},\uv^{\inn/\out}_0,\pv^{\inn/\out}_0,\vv^{\inn/\out}_0\rbrace$, as defined in \eqref{eq:sigma_algebras}.  

By the reasoning in the paragraph just before the lemma statement, conditioning on $\mathcal{Q}_k$ is equivalent to conditioning on the linear constraints $\Cm^\inn_{pk}=\Vm\Cm^\inn_{uk}$ and $\Cm^\out_{pk}=\Um\Cm^\out_{uk}$. Thus, 
\ben
\Vm\lvert_{\mathcal{Q}_k} = \Vm\lvert_{\left\lbrace\Cm^\inn_{pk}=\Vm\Cm^\inn_{uk},\Cm^\out_{pk}=\Um\Cm^\out_{uk}\right\rbrace} \stackrel{(a)}{=} \Vm\lvert_{\left\lbrace\Cm^\inn_{pk}=\Vm\Cm^\inn_{uk}\right\rbrace} = \Vm\lvert_{\mathcal{Q}^\inn_k},
\een
where $(a)$ follows as $\Um$ and $\Vm$ are assumed to be independent.
The right-hand side can be handled by Lemma~\ref{lem:V_cond}. In particular, the conditional distribution of $\Vm$ given $\mathcal{Q}^\inn_k$ can be decomposed as
  \begin{equation}
  \label{eq:cond_lem_app}
 \Vm\lvert_{\mathcal{Q}^\inn_k} \stackrel{d}{=} \Cm^\inn_{pk}([\Cm^\inn_{pk}]^T\Cm^\inn_{pk})^{-1}[\Cm^\inn_{uk}]^T+\Bm^\perp_{\Cm_{pk}^\inn}\widetilde{\Vm}[\Bm^\perp_{\Cm_{uk}^\inn}]^T,
\end{equation}
where $\widetilde{\Vm}$ is a random $(N-(2k+1))\times (N-(2k+1))$ matrix that is independent of $\Vm$, but taking the same distribution, in that it is Haar (or uniformly) distributed on the group of orthogonal matrices of its dimensions. 
In what follows, we drop the explicit ``$\inn$'' superscript to save space because throughout we will be working with the input version of the variables.

Thus, for $k=0$, conditioning on $\mathcal{Q}_0 = \sigma\lbrace \wv_q,\uv_0,\pv_0,\vv_0\rbrace$ is equivalent to conditioning on $\Cm_{p0}=\Vm\Cm_{u0}$ (since $\Vm$ and $\wv_q$ are independent), and the result in \eqref{eq:cond_lem_app} gives
\be
 \Vm^T\vv_0\lvert_{\mathcal{Q}_0} \stackrel{d}{=} \Cm_{u0}(\Cm_{p0}^T\Cm_{p0})^{-1}\Cm_{p0}^T\vv_0+\Bm^\perp_{\Cm_{u0}}\widetilde{\Vm}[\Bm^\perp_{\Cm_{p0}}]^T\vv_0,
 \label{eq:q_decomp_1}
\ee
where we have removed the ``$\inn$'' superscript from the vector $\vv_0$ and the matrices $\Cm_{u0}$ and $\Cm_{p0}$.
Now observe that $\Cm_{p0} = \pv_0$ and $\Cm_{u0} = \uv_0$, by definition in \eqref{eq:Cmat}, and therefore,
\be
\Cm_{u0}(\Cm_{p0}^T\Cm_{p0})^{-1}\Cm_{p0}^T\vv_0 = \uv_0 {\|\pv_0\|^{-2}}\pv_0^T\vv_0.
\label{eq:q0_T1}
\ee

Next, as $\pv_0$ and $\vv_0$ are both measurable with respect to $\mathcal{Q}_0$, so is the vector $[\Bm^\perp_{\Cm_{p0}}]^T\vv_0$. We will next argue that the vector $\widetilde{\Vm}[\Bm^\perp_{\Cm_{p0}}]^T\vv_0$ is isotropically invariant as defined in Definition~\ref{def:isotropic_inv}. To see this,  recall that $\widetilde{\Vm}$ is Haar-distributed (see Definition~\ref{def:Haar}), thus $\Vm' \widetilde{\Vm} \overset{d}{=} \widetilde{\Vm}$ for any other orthogonal matrix $\Vm' \in \mathbb{R}^{N-1 \times N-1}$. Hence, $\Vm' \widetilde{\Vm}[\Bm^\perp_{\Cm_{p0}}]^T\vv_0 \overset{d}{=} \widetilde{\Vm}[\Bm^\perp_{\Cm_{p0}}]^T\vv_0$, meaning $\widetilde{\Vm}[\Bm^\perp_{\Cm_{p0}}]^T\vv_0$ is isotropically invariant; therefore, by Lemma~\ref{lem:isotropic_inv},
its distribution is entirely determined by the distribution of its magnitude. In other words, since $\| \widetilde{\Vm} [\Bm^\perp_{\Cm_{p0}}]^T\vv_0\| = \|[\Bm^\perp_{\Cm_{p0}}]^T\vv_0\|$, we have
\be
\widetilde{\Vm}[\Bm^\perp_{\Cm_{p0}}]^T\vv_0 \stackrel{d}{=} \|[\Bm^\perp_{\Cm_{p0}}]^T\vv_0\| \left({\Zv_{q0}}/{\|\Zv_{q0}\|}\right),
\label{eq:q0_T2_firsteq}
\ee
for $\Zv_{q0} \sim \mathcal{N}(\mathbf{0}, \mathbb{I}_{N-1 \times N-1})$, independent of $\mathcal{Q}_0$. 
This gives us that, conditional on $\mathcal{Q}_0$,
\begin{align}
  \Bm^\perp_{\Cm_{u0}}\widetilde{\Vm}[\Bm^\perp_{\Cm_{p0}}]^T\vv_0&\stackrel{d}{=}\frac{\|[\Bm^\perp_{\Cm_{p0}}]^T\vv_0\|}{\|\Zv_{q0}\|} \,\, \Bm^\perp_{\Cm_{u0}} \Zv_{q0}\nonumber\\
  &=\left[\frac{\|[\Bm^\perp_{\Cm_{p0}}]^T\vv_0\|}{\|\Zv_{q0}\|}-\sqrt{\rho_{q0}}\right]\Bm^\perp_{\Cm_{u0}} \, \Zv_{q0}-\sqrt{\rho_{q0}} \, \Bm_{\Cm_{u0}} \, \breve{\Zv}_{q0}+\sqrt{\rho_{q0}} \, \Om_{q0} \, \overline{\Zv}_{q0}.
\label{eq:q0_T2}
\end{align}
The second equality in the above uses that $\Om_{q0} = [\Bm^\perp_{\Cm_{u0}}\;\Bm_{\Cm_{u0}}]\in \mathbb{R}^{N\times N}$ by \eqref{eq:Up_Uq_matrices} for $\Bm_{\Cm_{u0}} \in \mathbb{R}^{N\times 1}$ and $\Bm^\perp_{\Cm_{u0}} \in \mathbb{R}^{N\times (N-1)}$; therefore, $\Bm^\perp_{\Cm_{u0}}\Zv_{q0} = \Om_{q0}\overline{\Zv}_{q0} - \Bm_{\Cm_{u0}}\breve{\Zv}_{q0}$. 

 Plugging result \eqref{eq:q0_T1} and \eqref{eq:q0_T2} into \eqref{eq:q_decomp_1}, we find the result in \eqref{eq:q_conds}-\eqref{eq:Deltaq0}, namely
\ben
 \Vm^T\vv_0\lvert_{\mathcal{Q}_0} \stackrel{d}{=}\left(\frac{\pv_0^T\vv_0}{\|\pv_0\|^2}\right)\uv_0 +\left[\frac{\|[\Bm^\perp_{\Cm_{p0}}]^T\vv_0\|}{\|\Zv_{q0}\|}-\sqrt{\rho_{q0}}\right]\Bm^\perp_{\Cm_{u0}} \,  \Zv_{q0}-\sqrt{\rho_{q0}} \,  \Bm_{\Cm_{u0}} \, \breve{\Zv}_{q0}+\sqrt{\rho_{q0}} \,  \Om_{q0} \, \overline{\Zv}_{q0}.
\een

To establish the statement for $\qv^\inn_k$,  notice that $\qv^\inn_k = \Vm^T\vv^\inn_k$ by Algorithm~\ref{alg:general_gvamp}, so that $\qv^\inn_k\lvert_{\mathcal{Q}_k} = \Vm^T\vv^\inn_k \lvert_{\mathcal{Q}_k}$ (where $\vv^\inn_{k}\in\mathcal{Q}_k$). Recall from the discussion for the $k=0$ case above that $\Vm^T\lvert_{\mathcal{Q}_k}\stackrel{d}{=}\Vm^T\lvert_{\mathcal{Q}_k^\inn}$. Thus, as before, we drop  the explicit ``$\inn$'' superscript, and we can use Lemma \ref{lem:V_cond} to write
\be
\begin{split}
\qv_k\lvert_{\mathcal{Q}_k} &\stackrel{d}{=} \Cm_{uk}(\Cm_{pk}^T\Cm_{pk})^{-1}\Cm_{pk}^T \vv_k + \Bm^\perp_{\Cm_{uk}}\widetilde{\Vm}[\Bm^\perp_{\Cm_{pk}}]^T\vv_k\\
&= \sum_{\ell=0}^{k-1}[ \betav_{qk}]_{\ell+1}\qv_{\ell} + \Cm_{uk}\left((\Cm_{pk}^T\Cm_{pk})^{-1}\Cm_{pk}^T \vv_k - \left[\begin{matrix}  \mathbf{0}_{k+1} \\ \betav_{qk}\end{matrix}\right]\right) + \Bm^\perp_{\Cm_{uk}}\widetilde{\Vm}[\Bm^\perp_{\Cm_{pk}}]^T\vv_k,
\label{eq:qk_T2}
\end{split}
\ee
where we recall that $\widetilde{\Vm}$ is Haar-distributed and independent of the conditioning sigma-algebra. The second equality in \eqref{eq:qk_T2} follows from the fact that $\Cm_{uk} = [\Um_k\;\Qm_{k-1}]$ from \eqref{eq:Cmat}; thus,
\[
\Cm_{uk}\left[\begin{matrix} \mathbf{0}_{k+1} \\ \betav_{qk}\end{matrix}\right] = \Um_k \mathbf{0}_{k+1} + \Qm_{k-1} \betav_{qk}=  \sum_{\ell=0}^{k-1}[ \betav_{qk}]_{\ell+1}\qv_{\ell}.
\]

As above in \eqref{eq:q0_T2_firsteq}-\eqref{eq:q0_T2}, because $\widetilde{\Vm}$ is Haar distributed and because $\Bm^\perp_{\Cm_{uk}}\Zv_{qk} = \Om_{qk}\overline{\Zv}_{qk} - \Bm_{\Cm_{uk}}\breve{\Zv}_{qk}$, we find that conditional on $\mathcal{Q}_k$,
\begin{align*}
\Bm^\perp_{\Cm_{uk}}\widetilde{\Vm}[\Bm^\perp_{\Cm_{pk}}]^T\vv_k &\stackrel{d}{=} \frac{\|[\Bm^\perp_{\Cm_{pk}}]^T\vv_k\|}{\|\Zv_{qk}\|} \,\, \Bm^\perp_{\Cm_{uk}}\,\Zv_{qk}\\
 &=   \left[\frac{\|[\Bm^\perp_{\Cm_{pk}}]^T\vv_k\|}{\|\Zv_{qk}\|} -\sqrt{\rho_{qk}}\right]\Bm^\perp_{\Cm_{uk}} \, \Zv_{qk}  - \sqrt{\rho_{qk}} \, \Bm_{\Cm_{uk}} \, \breve{\Zv}_{qk} + \sqrt{\rho_{qk}}\, \Om_{qk} \, \overline{\Zv}_{qk}.
\end{align*}
Combining with \eqref{eq:qk_T2} and using the definition of $\Delm_{qk}$ in \eqref{eq:Deltaqk}, we find the desired result:
\begin{align*}
\qv_k\lvert_{\mathcal{Q}_k} &\stackrel{d}{=}\sum_{\ell=0}^{k-1} [\betav_{qk}]_{\ell+1}\qv_{\ell} + \Cm_{uk}\left((\Cm_{pk}^T\Cm_{pk})^{-1}\Cm_{pk}^T \vv_k - \left[\begin{matrix} \mathbf{0}_{k+1} \\ \betav_{qk}\end{matrix}\right]\right) + \sqrt{\rho_{qk}} \, \Om_{qk} \, \overline{\Zv}_{qk}  \\
&\hspace{1cm} -\sqrt{\rho_{qk}} \, \Bm_{\Cm_{uk}} \, \breve{\Zv}_{qk} +  \left[\frac{\|[\Bm^\perp_{\Cm_{pk}^\inn}]^T\vv^\inn_k\|}{\|\Zv_{qk}\|} -\sqrt{\rho_{qk}}\right]\Bm^\perp_{\Cm_{uk}}\Zv_{qk}\\
&=\sum_{\ell=0}^{k-1} [\betav_{qk}]_{\ell+1}\qv_{\ell} +  \sqrt{\rho_{qk}} \, \Om_{qk} \, \overline{\Zv}_{qk}  +\Delm_{qk}.
\end{align*}

Finally, we establish the conditional independence of $\qv^\inn_k$ and $\qv^\out_k$ given the $\sigma$-algebra $\mathcal{Q}_k$ for all $k\geq 0$. Notice that conditioning on $\mathcal{Q}_k$ is equivalent to conditioning on $\left\lbrace g(\Vm)=0,h(\Um)=0\right\rbrace$ where $g,h$ are the (deterministic) functions $\mathbf{X}\mapsto \Cm^\inn_{pk}-\mathbf{X}\Cm^\inn_{uk}$ and $\mathbf{X}\mapsto \Cm^\out_{pk}-\mathbf{X}\Cm^\out_{uk}$, respectively. This is just a reformulation of the linear constraints discussed earlier.  Hence,
\begin{equation}
\begin{split}
\label{eq:independence1}
  \P\left( \qv^\inn_k\in A,\qv^\out_k\in B\mid \mathcal{Q}_k \right)&=\P\left( \qv^\inn_k\in A,\qv^\out_k\in B\mid g(\Vm)=0,h(\Um)=0 \right)\\
  &=\P\left( \Vm^T\vv^\inn_k\in A,\Um^T\vv^\out_k\in B\mid g(\Vm)=0,h(\Um)=0 \right).
\end{split}
\end{equation}
In the above, the second equality follows because  $\qv^\inn_k = \Vm^T\vv^\inn_k$ and  $\qv^\out_k = \Um^T\vv^\out_k$ by Algorithm~\ref{alg:general_gvamp} step \eqref{eq:alg_gvamp_q_step}. Now applying the definition of conditional probability,
\begin{equation}
\begin{split}
\label{eq:independence2}
  \P\left( \qv^\inn_k\in A,\qv^\out_k\in B\mid \mathcal{Q}_k \right)&=\frac{\P\left( \Vm^T\vv^\inn_k\in A,\Um^T\vv^\out_k\in B, g(\Vm)=0,h(\Um)=0 \right)}{\P\left( g(\Vm)=0,h(\Um)=0 \right)}\\
                                                                      &=\frac{\P\left( \Vm^T\vv^\inn_k\in A,g(\Vm)=0\right)}{\P\left( g(\Vm)=0\right)}\times \frac{\P\left( \Um^T\vv^\out_k\in B, h(\Um)=0 \right)}{\P\left(h(\Um)=0 \right)}\\
                                                                      &= \P\left( \qv^\inn_k\in A\mid g(\Vm) = 0 \right)\P\left( \qv^\out_k\in B\mid h(\Um)=0 \right).
\end{split}
\end{equation}
The second equality of \eqref{eq:independence2} follows as $\Um$ and $\Vm$ are independent and $\vv^{\inn/\out}_k$ is measurable with respect to $ \mathcal{Q}^{\inn/\out}_k$ (and because conditioning on $\left\lbrace g(\Vm)=0\right\rbrace$ or $\mathcal{Q}^{\inn}_k$  is equivalent; similarly for $\left\lbrace h(\Um)=0\right\rbrace$ and $\mathcal{Q}^{\out}_k$).
This establishes conditional independence of $\qv^\inn_k$ and $\qv^\out_k$ given $\mathcal{Q}_k$ as
\begin{align*}
&\P\left( \qv^\inn_k\in A\mid g(\Vm) = 0 \right)\P\left( \qv^\out_k\in B\mid h(\Um)=0 \right) \\
& \qquad = \P\left( \qv^\inn_k\in A\mid g(\Vm) = 0,h(\Um)=0 \right)\P\left( \qv^\out_k\in B\mid g(\Vm)=0,h(\Um)=0 \right)\\
&\qquad  = \P\left( \qv^\inn_k\in A\mid \mathcal{Q}_k \right)\P\left( \qv^\out_k\in B\mid\mathcal{Q}_k \right).
\end{align*}
\end{proof}

Next we extend this lemma to a result characterizing the discrepancy between the finite sample iterates and the limiting behavior of the $\pv$ and $\qv$ iterates jointly. To do this, we introduce a new recursion in terms of Gaussian vectors, $\overline{\Zv}^{\inn/\out}_p$ and $\overline{\Zv}^{\inn/\out}_q$ defined in Lemma \ref{lem:cond_dist}. As we see in Lemma \ref{lem:joint_dists}, the iterates generated by this recursion,  referred to as the equivalent Gaussian recursion, are jointly equal in distribution to the iterates in the general recursion, Algorithm \ref{alg:general_gvamp}.
\begin{algorithm}[H]
  \setstretch{1.3}
\caption{Equivalent Gaussian Representation of General Recursion\label{alg:gaussian}}
\begin{algorithmic}[1]
  \Require{Separable denoisers $\fv^{\inn}_p$, $\fv_p^{\out}$, $\fv^{\inn}_q$, and $\fv_q^{\out}$;
    parameter update functions $\Gamma_i:\mathbb{R}^2\to\mathbb{R}$; initial data $\widetilde{\uv}_0^{\inn}\in\mathbb{R}^N$ and $\widetilde{\uv}_0^{\out}\in\mathbb{R}^M$; and disturbance vectors $\wv^{\inn}_p, \wv^{\inn}_q\in\mathbb{R}^N$ and $\wv^{\out}_p, \wv^{\out}_q\in\mathbb{R}^M$.}
\State  
$\widetilde{\pv}^{\inn}_k\gets\sum_{r=0}^k  [\cv_{pk}^\inn]_r \, \left( \sqrt{\rho^\inn_{pr}} \, \widetilde{\Om}_{pr}^\inn \, \overline{\Zv}^\inn_{pr} + \widetilde{\Delm}^\inn_{pr}\right),$ \hspace{6mm} $\widetilde{\pv}^{\out}_k\gets\sum_{r=0}^k  [\cv_{pk}^\out]_r \, \left(\sqrt{\rho^\out_{pr}} \, \widetilde{\Om}_{pr}^\out \, \overline{\Zv}^\out_{pr} + \widetilde{\Delm}^\out_{pr}\right),$   \label{gaussian1}
\State \begin{varwidth}[t]{\linewidth}
  $\widetilde{\vv}_k^{\inn}\gets \frac{1}{1-\alpha^\inn_{pk}}\left[ \fv_p^{\inn}\left( \widetilde{\pv}_k^{\inn},\widetilde{\pv}_k^{\out},\wv^\inn_p,\gamma_{pk}^\out,\gamma_{pk}^\inn \right)-\alpha_{pk}^\inn\widetilde{\pv}_k^{\inn} \right]$,
  \par \qquad \qquad  $\widetilde{\vv}_k^{\out}\gets \frac{1}{1-\alpha^\out_{pk}}\left[ \fv_p^{\out}\left( \widetilde{\pv}_k^{\inn},\widetilde{\pv}_k^{\out},\wv^\out_p,\gamma_{pk}^\out,\gamma_{pk}^\inn \right)-\alpha_{pk}^\out\widetilde{\pv}_k^{\out} \right],$
  \end{varwidth}
\State 
  $\widetilde{\qv}^{\inn}_k\gets\sum_{r=0}^k  [\cv_{qk}^\inn]_r \, \left(\sqrt{\rho^\inn_{qr}} \, \widetilde{\Om}_{qr}^\inn \, \overline{\Zv}^\inn_{qr} + \widetilde{\Delm}^\inn_{qr}\right),$
\hspace{7mm}  $\widetilde{\qv}^{\out}_k\gets\sum_{r=0}^k  [\cv_{qk}^\out]_r \, \left(\sqrt{\rho^\out_{qr}} \, \widetilde{\Om}_{qr}^\out \, \overline{\Zv}^\out_{qr} + \widetilde{\Delm}^\out_{qr}\right),$
\State \begin{varwidth}[t]{\linewidth}
  $\widetilde{\uv}_{k+1}^{\inn}\gets \frac{1}{1-\alpha^\inn_{qk}}\left[ \fv_q^{\inn}\left( \widetilde{\qv}_k^{\inn},\widetilde{\qv}_k^{\out},\wv^\inn_q,\gamma_{qk}^\out,\gamma_{qk}^\inn \right)-\alpha_{qk}^\inn\widetilde{\qv}_k^{\inn} \right]$, \par \qquad  \qquad 
  $\widetilde{\uv}_{k+1}^{\out}\gets \frac{1}{1-\alpha^\out_{qk}}\left[ \fv_q^{\out}\left( \widetilde{\qv}_k^{\inn},\widetilde{\qv}_k^{\out},\wv^\out_q,\gamma_{qk}^\out,\gamma_{qk}^\inn \right)-\alpha_{qk}^\out\widetilde{\qv}_k^{\out} \right].$
  \end{varwidth}
\end{algorithmic}
\end{algorithm}

In the above, all relevant scalar quantities (e.g.\ $\alpha^{\inn/\out}_{pk}$) are understood to be defined analogously as in the general recursion, Algorithm \ref{alg:general_gvamp}, but now in terms of the Gaussian equivalent iterates (e.g.\ $\widetilde{\pv}^\inn_k$ and $\widetilde{\pv}^\out_k$). Likewise, the discrepancy terms $\widetilde{\Delm}^{\inn/\out}_{pk}$ and $\widetilde{\Delm}^{\inn/\out}_{qk}$ are defined as those in Lemma \ref{lem:cond_dist}, but using the Gaussian equivalent iterates. Constant vectors $\cv_{pk}^{\inn/\out}$ and $\cv_{qk}^{\inn/\out}$ are defined explicitly \eqref{eq:lemma_c_def} below in terms of the $\betav_{pk}^{\inn/\out}$ and $\betav_{qk}^{\inn/\out}$ vectors. These vectors, along with $\rho_{pr}^{\inn/\out}$ and $\rho_{qk}^{\inn/\out}$ remain unspecified, as the statement of Lemma \ref{lem:joint_dists} is valid for any choice of these constants. 

The Gaussian equivalent recursion is initialized in the same way as the GVAMP recursion, so in particular $\widetilde{\uv}^{\inn/\out}_0=\uv^{\inn/\out}_0$, and thus $\widetilde{\Delm}^{\inn/\out}_{p0}=\Delm^{\inn/\out}_{p0}$. Furthermore, $\widetilde{\Om}_{p0}^{\inn/\out}$ and $\Om^{\inn/\out}_{p0}$ are deterministic orthogonal matrices, so we choose them to be equal as well.

In order to state Lemma \ref{lem:joint_dists}, we first introduce some new notation. For all $j\geq 0$, we will let
\begin{equation}
\rv_j = \left( \pv^\inn_j,\pv^\out_j,\qv^\inn_j,\qv^\out_j \right),\hspace{1cm}\widetilde{\rv}_j = \left( \widetilde{\pv}^\inn_j,\widetilde{\pv}^\out_j,\widetilde{\qv}^\inn_j,\widetilde{\qv}^\out_j \right),
\label{eq:r_def}
\end{equation}
where the terms in $\rv_j $ are those in Algorithm~\ref{alg:general_gvamp}, while those in $\widetilde{\rv}_j$ are from Algorithm~\ref{alg:gaussian}.
We also define the concatenation of these up through iteration $k$ as
\begin{equation}
\label{eq:under_r}
\underline{\rv}_k = \left( \rv_0,\ldots,\rv_k \right), \hspace{1cm}\underline{\widetilde{\rv}}_k = \left(\widetilde{\rv}_0,\ldots,\widetilde{\rv}_k  \right).
\end{equation}

\begin{lem}
\label{eq:lemma_res0}
Let $\{\overline{\Zv}_{pk}^{\inn}\}_{k \geq 0}$ and $\{\overline{\Zv}_{pk}^{\out}\}_{k \geq 0}$ be the i.i.d.\ standard Gaussian vector sequences defined in Lemma \ref{lem:cond_dist}. Then, with respect to the Gaussian equivalent recursion in Algorithm~\ref{alg:gaussian}, define
\begin{equation}
\begin{split}
\ide{\pv}^{\inn}_0 = \sqrt{\rho_{p0}^{\inn}} \, \widetilde{\Om}_{p0}^\inn \, \overline{\Zv}^{\inn}_{p0}, &\qquad \text{ and } \qquad  \ide{\pv}^\inn_k = \sum_{r=0}^{k-1}[\betav^\inn_{pk}]_{r+1} \ide{\pv}_r^\inn + \sqrt{\rho^\inn_{pk}}\, \widetilde{\Om}^\inn_{pk} \, \overline{\Zv}^\inn_{pk}, \\
\ide{\pv}^{\out}_0 = \sqrt{\rho_{p0}^{\out}} \, \widetilde{\Om}^\out_{p0} \, \overline{\Zv}^{\out}_{p0}, & \qquad \text{ and } \qquad  \ide{\pv}_k^\out = \sum_{r=0}^{k-1}[\betav^\out_{pk}]_{r+1} \ide{\pv}_r^\out + \sqrt{\rho_{pk}^\out}\, \widetilde{\Om}^\out_{pk}\, \overline{\Zv}^\out_{pk}.
\label{eq:lemma_res0}
\end{split}
\end{equation}
Similarly define $ \ide{\qv}^\inn_k$ and $\ide{\qv}^\out_k$. With these definitions, we have
\begin{enumerate}
\item 
\begin{equation}
\ide{\pv}_k^\inn = \sum_{r=0}^k \sqrt{\rho^\inn_{pr}} \, [\cv^\inn_{pk}]_r \, \widetilde{\Om}^\inn_{pr} \, \overline{\Zv}^\inn_{pr},\qquad \ide{\pv}_k^\out = \sum_{r=0}^k \sqrt{\rho^\out_{pr}} \, [\cv^\out_{pk}]_r \, \widetilde{\Om}^\out_{pr} \, \overline{\Zv}^\out_{pr},
\label{eq:lemma_res1}
\end{equation}
where for $k\geq 1$ and $0 \leq r\leq k-1$, $[\cv^\inn_{pk}]_r$ and $[\cv^\out_{pk}]_r$ are defined recursively as
\begin{equation}
[\cv^\inn_{pk}]_r = \sum_{i=r}^{k-1}[\cv^\inn_{pi}]_r[\betav^\inn_{pk}]_{i+1}, \qquad \text{ and } \qquad [\cv^\out_{pk}]_r = \sum_{i=r}^{k-1}[\cv^\out_{pi}]_r[\betav^\out_{pk}]_{i+1},
\label{eq:lemma_c_def}
\end{equation}
with $[\cv^\inn_{pk}]_k=[\cv^\out_{pk}]_k = 1$ for all $k\geq 0$. 
\item
  Let $\underline{\ide{\rv}}_k = \left( \ide{\rv}_0,\ldots,\ide{\rv}_k \right)$ where $\ide{\rv}_j = (\ide{\pv}^{\inn}_j, \ide{\pv}^{\out}_j, \ide{\qv}^{\inn}_j, \ide{\qv}^{\out}_j)$ are defined as above for the ideal variables $\ide{\pv}$ and $\ide{\qv}$. Then we have that
  \begin{equation}
    \label{eq:lemma_res2}
    \widetilde{\underline{\rv}}_k - \underline{\ide{\rv}}_k = \left( \dv_0,\ldots,\dv_k \right),
  \end{equation}
 where $ \underline{\widetilde{\rv}}_k$ is defined in \eqref{eq:under_r} with
  \be
  \label{eq:dv_vec}
    \dv_k = \left(\sum_{r=0}^k[\cv^\inn_{pk}]_r\widetilde{\Delm}^\inn_{pr}, \, \sum_{r=0}^k[\cv^\out_{pk}]_r\widetilde{\Delm}^\out_{pr}, \, \sum_{r=0}^k[\cv^\inn_{qk}]_r\widetilde{\Delm}^\inn_{qr}, \, \sum_{r=0}^k[\cv^\out_{qk}]_r\widetilde{\Delm}^\out_{qr} \right).
  \ee
  Furthermore, for all $k\geq 0$, we have that, conditional on $\left( \wv_p^{\inn/\out},\wv_q^{\inn/\out} \right)$,
  \begin{equation}
    \widetilde{\underline{\rv}}_k \stackrel{d}{=} \underline{\rv}_k.
    \label{eq:lemma_eq_dist}
  \end{equation}
\item For all $i\geq 1$, we have $\left( [\ide{\pv}_0^\inn]_i,\ldots,[\ide{\pv}_k^\inn]_i,[\ide{\pv}_0^\out]_i,\ldots,[\ide{\pv}_k^\out]_i \right) \stackrel{d}{=}\left( P^\inn_0,\ldots,P^\inn_k ,P_0^\out,\ldots,P_k^\out\right)$ where $(P_0^\inn,\ldots,P_k^\inn)$ and $(P_0^{\out},\ldots,P_k^\out)$ are independent, zero-mean, jointly Gaussian vectors (and hence the right hand side is itself jointly Gaussian). 
\end{enumerate}
Analogous statements for (1.)-(3.) hold for the $\ide{\qv}$ variables.
\label{lem:joint_dists}
\end{lem}

Recall that we introduced the Gaussian equivalent recursion in Algorithm~\ref{alg:gaussian} as an idealized `Gaussian' version of the GVAMP general recursion in Algorithm~\ref{alg:general_gvamp}. Lemma~\ref{lem:joint_dists} makes this precise. Indeed, 
Lemma~\ref{lem:joint_dists} introduces a recursion in \eqref{eq:lemma_res0} that is shown to be purely   Gaussian in point (3.), in that its iterates are elementwise equal in distribution to Gaussian vectors who covariance structure can be given exactly.
We see from point (2.) of the lemma, that the difference between the terms in Algorithm~\ref{alg:gaussian} (given in     $\widetilde{\underline{\rv}}_k$) and the purely Gaussian terms from  \eqref{eq:lemma_res0}  (collected in $\underline{\ide{\rv}}_k$) is only in the deviation terms (see \eqref{eq:dv_vec}). Moreover, the terms of  Algorithm~\ref{alg:gaussian} (given in   $\widetilde{\underline{\rv}}_k$) are equal in distribution to those from Algorithm~\ref{alg:general_gvamp} (given in  $\underline{\rv}_k$). Taken together, this shows that the difference between the GVAMP general recursion and the purely Gaussian recursion can be summarized via the deviation terms. The major technical part of the proof, given in Lemma~\ref{lem:main_general} is in showing that these deviation terms concentrate exponentially fast to zero, so that the dynamics of the GVAMP general recursion can be characterized by studying the purely Gaussian equivalent.

Before proving Lemma~\ref{lem:joint_dists}, we note that the Gaussian vectors used to characterize the dynamics of the purely Gaussian recursion in point (3.), namely $(P_0^{\inn/out},\ldots,P_k^{\inn/\out})$ and $(Q_0^{\inn/\out},\ldots,Q_k^{\inn/\out})$, may be explicitly defined in terms of the $\overline{\Zv}$ variables defined in Lemma~\ref{lem:cond_dist} as follows. For $k\geq 1$,
\begin{equation}
  \begin{split}
    P^{\inn/\out}_k = \sum_{r=0}^{k-1}[\betav_{pk}]_{r+1}P_r^{\inn/\out} + \sqrt{\rho_{pk}^{\inn/\out}}[\overline{\Zv}_{pk}^{\inn/\out}]_1, \quad Q^{\inn/\out}_k = \sum_{r=0}^{k-1}[\betav_{qk}]_{r+1}Q_r^{\inn/\out} + \sqrt{\rho_{qk}^{\inn/\out}}[\overline{\Zv}_{qk}^{\inn/\out}]_1,
  \end{split}
  \label{eq:PQ_def}
\end{equation}
with initializations $P_0^{\inn/\out}=\sqrt{\rho_{p0}^{\inn/\out}}[\overline{\Zv}_{p0}^{\inn/\out}]_1$ and $Q_0^{\inn/\out}=\sqrt{\rho_{q0}^{\inn/\out}}[\overline{\Zv}_{q0}^{\inn/\out}]_1$.

\begin{proof}
We prove the results for $\pv^\inn_k$; the arguments for $\qv^\inn_k$ and the output iterates are completely analogous.
To establish the first part, namely the result of \eqref{eq:lemma_res1}, we proceed by induction. First consider the base case $k=0$. We have $[c^\inn_{p0}]_0=1$ by definition \eqref{eq:lemma_c_def}, and thus the statement in \eqref{eq:lemma_res1} reduces to $\ide\pv_0^\inn=\sqrt{\rho^\inn_{p0}} \, \widetilde{\Om}^\inn_{p0} \, \overline{\Zv}^\inn_{p0}$,  which is just the definition of $\ide\pv_0^\inn$ in \eqref{eq:lemma_res0}. 

Now assume the statement in \eqref{eq:lemma_res1} holds for $\ide\pv_j^\inn$ for $0\leq j\leq k-1$. Then, combining the definition of $\ide\pv_k^\inn$ in \eqref{eq:lemma_res0} with the induction hypothesis, we get
\begin{align*}
  \ide\pv^\inn_k &=  \sum_{r=0}^{k-1}[\betav^\inn_{pk}]_{r+1} \ide{\pv}_r^\inn + \sqrt{\rho^\inn_{pk}}\, \widetilde{\Om}^\inn_{pk} \, \overline{\Zv}^\inn_{pk} \\
  &= \sum_{r=0}^{k-1}[\betav^\inn_{pk}]_{r+1}\left[ \sum_{j=0}^r \sqrt{\rho^\inn_{pj}} \, [\cv^\inn_{pr}]_j \, \widetilde{\Om}^\inn_{pj} \, \overline{\Zv}^\inn_{pj}\right] + \sqrt{\rho^\inn_{pk}}\, \widetilde{\Om}^\inn_{pk} \,\overline{\Zv}^\inn_{pk}\\
          &= \sum_{j=0}^{k-1}\left[\sum_{r=j}^{k-1}[\betav^\inn_{pk}]_{r+1} \, [\cv^\inn_{pr}]_j\right]\sqrt{\rho^\inn_{pj}} \, \widetilde{\Om}^\inn_{pj} \, \overline{\Zv}^\inn_{pj}+\sqrt{\rho^\inn_{pk}}\, \widetilde{\Om}^\inn_{pk} \, \overline{\Zv}^\inn_{pk} =\sum_{j=0}^k \, [\cv^\inn_{pk}]_j \, \sqrt{\rho^\inn_{pj}} \, \widetilde{\Om}^\inn_{pj} \, \overline{\Zv}^\inn_{pj}.
\end{align*}
as needed, where the final equality uses the definition of $\cv^\inn_{pk}$ in \eqref{eq:lemma_c_def}.

For the second part, we observe that \eqref{eq:lemma_res2} becomes trivially true in light of the first result where we showed, for example, that $\ide{\pv}_k^\inn = \sum_{r=0}^k \sqrt{\rho^\inn_{pr}} \, [\cv^\inn_{pk}]_r \, \widetilde{\Om}^\inn_{pr} \, \overline{\Zv}^\inn_{pr}$, and the Gaussian Algorithm \ref{alg:gaussian} step \eqref{gaussian1} definition, for which $\widetilde{\pv}^{\inn}_k = \sum_{r=0}^k \sqrt{\rho^\inn_{pr}} \, [\cv_{pk}^\inn]_r \, \widetilde{\Om}_{pr}^\inn \, \overline{\Zv}^\inn_{pr} + \sum_{r=0}^k[\cv^\inn_{pk}]_r\widetilde{\Delm}^\inn_{pr}$. 

To establish the distributional equalities in  \eqref{eq:lemma_eq_dist} of part (3.), we make repeated use of two facts: 
\begin{fact}
\label{fact:conds}
Let $\xv,\widetilde{\xv}\in\mathbb{R}^{n_1}$ and $\yv,\widetilde{\yv}\in\mathbb{R}^{n_2}$ be random vectors. Then we have:
\begin{enumerate}[label=(\Alph*)]
\item If $\yv\mid_{\xv} \stackrel{d}{=} \widetilde{\yv}$, then $(\xv,\yv)\stackrel{d}{=}(\xv,\widetilde{\yv})$.
\item If $\yv \stackrel{d}{=} \widetilde{\yv}$ and $\xv$ is independent of $\yv$ and $\widetilde{\yv}$, then $f(\xv,\yv) \stackrel{d}{=} f(\xv,\widetilde{\yv})$ for any measurable $f$.
\end{enumerate}
\end{fact}

Now we can show $\underline{\rv}_k \stackrel{d}{=} \widetilde{\underline{\rv}}_k$ given in \eqref{eq:lemma_eq_dist} inductively. For the initialization step, we show $\underline{\rv}_0=\rv_0\stackrel{d}{=}\widetilde{\rv}_0=\widetilde{\underline{\rv}}_0$ where by \eqref{eq:r_def} we have that $\rv_0 = \left( \pv^\inn_0,\pv^\out_0,\qv^\inn_0,\qv^\out_0 \right)$ and $\widetilde{\rv}_0 = \left( \widetilde{\pv}^\inn_0,\widetilde{\pv}^\out_0,\widetilde{\qv}^\inn_0,\widetilde{\qv}^\out_0 \right).$ We prove the distributional equality for the $ \pv^\inn$ and $\pv^\out$ terms, while the proof for $ \qv^\inn$ and $\qv^\out$ follows similarly.
By \textbf{Condition 0} and the definition of $\widetilde{\pv}_0^{\inn/\out}$, we have that
\[\pv^{\inn}_0 \lvert_{\mathcal{P}_0} \stackrel{d}{=} \sqrt{\rho^{\inn}_{p0}}  \, \overline{\Zv}^{\inn}_{p0} + \Delm^{\inn}_{p0} \qquad \text{ and } \qquad \widetilde{\pv}^{\inn}_0\lvert_{\mathcal{P}_0} = \sqrt{\rho^\inn_{p0}} \, [\cv_{p0}^\inn]_0 \, \widetilde{\Om}_{p0}^\inn \, \overline{\Zv}^\inn_{p0} + [\cv^\inn_{p0}]_0\widetilde{\Delm}^\inn_{p0}.\]
To see the equivalence of the two terms above, notice that $[\cv^{\inn}_{p0}]_0 = 1$ and $\widetilde{\Om}_{p0}^{\inn / \out} = \Om_{p0}^{\inn / \out} = \mathbf{I}$ by definition, and $\Delm^\inn_{p0}=\widetilde{\Delm}^{\inn}_{p0}$ since the GVAMP recursion and general recursion use identical initialization $\uv_0^\inn$. Furthermore, observe that $\pv_0^{\inn}$ and $\pv_0^{\out}$ are independent since $\Um$ and $\Vm$ are independent, and $\widetilde{\pv}_0^\inn$ and $\widetilde{\pv}^\out_0$ are independent since $\overline{\Zv}^\inn_{p0}$ and $\overline{\Zv}^\out_{p0}$ are independent (by definition), and $\uv_0^{\inn}$ and $\uv_0^{\out}$ are independent by \textbf{Condition 0}.

Next, by definition, conditioning on $\mathcal{Q}_0$ is equivalent to conditioning on $\sigma\left( \pv_0^\inn,\pv_0^\out \right)$ (since $\vv_0$ is a deterministic function of $\pv_0$ conditional on the disturbance vectors; see Algorithm~\ref{alg:general_gvamp}). So Lemma~\ref{lem:cond_dist} gives that
$\qv_0^{\inn / \out}\mid_{\pv_0^\inn,\pv_0^\out} = \sqrt{\rho_{q0}^{\inn /\out}}\Om_{q0}^{\inn / \out}\overline{\Zv}_{q0}^{\inn / \out} + \Delm^{\inn / \out}_{q0}.$ Furthermore, $\qv_0^\inn$ and $\qv_0^\out$ are again conditionally independent given $\mathcal{Q}_0$ since $\Vm$ and $\Um$ are independent, so the above holds for $(\qv_0^{\inn},\qv_0^{\out})$ jointly. Together, this implies
\[
\left( \pv_0^{\inn},\pv_0^{\out},\qv_0^{\inn},\qv_0^{\out} \right) \stackrel{d}{=} \left(\pv_0^{\inn},\pv_0^{\out},\sqrt{\rho_{q0}^{\inn}}\Om_{q0}^{\inn}\overline{\Zv}_{q0}^{\inn} + \Delm^{\inn}_{q0},\sqrt{\rho_{q0}^{\out}}\Om_{q0}^{\out}\overline{\Zv}_{q0}^{\out} + \Delm^{\out}_{q0}  \right),
\]
by (A) in Fact~\ref{fact:conds}. Now note that $\Om_{q0}^{\inn/\out}$ and $\Delm_{q0}^{\inn/\out}$ are deterministic functions of $(\pv_0^{\inn},\pv_0^{\out},\overline{\Zv}_{q0}^{\inn},\overline{\Zv}_{q0}^{\out})$, and thus the entire above vector is also a function of these variables. Furthermore, by Lemma 3, $(\pv_0^{\inn},\pv_0^{\out})$ and $(\overline{\Zv}_{q0}^{\inn},\overline{\Zv}_{q0}^{\out})$ are independent. Likewise, $(\widetilde{\pv}_0^{\inn},\widetilde{\pv}_0^{\out})$ and $(\overline{\Zv}_{q0}^{\inn},\overline{\Zv}_{q0}^{\out})$ are independent since the former only depends on $\overline{\Zv}^{\inn}_{p0}$ and $\overline{\Zv}^{\out}_{p0}$. Thus, equality in distribution is maintained if we substitute $(\widetilde{\pv}_0^{\inn},\widetilde{\pv}_0^{\out})$ for $(\pv_0^{\inn},\pv_0^{\out})$ in the right side of the above by (B)  in Fact~\ref{fact:conds}. But this is just $(\widetilde{\pv}_0^{\inn},\widetilde{\pv}_0^{\out},\widetilde{\qv}_0^{\inn},\widetilde{\qv}_0^{\out})$ by the definitions in the Gaussian Algorithm \ref{alg:gaussian} and the fact that $[\cv_{q0}^{\inn/\out}]_0 = 1$. This completes the proof of the base case.

For the inductive step, assume that $\widetilde{\underline{\rv}}_{k-1} \stackrel{d}{=} \underline{\rv}_{k-1}$ for some $k\geq 2$.
First we will establish that
$\left( \underline{\rv}_{k-1},\pv_k^\inn,\pv_k^\out \right) \stackrel{d}{=}  \left( \widetilde{\underline{\rv}}_{k-1},\widetilde{\pv}^\inn_k,\widetilde{\pv}^\out_k \right).$
To do this, observe that
\begin{align}
 & \left( \underline{\rv}_{k-1},\pv_k^\inn,\pv_k^\out \right) \notag \\
 &\stackrel{d}{=} \left( \underline{\rv}_{k-1},\sum_{\ell=0}^{k-1}[\betav^\inn_{pk}]_{\ell +1} \, \pv^\inn_{\ell} + \sqrt{\rho^\inn_{pk}} \, \Om^\inn_{pk} \, \overline{\Zv}^\inn_{pk} + \Delm^\inn_{pk},\sum_{\ell=0}^{k-1}[\betav^\out_{pk}]_{\ell +1} \, \pv^\out_{\ell} + \sqrt{\rho^\out_{pk}} \, \Om^\out_{pk} \, \overline{\Zv}^\out_{pk} + \Delm^\out_{pk} \right)\label{eq:apply_cond}\\
&\stackrel{d}{=} \Bigg( \widetilde{\underline{\rv}}_{k-1},\sum_{\ell=0}^{k-1}[\betav^\inn_{pk}]_{\ell +1} \, \widetilde{\pv}^\inn_{\ell} + \sqrt{\rho^\inn_{pk}} \, \widetilde{\Om}^\inn_{pk} \, \overline{\Zv}^\inn_{pk} + \widetilde{\Delm}^\inn_{pk},\sum_{\ell=0}^{k-1}[\betav^\out_{pk}]_{\ell +1} \, \widetilde{\pv}^\out_{\ell} + \sqrt{\rho^\out_{pk}} \, \widetilde{\Om}^\out_{pk} \, \overline{\Zv}^\out_{pk} + \widetilde{\Delm}^\out_{pk} \Bigg), \label{eq:apply_induction}
\end{align}
where \eqref{eq:apply_cond} follows from (A)  in Fact~\ref{fact:conds} (using that conditioning on $\underline{\rv}_{k-1}$ is equivalent to conditioning on $\mathcal{P}_{k-1}$ because the additional $\uv$ and $\vv$ terms that generate the latter are deterministic functions of $\underline{\rv}_{k-1}$) and the conclusions of Lemma~\ref{lem:cond_dist} (particularly, that $\pv_k^\inn$ and $\pv_k^\out$ are independent given ${\underline{\rv}_{k-1}}$), and where \eqref{eq:apply_induction} follows from substituting $(\widetilde{\pv}^{\inn/\out}_0,\ldots,\widetilde{\pv}^{\inn/\out}_{k-1})$ for $(\pv^{\inn/\out}_0,\ldots,\pv^{\inn/\out}_{k-1})$. This substitution is justified by (B)  in Fact~\ref{fact:conds} and the induction hypothesis as $\left( \overline{\Zv}_{pk}^\inn, \overline{\Zv}_{pk}^\out \right)$ is independent of $(\pv^{\inn/\out}_0,\ldots, \pv^{\inn/\out}_{k-1})\subset \mathcal{P}_{k-1}$ by construction, and independent of $\widetilde{\underline{\rv}}_{k-1}$ as $\widetilde{\underline{\rv}}_{k-1}$ only depends on
\[
  \left( \overline{\Zv}^{\inn/\out}_{p0},\ldots \overline{\Zv}^{\inn/\out}_{p(k-1)},\overline{\Zv}^{\inn/\out}_{q0},\ldots,\overline{\Zv}^{\inn/\out}_{q(k-1)} \right).
 \]

Next observe that
\begin{align}
  \sum_{\ell=0}^{k-1}&[\betav^\inn_{pk}]_{\ell +1} \, \widetilde{\pv}^\inn_{\ell} + \sqrt{\rho^\inn_{pk}} \, \widetilde{\Om}^\inn_{pk} \, \overline{\Zv}^\inn_{pk} + \widetilde{\Delm}^\inn_{pk}\nonumber\\
  &= \sum_{\ell=0}^{k-1}[\betav^\inn_{pk}]_{\ell +1} \, \left[ \sum_{j=0}^\ell[\cv^\inn_{p\ell}]_j \left(\sqrt{\rho_{pj}^\inn} \, \widetilde{\Om}^\inn_{pj}\, \overline{\Zv}^\inn_{pj} + \widetilde{\Delm}^\inn_{pj}]  \right)\right] + \sqrt{\rho^\inn_{pk}} \, \widetilde{\Om}^\inn_{pk} \, \overline{\Zv}^\inn_{pk} + \widetilde{\Delm}^\inn_{pk}\label{eq:plugin_GER}\\
                     &= \sum_{j=0}^{k-1}\left[ \sum_{\ell=j}^{k-1}\left[ \betav^\inn_{pk} \right]_{\ell+1}[\cv^\inn_{p\ell}]_j \right]\left( \sqrt{\rho^\inn_{pj}}\widetilde{\Om}^\inn_{pj}\overline{\Zv}_{pj}^\inn+\widetilde{\Delm}^\inn_{pj}\right)+\sqrt{\rho^\inn_{pk}} \, \widetilde{\Om}^\inn_{pk} \, \overline{\Zv}^\inn_{pk} + \widetilde{\Delm}^\inn_{pk}\nonumber\\
                     &= \sum_{j=0}^{k}[\cv^\inn_{pk}]_j\sqrt{\rho_{pj}^\inn}\widetilde{\Om}^\inn_{pj}\overline{\Zv}_{pj}+ \sum_{j=0}^k[\cv^\inn_{pk}]_j\widetilde{\Delm}^\inn_{pj}\label{eq:apply_constants}\\
                     &= \widetilde{\pv}_k^\inn, \label{eq:apply_GER_def}
\end{align}
where \eqref{eq:plugin_GER} follows from the definitions of $\widetilde{\pv}_j^\inn$ in Algorithm \ref{alg:gaussian}, \eqref{eq:apply_constants} follows from the definitions of the constants $[\cv^\inn_{pk}]_j$, and \eqref{eq:apply_GER_def} again follows from the definition of $\widetilde{\pv}_k^\inn$ in Algorithm \ref{alg:gaussian}. The same reasoning applies to $\widetilde{\pv}_k^\out$.

Thus, as claimed, plugging these into the above gives 
$\left( \underline{\rv}_{k-1},\pv_k^\inn,\pv_k^\out \right)\stackrel{d}{=} \left( \widetilde{\underline{\rv}}_{k-1},\widetilde{\pv}^\inn_k,\widetilde{\pv}^\out_k \right).$ Next we need to extend this to the $\left( \qv_k^\inn,\qv_k^\out \right)$ variables. The proof of this works in exactly the same way as above, conditioning now on $(\underline{\rv}_{k-1},\pv_k^\inn,\pv_k^\out)$ and applying Lemma 2 with (A) and (B) above to get the proper joint distributional equality.

For the third part, we claim
$\left( \widetilde{\Om}^\inn_{p0}\overline{\Zv}_{p0}^\inn,\widetilde{\Om}^\inn_{q0}\overline{\Zv}_{q0}^\inn,\ldots, \widetilde{\Om}^\inn_{pk}\overline{\Zv}_{pk}^\inn,\widetilde{\Om}^\inn_{qk}\overline{\Zv}_{qk}^\inn \right)\stackrel{d}{=}\left( \overline{\Zv}_{p0}^\inn,\overline{\Zv}_{q0}^\inn,\ldots, \overline{\Zv}_{pk}^\inn,\overline{\Zv}_{qk}^\inn \right).$
As a first step, we show that the two above vectors generate the same $\sigma$-algebras. For this it suffices to show that the right is a function of the left. (The reverse is immediate since the $\widetilde{\Om}$ matrices (which are equal to their $\Om$ equivalents defined in  \eqref{eq:Up_Uq_matrices}) are functions of the $\{\widetilde{\pv}_i,\widetilde{\qv}_i\}_{i=1}^{k-1}$, and these are functions of the right hand side.) We can prove this inductively. For the first component, this is obvious since $\widetilde{\Om}^\inn_{p0}=\Om^\inn_{p0}$ is a deterministic orthogonal matrix. Now suppose the statement is true for all components up to $\widetilde{\Om}^\inn_{(p/q)r}\overline{\Zv}_{(p/q)r}^\inn$. Since $\widetilde{\Om}^\inn_{(p/q)r}$ is a function of the $\overline{\Zv}^\inn_{(p/q)j}$ for $j<r$, it follows by the induction hypothesis that $[\widetilde{\Om}^\inn_{(p/q)r}]^{-1}$ is a function of the prior $\widetilde{\Om}^\inn_{(p/q)j}\overline{\Zv}_{(p/q)j}^\inn$ components. But then $\overline{\Zv}_{(p/q)r}^\inn = [\widetilde{\Om}^\inn_{(p/q)r}]^{-1}\widetilde{\Om}^\inn_{(p/q)r}\overline{\Zv}_{(p/q)r}^\inn $ can be written as a function of the components of the left hand vector up through $\widetilde{\Om}^\inn_{(p/q)r}\overline{\Zv}_{(p/q)r}^\inn$, as claimed.

Now we can prove the asserted distributional equality inductively using (A)  in Fact~\ref{fact:conds} by establishing distributional equality for each component conditional on the prior components. For the base case, again we know that $\widetilde{\Om}^\inn_{p0}$ is a deterministic orthogonal matrix, so by standard properties of Gaussian random vectors we have that $\widetilde{\Om}^\inn_{p0}\overline{\Zv}_{p0}^\inn \stackrel{d}{=}\overline{\Zv}_{p0}$. Next suppose that the statement is true up through $\widetilde{\Om}^\inn_{(p/q)(r-1)}\overline{\Zv}_{(p/q)(r-1)}^\inn$. Then conditioning on
\[
  \left( \widetilde{\Om}^\inn_{p0}\overline{\Zv}_{p0}^\inn,\widetilde{\Om}^\inn_{q0}\overline{\Zv}_{q0}^\inn,\ldots, \widetilde{\Om}^\inn_{p(r-1)}\overline{\Zv}_{p(r-1)}^\inn,\widetilde{\Om}^\inn_{q(r-1)}\overline{\Zv}_{q(r-1)}^\inn \right),
\]
is equivalent to conditioning on
\begin{equation}
\label{eq:new_Z_vec}
\left( \overline{\Zv}_{p0}^\inn,\overline{\Zv}_{q0}^\inn,\ldots, \overline{\Zv}_{p(r-1)}^\inn,\overline{\Zv}_{q(r-1)}^\inn \right),
\end{equation}
by the induction hypothesis. But, conditional on the above, $\widetilde{\Om}^\inn_{(p/q)r}$ is a deterministic orthogonal matrix (as it depends only on $\{\widetilde{\pv}_i,\widetilde{\qv}_i\}_{i=1}^{r-1}$, which are functions of \eqref{eq:new_Z_vec}). Together, this gives that $\widetilde{\Om}^\inn_{pr}\overline{\Zv}_{pr}^\inn\mid_{\{\widetilde{\Om}^\inn_{pi}\overline{\Zv}_{pi}^\inn\}_{i=0}^{r-1}} \stackrel{d}{=} \widetilde{\Om}^\inn_{pr}\overline{\Zv}_{pr}^\inn\mid_{\{\overline{\Zv}_{pi}^\inn\}_{i=0}^{r-1}} \stackrel{d}{=} \overline{\Zv}^\inn_{pr}.$  Putting this together, we find
\begin{align*}
  \left( \widetilde{\Om}^\inn_{p0}\overline{\Zv}_{p0}^\inn,\widetilde{\Om}^\inn_{q0}\overline{\Zv}_{q0}^\inn,\ldots, \widetilde{\Om}^\inn_{pr}\overline{\Zv}_{pr}^\inn\right) &\stackrel{d}{=} \left( \widetilde{\Om}^\inn_{p0}\overline{\Zv}_{p0}^\inn,\widetilde{\Om}^\inn_{q0}\overline{\Zv}_{q0}^\inn,\ldots,\overline{\Zv}_{pr}^\inn\right) \stackrel{d}= \left( \overline{\Zv}_{p0}^\inn,\overline{\Zv}_{q0}^\inn,\ldots,\overline{\Zv}_{pr}^\inn \right), 
\end{align*}
where the first equality follows from the above and (A) in Fact~\ref{fact:conds} and the second because $\overline{\Zv}^\inn_{pr}$ is independent of $\left( \overline{\Zv}_{p0}^\inn,\overline{\Zv}_{q0}^\inn,\ldots,\overline{\Zv}_{q(r-1)}^\inn,\overline{\Zv}_{p(r-1)}^\inn\right)$ using (B)  in Fact~\ref{fact:conds}. The same distributional equality follows in the same way for the vectors of output variables.

Now we can proceed with the proof of the statement in part (iii). For the base case of $k=0$, we have $\ide{\pv}_0^{\inn}=\sqrt{\rho^\inn_{p0}} \, \widetilde{\Om}^\inn_{p0} \, \overline{\Zv}_{p0}^\inn$ and $\ide{\pv}_0^{\out}=\sqrt{\rho^\out_{p0}} \,\widetilde{\Om}^\out_{p0} \, \overline{\Zv}_{p0}^\out.$ Because the $\widetilde{\Om}^{\inn/\out}_{p0}=\Om^{\inn/\out}_{p0}$ are deterministic orthogonal matrices, we have that
\[
  \ide{\pv}_0^\inn = \sqrt{\rho^\inn_{p0}} \, \widetilde{\Om}^\inn_{p0} \, \overline{\Zv}_{p0}^\inn  = \sqrt{\rho^\inn_{p0}} \,\Om^\inn_{p0} \, \overline{\Zv}_{p0}^\inn \stackrel{d}{=} \sqrt{\rho_{p0}^{\inn}} \, \overline{\Zv}_{p0}^\inn\stackrel{i.i.d.}{\sim}N(0,\sqrt{\rho^\inn_{p0}}),
\]
and similarly for the output case. The independence of $\ide{\pv}_0^{\inn}$ and $\ide{\pv}_0^{\out}$ follows because $\overline{\Zv}_{p0}^\inn$ and $\overline{\Zv}_{p0}^\out$ are independent by definition.

To prove the induction step, assume the statement holds for independent Gaussian vectors $(P^\inn_0,\ldots,P^\inn_{k-1})$ and $(P^\out_0,\ldots,P^\out_{k-1})$ (whose values are given in \eqref{eq:PQ_def}). We first notice that the induction hypothesis along with the definition of $\ide{\pv}_k^{\inn/\out}$ in \eqref{eq:lemma_res0} and the above implies that
\begin{align*}
  &\left([\ide{\pv}_0^\inn]_i,[\ide{\pv}_0^\out]_i,\ldots, [\ide{\pv}_k^\inn]_i,[\ide{\pv}_k^\out]_i\right)\\
  &= \left([\ide{\pv}_0^\inn]_i,[\ide{\pv}_0^\out]_i,\ldots,\sum_{r=0}^{k-1}[\betav^\inn_{pk}]_r[\ide{\pv}_r^\inn]_i + \sqrt{\rho^\inn_{pk}} \, [\widetilde{\Om}^\inn_{pk} \, \overline{\Zv}^\inn_{pk}]_i,\sum_{r=0}^{k-1}[\betav^\out_{pk}]_r[\ide{\pv}_r^\out]_i + \sqrt{\rho^\out_{pk}} \, [\widetilde{\Om}^\out_{pk} \, \overline{\Zv}^\out_{pk}]_i\right)\\
  &\stackrel{d}{=} \left( P_0^\inn,P_0^\out, \ldots,\sum_{r=0}^{k-1}[\betav_{pk}]_rP^\inn_r + \sqrt{\rho^\inn_{pk}}  \,  [\widetilde{\Om}^\inn_{pk} \, \overline{\Zv}^\inn_{pk}]_i,\sum_{r=0}^{k-1}[\betav_{pk}]_rP^\out_r + \sqrt{\rho^\out_{pk}}  \,  [\widetilde{\Om}^\out_{pk} \, \overline{\Zv}^\out_{pk}]_i\right)\\
  &\stackrel{d}{=} \left( P_0^\inn,P_0^\out, \ldots,\sum_{r=0}^{k-1}[\betav_{pk}]_rP^\inn_r + \sqrt{\rho^\inn_{pk}}  \,  [\overline{\Zv}^\inn_{pk}]_i,\sum_{r=0}^{k-1}[\betav_{pk}]_rP^\out_r + \sqrt{\rho^\out_{pk}}  \,  [\overline{\Zv}^\out_{pk}]_i\right),
\end{align*}
where, on the right-hand side, we treat $\overline{\Zv}^{\inn/\out}_{pk}$ as independent of the $P^{\inn/\out}_0,\ldots,P^{\inn/\out}_{k-1}$ (as the $\ide{\pv}^{\inn/\out}_i$ are independent of $\overline{\Zv}^{\inn/\out}_{pk}$ for $i<k$, which follows directly from the definition of the ideal variables $\ide{\pv}^\inn_i$). In particular, this shows that the last expression is jointly Gaussian (e.g., by Cramer-Wold). The zero-mean property and the independence of the input and output variables follows immediately from the definition of the $\ide{\pv}$ variables. Thus this completes the proof of the third part.

\end{proof}

\subsection{General State Evolution}
\label{sec:generalSE}
Next we define a scalar recursion that is the analog of the GVAMP state evolution in \eqref{eq:SE}, but for the general algorithm in Algorithm~\ref{alg:general_gvamp}.
One consequence of our general concentration result, Lemma \ref{lem:main_general}, is that the iterates $\pv_k^{\inn/\out}$ and $\qv_k^{\inn/\out}$ converge empirically to the zero-mean Gaussian variables $P_k^{\inn/\out}$ and $Q_k^{\inn/\out}$, respectively, with both defined in $\eqref{eq:PQ_def}$. The variances of these Gaussian variables are completely characterized by the state evolution equations in \eqref{eq:se_general}.  The general algorithm state evolution is given by the following recursive system for $k\geq 0$:
\begin{equation}
\begin{split}
  \label{eq:se_general}
  \overline{\alpha}^{\inn/\out}_{pk} &= \mathbb{E}\left\{\left[f^{\inn/\out}_p\right]'\left(P_k^\inn,P_k^\out,W^p,\overline{\gamma}^\inn_{pk},\overline{\gamma}^\out_{pk} \right)\right\}, \hspace{15mm}  \overline{\gamma}^{\inn/\out}_{qk} = \Gamma^{\inn/\out}_{pk}\left( \overline{\gamma}^{\inn/\out}_{pk},\overline{\alpha}^{\inn/\out}_{pk} \right), \\
    \tau^{\inn/\out}_{qk} &= \left( 1-\overline{\alpha}^{\inn/\out}_{pk} \right)^{-2}\left[ \mathbb{E}\left\lbrace \left[f^{\inn/\out}_p\right]^2\left( P_k^\inn,P_k^\out,W^p,\overline{\gamma}^\inn_{pk},\overline{\gamma}^\out_{pk} \right) - \left[\overline{\alpha}^{\inn/\out}_{pk}\right]^2\tau_{pk}^{\inn/\out}\right\rbrace \right],\\
   \overline{\alpha}^{\inn/\out}_{qk} &= \mathbb{E} \left\{\left[f^{\inn/\out}_q\right]'\left(Q_k^\inn,Q_k^\out,W^q,\overline{\gamma}^\inn_{qk},\overline{\gamma}^\out_{qk} \right)\right\}, \hspace{15mm}   \overline{\gamma}^{\inn/\out}_{p(k+1)} = \Gamma^{\inn/\out}_{qk}\left( \overline{\gamma}^{\inn/\out}_{qk},\overline{\alpha}^{\inn/\out}_{qk} \right), \\
  \tau^{\inn/\out}_{p(k+1)} &= \left( 1-\overline{\alpha}^{\inn/\out}_{qk} \right)^{-2}\left[ \mathbb{E}\left\lbrace \left[f^{\inn/\out}_q\right]^2\left( Q_k^\inn,Q_k^\out,W^q,\overline{\gamma}^\inn_{qk},\overline{\gamma}^\out_{qk} \right) - \left[\overline{\alpha}^{\inn/\out}_{qk}\right]^2\tau_{qk}^{\inn/\out}\right\rbrace \right].
\end{split}
\end{equation}

The initial terms of this recursion $\overline{\gamma}^{\inn/\out}_{p0}$ and the random variable $W^p$ are well-defined by \textbf{Condition 0}, and the terms $\tau^{\inn/\out}_0$ are initialized with $\rho_{p0}^{\inn/\out}$, which we have thought of as a free parameter of Algorithm \ref{alg:gaussian} and which we will define explicitly along with the other limiting quantities in the next section. In terms of these state evolution quantities, we define two more sequences of random variables. For $k\geq 0$:
\be
\begin{split}
  U^{\inn/\out}_{k+1} = \left( 1-\overline{\alpha}^{\inn/\out}_{qk} \right)^{-1}\left[ f^{\inn/\out}_q\left( Q_k^\inn,Q_k^\out,W^q,\overline{\gamma}^{\inn}_{qk},\overline{\gamma}^\out_{qk} \right)-\overline{\alpha}^{\inn/\out}_{qk}Q^{\inn/\out}_{k} \right],\\
V^{\inn/\out}_k = \left( 1-\overline{\alpha}^{\inn/\out}_{pk} \right)^{-1}\left[ f^{\inn/\out}_p\left( P_k^\inn,P_k^\out,W^p,\overline{\gamma}^{\inn}_{pk},\overline{\gamma}^\out_{pk} \right)-\overline{\alpha}^{\inn/\out}_{pk}P^{\inn/\out}_{k} \right],
\end{split}
\label{eq:UV_def}
\ee
with $Q^{\inn/\out}_k$ and $P^{\inn/\out}_k$ defined above in \eqref{eq:PQ_def}, and where $U_0$ is the random variable determined by \textbf{Condition 0} to which the initial condition $\uv_0$ concentrates.

Lemma \ref{lem:main_general}  part $(d)$ (which we prove in what follows) implies that ${\|\pv^{\inn/\out}_k\|^2}/{N}$ concentrates on $\mathbb{E}\{[P^{\inn/\out}_k]^2\}$ and part $(e)$ implies that ${\|\uv^{\inn/\out}_k\|^2}/{N}$ concentrates to $\mathbb{E}\{[U^{\inn/\out}_k]^2\}$.

Following \cite{VAMP}, we can use these results to see that
\be
\begin{split}
  \mathbb{E}\{[P^{\inn/\out}_k]^2\} = \lim_{N\to\infty}\frac{\|\pv^{\inn/\out}_k\|^2}{N}&=\lim_{N\to\infty}\frac{\|\uv^{\inn/\out}_k\|^2}{N}\\
  &= \frac{\mathbb{E}\Big[ \Big( f^{\inn/\out}_q\left( Q_k^\inn,Q_k^\out,W^q,\overline{\gamma}^{\inn}_{qk},\overline{\gamma}^\out_{qk} \right)-\overline{\alpha}^{\inn/\out}Q^{\inn/\out}_{k} \Big)^2 \Big]}{(1-\overline{\alpha}^{\inn/\out}_{qk})^{2}},
  \label{eq:tau_var_argument}
\end{split}
\ee
where the second equality in the above follows because $\pv^{\inn}_k = \Vm \uv^{\inn}_k$, by definition in Algorithm~\ref{alg:general_gvamp}; hence, $\|\pv^{\inn/\out}_k\|^2 = (\uv^{\inn/\out}_k)^T \Vm^T \Vm \uv^{\inn/\out}_k = \|\uv^{\inn/\out}_k\|^2$.
Now expanding terms and applying Stein's lemma along with the state evolution definitions in \eqref{eq:se_general} shows that indeed $\mathbb{E}\{[P^{\inn/\out}_k]^2\} = \tau^{\inn/\out}_{pk}$. The same reasoning can be applied to show that $\mathbb{E}\{[Q^{\inn/\out}_k]^2\} = \tau^{\inn/\out}_{qk}$.

We can also show, assuming again the result of Lemma \ref{lem:main_general}, that the GVAMP state evolution \eqref{eq:SE} is equivalent to the general state evolution \eqref{eq:se_general} under the translation defined in \eqref{eq:general_identity} - \eqref{eq:gamma_trans}. The equivalences $\overline{\alpha}_{1k}=\overline{\alpha}^\inn_{pk}$ and $\overline{\beta}_{1k}=\overline{\alpha}^\out_{pk}$ follow directly from the definitions of the sensitivity functions $A_{x1}$ and $A_{z1}$ and the definitions in \eqref{eq:translation_2}. The equivalences $\sigma^2_{2k}=\tau^\inn_{qk}$ and $\rho^2_{2k}=\tau^\out_{qk}$ follow similarly. Then $\overline{\gamma}_{2k} = \overline{\gamma}^\inn_{qk}$ and $\overline{\tau}_{2k}=\overline{\gamma}^\out_{qk}$ follows directly from the definition of the $\Gamma^{\inn/\out}_{pk}$ above.

The equivalences $\overline{\alpha}_{2k}=\overline{\alpha}^\inn_{qk}$ and $\overline{\beta}_{2k}=\overline{\alpha}^\out_{qk}$ follow from as the definitions of the sensitivity functions $A_{x2}$ and $A_{z2}$ are the limits of the averages of the derivatives of $f_q^\inn$ and $f_q^\out$ respectively (again by the definitions in \eqref{eq:translation_2}), and these averages converge to $\overline{\alpha}^{\inn/\out}_{qk}$ since we have empirical convergence of the arguments. (Indeed, in this case the derivatives do not depend on the $\qv^{\inn/\out}_k$, so we only need empirical convergence of the $\wv^{\inn/\out}_q$, which is guaranteed by \textbf{Condition 0}.)

Then, the equivalences $\sigma^2_{1k}=\tau^\inn_{pk}$ and $\rho^2_{1k}=\tau^\out_{pk}$ are established by the definitions \eqref{eq:translation_2}, the uniformly Lipschitz property of $g_{x2}$ and $g_{z2}$, and the fact that the i.i.d.\ sequences $\rv_2$ and $\pv_2$ in the definition of the error functions converge empirically to the proper normal variables. Finally, the equivalences $\overline{\gamma}_{qk}=\overline{\gamma}^\inn_{pk}$ and $\overline{\tau}_{1k}=\overline{\gamma}^\out_{pk}$ follow from the definitions of the $\Gamma^{\inn/\out}_{qk}$ functions above.

\subsection{Definitions of Limiting Quantities}\label{sec:limit_defs}

The statements of lemmas \ref{lem:cond_dist} and \ref{lem:joint_dists} are valid for any choice of  parameters $\betav_{pk}^{\inn/\out}$, $\betav_{qk}^{\inn/\out}$, $\rho_{pk}^{\inn/\out}$, and $\rho_{qk}^{\inn/\out}$, and the exact values of these parameters have yet to be  specified.  To make our analysis work, we choose these parameters to be the limiting values of the quantities that occur in the discrepancy  terms $\Delm$ defined in Lemma \ref{lem:cond_dist}, so that these discrepancy terms concentrate around zero.

To formally define these parameters, we first introduce come intermediate quantities: for $k \geq 0$,
\be
\begin{split}
  &\bv^{\inn/\out}_{uk} = (\mathbb{E}[U^{\inn/\out}_0U^{\inn/\out}_k],\ldots,\mathbb{E}[U^{\inn/\out}_{k-1}U^{\inn/\out}_k]),\\
  &\bv^{\inn/\out}_{vk} = (\mathbb{E}[V^{\inn/\out}_0V^{\inn/\out}_k],\ldots,\mathbb{E}[V^{\inn/\out}_{k-1} V^{\inn/\out}_k]),
\label{eq:bvecs}
\end{split}
\ee
and define matrices $\covm^{\inn/\out}_{uk},\covm^{\inn/\out}_{vk}\in\mathbb{R}^{k\times k}$ elementwise as
\begin{equation}
\begin{split}
  \left[\covm^{\inn/\out}_{pk}\right]_{ij} = \E\left[ U^{\inn/\out}_i U^{\inn/\out}_j \right],& \qquad \qquad \left[\covm^{\inn/\out}_{qk}\right]_{ij} = \E\left[ V^{\inn/\out}_i V^{\inn/\out}_j \right],
\label{eq:cov_def}
\end{split}
\end{equation}
where the $U$ and $V$ variables are defined in $\eqref{eq:UV_def}$. In terms of these quantities, we define our parameters as follows: for $k \geq 1$
\be
\begin{split}
\betav^{\inn/\out}_{pk}&=[\covm^{\inn/\out}_{u(k-1)}]^{-1}\bv^{\inn/\out}_{uk}, \quad  \quad \rho^{\inn/\out}_{pk} = \mathbb{E}\{[U^{\inn/\out}_k]^2\} - (\bv^{\inn/\out}_{uk})^T[\covm^{\inn/\out}_{u(k-1)}]^{-1}\bv^{\inn/\out}_{uk}, \\
\betav^{\inn/\out}_{qk}&=[\covm^{\inn/\out}_{v(k-1)}]^{-1}\bv^{\inn/\out}_{vk}, \quad \quad \rho^{\inn/\out}_{qk} = \mathbb{E}\{[V^{\inn/\out}_k]^2\} - (\bv^{\inn/\out}_{vk})^T[\covm^{\inn/\out}_{v(k-1)}]^{-1}\bv^{\inn/\out}_{vk}.
\label{eq:rhos}
\end{split}
\ee
These parameters are initialized as $\betav^{\inn/\out}_{p0}=\betav^{\inn/\out}_{q0}=0$, $\rho^{\inn/\out}_{p0}$ given by \textbf{Condition 0}, and $\rho^{\inn/\out}_{q0}=\mathbb{E}\left\lbrace [V^{\inn/\out}_0]^2\right\rbrace$.

These definitions may appear circular at first, as the definitions of the $P$ and $Q$ variables in Lemma \ref{lem:joint_dists} depend on these parameters, which are then defined in terms of moments involving the $P$ and $Q$. However, it is easy to check that the definitions can be made recursively, avoiding any circular dependence. For instance $U_0^{\inn/\out}$ is determined by the initial condition $\uv_0^{\inn/\out}$. Since $\betav_{p0}^{\inn/\out}=0$ and $\rho_{p0}^{\inn/\out}$ is also determined by the initial conditions, $P_0$ is well-defined in terms of $U_0$. Then it is easy to check from the definition and state evolution equations that $V_0$ is well-defined in terms of $P_0$, and $Q_0$ is well-defined in terms of $V_0$. Continuing in this fashion allows us to recursively build up the sequences of $P$, $Q$, $U$, and $V$ variables in terms of the previously defined variables.

Finally, we define limiting precision matrices. Let $\presm^{\inn/\out}_{u0}=[\covm^{\inn/\out}_{u0}]^{-1}$ and for $k\geq 0$:
\begin{align}
\label{eq:Rk_Sk}
  \presm^{\inn/\out}_{uk} &= \left[\begin{matrix}[\covm^{\inn/\out}_{uk}]^{-1} & 0\\ 0 & [\covm^{\inn/\out}_{v(k-1)}]^{-1}\end{matrix}\right], \qquad \presm^{\inn/\out}_{vk} = \left[\begin{matrix}[\covm^{\inn/\out}_{uk}]^{-1} & 0\\ 0 & [\covm^{\inn/\out}_{vk}]^{-1}\end{matrix}\right],
\end{align}
\begin{lem}\label{lem:main_general}
Throughout the lemma, we use $c$ and $C$ to denote constants that do not depend on the iteration number, but which may vary in their exact value between usages. For $t\geq 0$, let
\begin{equation}
\begin{split}
\label{eq:Cdef}
C_k = C^{2k} (k!)^{16}, \quad \qquad & c_k = \frac{1}{c^{2k} (k!)^{26}}, \quad \qquad C'_k = C^{k} (k+1)^{8}, \quad \qquad  c'_k = \frac{c_k}{c (k+1)^{15}}.
\end{split}
\end{equation}

\begin{enumerate}[label=(\alph*)]
\item 
\begin{align*}
\P\left(\frac{1}{N}\|\Delm^{\inn/\out}_{qk}\|^2\geq\epsilon\right)\leq Ck^2C'_{k-1}e^{-cc'_{k-1} n\epsilon/k^4}, \qquad \P\left(\frac{1}{N}\|\Delm^{\inn/\out}_{pk}\|^2\geq\epsilon\right)\leq Ck^2 C_{k-1}e^{-cc_{k-1} n\epsilon/k^4}.
\end{align*}
\item %
For $\phi:\mathbb{R}^{k+2}\to\mathbb{R}$ pseudo-Lipschitz, we have
\begin{align*}
  \P\Big( \Big| \frac{1}{N}\sum_{i=1}^N \phi&\left([\qv^{\inn}_0]_i,...,[\qv^{\inn}_{k+1}]_i,[\qv^{\out}_0]_i,...,[\qv^{\out}_{k+1}]_i,[\wv^{\inn}_q]_i,[\wv^{\out}_q]_i\right)\nonumber\\
  &-\mathbb{E}\left\{\phi\left(Q^{\inn}_0,..., Q^{\inn}_{k+1},Q^{\out}_0,..., Q^{\out}_{k+1},W_q^{\inn},W_q^{\out}\right)\right\} \Big|\geq \epsilon\Big)\leq Ck^4C'_{k-1}e^{-cc'_{k-1}n\epsilon^2/k^7}, \\
  \P\Big( \Big| \frac{1}{N}\sum_{i=1}^N \phi&\left([\pv^{\inn}_0]_i,...,[\pv^{\inn}_{k+1}]_i,[\pv^{\out}_0]_i,...,[\pv^{\out}_{k+1}]_i,[\wv^{\inn}_p]_i,[\wv^{\out}_p]_i\right)\nonumber\\
  &-\mathbb{E}\left\{\phi\left(P^{\inn}_0,..., P^{\inn}_{k+1},P^{\out}_0,..., P^{\out}_{k+1},W_p^{\inn},W_p^{\out}\right)\right\} \Big|\geq \epsilon\Big) \leq Ck^4C'_{k-1}e^{-cc'_{k-1}n\epsilon^2/k^7}.
\end{align*} 
where the $Q^{\inn/\out}_i$ are jointly Gaussian with $Q^{\inn/\out}_i\sim N(0,\tau^{\inn/\out}_{qi})$ and $\mathbb{E}\{Q^{\inn/\out}_iQ^{\inn/\out}_j\}= [\covm^{\inn/\out}_{v,k+1}]_{ij}$, the $P^{\inn/\out}_i$ are jointly Gaussian with $P^{\inn/\out}_i\sim N(0,\tau^{\inn/\out}_{pi})$ and $\mathbb{E}\{P^{\inn/\out}_iP^{\inn/\out}_j\}=[\covm^{\inn/\out}_{uk}]_{ij}$.
\end{enumerate}
\end{lem}
We will use Lemma~\ref{lem:main_general}  to prove Theorem~\ref{thm:main}; details in Section~\ref{sec:proof_thm1}. Because the proof of Lemma~\ref{lem:main_general} is quite technical, in the main body of the text, we give a high level discussion in Section~\ref{sec:discuss} and relegate the details to Appendix~\ref{sec:lemma_proof}. In particular, Section~\ref{sec:discuss} discusses the innovation required to prove Lemma~\ref{lem:main_general}, and more generally, describes  differences between the finite sample analyses for generalized AMP algorithms and for  standard AMP algorithms, as in \cite[Lemma 5]{AMP_FS}.

\subsection{Theorem~\ref{thm:main} Proof} \label{sec:proof_thm1}
Theorem~\ref{thm:main} follows from Lemma~\ref{lem:main_general} part \textbf{(b)}. In translating GVAMP (Algorithm~\ref{alg:GVAMP}) to the general iteration (Algorithm~\ref{alg:general_gvamp}), we set $\pv^\inn_{k}=\rv_{1k}$ and $\gamma_{pk}^\inn=\gamma_{1k}$ in \eqref{eq:general_identity}. Along with the definition of $\hat{\xv}_{1k}$ in Algorithm~\ref{alg:GVAMP}, namely $\hat{\xv}_{1k} = g_{x1}(\rv_{1k}, \gamma_{1k})$, it follows that
$\phi( [\hat{\xv}_{1k}]_{i}, [\xv_0]_{i})= \phi( g_{x1}( [\pv^\inn_{k}]_i, \gamma_{pk}^\inn), [\xv_0]_i)$, where now quantities on the right are all from Algorithm~\ref{alg:general_gvamp}.
Hence, we bound the probability in \eqref{eq:thm1_result} by introducing the limiting state evolution quantities $\overline{\gamma}_{pk}^\inn$ defined in \eqref{eq:se_general} as follows:
\begin{align} 
&\mathbb{P}\Big(\Big|\frac{1}{N}\sum_{i=1}^N \phi([\hat{\xv}_{1k}]_{i}, [\xv_0]_{i}) - \mathbb{E}\{\phi(\hat{X}_{1k},X_0)\}\Big|\geq\epsilon\Big)  \nonumber \\
& \qquad = \mathbb{P}\Big(\Big|\frac{1}{N}\sum_{i=1}^N \phi( g_{x1}( [\pv^\inn_{k}]_i, \gamma_{pk}^\inn), [\xv_0]_i) - \mathbb{E}\{\phi(\hat{X}_{1k},X_0)\}\Big|\geq\epsilon\Big)\nonumber \\
& \qquad \leq \mathbb{P}\Big(\frac{1}{N}\sum_{i=1}^N\Big|\phi( g_{x1}( [\pv^\inn_{k}]_i, \gamma_{pk}^\inn), [\xv_0]_i) -  \phi( g_{x1}( [\pv^\inn_k]_i,\overline{\gamma}_{pk}^\inn),[\xv_0]_i )  \Big|\geq \frac{\epsilon}{2}\Big)  \label{eq:thm_first_split1} \\
&\qquad \qquad +  \mathbb{P}\Big(\Big|\frac{1}{N}\sum_{i=1}^N \phi( g_{x1}( [\pv^\inn_k]_i,\overline{\gamma}_{pk}^\inn),[\xv_0]_i ) - \mathbb{E}\{\phi(\hat{X}_{1k},X_0)\}\Big|\geq\frac{\epsilon}{2}\Big). \label{eq:thm_first_split2}
\end{align}
Now we consider the two terms in \eqref{eq:thm_first_split1} and   \eqref{eq:thm_first_split2} separately and for both we give an upper bound of $CC_{k}e^{-cc_{k}N\epsilon^2}$ to establish the desired result. First, we notice that bound on  \eqref{eq:thm_first_split2} follows from Lemma~\ref{lem:main_general} part $(b)$. Indeed,
because $g_{x1}(\cdot,\overline{\gamma}_{pk}^\inn)$ is Lipschitz, it is easy to verify that the composition $
(p,x)\mapsto\phi(g_{x1}(p,\overline{\gamma}_{pk}^\inn), x)$ is pseudo-Lipschitz, and because $\xv_0$ concentrates to $X_0$ at the correct rate, we can apply Lemma~\ref{lem:main_general} part $(b)$ to conclude that $\frac{1}{N}\sum_{i=1}^N \phi( g_{x1}( [\pv^\inn_k ,\overline{\gamma}_{pk}^\inn),[\xv_0]_i )$ concentrates to 
$$\mathbb{E}\Big[ \phi\Big( g_{x1}\Big( P^\inn_k, \overline{\gamma}_{pk}^\inn\Big), X_0 \Big)\Big] = \mathbb{E}\Big[ \phi\Big( g_{x1}\Big( \sqrt{\tau_{pi}^{\inn}} Z , \overline{\gamma}_{pk}^\inn\Big), X_0 \Big)\Big]= \mathbb{E}\Big[ \phi\Big(\hat{X}_{1k}, X_0 \Big)\Big].$$
In the final equality above we have used that $\hat{X}_{1k} = g_1(R_{1k}, \bar{\gamma}_{1k})$ with $R_{1k} = \sigma_{1k}Z$ for $Z \sim \mathcal{N}(0,1)$ independent of $X_0$, in addition to the facts that $ \tau_{pi}^{\inn} = \sigma^2_{1k}$, which is discussed when introducing the general state evolution in Section~\ref{sec:generalSE} and that $\gamma_{pk}^\inn=\gamma_{1k}$; hence, $\overline{\gamma}_{pk}^\inn= \overline{\gamma}_{1k}$.

Now we will upper bound the probability in \eqref{eq:thm_first_split1}. First, notice that by the pseudo-Lipschitz property of $\phi$, we have that $|\phi(a_1,b) - \phi(a_2,b)| \leq L(1 + \norm{(a_1,b)} + \norm{(a_2,b)}) |a_1-a_2|$. Moreover, by the Triangle Inequality:
\begin{align*}
  \norm{(a_1,  b)} + \norm{(a_2,  b)} = \norm{(a_1,b)-(a_2,b)+(a_2,b)}+\norm{(a_2,b)} &\leq \norm{(a_1,b)-(a_2,b)} + 2\norm{(a_2,b)}\\
  &= |a_1-a_2|+2\norm{(a_2,b)}.
\end{align*}
Putting this together, we have the following bound.
\begin{equation}
\begin{split}
\label{eq:thshow1}
&\frac{1}{N}\sum_{i=1}^N\Big| \phi( g_{x1}( [\pv^\inn_k ]_i, \gamma_{pk}^\inn), [\xv_0]_i) -  \phi( g_{x1}( [\pv^\inn_k ]_i,\overline{\gamma}_{pk}^\inn),[\xv_0]_i ) \Big| \\
&\leq  \frac{L}{N}\sum_{i=1}^N\Big[ 1+\Big| g_{x1}([\pv^\inn_k  ]_i,\gamma_{pk}^\inn)-g_{x1}([\pv^\inn_k]_i,\overline{\gamma}_{pk}^\inn) \Big| +2\Big\|\Big( g_{x1}\Big([\pv^\inn_k]_i,\overline{\gamma}_{pk}^\inn\Big),  [\xv_0]_i \Big)\Big\|\Big]\\
 &\hspace{2cm}\times\Big| g_{x1}\Big([\pv^\inn_k]_i,\gamma_{pk}^\inn\Big)-g_{x1}\Big([\pv^\inn_k ]_i,\overline{\gamma}_{pk}^\inn\Big) \Big| \\
 &\leq  L\sqrt{3} \sqrt{ 1+ \frac{1}{N}\|g_{x1}\Big(\pv^\inn_k,\gamma_{pk}^\inn\Big)-g_{x1}\Big(\pv^\inn_k ,\overline{\gamma}_{pk}^\inn\Big) \|^2 +\frac{4}{N}\Big\|\Big( g_{x1}\Big(\pv^\inn_k ,\overline{\gamma}_{pk}^\inn\Big),  \xv_0 \Big)\Big\|^2}\\
 &\hspace{2cm}\times \frac{1}{\sqrt{N}} \|g_{x1}\Big(\pv^\inn_k ,\gamma_{pk}^\inn\Big)-g_{x1}\Big(\pv^\inn_k ,\overline{\gamma}_{pk}^\inn\Big) \|,
\end{split}
\end{equation}
where the final inequality follows by Cauchy-Schwarz and Lemma~\ref{lem:squaredsums}. 
To see that \eqref{eq:thshow1} concentrates around zero at the correct rate, notice that it suffices to show 
\begin{equation}
\label{eq:conc1}
\mathbb{P}\Big(\Big| \frac{1}{N}\sum_{i=1}^N \|(g_{x1}([\pv^\inn_k ]_i,\overline{\gamma}_{pk}^\inn),[\xv_0]_i)\|^2 - c_g\Big|  \geq \epsilon  \Big) \leq CC_{k}e^{-cc_{k}N\epsilon^2},
\end{equation}
for some universal constant $c_g>0$, and
\begin{equation}
\label{eq:conc2}
\mathbb{P}\Big(\frac{1}{N}\sum_{i=1}^N\Big| g_{x1}\Big([\pv^\inn_k ]_i,\gamma_{pk}^\inn\Big)-g_{x1}\Big([\pv^\inn_k ]_i,\overline{\gamma}_{pk}^\inn\Big) \Big|^2 \geq \epsilon^2 \Big) \leq CC_{k}e^{-cc_{k}N\epsilon^2}.
\end{equation}
Indeed if these concentration inequalities hold, then applications of Lemmas \ref{sums}, \ref{sqroots}, and \ref{products_0} to \eqref{eq:thshow1} give concentration with the right rates in terms of $N$ and $\epsilon$.
For \eqref{eq:conc1}, observe that $g_{x1}(\cdot,\overline{\gamma}_{pk}^\inn)$ is uniformly Lipschitz by \textbf{Assumption 3}. Thus, the composition, $(p,x)\mapsto \|(g_{x1}(p,\overline{\gamma}_{pk}^\inn),x)\|^2$ is pseudo-Lipschitz, and so with $\text{constant} = \mathbb{E}[\|(g_{x1}(P_k^{\inn},\overline{\gamma}_{pk}^\inn),X_0)\|]^2$ the result in \eqref{eq:conc1} follows by part $(b)$ of Lemma \eqref{lem:main_general}. We also use \textbf{Assumption 4} that $\gamma_{pk}$ and, hence, $\bar{\gamma}_{pk}$ is bounded.
To demonstrate the result in \eqref{eq:conc2}, observe that by \textbf{Assumption 3}, $g_{x1}$ is uniformly Lipschitz, meaning $|g_{x1}(a, b_1) - g_{x1}(a, b_2)| \leq L(1+|a|)|b_1 - b_2|$; hence, $|g_{x1}(a, b_1) - g_{x1}(a, b_2)|^2 \leq 2L^2(1+|a|^2)|b_1 - b_2|^2$ using Lemma~\ref{lem:squaredsums}.
 Therefore, applying Lemma~\ref{lem:squaredsums} again,
\begin{align*}
\frac{1}{N}\sum_{i=1}^N\Big| g_{x1}\Big([\pv^\inn_k]_i,\gamma_{pk}^\inn\Big)-g_{x1}\Big([\pv^\inn_k]_i,\overline{\gamma}_{pk}^\inn\Big) \Big|^2 &\leq \Big| \gamma_{pk}^\inn-\overline{\gamma}_{pk}^\inn \Big|^2\frac{2L^2}{N}\sum_{i=1}^N \Big( 1+|[\pv^\inn_k ]_i|^2 \Big) \\
&\leq \Big| \gamma_{pk}^\inn-\overline{\gamma}_{pk}^\inn \Big|^2\frac{2L^2}{N}\sum_{i=1}^N \Big( 1+2|[\pv^\inn_k]_i|^2 \Big) \\
&\leq \Big| \gamma_{pk}^\inn-\overline{\gamma}_{pk}^\inn \Big|^2 \Big( 2L^2+ \frac{4L^2}{N}\sum_{i=1}^N|[\pv^\inn_k]_i|^2 \Big) .
\end{align*}
Using the above abound, we  can bound the probability in \eqref{eq:conc1} as follows:
\begin{equation*}
\begin{split}
\mathbb{P}\Big(\frac{1}{N}\sum_{i=1}^N\Big| g_{x1}\Big([\pv^\inn_k]_i,\gamma_{pk}^\inn\Big)-g_{x1}\Big([\pv^\inn_k ]_i,\overline{\gamma}_{pk}^\inn\Big) \Big|^2 \geq \epsilon^2 \Big) 
& \leq \mathbb{P}\Big(\Big| \gamma_{pk}^\inn-\overline{\gamma}_{pk}^\inn \Big|^2 \Big[ 1+ \frac{2}{N}\sum_{i=1}^N|[\pv^\inn_k]_i|^2 \Big]  \geq \frac{\epsilon^2}{2L^2} \Big) \\
&\hspace{-3cm} \leq \mathbb{P}\Big(\Big| \gamma_{pk}^\inn-\overline{\gamma}_{pk}^\inn \Big| \geq \frac{\epsilon}{2L} \Big)  + \mathbb{P}\Big(\Big| \gamma_{pk}^\inn-\overline{\gamma}_{pk}^\inn \Big| \sqrt{ \frac{1}{N}\|\pv^\inn_k\|^2}  \geq  \frac{\epsilon}{2\sqrt{2}L} \Big).
\end{split}
\end{equation*}
The first term on the right side of the above has the desired bound by Lemma~\ref{lem:main_general_long} part $(f)$, proved in Appendix~\ref{sec:lemma_proof}.
The second term is bounded as above by using Lemmas \ref{sums}, \ref{sqroots}, and \ref{products_0} to combine the concentration results for the individual factors and terms that are given by Lemma~\ref{lem:main_general_long} part $(f)$ (for $\gamma^{\inn}_{pk}$),
and the following bound given in  Lemma~\ref{lem:main_general_long} part $(d)$:
\[
  \mathbb{P}\Big(\Big \lvert \frac{1}{N}\|\pv^\inn_k\|^2 -  [\Sigma_{pk}^{\inn}]_{k+1, k+1}  \Big \lvert \leq 1\Big) \leq CC_{k}e^{-cc_{k}N}.
\]
 
\subsection{Lemma \ref{lem:main_general} Proof Discussion} \label{sec:discuss}

Structurally, Lemma \ref{lem:main_general} is similar to \cite[Lemma 5]{AMP_FS}, which is used to establish the AMP analog of Theorem~\ref{thm:main} for the linear model in \eqref{eqn:lin_reg}. However, there are a number of differences between the algorithms that lead to important differences in the proofs, discussed here.

\textbf{(1)} In the finite sample analysis of AMP, one often  studies vectors of the form $\Am \xv$ where $\Am$ has i.i.d.\ Gaussian entries and $\xv$ is deterministic. For a vector $\mathbf{Z}$ having i.i.d.\ $N(0,1)$ elements, $\Am \xv \stackrel{d}{=}\|\xv\| \mathbf{Z}$, since $\|\xv\| $ is deterministic.  Importantly, the elements of $\|\xv\| \mathbf{Z}$ are independent.

With GVAMP, the analogous object of interest are  vectors of the form $\Vm \xv$ where $\Vm$ is uniformly distributed on the group of orthogonal matrices. This property of $\Vm$ implies that $\Vm \xv$ has a rotationally invariant distribution, or mathematically, for a vector $\mathbf{Z}$ having i.i.d.\ $N(0,1)$ elements, $\Vm \xv\stackrel{d}{=}({\|\xv\|}/{\|\mathbf{Z}\|})\mathbf{Z}.$
Thus, we pay for relaxing the Gaussianity condition by picking up a $\|\mathbf{Z}\|$ factor in the denominator, causing dependencies in the elements of $({\|\xv\|}/{\|\mathbf{Z}\|})\mathbf{Z}$ that complicate the concentration arguments. By observing that the Gaussiantiy of $\mathbf{Z}$ gives concentration of $\|\mathbf{Z}\|/\sqrt{N}$ around $1$, it can be shown that $\Vm \xv$ will concentrate around the same limit as $(\|\xv\|/\sqrt{N})\mathbf{Z}$, which essentially returns us to the AMP case.

\textbf{(2)} As in the analysis of GAMP, we need to ensure that the $\zv_0$ vector is available at the output denoising stage of the general recursion (as it is an input to the likelihood denoiser). Adding $\zv_0$ to the disturbance vectors $\wv$ would create dependence between the disturbance vectors and other quantities which would violate our conditions, so instead we must generate it from $\xv$ in the general recursion itself. To achieve this, we consider a version of the general recursion where vector iterates are replaced by matrix iterates in order to track both the GVAMP iterates and to recreate $\zv_0$ at each loop through the recursion. This case is not handled explicitly in our proof of Lemma \ref{lem:main_general}, but in Appendix \ref{app:matrix_case}, we demonstrate that when $\xv$ is a matrix rather than a vector, we recover a characterization of the form $\Vm\xv \stackrel{d}{=}\mathbf{Z} \gram^{-1/2}_{\mathbf{Z}}\gram^{1/2}_{\xv}$, where $\mathbf{Z}$ is now a matrix of equivalent dimensions with i.i.d. $N(0,1)$ entries and $\gram_{\mathbf{M}}$ is the Gram matrix corresponding to matrix $\mathbf{M}$. This further generalizes the expression $({\|\xv\|}/{\|\mathbf{Z}\|})\mathbf{Z}$ above.

\textbf{(3)} When studying AMP, one often works with projection matrices $\mathsf{P}_{\mathbf{M}^\perp}$, which map vectors into the orthogonal complement of the range of a matrix $\mathbf{M}$. On the other hand, when studying GVAMP, we instead work with the transformations of the form $\Bm^\perp_\mathbf{M}$ that have columns spanning $\mathrm{range}(\mathbf{M})^\perp$. These are related to the corresponding projections by $\mathsf{P}_{\mathbf{M}^\perp} = \Bm^\perp_\mathbf{M}[\Bm^\perp_\mathbf{M}]^T$. This creates two main differences in the decomposition in Lemma \eqref{lem:cond_dist}. 

First, in the deviation terms that describe the distance of the algorithm in finite samples from its idealized (and easily analyzable) counterpart, the transformation $\mathbb{I}-\mathsf{P}_{\Cm_k}$ that shows up in the AMP case is replaced with the matrices $\Bm_{\Cm^{\inn/\out}_{vk}}\in\mathbb{R}^{N\times 2k}$, where all we know about this matrix is that it has orthonormal columns that span $\mathrm{range}\Big( \Cm^{\inn/\out}_{vk} \Big)$. Notice, on the other hand, the projection matrix $\mathbb{I}-\mathsf{P}_{\Cm_k}$ is fully specified.  However, this knowledge is enough to force images of Gaussian vectors under the transformation given by $\Bm_{\Cm^{\inn/\out}_{vk}}$ to concentrate around zero, because the images are constrained to live in a $k-$dimensional subspace and $\Bm_{\Cm^{\inn/\out}_{vk}}$ preserves norms.

The second difference is the presence of the orthogonal matrices $\Om^{\inn/\out}_{pk}$ and $\Om^{\inn/\out}_{qk}$ in the decomposition. These appear because there is no direct way to relate $\Bm^\perp_{\Cm_{vk}^{\inn/\out}}\Zv^{\inn/\out}_{pk}$ and $\Bm^\perp_{\Cm_{uk}^{\inn/\out}}\Zv^{\inn/\out}_{qk}$ to $\Zv^{\inn/\out}_{pk}$ and $\Zv^{\inn/\out}_{qk}$. However, these matrices end up having no effect on the resulting joint distribution as they are independent of $\overline{\Zv}^{\inn/\out}_{pk}$ and $\overline{\Zv}^{\inn/\out}_{qk}$, respectively. 

\textbf{(4)}  Because the parameters $\gamma_{1k}, \gamma_{2k}$ and $\tau_{1k},\tau_{2k}$ can vary with $N$, so too can the behavior of the denoisers $g_{x1}, g_{x2}$ and $g_{z1},g_{z2}$. This is in contrast to AMP, where the denoisers can vary by iteration but not with $N$. Furthermore, in the AMP case, the denoisers are Lipschitz in both the iterates and the (analog of) the disturbance vectors $\wv$, whereas in the GVAMP case, we must consider a denoiser of the form $[\wv]_i[\qv]_i$, which is not Lipschitz in both $[\wv]_i$ and $[\qv]_i$.

These differences are accounted for by using two different extensions of the Lipschitz condition, the \emph{uniformly} Lipschitz condition and the uniformly bounded conditionally Lipschitz condition. We ultimately view the $\gamma$, $\tau$, and disturbance vectors $\wv$ as parametrizing families of Lipschitz functions, and these two conditions provide sufficient control to ensure that these Lipschitz functions do not vary too rapidly with the parameters and that the corresponding concentration arguments continue to hold.

\textbf{(5)}  The order of denoising and subtracting the Onsager term is interchanged in AMP and GVAMP. This adds an additional layer of complication in obtaining concentration for $\vv^{\inn/\out}_k$: from Algorithm~\ref{alg:general_gvamp},
\[
\vv^{\inn/\out}_k =  \frac{1}{1-\alpha^{\inn/\out}_{pk}}\Big[f^{\inn/\out}_p(\pv^{\inn}_k,\pv^\out_k,\wv^{\inn/\out}_p,\gamma^\inn_{pk},\gamma^\out_{pk})-\alpha^{\inn/\out}_{pk}\pv^{\inn/\out}_{k}\Big],
\]
with $f^{\inn/\out}_p(\pv^\inn_k,\pv^\out_k,\wv^{\inn/\out}_p,\gamma^\inn_{pk},\gamma^\out_{pk})-\alpha^{\inn/\out}_{pk}\pv^{\inn/\out}_{k}$ failing to be Lipschitz (since $\alpha^{\inn/\out}_{pk}$ also has a dependence on $\pv^{\inn/\out}_k$). This prevents us from simply applying Lemma~\ref{lem:main_general}~$\textbf{(b)}$ when proving concentration results for $\vv^{\inn/\out}_k$. However, as with the $1/\|\mathbf{Z}\|$ dependence discussed in \textbf{(1)} above, we handle this by using concentration of $\alpha^{\inn/\out}_{pk}$ to its limit $\overline{\alpha}^{\inn/\out}_{pk}$, allowing one to show that $\vv^{\inn/\out}_k$ will concentrate around the same limit as $ \frac{1}{1-\overline{\alpha}^{\inn/\out}_{pk}}\Big[f^{\inn/\out}_p(\pv^\inn_k,\pv^\out_k,\wv^{\inn/\out}_p,\gamma^\inn_{pk},\gamma^\out_{pk})-\overline{\alpha}^{\inn/\out}_{pk}\pv^{\inn/\out}_{k}\Big]$. This function is Lipchitz, so (inductively) applying Lemma~\ref{lem:main_general}~$\textbf{(b)}$ gives the desired concentration.

\section{Conclusion and Future Work}
This work presents rigorous non-asymptotic performance guarantees for vector approximate message passing (VAMP) and its generalized version GVAMP, characterized through their respective state evolution recursions.  We show, in both cases, that the probability of deviation from the state evolution predictions decay exponentially in the problem size $N$ up through a certain number of iterations related to $N$. This analysis also implies a similar concentration result for generalized approximate message passing (GAMP) and its state evolution.

We expect these concentration results to extend beyond the generalized models \eqref{eqn:gen_reg} and assumptions considered in this work to other settings where a state evolution has been rigorously proved to characterize performance of an AMP-style algorithm in the large system limit.  For example, when the denoisers are non-separable \cite{VAMP_nonsep}, to settings where the distributional parameters of the noise and signal must be learned \cite{VAMP_learn}, or to AMP algorithms that work under more general assumptions on the measurement matrix  \cite{chen2021universality, fan2022approximate, wang2022universality, opper2016theory,dudeja2022universality}. Another interesting direction for future study is whether the analysis of Li and Wei \cite{li2022non} can be used for the type of AMP algorithms studied here to improve the dependence between the problem size and number of iterations for which the concentration results hold.

\section*{Acknowledgment}
The authors acknowledge NSF CCF $\#1849883$ support.

\section{Proof of Lemma~\ref{lem:main_general}} \label{sec:lemma_proof}

To prove Lemma~\ref{lem:main_general}, we prove Lemma~\ref{lem:main_general_long} below. Notice that part (a) and (b) of Lemma~\ref{lem:main_general} are also part (a) and (b) of Lemma~\ref{lem:main_general_long}, only now Lemma~\ref{lem:main_general_long} also has parts (c)-(g) as well.

\begin{lem}\label{lem:main_general_long}

Throughout the lemma, we use $c$ and $C$ to denote constants that do not depend on the iteration number, but which may vary in their exact value between usages. The values of $C_k, C'_k, c_k,$ and $c'_k$ for $k \geq 0$ are given in \eqref{eq:Cdef}.
Let $X_n\stackrel{\cdot\cdot}{=} c$ mean 
\[
\P(|X_n-c|\geq \epsilon) \leq Ck^4C'_{k-1}\exp\left(-cc'_{k-1}n\epsilon^2/k^7\right),
\]
and let $X_n\stackrel{\cdot}{=}c$ mean
\[
\P\left(|X_n-c|\geq \epsilon\right) \leq Ck^4 C_{k-1}\exp\left(-cc_{k-1} n\epsilon^2/k^7\right).
\]

\setcounter{equation}{38}
\begin{enumerate}[label=(\alph*)]
\item \begin{align}
\P\left(\frac{1}{N}\|\Delm^{\inn/\out}_{qk}\|^2\geq\epsilon\right)\leq CkC'_{k-1}\exp\left(-cc'_{k-1} n\epsilon/k^2\right), \label{eq:Qka} \\
\P\left(\frac{1}{N}\|\Delm^{\inn/\out}_{pk}\|^2\geq\epsilon\right)\leq CkC_{k-1}\exp\left(-cc_{k-1} n\epsilon/k^2\right). \label{eq:Pka}
\end{align}
\item %
  Let $\phi:\mathbb{R}^{k+2}\to\mathbb{R}$ pseudo-Lipschitz, or satisfy the following bounded conditionally pseudo-Lipschitz condition:
  \begin{enumerate}
  \item $\phi$ is continuous in all inputs.
  \item $\phi(\cdot,w)$ is pseudo-Lipschitz for each $w$ with PL constants continuous in $w$.
  \item The domain of $w$ is compact.
  \end{enumerate}
Then we have
\begin{align}
  \frac{1}{N}\sum_{i=1}^N \phi&\left([\qv^{\inn}_0]_i,...,[\qv^{\inn}_{k+1}]_i,[\qv^{\out}_0]_i,...,[\qv^{\out}_{k+1}]_i,[\wv]_i\right)  \stackrel{\cdot\cdot}{=}\mathbb{E}\left\{\phi\left(Q^{\inn}_0,..., Q^{\inn}_{k+1},Q^{\out}_0,..., Q^{\out}_{k+1},W\right)\right\}, \label{eq:Qkb} \\
  \frac{1}{N}\sum_{i=1}^N \phi&\left([\pv^{\inn}_0]_i,...,[\pv^{\inn}_{k}]_i,[\pv^{\out}_0]_i,...,[\pv^{\out}_{k}]_i,[\wv]_i\right)  \stackrel{\cdot\cdot}{=}\mathbb{E}\left\{\phi\left(P^{\inn}_0,..., P^{\inn}_k,P^{\out}_0,..., P^{\out}_k,W\right)\right\},\label{eq:Pkb}
\end{align} 
where the $Q^{\inn/\out}_i$ are jointly Gaussian with $Q^{\inn/\out}_i\sim N(0,\tau^{\inn/\out}_{qi})$ and $\mathbb{E}\{Q^{\inn/\out}_iQ^{\inn/\out}_j\}= [\covm^{\inn/\out}_{v,k+1}]_{ij}$, the $P^{\inn/\out}_i$ are jointly Gaussian with $P^{\inn/\out}_i\sim N(0,\tau^{\inn/\out}_{pi})$ and $\mathbb{E}\{P^{\inn/\out}_iP^{\inn/\out}_j\}=[\covm^{\inn/\out}_{uk}]_{ij}$, and $\wv$ is any vector independent of the $\pv$ and $\qv$ iterates that concentrates exponentially fast (see Definition~\ref{def:concentration}) around some limiting variable $W$ with finite second moment.
\item For all $k\geq 0$,
  \begin{align}
    \frac{1}{N}[\qv^{\inn/\out}_{k+1}]^T\uv^{\inn/\out}_0 &\stackrel{\cdot\cdot}{=} 0, \;\;\;  \frac{1}{N}[\qv^{\inn/\out}_{k}]^T\wv^{\inn/\out}_q \stackrel{\cdot\cdot}{=} 0, \label{eq:Qkc} \\
\frac{1}{N}[\pv^{\inn/\out}_k]^T\wv^{\inn/\out}_p &\stackrel{\cdot}{=} 0. \label{eq:Pkc}
\end{align}
\item For all $0\leq j\leq k$,
\begin{align}
\frac{1}{N}[\qv^{\inn/\out}_{j}]^T\qv^{\inn/\out}_{k} &\stackrel{\cdot\cdot}{=} [\covm^{\inn/\out}_{v,k+1}]_{j+1,k+1}, \label{eq:Qkd} \\
\frac{1}{N}[\pv^{\inn/\out}_j]^T\pv^{\inn/\out}_k &\stackrel{\cdot}{=} [\covm^{\inn/\out}_{uk}]_{j+1,k+1}\label{eq:Pkd}.
\end{align} 

\item For all $1\leq j\leq k+1$ and for all $0\leq i\leq k$,
\begin{align}
\frac{1}{N}[\uv^{\inn/\out}_j]^T\uv^{\inn/\out}_{k+1} \stackrel{\cdot\cdot}{=} [\covm^{\inn/\out}_{u,k+1}]_{j+1,k+2},  & \;\;\;   \frac{1}{N}[\uv^{\inn/\out}_{j}]^T\uv^{\inn/\out}_{0} \stackrel{\cdot\cdot}{=} 0,\label{eq:Qke} \\
  \frac{1}{N}[\vv^{\inn/\out}_i]^T\vv^{\inn/\out}_k \stackrel{\cdot}{=} [\covm^{\inn/\out}_{vk}]_{i+1,k+1}.& \label{eq:Pke}
\end{align}

\item For $0\leq j\leq k$, 
\begin{equation}
\begin{split}
\P\left(\left|\alpha^{\inn/\out}_{qk}-\overline{\alpha}^{\inn/\out}_{qk}\right|\geq\epsilon\right) &\leq C'k^4C'_{k-1}\exp\left(-nc'c'_{k-1}\epsilon^2/k^7\right), \\
\P\left(\left|\gamma^{\inn/\out}_{p(k+1)}-\overline{\gamma}^{\inn/\out}_{p(k+1)}\right|\geq\epsilon\right) &\leq C'k^4C'_{k-1}\exp\left(-nc'c'_{k-1}\epsilon^2/k^7\right), \\
\frac{1}{N}[\qv^{\inn/\out}_{k}]^T\uv^{\inn/\out}_{j+1} &\stackrel{\cdot\cdot}{=} 0 \qquad \;\;\;\frac{1}{N}[\qv^{\inn/\out}_{j}]^T\uv^{\inn/\out}_{k+1} \stackrel{\cdot\cdot}{=} 0. \label{eq:Qkf}
\end{split}
\end{equation}
For $0\leq j\leq k$,
\begin{equation}
\begin{split}
\P\left(\left|\alpha^{\inn/\out}_{pk}-\overline{\alpha}^{\inn/\out}_{pk}\right|\geq\epsilon\right) & \leq Ck^4C_{k-1}\exp\left(-ncc_{k-1}\epsilon^2/k^7\right), \\
\P\left(\left|\gamma^{\inn/\out}_{qk}-\overline{\gamma}^{\inn/\out}_{qk}\right|\geq\epsilon\right) &\leq Ck^4C_{k-1}\exp\left(-ncc_{k-1}\epsilon^2/k^7\right), \\
\frac{1}{N}[\pv^{\inn/\out}_j]^T\vv^{\inn/\out}_k \stackrel{\cdot}{=} 0,  \qquad &\;\;\; \frac{1}{N}[\pv^{\inn/\out}_k]^T\vv^{\inn/\out}_j \stackrel{\cdot}{=} 0. \label{eq:Pkf} 
\end{split}
\end{equation}

\item 
\begin{enumerate}[label=(\roman*)]
\item 
\begin{align}
\P\left(\left|\frac{1}{N}\|[\Bm_{\Cm_{uk}^{\inn/\out}}^\perp]^T\qv^{\inn/\out}_k\|^2 - \rho^{\inn/\out}_{qk}\right|\geq\epsilon\right) &\leq Ck^6C'_{k-1}\exp\left(-ncc'_{k-1}/k^{9}\right), \label{eq:Qkg1}\\
\P\left(\left|\frac{1}{N}\|[\Bm^\perp_{\Cm_{vk}^{\inn/\out}}]^T\pv^{\inn/\out}_k\|^2 - \rho^{\inn/\out}_{pk}\right|\geq\epsilon\right) &\leq Ck^6C_{k-1}\exp\left(-ncc_{k-1}/k^{9}\right). \label{eq:Pkg1} 
\end{align}
\item For $1\leq i,j\leq 2(k+1)$,
\begin{equation}
\begin{split}
\P\left([\Cm_{q,k+1}^{\inn/\out}]^T\Cm_{q,k+1}^{\inn/\out}\text{ singular}\right)\leq Ck^6C'_{k-1}\exp\left(-ncc'_{k-1}/k^{9}\right), \label{eq:Qkg2} \\
\P\left(\left|\left[\left([\Cm_{q,k+1}^{\inn/\out}]^T\Cm_{q,k+1}^{\inn/\out}\right)^{-1}\right]_{ij}-\left[\presm^{\inn/\out}_{vk}\right]_{ij}\right|\geq\epsilon\right) \leq Ck^6C'_{k-1}\exp\left(-ncc'_{k-1}/k^{9}\right).
\end{split}
\end{equation}

For $1\leq i,j\leq 2k+1$,
\begin{equation}
\begin{split}
\P\left([\Cm^{\inn/\out}_{pk}]^T\Cm^{\inn/\out}_{pk}\text{ singular}\right)\leq Ck^8C_{k-1}\exp\left(-ncc_{k-1}/k^{11}\right), \label{eq:Pkg2} \\
\P\left(\left|\left[([\Cm^{\inn/\out}_{pk}]^T\Cm^{\inn/\out}_{pk})^{-1}\right]_{ij} - \left[\presm^{\inn/\out}_{uk}\right]_{ij}\right|\geq\epsilon\right)\leq Ck^8C_{k-1}\exp\left(-ncc_{k-1}/k^{11}\right).
\end{split}
\end{equation}

\item 
\begin{align}
\P\left(\left|\frac{1}{N}\|[\Bm^\perp_{\Cm^{\inn/\out}_{q,k+1}}]^T\uv^{\inn/\out}_{k+1}\|^2- \rho^{\inn/\out}_{p,k+1}\right|\geq\epsilon\right) \leq Ck^8C'_{k-1}\exp\left(-ncc'_{k-1}\epsilon^2/k^{11}\right), \label{eq:Qkg3}\\
\P\left(\left|\frac{1}{N}\|[\Bm^\perp_{\Cm^{\inn/\out}_{pk}}]^T\vv^{\inn/\out}_k\|^2- \rho^{\inn/\out}_{qk}\right|\geq\epsilon\right) \leq Ck^8C_{k-1}\exp\left(-ncc_{k-1}\epsilon^2/k^{11}\right). \label{eq:Pkg3}
\end{align}

\item For $1\leq i,j\leq 2k+3$,
\begin{equation}
\begin{split}
\P\left([\Cm^{\inn/\out}_{u,k+1}]^T\Cm^{\inn/\out}_{u,k+1}\text{ singular}\right)\leq Ck^8C'_{k-1}\exp\left(-ncc'_{k-1}/k^{11}\right), \label{eq:Qkg4}\\
\P\left(\left|\left[([\Cm^{\inn/\out}_{u,k+1}]^T\Cm^{\inn/\out}_{u,k+1})^{-1}\right]_{ij} - \left[\presm^{\inn/\out}_{u(k+1)}\right]_{ij}\right|\geq\epsilon\right)\leq Ck^8C'_{k-1}\exp\left(-ncc'_{k-1}/k^{11}\right).
\end{split}
\end{equation}

For $1\leq i,j\leq 2(k+1)$,
\begin{equation}
\begin{split}
\P\left([\Cm^{\inn/\out}_{v,k+1}]^T\Cm^{\inn/\out}_{v,k+1}\text{ singular}\right)\leq Ck^8C_{k-1}\exp\left(-ncc_{k-1}/k^{11}\right), \label{eq:Pkg4}\\
\P\left(\left|\left[([\Cm^{\inn/\out}_{v,k+1}]^T\Cm^{\inn/\out}_{v,k+1})^{-1}\right]_{ij} - \left[\presm^{\inn/\out}_{vk}\right]_{ij}\right|\geq\epsilon\right)\leq Ck^8C_{k-1}\exp\left(-ncc_{k-1}/k^{11}\right).
\end{split}
\end{equation}

\end{enumerate}
\end{enumerate}
\end{lem}

\begin{rem}
In the course of the Lemma \ref{lem:main_general_long} proof, it often happens that the concentrating value of some quantity is a factor in the rate of concentration for another quantity (e.g.\ when we apply Lemma $A.3$). In theory, we must therefore take care about the varying magnitudes of these concentrating values as the iteration number changes. However, in our case, we can ignore these issues by bounding the relevant concentrating values independently of $n$ and the concentration number.

In particular, it suffices to find upper and lower bounds on $\tau^{\inn/\out}_{(p/q)k}$ and $\rho^{\inn/\out}_{(p/q)k}$ and upper bounds on $\overline{\alpha}^{\inn/\out}_{(p/q)k}$ for all $k\geq 1$. Our stopping criteria allows us to assume that the $\tau^{\inn/\out}_{(p/q)k}$ and $\rho^{\inn/\out}_{(p/q)k}$ are lower bounded independently of $k$ in our concentration analysis. For an upper bound, observe that $\rho^{\inn/\out}_{(p/q)k}\leq \tau^{\inn/\out}_{(p/q)k}$. If either of the $\tau^{\inn}_{(p/q)k}$ is a not bounded, then by the translation $\sigma_{1k}=\tau^{\inn}_{pk}$ and $\sigma_{2k}=\tau^{\inn}_{qk}$ and Theorem \ref{thm:main} with $\phi(x,y)=(x-y)^2$, the state evolution prediction of the mean squared error of $\hat{x}_{1k}$ or $\hat{x}_{2k}$ will diverge as $k\to\infty$. Similarly, we can relate non-bounded behavior for $\tau^{\out}_{(p/q)k}$ to predicted divergence of iterates in the state evolution.

Finally, recall that by definition, $\overline{\alpha}^{\inn/\out}_{pk} = \mathbb{E}\{(f^{\inn/\out}_p)'(P^{\inn}_k,P^\out_k,W^{\inn/\out}_p,\overline{\gamma}^{\inn}_{pk},\overline{\gamma}^\out_{pk})\}.$ Now, under the strong concavity assumption (\textbf{Assumption 2}), $(f^{\inn/\out}_p)'$ is bounded independently of $k$, so we trivially have an upper bound on the size of $\overline{\alpha}_{pk}^{\inn/\out}$ tat is independent of $k$, and similarly for $\overline{\alpha}^{\inn/\out}_{qk}$.
\end{rem}

The proof of Lemma~\ref{lem:main_general} follows by induction on iteration $k$. We will refer to results \eqref{eq:Qka}, \eqref{eq:Qkb}, \eqref{eq:Qkc}, \eqref{eq:Qkd}, \eqref{eq:Qke}, \eqref{eq:Qkf}, \eqref{eq:Qkg1}, \eqref{eq:Qkg2}, \eqref{eq:Qkg3}, and \eqref{eq:Qkg4} as $\mathbf{Q_{k+1}.(a)}$ - $\mathbf{Q_{k+1}.(g)}$ and results  \eqref{eq:Pka}, \eqref{eq:Pkb}, \eqref{eq:Pkc}, \eqref{eq:Pkd}, \eqref{eq:Pke}, \eqref{eq:Pkf}, \eqref{eq:Pkg1}, \eqref{eq:Pkg2}, \eqref{eq:Pkg3}, and \eqref{eq:Pkg4} as $\mathbf{P_{k}.(a)}$ - $\mathbf{P_{k}.(g)}$. Then the proof proceeds as follows: \\
\textbf{1.} Prove $\mathbf{P_0}$. \\
\textbf{2.}  Assuming $\mathbf{P_0}$, prove $\mathbf{Q_1}$. \\
\textbf{3.}  Assuming $\mathbf{P_{r-1}}$ and $\mathbf{Q_r}$ for $1 \leq \mathbf{r} \leq k$, prove $\mathbf{P_{k}}$. \\
\textbf{4.}  Assuming $\mathbf{P_{r}}$ and $\mathbf{Q_r}$ for $1 \leq \mathbf{r} \leq k$, prove $\mathbf{Q_{k+1}}$.

\subsection{Showing $\mathbf{P_0}$ Holds}
 
 
$\mathbf{P_0.(a)}$ We prove the result for $\Delm^\inn_{p0} $ and the result for $\Delm^\out_{p0} $ follows in the same way.  From Lemma~\ref{lem:cond_dist}, recall
$
\Delm^\inn_{p0} = \left(\frac{\|\uv^\inn_0\|}{\|\Zv^\inn_{p0}\|}-\sqrt{\rho^\inn_{p0}}\right)\Bm^\perp_{\Cm^\inn_{v0}}\Zv^\inn_{p0}.
$
Therefore, we write
\be
\begin{split}
 \P\left(\frac{ \|\Delm^\inn_{p0}\|}{\sqrt{N}}\geq \sqrt{\epsilon}\right) &= \P\left(\frac{\|\Bm^\perp_{\Cm^\inn_{v0}}\Zv^\inn_{p0}\|}{\sqrt{N}}\left|\frac{\|\uv^\inn_0\|}{\|\Zv^\inn_{p0}\|}-\sqrt{\rho^\inn_{p0}}\right|\geq \sqrt{\epsilon} \right) \stackrel{(a)}{=} \P\left(\frac{\|\Zv^\inn_{p0}\|}{\sqrt{N}}\left|\frac{\|\uv^\inn_{p0}\|}{\|\Zv^\inn_{p0}\|}-\sqrt{\rho^\inn_{p0}}\right|\geq \sqrt{\epsilon}  \right)\\
& \hspace{2.5cm} \stackrel{(b)}{\leq} \P\left(\left|\frac{\|\uv^\inn_{p0}\|}{\sqrt{N}}-\sqrt{\rho^\inn_{p0}}\right|\geq\frac{\sqrt{\epsilon}}{2}\right) + \P\left(\sqrt{\rho^\inn_{p0}}\left|\frac{\|\Zv^\inn_{p0}\|}{\sqrt{N}}-1\right|\geq\frac{\sqrt{\epsilon}}{2}\right),
\label{eq:P0a}
\end{split}
\ee
where step $(a)$ follows because  $[\Bm^\perp_{\Cm^\inn_{v0}}]^T\Bm^\perp_{\Cm^\inn_{v0}}= \Idm$ 
and step $(b)$ from Lemma~\ref{sums} using that
\[\frac{\|\Zv^\inn_{p0}\|}{\sqrt{N}}\left|\frac{\|\uv^\inn_{p0}\|}{\|\Zv^\inn_{p0}\|}-\sqrt{\rho^\inn_{p0}}\right| =\left| \frac{\|\uv^\inn_{p0}\|}{\sqrt{N}} -\sqrt{ \frac{\rho^\inn_{p0}}{N}}\|\Zv^\inn_{p0}\|\right| \leq \left| \frac{ \|\uv^\inn_{p0}\|}{\sqrt{N}} - \sqrt{\rho^\inn_{p0}} \right|  + \sqrt{\rho^\inn_{p0}} \left| 1  - \frac{\|\Zv^\inn_{p0}\|}{\sqrt{N}}\right|.\]

Now, we bound the first term on the right side of \eqref{eq:P0a} by $C\exp\left(-c N\epsilon \right)$ using \textbf{Condition 0} and Lemma \ref{sqroots} along with
the fact that $c$ depends on $\rho^\inn_{p0} >\epsilon'_1$ by the stopping criteria discussed in \textbf{Condition 5}. Similarly, using Lemmas \ref{subexp} and \ref{sqroots} and that $0 < \rho^\inn_{p0} < \infty$, we find
\[
\P\left(\sqrt{\rho^\inn_{p0}}\left|\frac{\|\Zv^\inn_{p0}\|}{\sqrt{N}}-1\right|\geq\frac{\sqrt{\epsilon} }{2}\right) = \P\left(\left|\frac{\|\Zv^\inn_{p0}\|}{\sqrt{N}}-1\right|\geq \sqrt{\frac{\epsilon} {4\rho^\inn_{p0}}}\right) \leq C\exp\left(-cN\epsilon/\rho^\inn_{p0} \right).
\]


$\mathbf{P_0.(b)}$ We prove the statement using a general vector $\wv$, independent of $\pv$ and $\qv$, that concentrates exponentially fast (see Definition~\ref{def:concentration}) around some limiting variable $W$ with finite second moment. When this result is applied in later steps in the induction, this value will be $\wv^{\inn/\out}_p$.

Now for the proof, first we observe, using Lemma~\ref{lem:cond_dist}, that
$
\pv^\inn_0 \lvert_{\mathcal{P}_0} \stackrel{d}{=} \sqrt{\rho^\inn_{p0}} \,\Om^\inn_{p0} \, \overline{\Zv}^\inn_{p0} + \Delm^\inn_{p0},
$
where we recall that $\overline{\Zv}^\inn_{p0}$ is a length-$N$ vector with independent, standard Gaussian entries that are independent of the conditioning sigma-algebra and $\Om^\inn_{p0} \in \mathbb{R}^{N \times N}$ is a deterministic orthogonal matrix (since we have $ \Cm^\inn_{v0} = \emptyset$ as defined in \eqref{eq:Up_Uq_matrices}). 
Using the notation defined in Lemma~\ref{lem:cond_dist},
\begin{equation}
\label{eq:Orep}
\Om^\inn_{p0} \, \overline{\Zv}^\inn_{p0} = \Bm^\perp_{\Cm_{v0}^\inn}\Zv^\inn_{p0} + \Bm_{\Cm_{v0}^\inn}\breve{\Zv}^\inn_{p0} = \Bm^\perp_{\Cm_{v0}^\inn}\Zv^\inn_{p0} \stackrel{d}{=}\Zv^{\inn}_{p0}.
\end{equation}
An analogous representation holds for $\pv^\out_0$.
Therefore, using Lemma~\ref{sums}, we find
\begin{equation}
\begin{split}
  \P&\left( \left| \frac{1}{N}\sum_{i=1}^N\phi\left( [\pv^\inn_0]_i,[\pv^\out_0]_i,[\wv]_i \right)-\mathbb{E}\left\{\phi\left( P^\inn_0,P^\out_0,W \right) \right\}\right|\geq \epsilon \right) \\
&\leq  \P\Big( \frac{1}{N}\sum_{i=1}^N\Big| \phi\left( \sqrt{\rho^\inn_{p0}}[\Zv^\inn_{p0}]_i + [\Delm^\inn_{p0}]_i,\sqrt{\rho^\out_{p0}}[\Zv^\out_{p0}]_i + [\Delm^\out_{p0}]_i,[\wv]_i \right)\\
  &\hspace{8cm}-\phi\left( \sqrt{\rho^\inn_{p0}}[\Zv^\inn_{p0}]_i,\sqrt{\rho^\out_{p0}}[\Zv^\out_{p0}]_i,[\wv]_i \right) \Big|\geq \frac{\epsilon}{2} \Big)\\
  & + \P\left( \left| \frac{1}{N}\sum_{i=1}^N\phi\left( \sqrt{\rho^\inn_{p0}}[\Zv^\inn_{p0}]_i,\sqrt{\rho^\out_{p0}}[\Zv^\out_{p0}]_i,[\wv]_i \right)-\mathbb{E}\left\{\phi\left( P^\inn_0,P^\out_0,W \right)\right\} \right|\geq \frac{\epsilon}{2} \right), 
  \label{eq:B0_b_eq1}
\end{split}
\end{equation}
where to make the above well-defined, we continue our convention of letting $[\Zv^\out_{p0}]_i= [\Delm^\out_{v0}]_i=0$ 
for $M<i\leq N$. Label the terms on the right side of \eqref{eq:B0_b_eq1} as $T_1$ and $T_2$. If $\phi$ is pseudo-Lipschitz over all inputs jointly, we can bound $T_1$ using the triangle inequality 
and Cauchy-Schwarz as follows:
\begin{align*}
  &T_1 \leq     \P\Big( \frac{L}{N}\sum_{i=1}^N \Big[ 1+2\sqrt{\rho^\inn_{p0}}\left|[\Zv^\inn_{p0}]_i\right|+2\sqrt{\rho^\out_{p0}}\left|[\Zv^\out_{p0}]_i\right|+2\left| [\wv]_i \right|+\left|[\Delm^\inn_{p0}]_i\right|+\left|[\Delm^\out_{p0}]_i\right| \Big]\\
  &\hspace{11.5cm} \times \Big[\left| [\Delm_{p0}^\inn]_i \right|+\left| [\Delm_{p0}^\out]_i \right|\Big]\geq \frac{\epsilon}{2} \Big)\\
      &\leq \P\left( \frac{\|\Delm_{p0}^\inn\|+\|\Delm_{p0}^\out\|}{\sqrt{N}}\left[ 1+2\sqrt{\rho^\inn_{p0}}\frac{\|\Zv^\inn_{p0}\|}{\sqrt{N}}+2\sqrt{\rho^\out_{p0}}\frac{\|\Zv^\out_{p0}\|}{\sqrt{N}}+2\frac{\|\wv\|}{\sqrt{N}}+\frac{\|\Delm_{p0}^\inn\|}{\sqrt{N}}+\frac{\|\Delm_{p0}^\out\|}{\sqrt{N}} \right] \geq \frac{\epsilon}{4L}\right)\\
  &\leq \P\left( \left[\frac{\|\Delm_{p0}^\inn\|}{\sqrt{N}}+\frac{\|\Delm_{p0}^\out\|}{\sqrt{M}}\right]\left[ 1+\sqrt{\rho^\inn_{p0}}\frac{\|\Zv^\inn_{p0}\|}{\sqrt{N}}+\sqrt{\rho^\out_{p0}}\frac{\|\Zv^\out_{p0}\|}{\sqrt{M}}+\frac{\|\wv\|}{\sqrt{N}}+\frac{\|\Delm_{p0}^\inn\|}{\sqrt{N}}+\frac{\|\Delm_{p0}^\out\|}{\sqrt{M}} \right] \geq \frac{\epsilon}{8L}\right),
\end{align*}
where the last inequality follows because $M\leq N$. Now we observe that $\frac{1}{\sqrt{N}}\|\Delm_{p0}^{\inn}\|$ and $\frac{1}{\sqrt{M}}\|\Delm^\out_{p0}\|$ concentrate around $0$ by $\mathbf{P_0.(a)}$, $\frac{1}{\sqrt{N}}\|\Zv^{\inn}_{p0}\|$ and $\frac{1}{\sqrt{M}}\|\Zv^{\out}_{p0}\|$ concentrate around $1$ by Lemmas~\ref{sqroots} and ~\ref{subexp}, and $\frac{1}{\sqrt{N}}\|\wv\|$ concentrates around $\sqrt{\mathbb{E}W^2}$ by the assumption of the part (b) statement. The overall concentration of term $T_1$ then follows from application of Lemmas~\ref{products_0} and \ref{sums}. We note that the rate of concentration will depend on the values $\rho^\inn_{p0}, \rho^\out_{p0},$ and $\mathbb{E}W^2$ all of which are assumed to be finite, with $\rho^\inn_{p0}$ and $\rho^\out_{p0}$ not too small  by the stopping criteria discussed in \textbf{Condition 5}; so all are absorbed into the constants.

If on the other hand $\phi$ satisfies the bounded conditionally pseudo-Lipschitz property, then the functions $f_i=\phi\left( \cdot,\cdot, [\wv]_i \right)$ are each pseudo-Lipschitz with PL constants that satisfy a common upper bound (since the PL constants are continuous in the $[\wv]_i$, which are themselves contained in a compact set by defintion). In this case, we can apply the exact same strategy as above with the $f_i$ in place of the $\phi$ and obtain an analogous bound where the $\|\wv\|/\sqrt{N}$ term is dropped and $L$ is now an upper bound for the PL constants of the $f_i$.

For $T_2$, note that for all $i$ for which the expression is defined,
we have by part 1 of Lemma \ref{lem:joint_dists} that
\begin{equation}
\label{eq:pstar_equality0}
 [\ide{\pv}^{\inn/\out}_0]_i = \sqrt{\rho^{\inn/\out}_{p0}} [\widetilde{\Om}^\inn_{p0} \bar{\Zv}^{\inn/\out}_{p0}]_i \stackrel{d}{=}  \sqrt{\rho^{\inn/\out}_{p0}}[\Zv^{\inn/\out}_{p0}]_i,
\end{equation}
where the final equality uses that $\widetilde{\Om}^\inn_{p0} = \Om^\inn_{p0}$, due to the initialization of Algorithm~\ref{alg:gaussian} (see the discussion following the presentation of the algorithm) and \eqref{eq:Orep}; hence, $\widetilde{\Om}^\inn_{p0} \bar{\Zv}^{\inn/\out}_{p0}  \overset{d}{=}\Zv^\inn_{p0}$.
Moreover, by part 3 of Lemma \ref{lem:joint_dists},
\begin{equation}
\label{eq:pstar_equality}
 [\ide{\pv}^{\inn/\out}_0]_i\stackrel{d}{=} P_0^{\inn/\out}.
\end{equation}
Thus, if $\phi$ is jointly pseudo-Lipschitz in all inputs, we can apply Lemma \ref{lem:PLsubgaussconc} to get concentration for $T2$.

If instead $\phi$ satsifies the conditionally bounded pseudo-Lipschitz condition, then letting $\mathbb{E}_P$ denote the expectation with respect to $(P_0^\inn,P_0^\out)$ and using Lemma~\ref{sums}, we get the further inequality
\begin{equation}
\begin{split}
  T_2 &= \P\left( \left| \frac{1}{N}\sum_{i=1}^N\phi\left( \sqrt{\rho^\inn_{p0}}[\Zv^\inn_{p0}]_i,\sqrt{\rho^\out_{p0}}[\Zv^\out_{p0}]_i,[\wv]_i \right)-\mathbb{E}\left\{\phi\left( P^\inn_0,P_0^\out,W \right)\right\} \right|\geq \frac{\epsilon}{2} \right)\\
      &\leq \P\left( \left| \frac{1}{N}\sum_{i=1}^N\phi\left( \sqrt{\rho^\inn_{p0}}[\Zv^\inn_{p0}]_i,\sqrt{\rho^\out_{p0}}[\Zv^\out_{p0}]_i,[\wv]_i \right)-\mathbb{E}_P\left\{\phi\left( P^\inn_0,P_0^\out,[\wv]_i \right)\right\}\right|\geq \frac{\epsilon}{4} \right)\\
  &\hspace{2cm} + \P\left( \left| \frac{1}{N}\sum_{i=1}^N \mathbb{E}_P\left\{\phi\left(P_0^\inn,P_0^\out,[\wv]_i  \right)\right\} - \mathbb{E}\left\{\phi(P_0^\inn,P_0^\out,W)\right\}\right|\geq\frac{\epsilon}{4} \right),
  \label{eq:t2_first_bound_partb}
  \end{split}
\end{equation}
Let the first term on the right side of \eqref{eq:t2_first_bound_partb} be denoted $T21$ and the second term be denoted $T22$.  For $T21$, note that if $f_i(p_1,p_2) = \phi(p_1,p_2,[\wv]_i)$, then by the conditionally bounded pseudo-Lipschitz condition, the $f_i$ are pseudo-Lipschitz with constant continuous in $[\wv]_i$. But since the $[\wv]_i$ are bounded, so too are these constants. Thus, the $f_i$ are all pseudo-Lipschitz with some common constant $B$, and Lemma \ref{lem:PLsubgaussconc} again gives concentration for T21.

For T22, the continuity of $x\mapsto\mathbb{E}_P\left\{\phi(P_0^\inn,P_0^\out,x)\right\}$ and the fact that the $[\wv]_i$ are bounded i.i.d. implies that the terms of $T22$ are again bounded i.i.d.. Thus, the desired concentration rate is given by Hoeffding's inequality.


$\mathbf{P_0.(c)}$ The result follows using $\mathbf{P_0.(b)}$ with the pseudo-Lipschitz functions $\phi(p^{\inn}, p^{\out}, w) =  p^{\inn} w$ and  $\phi(p^{\inn}, p^{\out}, w) =  p^{\out} w$ along with the facts that $ \mathbb{E}\{P^\inn_0 W^\inn_p\} = 0$ and $\mathbb{E}\{P_0^\out W_p^\out\}=0$, since $W^{\inn/\out}_p$ and $P_0^{\inn/\out}$ are independent.


$\mathbf{P_0.(d)}$
The result follows using $\mathbf{P_0.(b)}$ with the pseudo-Lipschitz functions $\phi(p^{\inn}, p^{\out}, w) =  (p^{\inn})^2$ and  $\phi(p^{\inn}, p^{\out}, w) =  (p^{\out})^2$.


$\mathbf{P_0.(e)-(f)}$ These are proved in exactly the same way as $\mathbf{P_k.(e)}$ and $\mathbf{P_k.(f)}$ in Section~\ref{sec:stepPK} with $j=k=0$ throughout and where references to $\mathbf{Q_k.(f)}$ are replaced by our assumption that the initial sequences $\gamma^{\inn/\out}_{p0}$ converges to $\overline{\gamma}^{\inn/\out}_{p0}$. Thus, we omit the details here and refer to the proof of the more general case of $\mathbf{P_k.(e)}$ and $\mathbf{P_k.(f)}$.


$\mathbf{P_0.(g).(i)}$ Because $\Cm^\inn_{v0} = \emptyset$ is empty, $\Bm^\perp_{\Cm_{v0}^\inn}$ is a determinsitic orthogonal matrix; therefore, $\|[\Bm^\perp_{\Cm_{v0}^\inn}]^T\pv^\inn_0\|^2 = \|\pv^\inn_0\|^2=\|\uv^\inn_0\|^2$ where the second equality follows as $\pv^\inn_0 = \Vv \uv^\inn_0$ in Algorithm~\ref{alg:general_gvamp}. The result then follows from the assumed concentration of $\uv^\inn_0$ and the definition of $\rho^{\inn}_{p0}$ in \eqref{eq:rhos}.


$\mathbf{P_0.(g).(ii)}$ By \eqref{eq:Cmat}, $\Cm^{\inn/\out}_{p0} = \pv^{\inn/\out}_0$.  Thus, the invertibility of $[\Cm^{\inn/\out}_{p0}]^T\Cm^{\inn/\out}_{p0}$ just requires that $[\pv^{\inn/\out}_0]^T\pv^{\inn/\out}_0\neq 0$, which is true with probability one. Moreover, by $\mathbf{P_0.(d)}$, we know that $[\pv^{\inn/\out}_0]^T\pv^{\inn/\out}_0/N$ concentrates to $\E \left[P_0^{\inn/\out}\right]^2$ and since, by $\mathbf{P_0.(d)}$ and \eqref{eq:Rk_Sk}, 
\[
\left[\presm^{\inn/\out}_{u0} \right]^{-1} = \E \left[U_0^{\inn/\out}\right]^2 = \lim_{N\to\infty}\frac{1}{N}\|\uv_0^{\inn/\out}\|^2  = \lim_{N\to\infty}\frac{1}{N}\|\pv_0^{\inn/\out}\|^2=\E\left[ P_0^{\inn/\out} \right]^2,
\]
the desired result follows by Lemma~\ref{inverses}. Application of Lemma~\ref{inverses} will put a power of $\tau^{\inn/\out}_{p0}$ in the numerator of the rate of concentration; this can be lower bounded by the stopping time assumption.


$\mathbf{P_0.(g).(iii)}$ We prove the `$\inn$' version  of the result, while the `$\out$' version follows correspondingly. Recall,  $\Cm^\inn_{p0}=\pv^\inn_0$ by \eqref{eq:Cmat} and $\Bm^\perp_{\Cm_{p0}^\inn} \in\mathbb{R}^{N\times N-1}$ is an orthogonal matrix with columns that form an orthonormal basis for $\mathrm{range}( \Cm^\inn_{p0} )^\perp=[\pv_0^\inn]^\perp$. Thus, the matrix $[\Bm^\perp_{\Cm_{v0}^\inn}\;\pv^\inn_0/\|\pv^\inn_0\|] \in\mathbb{R}^{N\times N}$ is orthogonal and
\begin{equation}
\label{eq:Bperp1}
\|\vv^\inn_0\|^2 = \left\|\left[ \begin{matrix} [\Bm^\perp_{\Cm_{v0}^\inn}]^T\\\|\pv_0^\inn\|^{-1}[\pv_0^\inn]^T\end{matrix} \right]\vv^\inn_0\right\|^2 = \|[\Bm^\perp_{\Cm_{v0}^\inn}]^T\vv_0^\inn\|^2 + \frac{([\pv_0^\inn]^T\vv_0^\inn)^2}{\|\pv_0^\inn\|^2}.
\end{equation}
Rearranging, this gives us
\begin{equation}
\label{eq:Bperp2}
\|[\Bm^\perp_{\Cm_{p0}^\inn}]^T\vv^\inn_0\|^2= \|\vv^\inn_0\|^2 - \frac{([\vv^\inn_0]^T\pv^\inn_0)^2}{\|\pv^\inn_0\|^2}.
\end{equation}
Finally, recalling that $\rho^\inn_{q0} =\mathbb{E}\{[V^\inn_0]^2\}$ from \eqref{eq:rhos} and using Lemma~\ref{sums}, we have
\be
\P\left(\left|\frac{1}{N}\|[\Bm^\perp_{\Cm_{p0}^\inn}]^T\vv^\inn_0\|^2-\rho^\inn_{q0}\right|\geq \epsilon\right)\leq\P\left(\left|\frac{1}{N}\|\vv^\inn_0\|^2-\mathbb{E}\{[V^\inn_0]^2\}\right|\geq\frac{\epsilon}{2}\right) + \P\left(\frac{([\vv^\inn_0]^T\pv^\inn_0)^2}{\|\pv^\inn_{0}\|^2}\geq\frac{\epsilon}{2}\right).
\label{P0_giii}
\ee
The first term concentrates by $\mathbf{P_0.(e)}$ since $\covm^\inn_{v0} = \mathbb{E}\{[V^\inn_0]^2\}$ by definition.
The second term concentrates by $\mathbf{P_0.(d)}$, $\mathbf{P_0.(f)}$, Lemma~\ref{inverses}, and Lemma \ref{products_0}.


$\mathbf{P_0.(g).(iv)}$ To show that $[\Cm^\inn_{v1}]^T\Cm^\inn_{v1}$ is invertible with high probability, recall that $\Cm^\inn_{v1} = [\Pm^\inn_{0} \; \Vm^\inn_0] = [\pv^\inn_{0} \; \vv^\inn_0]$  by \eqref{eq:Cmat}; therefore, $\det\left([\Cm^\inn_{v1}]^T\Cm^\inn_{v1}\right) = \|\pv^\inn_0\|^2 \|\vv^\inn_0\|^2-([\pv^\inn_0]^T\vv^\inn_0).$
Noting that $\|\pv_0^{\inn/\out}\| = \|\uv_0^{\inn/\out}\|$, we have by $\mathbf{P_0.(e)(f)}$, \textbf{Condition 0}, and Lemmas~\ref{sums} and \ref{products} that this concentrates around $(\covm^\inn_{u0})(\covm^\inn_{v0})>0$.
Then we have
\begin{equation}
\label{eq:inverseC}
([\Cm^\inn_{v1}]^T\Cm^\inn_{v1})^{-1} = \frac{1}{\det\left([\Cm^\inn_{v1}]^T\Cm^\inn_{v1}\right)}\left[\begin{matrix}[\vv^\inn_0]^T\vv^\inn_0 & -[\pv^\inn_0]^T\vv^\inn_0\\ -[\pv^\inn_0]^T\vv^\inn_0 & [\pv^\inn_0]^T\pv^\inn_0\end{matrix}\right].
\end{equation}
By similar reasoning, namely via $\mathbf{P_0.(d)(e)(f)}$  and Lemma \ref{products}, the entries of \eqref{eq:inverseC} have the desired concentrate concentrating constant given below using the definition of $\presm^\inn_{v0}$ given in \eqref{eq:Rk_Sk}.
\[
\left[\begin{matrix} (\covm^\inn_{u0})^{-1} & 0 \\ 0 & (\covm^\inn_{v0})^{-1}\end{matrix}\right] = \presm_{v0}^\inn.
\]

 \subsection{Showing $\mathbf{Q_1}$ Holds}
$\mathbf{Q_1.(a)}$ First, recall that by Lemma~\ref{lem:cond_dist},
\[
  \Delm^\inn_{q0} = \frac{[\pv^\inn_0]^T\vv^\inn_0}{\|\pv^\inn_0\|^2}\uv^\inn_0 + \left(\frac{\|[\Bm^\perp_{\Cm_{p0}^\inn}]^T\vv^\inn_0\|}{\|\Zv^\inn_{q0}\|}-\sqrt{\rho^\inn_{q0}}\right)\Bm^\perp_{\Cm_{u0}^\inn}\Zv^\inn_{q0}- \sqrt{\rho^\inn_{q0}}\Bm_{\Cm_{u0}^\inn}\breve{\Zv}^\inn_{q0}.
\]
Labeling the above three terms $T_1$ - $T_3$, we have $\|\Delm^\inn_{q0} \|^2 = \|T_1 + T_2 + T_3\|^2 \leq 3\|T_1\|^2 + 3\|T_2\|^2 + 3\|T_3\|^2$ by Lemma~\ref{lem:squaredsums}. Hence, it suffices to derive bounds for $\mathbb{P}(\frac{1}{N}\|T_i\|^2\geq \frac{\epsilon}{9})$ for $i\in\{1,2,3\}$ by  Lemma~\ref{sums}.

For $T_1$, note that $\pv^\inn_0=\Vm\uv^\inn_0$ by Algorithm~\ref{alg:general_gvamp} step~\eqref{eq:alg_gvamp_p_step}, and so $\|\pv^\inn_0\|^2 = [\uv^\inn_0]^T \Vm^T \Vm\uv^\inn_0 = \|\uv^\inn_0\|^2$ as $\Vm^T \Vm = \Idm$. Thus $\|T_1\|^2 = \frac{([\pv^\inn_0]^T\vv^\inn_0)^2}{\|\pv^\inn_0\|^4}\|\uv^\inn_0\|^2 = \frac{([\pv^\inn_0]^T\vv^\inn_0)^2}{\|\pv^\inn_0\|^2}$ and an upper bound is given in \eqref{P0_giii}.

For $T_2$, observe that $\|\Bm^\perp_{\Cm_{u0}^\inn}\Zv^\inn_{q0}\|^2 = [\Zv^\inn_{q0}]^T[\Bm^\perp_{\Cm_{u0}^\inn}]^T \Bm^\perp_{\Cm_{u0}^\inn}\Zv^\inn_{q0} = \|\Zv^\inn_{q0}\|^2$ using $[\Bm^\perp_{\Cm_{u0}^\inn}]^T \Bm^\perp_{\Cm_{u0}^\inn} = \Idm$ and therefore
\begin{align*}
\|T_2\|^2 = \left\lvert \frac{\|[\Bm^\perp_{\Cm_{p0}^\inn}]^T\vv^\inn_0\|}{\|\Zv^\inn_{q0}\|}-\sqrt{\rho^\inn_{q0}}\right \lvert^2\|\Bm^\perp_{\Cm_{u0}^\inn}\Zv^\inn_{q0}\|^2 &= \left\lvert \frac{\|[\Bm^\perp_{\Cm_{p0}^\inn}]^T\vv^\inn_0\|}{\|\Zv^\inn_{q0}\|}-\sqrt{\rho^\inn_{q0}}\right \lvert^2\|\Zv^\inn_{q0}\|^2  \\
&= \left\lvert \|[\Bm^\perp_{\Cm_{p0}^\inn}]^T\vv^\inn_0\| -\sqrt{\rho^\inn_{q0}} \|\Zv^\inn_{q0}\|\right \lvert^2.
\end{align*}
By Lemma~\ref{sums} along with $\mathbf{P_0.(g).(iii)}$, Lemma~\ref{subexp}, and Lemma \ref{sqroots}, we have the desired bound using 
\begin{align*}
\mathbb{P}\left(\frac{1}{\sqrt{N}}\|T_2\|\geq \sqrt{\frac{\epsilon}{9}}\right)&\leq\P\left(\frac{1}{\sqrt{N}}\left|\|[\Bm^\perp_{\Cm_{p0}^\inn}]^T\vv^\inn_0\| - \sqrt{\rho^\inn_{q0}} \right| + \sqrt{\frac{\rho^{\inn}_{q0}}{N}}\left| \|\Zv^{\inn}_{q0}\| - 1 \right|\geq \sqrt{\frac{\epsilon}{9}}\right).
\end{align*}

Finally, notice that $\|T_3\| =\sqrt{\rho^\inn_{q0}}  \|\Bm_{\Cm_{u0}^\inn}\breve{\Zv}^\inn_{q0}\|$. Observe that because $\mathrm{rank}\left( \Cm^\inn_{u0} \right)=1$, the matrix $\Bm_{\Cm_{u0}^\inn}$ is just a unit column vector. Furthermore, $\Zv^\inn_{q0}$ has length $1$ (by definition); thus, is a standard normal. Combined, these give us that
$\|\Bm_{\Cm_{u0}^\inn}\breve{\Zv}^\inn_{q0}\| = |\breve{\Zv}^\inn_{q0}|\|\Bm_{\Cm_{u0}^\inn}\| = |\breve{\Zv}^\inn_{q0}| \stackrel{d}{=} |\breve{Z}|$
with $\breve{Z}\sim N(0,1)$. So the above concentrates by Lemma \ref{lem:normalconc}.

$\mathbf{Q_1.(b)}$ The first part of the proof of part $(b)$ follows as in $\mathbf{P_0}$.
Using Lemma \eqref{lem:joint_dists} part 2, we have $[\qv^{\inn/\out}_0]_i \overset{d}{=} [\widetilde{\qv}^{\inn/\out}_0]_i$; hence, by Lemma~\ref{sums} we have
\begin{align*}
  \P&\left( \left| \frac{1}{N}\sum_{i=1}^N\phi\left([\qv^{\inn}_0]_i,[\qv^{\out}_0]_i,[\wv^\inn_q]_i \right) - \mathbb{E}\phi\left( Q_0^\inn,Q_0^\out,W^\inn_q \right) \right|\geq \epsilon \right)\\
  &= \P\left( \left| \frac{1}{N}\sum_{i=1}^N\phi\left([\widetilde{\qv}^{\inn}_0]_i,[\widetilde{\qv}^{\out}_0]_i,[\wv^\inn_q]_i \right) - \mathbb{E}\phi\left( Q_0^\inn,Q_0^\out,W^\inn_q \right) \right|\geq \epsilon \right)\\
    &\leq \P\left( \left| \frac{1}{N}\sum_{i=1}^N\phi\left([\widetilde{\qv}^{\inn}_0]_i,[\widetilde{\qv}^{\out}_0]_i,[\wv^\inn_q]_i \right) - \frac{1}{N}\sum_{i=1}^N\phi\left([\ide{\qv}^{\inn}_0]_i,[\ide{\qv}^{\out}_0]_i,[\wv^\inn_q]_i \right) \right|\geq \frac{\epsilon}{2} \right)\\
  &\hspace{2cm}+ \P\left( \left| \frac{1}{N}\sum_{i=1}^N\phi\left([\ide{\qv}^{\inn}_0]_i,[\ide{\qv}^{\out}_0]_i,[\wv^\inn_q]_i \right) - \mathbb{E}\phi\left( Q_0^\inn,Q_0^\out,W^\inn_q \right) \right|\geq \frac{\epsilon}{2} \right).
\end{align*}

Now, using Lemma \ref{lem:PLsubgaussconc}/Hoeffding's inequality and Lemma \ref{lem:joint_dists} part 3, we can bound the second term above using exactly the same strategy as in the proof of $\mathbf{P_0.(b)}$. For the first term, if $\phi$ is pseudo-Lipschitz in all inputs, we can apply the pseudo-Lipschitz property and Lemma \ref{lem:joint_dists} part 2 to get
\begin{align*}
  \P&\left( \left| \frac{1}{N}\sum_{i=1}^N\phi\left([\widetilde{\qv}^{\inn}_0]_i,[\widetilde{\qv}^{\out}_0]_i,[\wv^\inn_q]_i \right) - \frac{1}{N}\sum_{i=1}^N\phi\left([\ide{\qv}^{\inn}_0]_i,[\ide{\qv}^{\out}_0]_i,[\wv^\inn_q]_i \right) \right|\geq \frac{\epsilon}{2} \right)\\
    &\leq \P\left( \frac{1}{N}\sum_{i=1}^N \left| \phi\left( [\widetilde{\qv}^\inn_0]_i,[\widetilde{\qv}^\out_0]_i,[\wv^\inn_q]_i \right)-\phi\left( [\ide{\qv}^\inn_0]_i,[\ide{\qv}^\out_0]_i,[\wv^\inn_q]_i \right) \right| \geq \frac{\epsilon}{2} \right)\\
    &\leq \P\left( \frac{L}{N}\sum_{i=1}^N[|[\Delm^\inn_{q0}]_i|+|[\Delm^\out_{q0}]_i|]\left[ 1 + 2|[\ide{\qv}^\inn_0]_i|+2|[\ide{\qv}^\out_0]_i| + 2|[\wv^\inn_q]_i| + |[\Delm^\inn_{q0}]_i| + |[\Delm^\out_{q0}]_i| \right] \geq \frac{\epsilon}{2}\right)\\
  &\leq \P\left( L\sum_{i=1}^N\left[\frac{|[\Delm^\inn_{q0}]_i|}{\sqrt{N}}+\frac{|[\Delm^\out_{q0}]_i|}{\sqrt{M}}\right]\left[ 1 + 2\frac{|[\ide{\qv}^\inn_0]_i|}{\sqrt{N}}+2\frac{|[\ide{\qv}^\out_0]_i|}{\sqrt{M}} + 2\frac{|[\wv^\inn_q]_i|}{\sqrt{N}} + \frac{|[\Delm^\inn_{q0}]_i|}{\sqrt{N}} + \frac{|[\Delm^\out_{q0}]_i|}{\sqrt{M}} \right] \geq \frac{\epsilon}{2}\right),
\end{align*}
where the last inequality follows from the fact that $M\leq N$. Now this term can be shown to concentrate in exactly the same way as in $\mathbf{P_0.(b)}$, namely applying Cauchy-Schwarz to re-express the above in terms of norms, then applying $\mathbf{Q_0.(a)}$ to the $\|\Delm^{\inn/\out}_{q0}\|$ terms, our concentration assumption to $\|\wv^q\|$, and Lemmas \ref{lem:joint_dists} part 3, \ref{sqroots}, and \ref{subexp} to the $\|\ide{\qv}^{\inn/\out}_0\|$ as in \eqref{eq:pstar_equality0}-\eqref{eq:pstar_equality}.

If instead $\phi$ satisfies the conditionally bounded pseudo-Lipschitz condition, then we can apply the analogous strategy to the functions $f_i = \phi\left( \cdot,\cdot,\left[ \wv_q^\inn \right]_i \right)$, which are themselves pseudo-Lipschitz with respect to a common constant (since these constants are a continuous function of the bounded $[\wv_q^\inn]_i$). The resulting bound drops the $2|[\wv_q^\inn]_i|/\sqrt{N}$ term, but is otherwise unchanged.

 Combining these inequalities, this completes the proof of $\mathbf{Q_0.(b)}$.

 $\mathbf{Q_1.(c)-(f)}$ For the inner product $[\uv_0^{\inn}]^T\qv_0^\inn$, note that $[\uv_0^{\inn}]_i=[\wv_q^{\inn}]_{2i}$ for $1\leq i\leq N$, and $(\qv^{\inn}_0,\qv^{\out}_0,\wv^\inn_q)\mapsto [\uv^{\inn}_0]^T\qv^{\inn}_0$ is pseudo-Lipschitz of order 2, so by $\mathbf{Q}_1.(b)$, we have that this concentrates to $\E\left[ Q_0U_0 \right] = \E[Q_0]\E[U_0]=0$ since $Q_0$ and $U_0$ are independent by hypothesis and $\E Q_0=0$.

 For the inner product $[\uv_0^{\inn}]^T\uv_1^{\inn}$, we expand the definition of $\uv_1^\inn$, obtaining
 \[
   \frac{1}{1-\alpha^{\inn}_{qk}}\left[ [\uv_0^{\inn}]^Tf_q^{\inn}\left( \qv^{\inn}_0,\qv^{\out}_0,\wv_q^{\inn},\gamma^{\inn}_{qk},\gamma^{\out}_{qk} \right) - \alpha^{\inn}_{qk}[\uv_0^{\inn}]^T\qv_0^{\inn}\right].
 \]
 Showing that this quantity concentrates around $\E U_0U_1$ uses the definition of $U_0$ and $U_1$ along with the previously esablished concentration of $\alpha^{\inn}_{qk}$ in $\mathbf{Q_1.(f)}$, the previously established concentration of $[\uv_0^{\inn}]^T\qv_0^\inn$ (above), and the concentration of
 $[\uv_0^{\inn}]^Tf_q^{\inn}\left( \qv^{\inn}_0,\qv^{\out}_0,\wv_q^{\inn},\gamma^{\inn}_{qk},\gamma^{\out}_{qk} \right),$
which follows from $\mathbf{Q_q.(b)}$ and the fact that the above is pseudo-Lipschitz in $(\qv^{\inn}_0,\qv^{\out}_0,\wv_q^{\inn})$. The details of the argument are entirely parallel to the details of proving $\mathbf{P_k.(e)}$, which we present explicitly as that is the more complicated case.
 
 The rest of these follow by applications of $\mathbf{Q_1.(b)}$ and concentration of $\alpha^{\inn/\out}_{q0}$ in the same way that $\mathbf{P_0.(c)}-\mathbf{(f)}$ were proved. 

$\mathbf{Q_1.(g).(i)}$ We will show that $\|[\Bm^\perp_{\Cm_{u0}^\inn}]^T\qv^\inn_0\|^2$ concentrates around $\rho^\inn_{q0}$. First, observe that
\begin{equation}
\|[\Bm^\perp_{\Cm_{u0}^\inn}]^T\qv^\inn_0\|^2 = [\qv^\inn_0]^T\qv^\inn_0 - \frac{([\uv^\inn_0]^T\qv^\inn_0)^2}{[\uv^\inn_0]^T\uv^\inn_0}.
\end{equation}
Because $[\qv^\inn_0]^T\qv^\inn_0 = [\vv^\inn_0]^T\vv^\inn_0$, using $\mathbf{P_0.(e)}$, $\mathbf{Q_1.(e)}$, and $\mathbf{Q_1.(f)}$, along with Lemmas \ref{sums} and \ref{products}, the above concentrates around $\mathbb{E}[V^\inn_0]^2 = \rho^\inn_{q0}$ as desired.

$\mathbf{Q_1.(g).(ii)}$ First observe that $[\Cm^\inn_{q1}]^T\Cm^\inn_{q1}$ is invertible if and only if
\[
\det([\Cm^\inn_{q1}]^T\Cm^\inn_{q1}) =([\uv^\inn_0]^T\uv^\inn_0)([\qv^\inn_0]^T\qv^\inn_0) - ([\uv^\inn_0]^T\qv^\inn_0) \neq 0.
\]
Using the fact that $\|\qv_0^{\inn/\out}\|=\|\vv_0^{\inn/\out}\|$, we have by \textbf{Conditon 0}, $\mathbf{P_0.(e)}$, $\mathbf{Q_1.(c)}$ and Lemmas \ref{sums} and \ref{products} that the above concentrates around $\covm^\inn_{u0}\covm^\inn_{v0}>0$. Now, by Lemma \ref{products}, we have that
\[
([\Cm^\inn_{q1}]^T\Cm^\inn_{q1})^{-1} = \frac{1}{\det([\Cm^\inn_{q1}]^T\Cm^\inn_{q1})}\left[\begin{matrix}[\qv^\inn_0]^T\qv^\inn_0 & -[\qv^\inn_0]^T\uv^\inn_0\\ -[\qv^\inn_0]^T\uv^\inn_0 & [\uv^\inn_0]^T\uv^\inn_0\end{matrix}\right],
\]
will concentrate around the desired constants
\[
\left[\begin{matrix}(\covm^\inn_{u0})^{-1} & 0\\ 0 & (\covm^\inn_{v0})^{-1}\end{matrix}\right].
\]

$\mathbf{Q_1.(g).(iii)}$ Recall that $\Cm^\inn_{q1}=[\uv^\inn_0\;\qv^\inn_0]$ by \eqref{eq:Cmat} and that $\Bm^\perp_{\Cm_{q1}^\inn} \in\mathbb{R}^{N\times N-2}$ is an orthogonal matrix with columns that form an orthonormal basis for $\mathrm{range}( \Cm^\inn_{q1} )^\perp=[\uv^\inn_0\;\qv^\inn_0]^\perp$. Thus, the matrix $[\Bm^\perp_{\Cm_{q1}^\inn}\; \Bm_{\Cm_{q1}^\inn}  ] \in\mathbb{R}^{N\times N}$ is orthogonal and
\be
\label{eq:Bperp1Q}
\|\uv^\inn_1\|^2 = \left\|\left[ \begin{matrix} [\Bm^\perp_{\Cm_{q1}^\inn}]^T\\ [\Bm_{\Cm_{q1}^\inn} ]^T\end{matrix} \right] \uv^\inn_1 \right\|^2 = \|[\Bm^\perp_{\Cm_{v0}^\inn}]^T \uv^\inn_1 \|^2 +  \|[\Bm_{\Cm_{v0}^\inn}]^T \uv^\inn_1 \|^2.
\ee
Rearranging, and noticing that $ \|[\Bm_{\Cm_{v0}^\inn}]^T \uv^\inn_1 \|^2 = [\uv^\inn_1]^T\Cm^\inn_{q1}([\Cm^\inn_{q1}]^T\Cm^\inn_{q1})^{-1}[\Cm^\inn_{q1}]^T\uv^\inn_1$, this gives us
\be
\label{eq:Bperp2Q}
  \|[\Bm^\perp_{\Cm_{q1}^\inn}]^T\uv^\inn_1\|^2 = \|\uv^\inn_1\|^2 - [\uv^\inn_1]^T\Cm^\inn_{q1}([\Cm^\inn_{q1}]^T\Cm^\inn_{q1})^{-1}[\Cm^\inn_{q1}]^T\uv^\inn_1.
\ee
Now $\|\uv^\inn_1\|^2$ concentrates around $(\covm^\inn_{u1})_{22}$ by $\mathbf{Q_1.(e)}$, and we have
\begin{align*}
&  [\uv^\inn_1]^T\Cm^\inn_{q1}([\Cm^\inn_{q1}]^T\Cm^\inn_{q1})^{-1}[\Cm^\inn_{q1}]^T\uv^\inn_1 =\\
  &  ([\uv^\inn_1]^T\uv^\inn_0)^2\frac{[\qv^\inn_0]^T\qv^\inn_0}{\det([\Cm^\inn_{q1}]^T\Cm^\inn_{q1})} + ([\uv^\inn_1]^T\qv^\inn_0)^2\frac{[\uv^\inn_0]^T\uv^\inn_0}{\det([\Cm^\inn_{q1}]^T\Cm^\inn_{q1})}- 2([\uv^\inn_1]^T\uv^\inn_0)([\uv^\inn_1]^T\qv^\inn_0)\frac{[\uv^\inn_0]^T\qv^\inn_0}{\det([\Cm^\inn_{q1}]^T\Cm^\inn_{q1})}.
\end{align*}
Label these terms $T_1-T_3$. We first notice that 
\[
\det\left([\Cm^\inn_{q1}]^T\Cm^\inn_{q1}\right) = \|\qv^\inn_0\|^2 \|\uv^\inn_0\|^2-([\qv^\inn_0]^T\uv^\inn_0)^2.
\]
By $\mathbf{Q_1.(d)(e)(f)}$, with \textbf{Condition 0} and Lemmas~\ref{sums} and \ref{products}, this concentrates to $(\covm^\inn_{q1})_{22}(\covm^\inn_{u1})_{11}>0$.

Then $T_1$ concentrates around $\frac{(\covm^\inn_{u1})_{12}^2}{(\covm^\inn_{u1})_{11}}$ by the above, $\mathbf{Q_1.(e)}$, $\mathbf{Q_1.(d)}$, and Lemma \ref{products}. Next notice that by the above, $\mathbf{Q_1.(f)}$, and Lemma \ref{products}, we have that $T_2$ concentrates around zero. And similarly by the above, $\mathbf{Q_1.(e)}$, $\mathbf{Q_1.(f)}$, and Lemma \ref{products}, $T_3$ concentrates around zero. So we ultimately get that $\|[\Bm^\perp_{\Cm_{q1}^\inn}]^T\uv^\inn_1\|^2 $ concentrates around
$(\covm^\inn_{u1})_{11}-\frac{(\covm^\inn_{u1})_{01}^2}{(\covm^\inn_{u1})_{00}} = \rho^\inn_{p1},$
where the equality follows using  \eqref{eq:bvecs}, \eqref{eq:cov_def},  and \eqref{eq:rhos} and which say
\[ \rho^\inn_{p1} =  \mathbb{E}\{[U^{\inn}_1]^2\} - (\bv^{\inn}_{u1})^T[\covm^{\inn}_{u0}]^{-1}\bv^{\inn}_{u1}  = \mathbb{E}\{[U^{\inn}_1]^2\} -\frac{(\mathbb{E}[U^{\inn}_0U^{\inn}_1])^2}{\covm^{\inn}_{u0}} = [\covm^{\inn}_{u1}]_{11} -\frac{ [\covm^{\inn}_{u1}]_{10}^2}{ [\covm^{\inn}_{u1}]_{00}}. \]


$\mathbf{Q_1.(g).(iv)}$ Notice that by \eqref{eq:Rk_Sk} for $k \geq 1$, we have $[\presm^\inn_{uk}]^{-1}=\left[\begin{matrix}\covm^\inn_{uk} & 0\\ 0 & \covm^\inn_{v(k-1)}\end{matrix}\right]$, so in particular,  $[\presm^\inn_{u1}]^{-1}=\left[\begin{matrix}\covm^\inn_{u1} & 0\\ 0 & \covm^\inn_{v0}\end{matrix}\right] $, where these matrices are defined in \eqref{eq:cov_def}.
Next, define $\Fm:=[\ev_2\; \ev_1\; \ev_3] \in \mathbb{R}^{3 \times 3}$, where $\ev_i$ is the $i^{th}$ canonical vector, namely the vector of zeros with a one in the $i^{th}$ position, so that for a matrix $\Am$, the transformation $\Am\Fm$ swaps the second and first column of $\Am$. Then, using that  $\bv^\inn_{u1} = \mathbb{E}[U^{\inn}_0U^{\inn}_1]$ by \eqref{eq:bvecs}, 
\ben
\begin{split}
\Fm^T[\presm^\inn_{u1}]^{-1}\Fm 
&= \left[\begin{matrix}\mathbb{E}[(U^{\inn}_1)^2] & [\bv^\inn_{u1}\;0]\\ \left[\begin{matrix}\bv^\inn_{u1}\\0\end{matrix}\right]& [\presm^\inn_{q0}]^{-1}\end{matrix}\right].
\end{split}
\een
Then, using $\Fm = \Fm^T  = \Fm^{-1},$ so that $(\Fm^T[\presm^\inn_{u0}]^{-1}\Fm)^{-1} = \Fm \presm^\inn_{u0} \Fm^T,$ the block matrix inversion formula gives
\be
\label{eqn:Fmatrix1}
\Fm\presm^\inn_{u1}\Fm^T = \left[\begin{matrix}0 & 0 \\ 0 & \presm^\inn_{q0}\end{matrix}\right] + \frac{1}{\rho^\inn_{p1}}\left[\begin{matrix}
1 & -[\betav^\inn_{p1}\; 0]\\-\left[\begin{matrix}\betav^\inn_{p1}\\0\end{matrix}\right] &\left[\begin{matrix}\betav^\inn_{p1}\\0\end{matrix}\right][\betav^\inn_{p1}\; 0] 
\end{matrix}\right].
\ee
In the above, we have used that $\rho^\inn_{p1} = \mathbb{E}[(U^{\inn}_1)^2] - (\bv^\inn_{u1})^2/\mathbb{E}[(U^{\inn}_0)^2] = \mathbb{E}[(U^{\inn}_1)^2] - [\bv^\inn_{u1}\;0] \presm^\inn_{q0}  [\bv^\inn_{u1}\;0]^T$  and $\betav^\inn_{p1} = \bv^\inn_{u1}/\mathbb{E}[(U^{\inn}_0)^2] $ by \eqref{eq:rhos}.

Similarly, using that $\Cm^\inn_{u1} = [\uv^\inn_0 \, \uv^\inn_1 \, \qv^\inn_0]$ and $\Cm^\inn_{q1} = [\uv^\inn_0 \, \qv^\inn_0]$, we get
\[
\Fm^T([\Cm^\inn_{u1}]^T\Cm^\inn_{u1})F = \left[\begin{matrix}
[\uv^\inn_1]^T\uv^\inn_1 & [\uv^\inn_1]^T \Cm^\inn_{q1} \\ [\Cm^\inn_{q1}]^T \uv^\inn_1 & [\Cm^\inn_{q1}]^T\Cm^\inn_{q1}
\end{matrix}\right].
\]
Letting $\boldsymbol{\nu} = ([\Cm^\inn_{q1}]^T\Cm^\inn_{q1})^{-1}[\Cm^\inn_{q1}]^T\uv^\inn_1$, again by block matrix inversion,
formula gives us
\be
\label{eqn:Fmatrix2}
\Fm^T([\Cm^\inn_{u1}]^T\Cm^\inn_{u1})^{-1}\Fm = \left[\begin{matrix}0 & 0 \\ 0 & ([\Cm^\inn_{q1}]^T\Cm^\inn_{q1})^{-1}\end{matrix}\right]+ \frac{1}{\|[\Bm^\perp_{\Cm_{q1}^\inn}]^T \uv^\inn_1\|^2}\left[\begin{matrix}
1 & -\boldsymbol{\nu}^T\\ -\boldsymbol{\nu} & -\boldsymbol{\nu}\boldsymbol{\nu}^T
\end{matrix}\right],
\ee
where we use that $\|[\Bm^\perp_{\Cm_{q1}^\inn}]^T \uv^\inn_1\|^2 = \|\uv^\inn_1\|^2 -\|\Bm_{\Cm_{q1}^\inn}\uv^\inn_1\|^2 = \|\uv^\inn_1\|^2 - [\uv^\inn_1]^T \Cm_{q1}^\inn ([\Cm_{q1}^\inn]^T\Cm_{q1}^\inn)^{-1}[\Cm_{q1}^\inn]^T \uv^\inn_1$, using an argument similar to that in \eqref{eq:Bperp1Q}-\eqref{eq:Bperp2Q}.

Now our goal is to show that, elementwise, \eqref{eqn:Fmatrix2} concentrates on \eqref{eqn:Fmatrix1}.
By Lemmas~\ref{sums} and \ref{products}, it suffices to show three results. Namely, that $([\Cm^\inn_{q1}]^T\Cm^\inn_{q1})^{-1}$ concentrates on $\presm^\inn_{v0}$, that $\|\Bm^\perp_{\Cm^\inn_{q1}} \uv^\inn_1\|^{-2}$ concentrates on $[\rho^\inn_{p1}]^{-1}$, and that $\boldsymbol{\nu}$ concentrates on $[\betav^\inn_{p1}\;0]^T$. Elementwise concentration for $([\Cm^\inn_{q1}]^T\Cm^\inn_{q1})^{-1}$ follows from $\mathbf{Q_{1}.(g).(i)}$, and concentration for $\|[\Bm^\perp_{\Cm_{q1}^\inn}]^T\uv^\inn_1\|^{-2}$ follows from $\mathbf{Q_1.(g).(ii)}$ along with Lemma~\ref{inverses}. Note that the $\rho$ terms are bounded above by our requirement that the state evolution not diverge in \textbf{Condition 5}.

Finally, let $\widetilde{\Cm}^\inn_{q1} = ([\Cm^\inn_{q1}]^T\Cm^\inn_{q1})^{-1}$ and observe that
\be
\label{eq:nu}
\boldsymbol{\nu} = \widetilde{\Cm}^\inn_{q1}\left[\begin{matrix}[\uv^\inn_0]^T\uv^\inn_1\\ [\qv^\inn_0]^T\uv^\inn_1\end{matrix}\right] = \left[\begin{matrix} [\widetilde{\Cm}^\inn_{q1}]_{11}([\uv^\inn_0]^T\uv^\inn_1) + [\widetilde{\Cm}^\inn_{q1}]_{12}([\qv^\inn_0]^T\uv^\inn_1)\\
[\widetilde{\Cm}^\inn_{q1}]_{21}([\uv^\inn_0]^T\uv^\inn_1) + [\widetilde{\Cm}^\inn_{q1}]_{22}([\qv^\inn_0]^T\uv^\inn_1)\end{matrix}\right].
\ee
Now we show concentration for the first element of  \eqref{eq:nu}. By $\mathbf{Q_1.(e)}$, $[\uv^\inn_0]^T\uv^\inn_1$ concentrates around $\mathbb{E}[U^\inn_0U^\inn_1]$ and by $\mathbf{Q_1.(g).(ii)}$, we have that $[\widetilde{\Cm}^\inn_{q1}]_{11}$ concentrates around $1/{\mathbb{E}[(U^\inn_0)^2]}$, so by Lemma \ref{products}, we have that $[\widetilde{\Cm}^\inn_{q1}]_{11}([\uv^\inn_0]^T\uv^\inn_1)$ concentrates around $\frac{\mathbb{E}[U^\inn_0U^\inn_1]}{\mathbb{E}[(U^\inn_0)^2]}=\betav^\inn_{q1}$ by \eqref{eq:rhos}. Similarly, we have that $[\widetilde{\Cm}^\inn_{q1}]_{12}$ concentrates around $0$ and $[\qv^\inn_0]^T\uv^\inn_1$ concentrates around $0$ by $\mathbf{Q_1.(f)}$. Therefore, using  so again by Lemmas \ref{sums}, we get that $[\widetilde{\Cm}^\inn_{q1}]_{11}([\uv^\inn_0]^T\uv^\inn_1) + [\widetilde{\Cm}^\inn_{q1}]_{12}([\qv^\inn_0]^T\uv^\inn_1)$ concentrates around $\betav^\inn_{q1}$. 

Noting that $[\widetilde{\Cm}^\inn_{q1}]_{21}$ and $[\qv^\inn_0]^T\uv^\inn_1$ concentrate around $0$ and proceeding similarly gives that $[\widetilde{\Cm}^\inn_{q1}]_{21}([\uv^\inn_0]^T\uv^\inn_1) + [\widetilde{\Cm}^\inn_{q1}]_{22}([\qv^\inn_0]^T\uv^\inn_1)$ concentrates around $0$. This completes the proof that the entries of $([\Cm^\inn_{u1}]^T\Cm^\inn_{u1})^{-1}$ concentrate around the entries of $\presm^\inn_{u1}$.

 \subsection{Showing $\mathbf{P_k}$ Holds} \label{sec:stepPK}
First we prove the following lemma.
\begin{lem}
For all $1\leq i\leq 2k$ and for $k \geq 1$,
\[
\P\left(\left|\left[([\Cm^\inn_{qk}]^T\Cm^\inn_{qk})^{-1}[\Cm^\inn_{qk}]^T\uv^\inn_k\right]_i-\left[\begin{matrix}\betav^\inn_{pk}\\ 0\end{matrix}\right]_i\right|\geq\epsilon\right)\leq CkC_{k-1}\exp\left(-ncc_{k-1}\epsilon^2/k^2\right).
\]
\label{lem:first_conc_lemma}
\end{lem}
\begin{proof} 
Let $\widetilde{\Cm}^\inn_k = ([\Cm^\inn_{qk}]^T\Cm^\inn_{qk})^{-1}$ and recall by \eqref{eq:Cmat} that $\Cm^\inn_{qk} = [ \Um_{k-1}^\inn \, \Qm_{k-1}^\inn]$. Then, for $1\leq i\leq 2k$,
\begin{equation}\label{eq:Bk_lem_1}
\left|\left[([\Cm^\inn_{qk}]^T\Cm^\inn_{qk})^{-1}[\Cm^\inn_{qk}]^T\uv^\inn_k\right]_i - \left[\begin{matrix}\betav^\inn_{pk}\\ 0\end{matrix}\right]_i\right| \leq \left|\sum_{j=1}^{k} [\widetilde{\Cm}^\inn_k]_{ij}([\uv^\inn_{j-1}]^T\uv^\inn_k) - [\betav^\inn_{pk}]_i\right|+\left|\sum_{j=1}^{k}[\widetilde{\Cm}^\inn_k]_{i,j+k}([\qv^\inn_{j-1}]^T\uv^\inn_k)\right|.
\end{equation}
Now by $\mathbf{P_{k-1}.(g).(ii)}$, we know that the entries of $\widetilde{\Cm}^\inn_k$ concentrate on finite values. In particular, we know that if $1\leq i,j,\leq k$, then $[\widetilde{\Cm}^\inn_k]_{ij}$ concentrates on  $[\presm^{\inn}_{q,k-1}]_{ij} = \left[\left(\covm^\inn_{u(k-1)}\right)^{-1}\right]_{ij}$ where the equality follows by \eqref{eq:Rk_Sk}. Moreover, by $\mathbf{Q_{k}.(e)}$ and $\mathbf{Q_{k}.(f)}$, we have that $[\uv^\inn_{j-1}]^T\uv^\inn_k$ and $[\qv^\inn_{j-1}]^T\uv^\inn_k$ concentrate on $ [\covm^{\inn}_{u,k}]_{j,k+1} = \mathbb{E}[U^{\inn}_{j-1} U^{\inn}_{k}]  =[\bv^\inn_{uk}]_{j}$ and $0$, respectively. 

Therefore, using Lemmas \ref{sums} and \ref{products}, $\sum_{j=1}^{k} [\widetilde{\Cm}^\inn_k]_{ij}([\uv^\inn_{j-1}]^T\uv^\inn_k)$ concentrates on
$\sum_{j=1}^{k} [(\covm^\inn_{u(k-1)})^{-1}]_{ij}(\bv_{uk}^\inn)_{j} = \left[(\covm^\inn_{u(k-1)})^{-1}(\bv_{uk}^\inn)\right]_i = [\betav^\inn_{pk}]_i$
and $\sum_{j=1}^{k}[\widetilde{\Cm}^\inn_k]_{i,j+k}([\qv^\inn_{j-1}]^T\uv^\inn_k)$ concentrates on $0$.

Now if $k+1\leq i\leq 2k$, then we can apply the same analysis, but now for $1\leq j\leq k$, we get $[\widetilde{\Cm}^\inn_k]_{i,j}$ concentrating on $0$, so $\left[([\Cm^\inn_{qk}]^T\Cm^\inn_{qk})^{-1}[\Cm^\inn_{qk}]^T\uv^\inn_k\right]_i$ concentrates on $0$, as needed.
\end{proof}


$\mathbf{P_k.(a)}$
Letting $\muv^\inn_{pk} = ([\Cm^\inn_{vk}]^T\Cm^\inn_{vk})^{-1}[\Cm^\inn_{qk}]^T\uv^\inn_k$, recall by Lemma~\ref{lem:cond_dist}, that we have
\begin{align*}
\Delm^\inn_{pk} &= \Cm^\inn_{vk}\left(\muv^\inn_{pk}-\left[\begin{matrix}\betav^\inn_{pk}\\0\end{matrix}\right]\right) +\left(\frac{\|[\Bm^\perp_{\Cm_{qk}^\inn}]^T\uv^\inn_k\|}{\|\Zv^\inn_{pk}\|}-\sqrt{\rho^\inn_{pk}}\right)\Bm^\perp_{\Cm_{vk}^\inn}\Zv^\inn_{pk} - \sqrt{\rho^\inn_{pk}} \Bm_{\Cm_{vk}^\inn}\breve{\Zv}^\inn_{pk}\\
&=\sum_{i=1}^{k} \left([\muv^\inn_{pk}]_i-[\betav^\inn_{pk}]_i\right)[\Cm^\inn_{vk}]_{(\cdot, i)}+\sum_{i=k+1}^{2k} [\muv^\inn_{pk}]_i[\Cm^\inn_{vk}]_{(\cdot, i)} \\
&\qquad \qquad +\left(\frac{\|[\Bm^\perp_{\Cm_{qk}^\inn}]^T\uv^\inn_k\|}{\|\Zv^\inn_{pk}\|}-\sqrt{\rho^\inn_{pk}}\right)\Bm^\perp_{\Cm_{vk}^\inn}\Zv^\inn_{pk} - \sqrt{\rho^\inn_{pk}} \Bm_{\Cm_{vk}^\inn}\breve{\Zv}_{pk},
\end{align*}
where we have used the notation $[\Cm^\inn_{vk}]_{(\cdot, i)}$ to indicate the $i^{th}$ column of the matrix $\Cm^\inn_{vk}$.

Therefore, by the triangle inequality, we find
\begin{align*}
\|\Delm_{pk}^\inn\| \leq &\sum_{i=1}^{k} \left\lvert [\muv^\inn_{pk}]_i-[\betav^\inn_{pk}]_i\right \lvert \|[\Cm^\inn_{vk}]_{(\cdot, i)}\| +\sum_{i=k+1}^{2k} \left \lvert [\muv^\inn_{pk}]_i\right \lvert \|[\Cm^\inn_{vk}]_{(\cdot, i)}\|\\
&\qquad \qquad +\left \lvert \frac{\|[\Bm^\perp_{\Cm_{qk}^\inn}]^T\uv^\inn_k\|}{\|\Zv^\inn_{pk}\|}-\sqrt{\rho^\inn_{pk}}\right \lvert \|\Zv^\inn_{pk}\| + \rho^\inn_{pk}  \|\breve{\Zv}^\inn_{pk}\|,
\end{align*}
where we have used that $[\Bm^\perp_{\Cm_{vk}^\inn}]^T\Bm^\perp_{\Cm_{vk}^\inn}= [\Bm_{\Cm_{vk}^\inn}]^T\Bm_{\Cm_{vk}^\inn}=\Idm$.
Then, using Lemma~\ref{sums}, we find
\begin{align}\label{eq:Bk_a_1}
\P\left(\frac{\|\Delm^\inn_{pk}\|}{\sqrt{N}}\geq \sqrt{\epsilon}\right) \leq& \sum_{i=1}^{k} \P\left(\left|[\muv^\inn_{pk}]_i-[\betav^\inn_{pk}]_i\right|\frac{\|[\Cm^\inn_{vk}]_{(\cdot, i)}\|}{\sqrt{N}}\geq \frac{\sqrt{\epsilon}}{4k}\right)+\sum_{i=k+1}^{2k}\P\left(\left|[\muv^\inn_{pk}]_i\right|\frac{\|[\Cm^\inn_{vk}]_{(\cdot, i)}\|}{\sqrt{N}}\geq \frac{\sqrt{\epsilon}}{4k}\right)\\\nonumber
  &+\P\left(\left|\frac{\|[\Bm^\perp_{\Cm_{qk}^\inn}]^T\uv^\inn_k\|}{\|\Zv^\inn_{pk}\|}-\sqrt{\rho^\inn_{pk}}\right|\frac{\|\Zv^\inn_{pk}\|}{\sqrt{N}}\geq \frac{\sqrt{\epsilon}}{4k}\right)+\P\left( \frac{\|\breve{\Zv}^\inn_{pk}\|}{\sqrt{N}} \geq\frac{\sqrt{\epsilon}}{4k \sqrt{\rho^\inn_{pk}}  }\right).
\end{align}
In what follows, we show upper bound for each of the terms on the right side of \eqref{eq:Bk_a_1}.

\textbf{First term of \eqref{eq:Bk_a_1}.} Observe that by the definition of $\Cm^\inn_{vk}$ in \eqref{eq:Cmat}, for $1\leq i\leq k$, that $\|[\Cm^\inn_{vk}]_{(\cdot, i)}\| = \|\pv^\inn_{i-1}\|$. Then for $0\leq j\leq k-1$,
\be
\begin{split}
\label{eq:Bk_a_2}
&\P\left(\left|[\muv^\inn_{pk}]_{j+1}-[\betav^\inn_{pk}]_{j+1}\right|\frac{\|[\Cm^\inn_{vk}]_{(\cdot, j+1)}\|}{\sqrt{N}}\geq \frac{\sqrt{\epsilon}}{4k}\right) \\
&\qquad =\P\left(\left|[\muv^\inn_{pk}]_{j+1}-[\betav^\inn_{pk}]_{j+1}\right|\left[\left|\frac{\|\pv^\inn_{j}\|}{\sqrt{N}}-(\covm^\inn_{uk})_{jj}\right|+(\covm^\inn_{uk})_{jj}\right]\geq \frac{\sqrt{\epsilon}}{4k}\right)\\
&\qquad \stackrel{(a)}{\leq} \P\left(\left|\frac{\|\pv^\inn_{j}\|}{\sqrt{N}}-(\covm^\inn_{uk})_{jj}\right|\geq\sqrt{\epsilon}\right)+\P\left(\left|[\muv^\inn_{pk}]_{j+1}-[\betav^\inn_{pk}]_{j+1}\right|\geq \frac{\sqrt{\epsilon}}{8k \max\{1, (\covm^\inn_{uk})_{jj}\}}\right).
\end{split}
\ee
The first term concentrates by $\mathbf{P_j.(d)}$ ($0 \leq j\leq k-1$) and the fact that
\[
\E \left[P_i^{\inn/\out}P_j^{\inn/\out}\right] = \lim_{N\to\infty} \left[\pv_i^{\inn/\out}\right]^T\pv_j^{\inn/\out}= \lim_{N\to\infty} \left[\uv_i^{\inn/\out}\right]^T\uv_j^{\inn/\out} = \E \left[U_i^{\inn/\out}U_j^{\inn/\out} \right],
\]
where the first and last equalities follow from $\mathbf{P_{k-1}.(d)}$ and $\mathbf{Q_{k-1}.(e)}$, and the middle equality follows from the fact that $\pv_i^{\inn/\out} \Vm \uv_i^{\inn/\out}$.

Then the second term concentrates by Lemma~\ref{lem:first_conc_lemma}. We note that the above also implies that $ (\covm^\inn_{uk})_{jj} = \tau^{\inn}_{pj}$, which is upper bounded independently of $N$ and $j$ by \textbf{Condition 5}. 

\textbf{Second term of \eqref{eq:Bk_a_1}.}  For $k+1 \leq i\leq 2k$, we have $\|[\Cm^\inn_{vk}]_{(\cdot, i)}\| = \|\vv^\inn_{i-(k+1)}\|$. Letting $i'=i-(k+1)$,
\begin{align}\label{eq:Bk_a_3}
\P\left(\left|[\muv^\inn_{pk}]_i\right|\frac{\|[\Cm^\inn_{vk}]_{(\cdot, i)}\|}{\sqrt{N}}\geq \frac{\sqrt{\epsilon}}{4k}\right)&\leq \P\left(\left|[\muv^\inn_{pk}]_i\right|\left[\left|\frac{\|\vv^\inn_{i'}\|}{\sqrt{N}}-(\covm^\inn_{vk})_{i'i'}\right|+(\covm^\inn_{vk})_{i'i'}\right]\geq \frac{\sqrt{\epsilon}}{4k}\right)\\\nonumber
&\stackrel{(b)}{\leq} \P\left(\left|\frac{\|\vv^\inn_{i'}\|}{\sqrt{N}}-(\covm^\inn_{vk})_{i'i'}\right|\geq\sqrt{\epsilon}\right)+\P\left(\left|[\muv^\inn_{pk}]_i\right|\geq \frac{\sqrt{\epsilon}}{8k\max\{1, (\covm^\inn_{uk})_{i'i'}\}}\right).
\end{align}
Again, the first term concentrates by $\mathbf{P_{i'}.(e)}$ where $0 \leq i' \leq k-1$, and the second term by Lemma~\ref{lem:first_conc_lemma}.

\textbf{Third term of \eqref{eq:Bk_a_1}.} By Lemma~\ref{products_0}, we find
\be
\begin{split}\label{eq:Bk_a_4}
&P\left(\left|\frac{\|[\Bm^\perp_{\Cm_{qk}^\inn}]^T\uv^\inn_k\|}{\|\Zv^\inn_{pk}\|}-\sqrt{\rho^\inn_{pk}}\right|\frac{\|\Zv^\inn_{pk}\|}{\sqrt{N}}\geq \frac{\sqrt{\epsilon}}{4k}\right) \\& \qquad \leq \P\left( \left|\frac{\|[\Bm^\perp_{\Cm_{qk}^\inn}]^T\uv^\inn_k\|}{\sqrt{N}}-\sqrt{\rho^\inn_{pk}}\right|\geq \frac{\sqrt{\epsilon}}{8k}\right) + \P\left( \left|\frac{\|\Zv^\inn_{pk}\|}{\sqrt{N}}-1\right|\geq\frac{\sqrt{\epsilon}}{8k\max\{1, \sqrt{\rho^\inn_{pk}}\}}\right).
\end{split}
\ee
The second term concentrates by Lemmas \ref{subexp} and \ref{sqroots}, and the first term by $\mathbf{Q_k.(g).(iii)}$.

\textbf{Fourth term of \eqref{eq:Bk_a_1}.}  Finally, recall from Lemma~\ref{lem:cond_dist} that $\breve{\Zv}^\inn_{pk}$ is a length-$2k$ vector. Thus,
\begin{align}\label{eq:Bk_a_5}
  \P\left( \frac{\|\breve{\Zv}^\inn_{pk}\|}{\sqrt{N}} \geq\frac{\sqrt{\epsilon}}{4k}\right)&= \P\left( \frac{1}{N}\sum_{i=1}^{2k}[\breve{\Zv}^\inn_{pk}]^2_i \geq\frac{\epsilon}{16k^2} \right) \leq \sum_{i=1}^{2k}\P\left( [\breve{\Zv}^\inn_{pk}]_i^2\geq \frac{N\epsilon}{32k^3} \right).                                                                                
\end{align}
Now the last expression concentrates by Lemma \ref{lem:normalconc}.


$\mathbf{P_k.(b)}$
In what follows, the notation $\underline{x}_k$ means $\left( x_0, x_1,\ldots,x_k \right)$  (e.g., $\underline{P}^{\inn/\out}_k = (P^{\inn/\out}_0,\ldots,P^{\inn/\out}_k)$ and $\left[\underline{\pv}^{\inn/\out}\right]_i = ([\pv^{\inn/\out}_0]_i,\ldots,[\pv^{\inn/\out}_k]_i)$).
Using Lemma \ref{lem:joint_dists} result \eqref{eq:lemma_eq_dist} and the triangle inequality,
\be
\begin{split}
  &\left|\frac{1}{N}\sum_{i=1}^N\phi\left(\left[\underline{\pv}^\inn_k\right]_i,\left[\underline{\pv}^\out_k\right]_i,[\wv]_i\right)-\mathbb{E}\left\{\phi(\underline{P}_k^\inn,\underline{P}_k^\out,W)\right\}\right| \\
  &\stackrel{d}{=}\left|\frac{1}{N}\sum_{i=1}^N\phi\left(\left[\underline{\widetilde{\pv}}^\inn_k\right]_i,\left[\underline{\widetilde{\pv}}^\out_k\right]_i,[\wv]_i\right) -\mathbb{E}\left\{\phi(\underline{P}_k^\inn,\underline{P}_k^\out,W)\right\}\right| \\
  &\leq \left|\frac{1}{N}\sum_{i=1}^N\phi\left( \left[\underline{\ide{\pv}}^\inn_k\right]_i,\left[\underline{\ide{\pv}}^\out_k\right]_i,[\wv]_i\right)-\mathbb{E}\left\{\phi(\underline{P}_k^\inn,\underline{P}_k^\out,W)\right\}\right|\\
  &\hspace{2cm}+\frac{1}{N}\sum_{i=1}^N\left|\phi\left(\left[\underline{\widetilde{\pv}}^\inn_k\right]_i,\left[\underline{\widetilde{\pv}}^\out_k\right]_i,[\wv]_i\right)-\phi\left(\left[\underline{\ide{\pv}}^\inn_k\right]_i,\left[\underline{\ide{\pv}}^\out_k\right]_i,[\wv]_i\right)\right|.
  \label{eq:Pkb_term1}
\end{split}
\ee
We label the two terms on the right side of \eqref{eq:Pkb_term1} as T1 and T2. First, we establish concentration for T1.  Wwe modify our usual convention and take $[\underline{\ide{\pv}}^\out_k]_i$ for $i>M$ to be i.i.d.\ copies of $\underline{P}^\out_k$. We can do this since $\phi(\cdot,\cdot,[\wv]_i)$ does not depend on its second argument for $i> M$. This modification implies that
$\left\lbrace \left(\left[\ide{\pv}^{\inn}_0\right]_i,\ldots,\left[\ide{\pv}^\inn_k\right]_i,\left[\ide{\pv}^{\out}_0\right]_i,\ldots,\left[\ide{\pv}^\out_k\right]_i\right)\right\rbrace_{i\geq 1}$
are independent samples from a common multivariate normal distribution by part 3 of Lemma \ref{lem:joint_dists}. If $\phi$ is jointly pseudo-Lipschitz in all inputs, then concentration for T1 already follows directly by Lemma \ref{lem:PLsubgaussconc}. If instead $\phi$ satisfies the conditionally bounded pseudo-Lipschitz property, we further decompose T1 as
\begin{align*}
  &\left|\frac{1}{N}\sum_{i=1}^N\phi\left( \left[\underline{\ide{\pv}}^\inn_k\right]_i,\left[\underline{\ide{\pv}}^\out_k\right]_i,[\wv]_i\right)-\mathbb{E}\left\{\phi(\underline{P}_k^\inn,\underline{P}_k^\out,W)\right\}\right|\\
  &\hspace{1cm} \leq \left|\frac{1}{N}\sum_{i=1}^N\phi\left( \left[\underline{\ide{\pv}}^\inn_k\right]_i,\left[\underline{\ide{\pv}}^\out_k\right]_i,[\wv]_i\right)-\mathbb{E}_P\left\{\phi(\underline{P}_k^\inn,\underline{P}_k^\out, [\wv]_i)\right\}\right|\\
  &\hspace{2cm} +\left|\frac{1}{N}\sum_{i=1}^N\mathbb{E}_P\left\{\phi(\underline{P}_k^\inn,\underline{P}_k^\out,[\wv]_i)\right\} -\mathbb{E}\left\{\phi(\underline{P}_k^\inn,\underline{P}_k^\out,W)\right\}\right|,
\end{align*}
where $\mathbb{E}_P$ is the expectation with respect to the random variables $(\underline{P}^\inn_k,\underline{P}^\out_k)$. We label these terms T1a and T1b.

Now, taking $f_i = \phi(\cdot,\cdot,[\wv]_i)$, Lemma \ref{lem:PLsubgaussconc} will give the desired concentration for the first term so long as the $f_i$ are pseudo-Lipschitz with a common constant. But by the bounded conditionally pseudo-Lischitz condition, the PL constant of $f_i$ is continuous in $[\wv]_i$. Since these are themselves bounded, the PL constants must also be bounded by some $B$, which we can take as our common constant.

To get concentration for T1b, we note that $x\mapsto \mathbb{E}_P \left\{\phi(\underline{P}^\inn_k,\underline{P}^\out_k,x)\right\}$ is continuous in $x$ (since $\phi$ is continuous in all inputs), and is thus bounded over the domain of $\wv$ since this domain is compact by definition. Thus, the terms in T1b are bounded and i.i.d., so concentration follows with desired rate by a direct application of Hoeffding's inequality.

Now we bound T2 of \eqref{eq:Pkb_term1}. First, the pseudo-Lipschitz property of $\phi$ gives the bound
\begin{align*}
&\left|\phi\left(\left[\underline{\widetilde{\pv}}^{\inn}_k\right]_i,\left[\underline{\widetilde{\pv}}^\out_k\right]_i,[\wv]_i\right)-\phi\left(\left[\underline{\ide{\pv}}^\inn_k\right]_i,\left[\underline{\ide{\pv}}^\out_k\right]_i,[\wv]_i\right)\right| \\
&\leq L\left[ 1 + \norm{\left(\left[\underline{\widetilde{\pv}}^{\inn}_k\right]_i,\left[\underline{\widetilde{\pv}}^\out_k\right]_i,[\wv]_i\right)} + \norm{\left(\left[\underline{\ide{\pv}}^\inn_k\right]_i,\left[\underline{\ide{\pv}}^\out_k\right]_i,[\wv]_i\right)}\right] \norm{\left(\left[\underline{\widetilde{\pv}}^{\inn}_k - \underline{\ide{\pv}}^\inn_k\right]_i, \left[\underline{\widetilde{\pv}}^\out_k - \underline{\ide{\pv}}^\out_k \right]_i\right)} \\
&\leq L\left[ 1 + 2\norm{\left(\left[\underline{\widetilde{\pv}}^{\inn}_k\right]_i,\left[\underline{\widetilde{\pv}}^\out_k\right]_i,[\wv]_i\right)} + \norm{\left[\underline{\widetilde{\pv}}^{\inn}_k - \underline{\ide{\pv}}^\inn_k \right]_i} + \norm{\left[\underline{\widetilde{\pv}}^\out_k - \underline{\ide{\pv}}^\out_k\right]_i}\right] \left[ \norm{\left[\underline{\widetilde{\pv}}^{\inn}_k - \underline{\ide{\pv}}^\inn_k\right]_i}  +\norm{ \left[\underline{\widetilde{\pv}}^\out_k - \underline{\ide{\pv}}^\out_k\right]_i}\right],
\end{align*}
where the final inequality uses that $\norm{(\underline{a}, \underline{b})} \leq \norm{\underline{a}} + \norm{\underline{b}}$ and that $$\norm{(\underline{a}, \underline{b}, w)} = \norm{(\underline{a}, \underline{b}, w) - (\tilde{\underline{a}}, \tilde{\underline{b}}, w) + (\tilde{\underline{a}}, \tilde{\underline{b}}, w)} \leq \norm{(\underline{a} - \tilde{\underline{a}}, \underline{b} - \tilde{\underline{b}})} + \norm{(\tilde{\underline{a}}, \tilde{\underline{b}}, w)}.$$

Next, applying the above bound, Cauchy-Schwarz, and Lemma~\ref{lem:squaredsums}, we get
\begin{align*}
&[T2]^2 = \left[\frac{1}{N}\sum_{i=1}^N\left|\phi\left(\left[\underline{\widetilde{\pv}}^{\inn}_k\right]_i,\left[\underline{\widetilde{\pv}}^\out_k\right]_i,[\wv]_i\right)-\phi\left(\left[\underline{\ide{\pv}}^\inn_k\right]_i,\left[\underline{\ide{\pv}}^\out_k\right]_i,[\wv]_i\right)\right|\right]^2\\
  &\leq \Big[\frac{L}{N}\sum_{i=1}^N \left(1+2\|(\left[\underline{\ide{\pv}}^\inn_k\right]_i,\left[\underline{\ide{\pv}}^\out_k\right]_i,[\wv]_i)\|+\|[\underline{\widetilde{\pv}}^\inn_k-\underline{\ide{\pv}}^\inn_k]_i\|+\|[\underline{\widetilde{\pv}}^\out_k-\underline{\ide{\pv}}^\out_k]_i\|\right)\\
  &\hspace{1.5cm}\times \left( \|[\underline{\widetilde{\pv}}^\inn_k-\underline{\ide{\pv}}^\inn_k]_i\|+\|[\underline{\widetilde{\pv}}^\out_k-\underline{\ide{\pv}}^\out_k]_i\| \right)\Big]^2\\
  &\leq \frac{8L^2}{N}\left[\sum_{i=1}^N\left( \|[\underline{\widetilde{\pv}}^\inn_k - \underline{\ide{\pv}}^\inn_k]_i\|^2+\|[\underline{\widetilde{\pv}}^\out_k - \underline{\ide{\pv}}^\out_k]_i\|^2 \right)\right]\\
  &\hspace{1.5cm}\times\left[1 + \frac{4}{N}\sum_{i=1}^N\|(\left[\underline{\ide{\pv}}^\inn_k\right]_i,\left[\underline{\ide{\pv}}^\out_k\right]_i,[\wv]_i)\|^2+\frac{1}{N}\sum_{i=1}^N\|\left[\underline{\widetilde{\pv}}^\inn_k - \underline{\ide{\pv}}^\inn_k\right]_i\|^2+ \frac{1}{N}\sum_{i=1}^N\|\left[\underline{\widetilde{\pv}}^\out_k - \underline{\ide{\pv}}^\out_k\right]_i\|^2\right].
\end{align*}

Next, using the state evolution  in \eqref{eq:se_general} and the discussion around \eqref{eq:tau_var_argument}, if $i\leq M$, then
$\mathbb{E}\|(\left[\underline{\ide{\pv}}^\inn_k\right]_i,\left[\underline{\ide{\pv}}^\out_k\right]_i,[\wv]_i)\|^2 = \sum_{j=0}^k(\tau^\inn_{pj}+\tau^\out_{pj})+ \mathbb{E}[|W|^2],$ 
and if $M< i \leq N$, then $\mathbb{E}\|(\left[\underline{\ide{\pv}}^\inn_k\right]_i,\left[\underline{\ide{\pv}}^\out_k\right]_i,[\wv]_i)\|^2 = \sum_{j=0}^k\tau^\inn_{pj}+  \mathbb{E}[|W|^2] \leq \sum_{j=0}^k(\tau^\inn_{pj}+\tau^\out_{pj})+ \mathbb{E}[|W|^2],$ using our convention that $\left[\underline{\ide{\pv}}^\out_k\right]_i=0$ for $i > M$.
Define $\mathbb{E}_k := \sum_{j=0}^k(\tau^\inn_{pj}+\tau^\out_{pj})+ \mathbb{E}[|W|^2]$. Using these facts,
\be
\begin{split}
 \P\left([T2]^2 \geq \frac{\epsilon^2}{4}\right) 
 & = \P\left(\left[\frac{1}{N}\sum_{i=1}^N\left|\phi\left(\left[\underline{\widetilde{\pv}}^\inn_k\right]_i,\left[\underline{\widetilde{\pv}}^\out_k\right]_i,[\wv]_i\right)-\phi\left(\left[\underline{\ide{\pv}}^\inn_k\right]_i,\left[\underline{\ide{\pv}}^\out_k\right]_i,[\wv]_i\right)\right|\right]^2\geq\frac{\epsilon^2}{4}\right)\\
&\leq \P\left(\frac{1}{N}\sum_{i=1}^N\left( \|[\underline{\widetilde{\pv}}^\inn_k - \underline{\ide{\pv}}^\inn_k]_i\|^2+\|[\underline{\widetilde{\pv}}^\out_k - \underline{\ide{\pv}}^\out_k]_i\|^2 \right)\geq \frac{\min\{1,1/(8L^2)\} \,(\epsilon^2/4)}{2+8\mathbb{E}_k}\right)\\
&\qquad +\P\left(\frac{1}{N}\sum_{i=1}^N\|(\left[\underline{\ide{\pv}}^\inn_k\right]_i,\left[\underline{\ide{\pv}}^\out_k\right]_i,[\wv]_i)\|^2\geq 2\mathbb{E}_k\right).
\label{eq:t2_bound_1}
\end{split}
\ee
Observe that the second term of \eqref{eq:t2_bound_1} equals
\begin{equation}\label{eq:Bk_b_2}
\P\left(\frac{1}{N}\sum_{i=1}^N\left(\|(\left[\underline{\ide{\pv}}^\inn_k\right]_i,\left[\underline{\ide{\pv}}^\out_k\right]_i,[\wv]_i)\|^2-\mathbb{E}\|(\left[\underline{\ide{\pv}}^\inn_k\right]_i,\left[\underline{\ide{\pv}}^\out_k\right]_i,[\wv]_i)\|^2\right)\geq \mathbb{E}_k\right).
\end{equation}
Concentration for this follows directly by Lemma \ref{lem:PLsubgaussconc} since $(p_1,p_2,w)\mapsto \|(p_1,p_2,w)\|^2$ is pseudo-Lipschitz.

To bound the first term of \eqref{eq:t2_bound_1}, observe that by Lemma \ref{lem:joint_dists} part 2,
\begin{align*}
\sum_{i=1}^N\|\left[\underline{\widetilde{\pv}}^\inn_k - \underline{\ide{\pv}}^\inn_k\right]_i\|^2 = \sum_{i=1}^N\sum_{j=0}^k\left(\left[\widetilde{\pv}^\inn_j  - \ide{\pv}_j^\inn\right]_i\right)^2
&= \sum_{i=1}^N\sum_{j=0}^k\left[\sum_{r=0}^j [\cv^\inn_{pj}]_r\left[\widetilde{\Delm}^\inn_{pr}\right]_i\right]^2.
\end{align*}
Next, by Cauchy-Schwarz, we have
\begin{align}
\label{eq:Bk_b_3a}
\sum_{i=1}^N\|\left[\underline{\widetilde{\pv}}^\inn_k - \underline{\ide{\pv}}^\inn_k\right]_i\|^2 \leq \sum_{i=1}^N\sum_{j=0}^k\left[\sum_{r=0}^j [\cv^\inn_{pj}]_r^2\right]\left[\sum_{r=0}^j \left(\left[\widetilde{\Delm}^\inn_{pr}\right]_i\right)^2\right] & =\sum_{j=0}^k\left[\sum_{r=0}^j [\cv^\inn_{pj}]_r^2\right]\left[\sum_{r=0}^j \|\widetilde{\Delm}^\inn_{pr}\|^2\right].
\end{align}
Observe that part 3 of Lemma $\ref{lem:joint_dists}$ implies that $\tau^\inn_{pk} = \mathbb{E}\left[([\ide{\pv}^\inn_k]_i)^2\right] = \sum_{r=0}^k \rho^\inn_{pr}[\cv^\inn_{pk}]_r^2$,where the second equality follows from part 3 of Lemma $\ref{lem:joint_dists}$ as $([\ide{\pv}^\inn_k]_i)^2 = (\sum_{r=0}^k \sqrt{\rho^\inn_{pr}} \, [\cv^\inn_{pk}]_r \, [\widetilde{\Om}^\inn_{pr} \, \overline{\Zv}^\inn_{pr}]_i)^2\overset{d}{=} (\sum_{r=0}^k \sqrt{\rho^\inn_{pr}} \, [\cv^\inn_{pk}]_r \, [\overline{\Zv}^\inn_{pr}]_i)^2$ with $[\overline{\Zv}^\inn_{pr}]_i$ being i.i.d.\ standard Gaussian; therefore,
\[
  \sum_{r=0}^k [\cv^\inn_{pk}]_r^2 \leq \frac{\tau^\inn_{pk}}{\min_{0\leq i\leq k}(\rho^\inn_{pi})}\leq \frac{\tau^\inn_{pk}}{\epsilon^*_1},
\]
using the stopping criterion. Now using this bound in \eqref{eq:Bk_b_3a}, we find
\begin{align*}
\sum_{i=1}^N\|\left[\underline{\widetilde{\pv}}^\inn_k - \underline{\ide{\pv}}^\inn_k\right]_i\|^2  \leq \sum_{j=0}^k\left[\sum_{r=0}^j [\cv^\inn_{pj}]_r^2\right]\left[\sum_{r=0}^j \|\widetilde{\Delm}^\inn_{pr}\|^2\right]  &\leq \sum_{j=0}^k\left[\sum_{r=0}^k [\cv^\inn_{pj}]_r^2\right]\left[\sum_{r=0}^j \|\widetilde{\Delm}^\inn_{pr}\|^2\right] \\
&\leq \left[ \frac{\tau^\inn_{pk}}{\epsilon^*_1}\right] \sum_{j=0}^k \sum_{r=0}^j \|\widetilde{\Delm}^\inn_{pr}\|^2.
\end{align*}
Now, combining the above, we get that
\begin{align}\label{eq:Bk_b_3}
  \P&\left(\frac{1}{N}\sum_{i=1}^N\|\left[\underline{\widetilde{\pv}}^\inn_k - \underline{\ide{\pv}}^\inn_k\right]_i\|^2\geq \frac{\min\{1,1/(8L^2)\} \,(\epsilon^2/4)}{2+8\sum_{j=0}^k(\tau^\inn_{pj}+\tau^\out_{pj})+ 8\mathbb{E}[|W|^2]}\right)\\
  &\leq\P\left(\frac{1}{N}  \sum_{j=0}^k \sum_{r=0}^j \|\widetilde{\Delm}^\inn_{pr}\|^2\geq \frac{\epsilon^*_1\min(1,\frac{1}{8L^2})\, (\epsilon^2/4)}{\tau^\inn_{pk}(2+8\sum_{j=1}^k(\tau^\inn_{pj}+\tau^\out_{pj})+ 8\mathbb{E}[|W|^2]) }\right)\\ 
  &\leq  k\sum_{r=0}^k \P\left(\frac{\|\widetilde{\Delm}^\inn_{pr}\|^2}{N}\geq \frac{c\epsilon^2}{k^3}\right) \stackrel{(a)}{\leq} Ck^4\exp\left(-c\epsilon^2/k^7\right),
\end{align}
where $c,C$ are constants, and the inequality $(a)$ follows from $\mathbf{P_i.(a)}$ for $0\leq i\leq k$. Next, consider
\begin{equation*}
  \frac{1}{N}\sum_{i=1}^N\|\left[\underline{\widetilde{\pv}}^\out_k\right]_i-\left[\underline{\ide{\pv}}^\out_k\right]_i\|^2=\delta
  \frac{1}{M}\sum_{i=1}^M\|\left[\underline{\widetilde{\pv}}^\out_k\right]_i-\left[\underline{\ide{\pv}}^\out_k\right]_i\|^2,
\end{equation*}
where $\delta = M/N$ and the equality follows from the convention for $i > N$. Concentration for this term follows exactly the same as above, with the only difference coming from the appearance of a $\delta^{-1}$ factor in the exponential rate.
Combining these with lemma \ref{sums} gives concentration for the first term above, which completes the proof of $\mathbf{P_k.(b)}$.


$\mathbf{P_k.(c)}$ 
Observe that, by Lemma \ref{lipprods},
\[
  \phi([\pv^\inn_1]_i,\ldots,[\pv^\inn_k]_i,[\pv^\out_1]_i,\ldots,[\pv^\out_k]_i,[\wv^\inn_p]_i) = [\pv^{\inn/\out}_k]_i[\wv^\inn_p]_i\in PL(2).
\]
Thus, by $\mathbf{P_k.(b)}$, we have that $\frac{1}{N}[\pv^{\inn/\out}_k]^T\wv^\inn_p\stackrel{\cdot}{=} \mathbb{E}[P^{\inn/\out}_kW^\inn_p] = 0$ (as $W^p$ is assumed independent of $(P^{\inn/\out}_0,\ldots,P^{\inn/\out}_k)$), as needed.


$\mathbf{P_k.(d)}$
Again by Lemma \ref{lipprods}, for all $0\leq j\leq k$,
\[
  \phi([\pv^\inn_1]_i,\ldots,[\pv^\inn_k]_i,[\pv^\out_1]_i,\ldots,[\pv^\out_k]_i,[\wv^\inn_p]_i)=[\pv^{\inn/\out}_j]_i[\pv^{\inn/\out}_k]_i\in PL(2).
\]
By $\mathbf{P_k.(b)}$, we find $\frac{1}{N}[\pv^{\inn/\out}_j]^T\pv^{\inn/\out}_k \stackrel{\cdot}{=}\mathbb{E}[P^{\inn/\out}_jP^{\inn/\out}_k]$, as needed.


$\mathbf{P_k.(f)}$ We note here that we prove $\mathbf{P_k.(f)}$ and then prove $\mathbf{P_k.(e)}$ afterwards, as the proof of $\mathbf{P_k.(f)}$ relies on the result of $\mathbf{P_k.(f)}$.

\textbf{Concentration for $\alpha$ terms.}  
First, we show $\alpha^\inn_{pk}\stackrel{\cdot}{=}\overline{\alpha}^\inn_{pk}$. To this end, recall that
\[
\alpha^\inn_{pk}=\mathcal{T}\frac{1}{N}\sum_{i=1}^N [f^\inn_p]'([\pv^\inn_k]_i,[\pv^\out_k]_i,[\wv^\inn_p]_i,\gamma^\inn_{pk},\gamma^\out_{pk}),
\]
where $\mathcal{T}$ is the truncation operator defined as $\mathcal{T}x = \min(t_{2},\max(t_{1},x))$. Recall that the upper truncation is only necessary if $\log p(x)$ or $\log p(y\mid z)$ is not strongly concave. Now, adopting the notation $[\pv_k]_i=\left( [\pv^{\inn}_k]_i,[\pv^{\out}_k]_i \right)$ and $\gamma_{pk} =( \gamma^{\inn}_{pk}, \gamma^{\out}_{pk})$, observe that by Lemma~\ref{sums},
\begin{align}
  \P\left( \left|\alpha^\inn_{pk}-\overline{\alpha}^\inn_{pk} \right|\geq\epsilon \right)\leq&\; \P\Big( \Big|\mathcal{T}\frac{1}{N}\sum_{i=1}^N [f^\inn_p]'\left( [\pv_k]_i,[\wv^\inn_p]_i,\gamma_{pk} \right)- \mathcal{T}\frac{1}{N}\sum_{i=1}^N[f^\inn_p]'\left( [\pv_k]_i,[\wv^\inn_p]_i,\overline{\gamma}_{pk}\right) \Big|\geq\frac{\epsilon}{2}\Big)\nonumber \\
 &+\P\left( \left| \mathcal{T}\frac{1}{N}\sum_{i=1}^N [f^\inn_p]'\left( [\pv_k]_i,[\wv^\inn_p]_i,\overline{\gamma}_{pk} \right)-\overline{\alpha}^\inn_{pk} \right|\geq\frac{\epsilon}{2} \right).\label{eq:Bk_f_1}
\end{align}
Label the two terms of \eqref{eq:Bk_f_1} as $T1$ and $T2$.

For term $T2$ of \eqref{eq:Bk_f_1}, because $\overline{\alpha}^\inn_{pk} \in [t_{\min},t_{\max}]$ (by \textbf{Condition 6}),
\begin{align*}
  &\P\left( \left| \mathcal{T}\frac{1}{N}\sum_{i=1}^N [f^\inn_p]'\left( [\pv_k]_i,[\wv^\inn_p]_i,\overline{\gamma}_{pk}\right)-\overline{\alpha}^\inn_{pk} \right|\geq\frac{\epsilon}{2} \right) \leq \P\left( \left| \frac{1}{N}\sum_{i=1}^N [f^\inn_p]'\left( [\pv_k]_i,[\wv^\inn_p]_i,\overline{\gamma}_{pk} \right)-\overline{\alpha}^\inn_{pk} \right|\geq\frac{\epsilon}{2} \right).
\end{align*}
By the assumption that $[f^\inn_p]'$ is assumed either uniformly Lipschitz or bounded conditionally Lipschitz in \textbf{Condition 3}, we have concentration for the right side by $\mathbf{P_k.(b)}$.

For term $T1$ of \eqref{eq:Bk_f_1}, observe that for any $x,y\in\mathbb{R}$, $|\mathcal{T}x - \mathcal{T}y|\leq |x-y|$; therefore,
\begin{align}
  &\P\Big( \Big|\mathcal{T}\frac{1}{N}\sum_{i=1}^N [f^\inn_p]'\left( [\pv_k]_i,[\wv^\inn_p]_i,\gamma_{pk}\right) - \mathcal{T}\frac{1}{N}\sum_{i=1}^N[f^\inn_p]'\left( [\pv_k]_i,[\wv^\inn_p]_i,\overline{\gamma}_{pk} \right) \Big|\geq\frac{\epsilon}{2}\Big)\nonumber \\
  &\qquad \leq \P\Big( \frac{1}{N}\sum_{i=1}^N\Big| [f^\inn_p]'\left( [\pv_k]_i,[\wv^\inn_p]_i,\gamma_{pk}\right)- [f^\inn_p]'\left( [\pv_k]_i,[\wv^\inn_p]_i,\overline{\gamma}_{pk} \right) \Big|\geq\frac{\epsilon}{2}\Big).
  \label{eq:T1_bound1_pkf}
\end{align}
Now, by \textbf{Condition 3}, the functions $ [f^\inn_p]'$ are uniformly Lipschitz around $\overline{\gamma}_{pk}$. This guarantees that there exists some $\delta>0$ such that if $\left| \gamma^\inn_{pk}-\overline{\gamma}^\inn_{pk} \right|<\delta$ and $\left| \gamma^\out_{pk}-\overline{\gamma}^\out_{pk} \right|<\delta$, then
\begin{align*}
  &\left| [f^\inn_p]'\left( [\pv_k]_i,[\wv^\inn_p]_i,\gamma_{pk} \right)- [f^\inn_p]'\left( [\pv_k]_i,[\wv^\inn_p]_i,\overline{\gamma}_{pk} \right)\right|\\
  &\hspace{4cm}\leq L_p\left( 1+\|([\pv_k]_i,[\wv^\inn_p]_i)\| \right)\left[\left| \gamma^\inn_{pk}-\overline{\gamma}^\inn_{pk} \right|+\left| \gamma^\out_{pk}-\overline{\gamma}^\out_{pk} \right|\right].
\end{align*}
In particular, this implies that if $\left| \gamma^\inn_{pk}-\overline{\gamma}^\inn_{pk} \right|<\delta$ and $\left| \gamma^\out_{pk}-\overline{\gamma}^\out_{pk} \right|<\delta$, and if
\[\left[ \left| \gamma^\inn_{pk}-\overline{\gamma}^\inn_{pk} \right|+\left| \gamma^\out_{pk}-\overline{\gamma}^\out_{pk} \right|\right]\frac{L_p}{N}\sum_{i=1}^N\left( 1+\|([\pv_k]_i,[\wv^\inn_p]_i)\| \right)\leq \frac{\epsilon}{2},\]
then
\[\frac{1}{N}\sum_{i=1}^N\Big| [f^\inn_p]'\left( [\pv_k]_i,[\wv^\inn_p]_i,\gamma_{pk}\right)- [f^\inn_p]'\left( [\pv_k]_i,[\wv^\inn_p]_i,\overline{\gamma}_{pk} \right) \Big|\leq\frac{\epsilon}{2}.\]
It then follows from \eqref{eq:T1_bound1_pkf} that $T1$ can be upper bounded by
\begin{align}\label{eq:Bk_f_2}
  \P&\left( \left[ \left| \gamma^\inn_{pk}-\overline{\gamma}^\inn_{pk} \right|+\left| \gamma^\out_{pk}-\overline{\gamma}^\out_{pk} \right|\right]\frac{L_p}{N}\sum_{i=1}^N\left( 1+\|([\pv_k]_i,[\wv^\inn_p]_i)\| \right)\geq \frac{\epsilon}{2} \right)\\
  &\qquad +\P\left( \left| \gamma^\inn_{pk} -\overline{\gamma}^\inn_{pk}\right|\geq \min\{1, \delta\} \right) +\P\left( \left| \gamma^\out_{pk} -\overline{\gamma}^\out_{pk}\right|\geq  \min\{1, \delta\} \right)\nonumber.
\end{align}

The second and third terms of \eqref{eq:Bk_f_2} concentrate by $\mathbf{Q_k.(f)}$, and the first term concentrates using {Lemma~\ref{products_0} along with $\mathbf{Q_k.(f)}$ and $\mathbf{P_k.(b)}$, because $1+\|([\pv^\inn_k]_i,[\pv^\out_k]_i,[\wv^\inn_p]_i)\|$ is pseudo-Lipschitz of order $2$ and $\mathbb{E}\|(P_k^{\inn}, P_k^{\out}, W^\inn_p)\|$ is upper bounded by a universal constant.

  \textbf{Concentration for $\gamma$ terms.}
We now show concentration of $\gamma^\inn_{qk}$ around $\overline{\gamma}^\inn_{qk}$. We can assume that the clipping of the $\gamma$ iterates is absorbed by the $\Gamma$ functions, because clipping a Lipschitz function to bound its range within a compact interval gives another Lipschitz function with the same constant $L_1$. Observe that {by \eqref{eq:alg_gvamp_gamma_step} in Algorithm~\ref{alg:general_gvamp},
\begin{align}
  \P\left(\left|\gamma^\inn_{qk}-\overline{\gamma}^\inn_{qk}\right|\geq \epsilon\right) &= \P\left(\left|\Gamma^\inn_q\left(\gamma^\inn_{pk},\alpha^\inn_{pk}\right)-\Gamma^\inn_q\left(\overline{\gamma}^\inn_{pk},\overline{\alpha}^\inn_{pk}\right)\right|\geq\epsilon\right) \stackrel{(a)}{\leq} \P\left(\left|\gamma^\inn_{pk}-\overline{\gamma}^\inn_{pk}\right|+\left|\alpha^\inn_{pk}-\overline{\alpha}^\inn_{pk}\right|\geq\frac{\epsilon}{L_1}\right) \nonumber \\
&\qquad \qquad  \stackrel{(b)}{\leq}  \P\left(\left|\gamma^\inn_{pk}-\overline{\gamma}^\inn_{pk}\right|\geq\frac{\epsilon}{2L_1}\right) + \P\left(\left|\alpha^\inn_{pk}-\overline{\alpha}^\inn_{pk}\right|\geq\frac{\epsilon}{2L_1}\right).\label{eq:Bk_f_3}
\end{align}
where $(a)$ follows from the assumption that $\Gamma^\inn_q$ is Lipschitz continuous on its domain along with the fact that $\|(\gamma^\inn_{pk}- \overline{\gamma}^\inn_{pk}, \alpha^\inn_{pk}- \overline{\alpha}^\inn_{pk})\| \leq |\gamma^\inn_{pk}-\overline{\gamma}^\inn_{pk}|+|\alpha^\inn_{pk}-\overline{\alpha}^\inn_{pk}|$ and $(b)$ by Lemma~\ref{sums}. The final terms in the above can be upper bounded by the work just above in \eqref{eq:Bk_f_1}-\eqref{eq:Bk_f_2} and $\mathbf{Q_k.(f)}$.

\textbf{Concentration for inner product terms.}
We show concentration for $[\vv^\inn_k]^T\pv^\inn_j$ where $0\leq j\leq k$. Recall that by \eqref{eq:alg_gvamp_v_step} in Algorithm~\ref{alg:general_gvamp}, $[\vv^\inn_k]_i = (1-\alpha^\inn_{pk})^{-1} g_p([\pv_k]_i,[\wv^\inn_p]_i,\gamma_{pk},\alpha^\inn_{pk}),$ where we have defined functions
\begin{align}
\label{eq:gp_func_def}
  g_p([\pv_k]_i,[\wv^\inn_p]_i,\gamma_{pk}, \alpha^\inn_{pk}) := \fv^\inn_p([\pv^\inn_k]_i,[\pv^\out_k]_i,[\wv^\inn_p]_i,\gamma_{pk}^\inn,\gamma_{pk}^\out)-\alpha^\inn_{pk}[\pv^\inn_k]_i,
\end{align}
where to simplify notation, we write $[\pv_k]_i = \left([ \pv_k^\inn]_i,[\pv_k^\out]_i \right)$ and $\gamma_{pk} = \left( \gamma_{pk}^\inn,\gamma_{pk}^\out \right)$. 
Observe that
\[
[\pv^\inn_j]_ig_p([\pv_k]_i,[\wv^\inn_p]_i,\gamma_{pk},\alpha^\inn_{pk}) = (1-\alpha^\inn_{pk})^{-1}\left[[\pv^\inn_j]_if^\inn_p([\pv_k]_i,[\wv^\inn_p]_i,\gamma_{pk})-\alpha^\inn_{pk}[\pv^\inn_k]_i[\pv^\inn_j]_i\right].
\] 
Now by Lemma~\ref{products_0} it suffices to show concentration separately for $(1-\alpha^\inn_{pk})^{-1}$ and
\[
  [\pv^\inn_j]_if_p([\pv_k]_i,[\wv^\inn_p]_i,\gamma_{pk})-\alpha^\inn_{pk}[\pv^\inn_k]_i[\pv^\inn_j]_i.
\]
For $(1-\alpha^\inn_{pk})^{-1}$, we have,
\begin{equation}\label{eq:Bk_f_4}
\P\left(\left|(1-\alpha^\inn_{pk})^{-1}-(1-\overline{\alpha}^\inn_{pk})^{-1}\right|\geq \epsilon\right) \leq \P\left(\left|\alpha^\inn_{pk}-\overline{\alpha}^\inn_{pk}\right|\geq\frac{\epsilon}{L_2}\right),
\end{equation}
where the inequality follows from the fact that $\frac{1}{1-a}$ is Lipschitz on $[t_1,t_2]$ with Lipschitz constant $L_2:=(1-t_2)^{-2}$ and the fact that the $\alpha^{\inn}_{pk}$ and $\overline{\alpha}^\inn_{pk}$ lie in $[t_1,t_2]$ by \textbf{Assumption 4}. The right hand side then concentrates by the work for $\alpha$ concentration above ($\mathbf{P_k.(f)}$; the first part).

Next, we observe that
\begin{align}\label{eq:Bk_f_5}
\P&\left(\frac{1}{N}\left|\sum_{i=1}^N \left[[\pv^\inn_j]_if^\inn_p([\pv_k]_i,[\wv^\inn_p]_i,\gamma_{pk})-\alpha^\inn_{pk}[\pv^\inn_k]_i[\pv^\inn_j]_i\right]\right|\geq\epsilon\right)\nonumber\\
  &\stackrel{(a)}{\leq} \P\left(\left|\alpha^\inn_{pk}-\overline{\alpha}^\inn_{pk}\right|\left|\frac{1}{N}\sum_{i=1}^N[\pv^\inn_k]_i[\pv^\inn_j]_i\right|\geq\frac{\epsilon}{2}\right)\nonumber\\
  &\hspace{2cm}+ \P\left(\left|\frac{1}{N}\sum_{i=1}^N[\pv^\inn_j]_i\left(f^\inn_p([\pv_k]_i,[\wv^\inn_p]_i,\gamma_{pk})-\overline{\alpha}^\inn_{pk}[\pv^\inn_k]_i\right)\right|\geq\frac{\epsilon}{2}\right),
\end{align}
where $(a)$ follows by adding and subtracting $\overline{\alpha}^\inn_{pk}[\pv^\inn_k]_i[\pv^\inn_j]_i$ from each term and applying the triangle inequality. Now the first term in \eqref{eq:Bk_f_5} concentrates by Lemma \ref{products_0}, $\mathbf{P_k.(f)}$ (the first part), and $\mathbf{P_k.(b)}$ since $[\pv^\inn_j]_i[\pv^\inn_k]_i\in PL(2)$. 

Adding and subtracting $[\pv^\inn_j]_i\left(f^\inn_p([\pv_k]_i,[\wv^\inn_p]_i,\overline{\gamma}_{pk})-\overline{\alpha}^\inn_{pk}[\pv^\inn_k]_i\right)$, we bound the second term in \eqref{eq:Bk_f_5} by
\begin{align}
  \P&\left(\left|\frac{1}{N}\sum_{i=1}^N[\pv^\inn_j]_i\left(f^\inn_p([\pv_k]_i,[\wv^\inn_p]_i,\overline{\gamma}_{pk})-\overline{\alpha}^\inn_{pk}[\pv^\inn_k]_i\right)\right|\geq\frac{\epsilon}{4}\right)\label{eq:Bk_f_6}\\
  &+ \P\left(\frac{1}{N}\sum_{i=1}^N[\pv^\inn_j]_i\left|f^\inn_p([\pv_k]_i,[\wv^\inn_p]_i,\overline{\gamma}_{pk})-f^\inn_p([\pv_k]_i,[\wv^\inn_p]_i,\gamma_{pk})\right|\geq\frac{\epsilon}{4}\right).\nonumber
\end{align}
The second of these terms can be bounded in exactly the same way as $T1b$  in $\mathbf{P_k.(e)}$ below.

For the first term in \eqref{eq:Bk_f_6}, note that if
$f^\inn_p([\pv_k]_i,[\wv^\inn_p]_i,\overline{\gamma}_{pk})$ is Lipschitz (in all inputs), then
\[
\phi([\pv^\inn_0]_i,\ldots,[\pv^\inn_k]_i,[\pv^\out_0]_i,\ldots,[\pv^\out_k]_i,[\wv^\inn_p]_i)=[\pv^\inn_j]_i\left(f^\inn_p([\pv_k]_i,[\wv^\inn_p]_i,\overline{\gamma}_{pk})-\overline{\alpha}^\inn_{pk}[\pv^\inn_k]_i\right),
\]
is in $PL(2)$ by Lemma \ref{lipprods}. If instead $f^\inn_p([\pv_k]_i,[\wv^\inn_p]_i,\overline{\gamma}_{pk})$ satisfies the bounded conditionally Lipschitz condition, then $\phi$ is bounded conditionally pseudo-Lipschitz. In either case, $\mathbf{P_k.(b)}$ tells us that this will concentrate around
\begin{align}
\label{eq:expectation0}
\mathbb{E}\left \{P^\inn_j\left(f^\inn_p(P^\inn_k,P^\out_k,W^\inn_p,\overline{\gamma}^\inn_{pk},\overline{\gamma}^\out_{pk})-\overline{\alpha}^\inn_{pk}P^\inn_k\right)\right\}.
\end{align}

Therefore, to complete the proof, we just need to show that the expression in \eqref{eq:expectation0} is equal to zero. To do this, recall that $\mathbb{E}[P_j^\inn]=0$, so \eqref{eq:expectation0} is equal to
\begin{align*}
  \mathrm{Cov}&\left( P^\inn_j,f^\inn_p(P^\inn_k,P^\out_k,W^\inn_p,\overline{\gamma}^\inn_{pk},\overline{\gamma}^\out_{pk})-\overline{\alpha}^\inn_{pk}P^\inn_k \right)\\
              &\hspace{2cm}= \mathbb{E}\left\{P^\inn_jP^\inn_k\right\} \times\mathbb{E}\left\{[f^\inn_p]'(P^\inn_k,P^\out_k,W^\inn_p,\overline{\gamma}^\inn_{pk},\overline{\gamma}^\out_{pk})-\overline{\alpha}^\inn_{pk}  \right\}=0,
\end{align*}
where the first equality follows from Stein's lemma, and the second equality follows from the definition of $\overline{\alpha}^\inn_{pk}$ in the GVAMP state evolution \eqref{eq:se_general}.

Concentration for $[\pv^\inn_k]^T\vv^\inn_j$ is proved in the same way as the above.


$\mathbf{P_k.(e)}$
By \eqref{eq:alg_gvamp_v_step} in Algorithm~\ref{alg:general_gvamp} and the definition of the function $g_p$ in \eqref{eq:gp_func_def}, 
\begin{align}\label{eq:Bk_e_1}
  &\P\left(\left|\frac{1}{N}[\vv^\inn_j]^T\vv^\inn_k-\left[ \covm^\inn_{vk} \right]_{(j+1)(k+1)}\right|\geq\epsilon\right) \\
  &\hspace{2cm}= \P\left(\left|\frac{1}{N}\sum_{i=1}^N \frac{g_p([\pv_k]_i,[\wv^\inn_p]_i,\gamma_{pk}, \alpha^\inn_{pk})g_p([\pv_j]_i,[\wv^\inn_p]_i,\gamma_{pj}, \alpha^\inn_{pj})}{(1-\alpha^\inn_{pk})(1- \alpha^\inn_{pj})}-\mathbb{E}[V^\inn_kV^\inn_j]\right|\geq\epsilon\right).\nonumber
\end{align}
Using Lemma \ref{products}, we can separately show concentration for $(1-\alpha^\inn_{pk})^{-1}(1- \alpha^\inn_{pj})^{-1}$ and
\[
  \frac{1}{N}\sum_{i=1}^Ng_p([\pv_k]_i,[\wv^\inn_p]_i,\gamma_{pk}, \alpha^\inn_{pk})g_p([\pv_j]_i,[\wv^\inn_p]_i,\gamma_{pj}, \alpha^\inn_{pj}).
\]
First observe that the function $f(a) = \frac{1}{1-a}$ is Lipschitz over the  interval $[t_1,t_2]$ within which the $\alpha$ iterates must lie given our truncation assumption (\textbf{Assumption 4}) and we can take $L_2=\frac{1}{(1-t_2)^2}$ to be the Lipschitz constant. Therefore,
\begin{equation}\label{eq:Bk_e_2}
  \P\left(\left|\frac{1}{1-\alpha^\inn_{pk}}- \frac{1}{1-\overline{\alpha}^\inn_{pk}}\right|\geq\epsilon\right) \leq \P\left(\left|\alpha^\inn_{pk}-\overline{\alpha}^\inn_{pk}\right|\geq\frac{\epsilon}{L_2}\right).
\end{equation}
The right side of \eqref{eq:Bk_e_2} concentrates by $\mathbf{P_k.(f)}$. Now $\frac{1}{1-\alpha^\inn_{pj}}$ concentrates in the same way using $\mathbf{P_j.(f)}$, and concentration for $(1-\alpha^\inn_{pk})^{-1}(1-\alpha^\inn_{pj})^{-1}$ follows by Lemma \ref{products}.

Next, 
using Lemma~\ref{sums}, we have that
\begin{equation}
\begin{split}
\P&\left(\left|\frac{1}{N}\sum_{i=1}^Ng_p([\pv_k]_i,[\wv^\inn_p]_i,\gamma_{pk},\alpha^\inn_{pk})g_p([\pv_j]_i,[\wv^\inn_p]_i,\gamma_{pj},\alpha^\inn_{pk})- \left(1-\overline{\alpha}^{\inn}_{pk} \right)\left( 1 - \overline{\alpha}^\inn_{pj} \right)\mathbb{E}[V^\inn_kV^\inn_j]\right|\geq\epsilon\right)\\
  &\leq \P\Big(\Big|\frac{1}{N}\sum_{i=1}^N g_p([\pv_k]_i,[\wv^\inn_p]_i,\gamma_{pk},\alpha^\inn_{pk})g_p([\pv_j]_i,[\wv^\inn_p]_i,\gamma_{pj},\alpha^\inn_{pk})\\
  &\hspace{2cm}-\frac{1}{N}\sum_{i=1}^N g_p([\pv_k]_i,[\wv^\inn_p]_i,\gamma_{pk},\overline{\alpha}^\inn_{pk})g_p([\pv_j]_i,[\wv^\inn_p]_i,\gamma_{pj},\overline{\alpha}^\inn_{pk})\Big|\geq\frac{\epsilon}{2}\Big)\\
& +\P\left(\left|\frac{1}{N}\sum_{i=1}^N g_p([\pv_k]_i,[\wv^\inn_p]_i,\gamma_{pk},\overline{\alpha}^\inn_{pk})g_p([\pv_j]_i,[\wv^\inn_p]_i,\gamma_{pj},\overline{\alpha}^\inn_{pk}) - \left[1-\overline{\alpha}^{\inn}_{pk} \right]\left[ 1 - \overline{\alpha}^\inn_{pj} \right]\mathbb{E}[V^\inn_kV^\inn_j]\right|\geq\frac{\epsilon}{2}\right).
\label{eq:pke_firstsplit}
\end{split}
\end{equation}
We label these latter two terms $T1$ and $T2$ and analyze them separately.

\textbf{Concentration for T1 of \eqref{eq:pke_firstsplit}.}
Observe that repeated application of the triangle inequality the definition of the function $g_p$ in \eqref{eq:gp_func_def} gives us
\begin{align*}
  &\Big|\frac{1}{N}\sum_{i=1}^N \Big[g_p([\pv_k]_i,[\wv^\inn_p]_i,\gamma_{pk},\alpha^\inn_{pk})g_p([\pv_j]_i,[\wv^\inn_p]_i,\gamma_{pj},\alpha^\inn_{pk})-g_p([\pv_k]_i,[\wv^\inn_p]_i,\gamma_{pk},\overline{\alpha}^\inn_{pk})g_p([\pv_j]_i,[\wv^\inn_p]_i,\gamma_{pj},\overline{\alpha}^\inn_{pk}) \Big]\Big|\\
  &\hspace{2cm}\leq \frac{1}{N}\left|\alpha^\inn_{pk}\alpha^\inn_{pj}-\overline{\alpha}^\inn_{pk}\overline{\alpha}^\inn_{pj}\right|\left|\sum_{i=1}^N[\pv^\inn_k]_i[\pv^\inn_j]_i\right| + \frac{1}{N}\left|\alpha^\inn_{pk}-\overline{\alpha}^\inn_{pk}\right|\left| \sum_{i=1}^N[\pv^\inn_k]_if^\inn_p([\pv_j]_i,[\wv^\inn_p]_i,\gamma_{pj})\right|\\
  &\hspace{4cm}+\frac{1}{N}\left| \alpha^\inn_{pj}-\overline{\alpha}^\inn_{pj} \right|\left| \sum_{i=1}^N[\pv^\inn_j]_if^\inn_p([\pv_k]_i,[\wv^\inn_p]_i,\gamma_{pk})\right|.\label{eq:Bk_e_3}
\end{align*}
By Lemmas \ref{sums} and \ref{products_0}, to establish concentration for these terms, it suffices to show concentration for each of the following:
\begin{align*}
  &\left|\alpha^\inn_{pk}-\overline{\alpha}^\inn_{pk}\right|, && \left|\alpha^\inn_{pj}-\overline{\alpha}^\inn_{pj}\right|, && \left|\alpha^\inn_{pk}\alpha^\inn_{pj}-\overline{\alpha}^\inn_{pk}\overline{\alpha}^\inn_{pj}\right|,\\
  &\frac{1}{N}\sum_{i=1}^N[\pv^\inn_k]_i[\pv^\inn_j]_i, &&\frac{1}{N}\sum_{i=1}^N[\pv^\inn_k]_if^\inn_p([\pv_j]_i,[\wv^\inn_p]_i,\gamma_{pj}), && \frac{1}{N}\sum_{i=1}^N[\pv^\inn_j]_if_p^k([\pv_k]_i,[\wv^\inn_p]_i,\gamma_{pk}).
\end{align*}

Concentration for the first two follows by $\mathbf{P_k.(f)}$ and $\mathbf{P_j.(f)}$  above, and concentration for the third follows from that and Lemma \ref{products}. Concentration for fourth term follows from $\mathbf{P_k.(b)}$ as
\[
  \phi([\pv^\inn_0]_i,\ldots,[\pv^\inn_k]_i,[\pv^\out_0]_i,\ldots,[\pv^\out_k]_i,[\wv^\inn_p]_i)=[\pv^\inn_k]_i[\pv^\inn_j]_i,
\]
is pseudo-Lipschitz of order $2$ by Lemma \ref{lipprods}.

Thus, it remains to show concentration for the last two terms above.
For simplicity, we show this only for the first of these, with the other following in the same way. First, defining
$\mu := \mathbb{E}\left\{P^\inn_jf_p\left(P^\inn_k,P^\out_k,W^\inn_p,\overline{\gamma}^\inn_{pk},\overline{\gamma}^\out_{pk}\right)\right\}$
and writing $\overline{\gamma}_{pk} = \left( \overline{\gamma}^{\inn}_{pk}, \overline{\gamma}^{\out}_{pk} \right)$,
\begin{equation}
\begin{split}
  \P&\left(\left|\frac{1}{N}\sum_{i=1}^N[\pv^\inn_j]_if_p([\pv_k]_i,[\wv^\inn_p]_i,\gamma_{pk})-\mu\right|\geq\epsilon\right) \leq \P\left(\left|\frac{1}{N}\sum_{i=1}^N[\pv^\inn_j]_if^\inn_p([\pv_k]_i,[\wv^\inn_p]_i,\overline{\gamma}_{pk})-\mu\right|\geq\frac{\epsilon}{2}\right)\\ 
&\hspace{1cm}+\P\left(\left|\frac{1}{N}\sum_{i=1}^N[\pv^\inn_j]_i \left(f^\inn_p([\pv_k]_i,[\wv^\inn_p]_i,\gamma_{pk})- f^\inn_p([\pv_k]_i,[\wv^\inn_p]_i,\overline{\gamma}_{pk}) \right)\right|\geq\frac{\epsilon}{2}\right). \label{eq:Bk_e_4}
\end{split}
\end{equation}
Let the terms on the right hand side of \eqref{eq:Bk_e_4} be $T1a$ and $T1b$, respectively. First we analyze $T1a$. By assumption, $f^\inn_p(p^\inn,p^\out,w,\gamma^\inn,\gamma^\out)$ is either uniformly Lipschitz or uniformly bounded conditionally Lipschitz (UBCL) at $(\gamma^\inn,\gamma^\out)=(\overline{\gamma}^\inn_{pk},\overline{\gamma}^\out_{pk})$. By definition, this implies eithery way that $f^\inn_p(p^\inn,p^\out,w,\overline{\gamma}^\inn_{pk},\overline{\gamma}^\out_{pk})$ is either Lipschitz in $(p^\inn,p^\out,w)$ or bounded conditionally Lipschitz in $(p^\inn,p^\out)$ for each $w^p$. Thus, the function
\[
\phi([\pv^\inn_0]_i,\ldots,[\pv^\inn_k]_i,[\pv^\out_0]_i,\ldots,[\pv^\out_k]_i,[\wv^\inn_p]_i)=[\pv^\inn_j]_if_p([\pv^\inn_k]_i,[\pv^\out_k]_i,[\wv^\inn_p]_i,\overline{\gamma}^\inn_{pk},\overline{\gamma}^\out_{pk})
\]
is either pseudo-Lipschitz of order $2$ or bounded conditionally pseudo-Lipschitz by Lemma \ref{lipprods}. Therefore, concentration for $T1a$ follows from $\mathbf{P_k.(b)}$.

Now for $T1b$, we have by the definition of uniform Lipschitz continuity and UBCL that there is some $\delta>0$ such that $\left|\gamma^\inn_{pk}-\overline{\gamma}^\inn_{pk}\right|<\delta$ and $\left|\gamma^\out_{pk}-\overline{\gamma}^\out_{pk}\right|<\delta$ implies that 
\begin{align*}
&\left|f^\inn_p(p^\inn,p^\out,w,\gamma^\inn_{pk},\gamma^\out_{pk})-f^\inn_p(p^\inn,p^\out,w,\overline{\gamma}^\inn_{pk},\overline{\gamma}^\out_{pk})\right| \\
&\qquad \qquad \leq L_p\left(1+\|(p^\inn,p^\out,w)\|\right)\left[\left|\gamma^\inn_{pk}-\overline{\gamma}^\inn_{pk}\right|+\left|\gamma^\out_{pk}-\overline{\gamma}^\out_{pk}\right|\right].
\end{align*}
Following the work in \eqref{eq:Bk_f_2}- \eqref{eq:T1_bound1_pkf}, we have the upper bound
\begin{align}\label{eq:Bk_e_5}
  T_1b &\leq \P\left(\left|\gamma^\inn_{pk}-\overline{\gamma}^\inn_{pk}\right| \geq \min\{1, \delta\} \right) +\P\left(\left|\gamma^\out_{pk}-\overline{\gamma}^\out_{pk}\right| \geq \min\{1, \delta\}  \right)\\
  &+ \P\left(L_p\left[\left|\gamma^\inn_{pk}-\overline{\gamma}^\inn_{pk}\right|+\left|\gamma^\out_{pk}-\overline{\gamma}^\out_{pk}\right|\right]\frac{1}{N}\sum_{i=1}^N\left[ |[\pv^\inn_j]_i|\left(1+\|([\pv^\inn_k]_i,[\pv^\out_k]_i,[\wv^\inn_p]_i)\|\right)\right]\geq\frac{\epsilon}{2}\right).\nonumber
\end{align}
The first two terms above concentrates by $\mathbf{Q_k.(f)}$, and the third concentrates using Lemma \ref{products_0} combined with $\mathbf{Q_k.(f)}$ and $\mathbf{P_k.(b)}$ since 
\[
\phi([\pv^\inn_0]_i,\ldots,[\pv^\inn_k]_i,[\pv^\out_0]_i,\ldots,[\pv^\out_k]_i,[\wv^\inn_p]_i)=|[\pv^\inn_j]_i|\left(1+\|([\pv^\inn_k]_i,[\pv^\out_k]_i,[\wv^\inn_p]_i)\|\right)
\]
is pseudo-Lipschitz of order $2$ by Lemma \ref{lipprods}. This establishes concentration for $T1b$.

\textbf{Concentration for T2 of \eqref{eq:pke_firstsplit}.}
We have to show concentration for
\begin{equation}
\label{eq:T2}
T2 = \P\left(\left|\frac{1}{N}\sum_{i=1}^N g_p([\pv_k]_i,[\wv^\inn_p]_i,\gamma_{pk},\overline{\alpha}^\inn_{pk})g_p([\pv_j]_i,[\wv^\inn_p]_i,\gamma_{pj},\overline{\alpha}^\inn_{pk}) - \left[1-\overline{\alpha}^{\inn}_{pk} \right]\left[ 1 - \overline{\alpha}^\inn_{pj} \right]\mathbb{E}[V^\inn_kV^\inn_j]\right|\geq\frac{\epsilon}{2}\right).
\end{equation}
Adding and subtracting $\frac{1}{N}\sum_{i=1}^N g_p([\pv_k]_i,[\wv^\inn_p]_i,\overline{\gamma}_{pk},\overline{\alpha}^\inn_{pk})g_p([\pv_j]_i,[\wv^\inn_p]_i,\overline{\gamma}_{pj},\overline{\alpha}^\inn_{pk})$, we find that from Lemma~\ref{sums},
\begin{equation}
\begin{split}
\label{eq:T2a}
&T2 \leq \P\left(\left|\frac{1}{N}\sum_{i=1}^N g_p([\pv_k]_i,[\wv^\inn_p]_i,\overline{\gamma}_{pk},\overline{\alpha}^\inn_{pk})g_p([\pv_j]_i,[\wv^\inn_p]_i,\overline{\gamma}_{pj},\overline{\alpha}^\inn_{pk}) - \left[1-\overline{\alpha}^{\inn}_{pk} \right]\left[ 1 - \overline{\alpha}^\inn_{pj} \right]\mathbb{E}[V^\inn_kV^\inn_j]\right|\geq\frac{\epsilon}{2}\right) \\
&+  \P\left(\left|\frac{1}{N}\sum_{i=1}^N \left[g_p([\pv_k]_i,[\wv^\inn_p]_i,\gamma_{pk},\overline{\alpha}^\inn_{pk})g_p([\pv_j]_i,[\wv^\inn_p]_i,\gamma_{pj},\overline{\alpha}^\inn_{pk}) \right. \right.\right. \\
&\hspace{2cm} - \left. \left.  \left. g_p([\pv_k]_i,[\wv^\inn_p]_i,\overline{\gamma}_{pk},\overline{\alpha}^\inn_{pk})g_p([\pv_j]_i,[\wv^\inn_p]_i,\overline{\gamma}_{pj},\overline{\alpha}^\inn_{pk}) \right]\right|\geq\frac{\epsilon}{2}\right).
\end{split}
\end{equation}
By the same reasoning used above, $g_p(p^\inn,p^\out,w,\overline{\gamma}^\inn_{pk},\overline{\gamma}^\out_{pk},\overline{\alpha}^\inn_{pk})$ is Lipschitz in $(p^\inn, p^\out, w)$ or UBCL in $(p^\inn,p^\out)$ for each $w$. Therefore, using Lemma \ref{lipprods} and $\mathbf{P_k.(b)}$, the function in the first term on the right side of \eqref{eq:T2a} is either pseudo-Lipschtz or bounded conditionally pseudo-Lipschitz, and thus the first term in \eqref{eq:T2a} concentrates
because
\[
 \left( 1-\overline{\alpha}^{\inn}_{pk} \right)\left( 1-\overline{\alpha}^{\inn}_{pj} \right)\mathbb{E}\{V^\inn_kV^\inn_j \}= \mathbb{E}\left\{g_p\left( P^\inn_k,P^\out_k,W^\inn_p,\overline{\gamma}^\inn_{pk},\overline{\gamma}^\out_{pk},\overline{\alpha}^\inn_{pk} \right)g_p\left( P^\inn_j,P^\out_j,W^\inn_p,\overline{\gamma}^\inn_{pj},\overline{\gamma}^\out_{pj},\overline{\alpha}^\inn_{pj} \right)\right\}.
\]

Thus, we just need to show that the second term on the right side of \eqref{eq:T2a} concentrates.
Applying the definition of the function $g_p$ in \eqref{eq:gp_func_def} and expanding terms, we see that
\begin{equation}
\begin{split}
&\left|\frac{1}{N}\sum_{i=1}^N \left[g_p([\pv_k]_i,[\wv^\inn_p]_i,\gamma_{pk},\overline{\alpha}^\inn_{pk})g_p([\pv_j]_i,[\wv^\inn_p]_i,\gamma_{pj},\overline{\alpha}^\inn_{pk})  - g_p([\pv_k]_i,[\wv^\inn_p]_i,\overline{\gamma}_{pk},\overline{\alpha}^\inn_{pk})g_p([\pv_j]_i,[\wv^\inn_p]_i,\overline{\gamma}_{pj},\overline{\alpha}^\inn_{pk}) \right]\right| \\
&\leq   \overline{\alpha}^\inn_{pk}\frac{1}{N}\sum_{i=1}^N |[\pv^\inn_k]_i|\left|f^\inn_p([\pv_j]_i,[\wv^\inn_p]_i,\gamma_{pj})-f_p([\pv_j]_i,[\wv^\inn_p]_i,\overline{\gamma}_{pj})\right| \\
&\quad +   \overline{\alpha}^\inn_{pj}\frac{1}{N}\sum_{i=1}^N |[\pv^\inn_j]_i|\left|f^\inn_p([\pv_k]_i,[\wv^\inn_p]_i,\gamma_{pk})-f_p([\pv_k]_i,[\wv^\inn_p]_i,\overline{\gamma}_{pk})\right| \\
& \quad + \frac{1}{N}\sum_{i=1}^N \Big|f^\inn_p([\pv_k]_i,[\wv^\inn_p]_i,\gamma_{pk})f^\inn_p([\pv_j]_i,[\wv^\inn_p]_i,\gamma_{pj})-f_p([\pv_k]_i,[\wv^\inn_p]_i,\overline{\gamma}_{pk})f^\inn_p([\pv_j]_i,[\wv^\inn_p]_i,\overline{\gamma}_{pj})\Big|.
\label{eq:gp_split1}
\end{split}
\end{equation}
Notice that the first two terms on the right side of \eqref{eq:gp_split1} concentrate in the same way as term T1b of \eqref{eq:Bk_e_4}. Namely, using the uniformly Lipschitz or UBCL property of $f_p^{\inn}$, we can bound the sums by sums of (bounded conditionally) pseudo-Lipschitz functions.
Considering the third term on the right side of \eqref{eq:gp_split1}, we apply Lemma~\ref{sums} and then to complete the proof, we obtain concentration for the following three quantities:
\begin{align*}
  & \P\Big(\frac{1}{N}\sum_{i=1}^N \left|f^\inn_p([\pv_k]_i,[\wv^\inn_p]_i,\overline{\gamma}_{pk})\right|\Big|f^\inn_p([\pv_j]_i,[\wv^\inn_p]_i,\gamma_{pj})-f^\inn_p([\pv_j]_i,[\wv^\inn_p]_i,\overline{\gamma}_{pj})\Big|\geq\epsilon\Big),\\
   &\P\Big(\frac{1}{N}\sum_{i=1}^N \left|f^\inn_p([\pv_j]_i,[\wv^\inn_p]_i,\overline{\gamma}_{pj})\right|\Big|f^\inn_p([\pv_k]_i,[\wv^\inn_p]_i,\gamma_{pk})-f^\inn_p([\pv_k]_i,[\wv^\inn_p]_i,\overline{\gamma}_{pk})\Big|\geq\epsilon\Big),\\
   & \P\Big(\frac{1}{N}\sum_{i=1}^N \Big|f^\inn_p([\pv_k]_i,[\wv^\inn_p]_i,\gamma_{pk})-f^\inn_p([\pv_k]_i,[\wv^\inn_p]_i,\overline{\gamma}_{pk})\Big|\Big| f^\inn_p([\pv_j]_i,[\wv^\inn_p]_i,\gamma_{pj})-f^\inn_p([\pv_j]_i,[\wv^\inn_p]_i,\overline{\gamma}_{pj})\Big|\geq\epsilon\Big).
\end{align*}
Let these probabilities be labeled as $T2a - T2c$. 

Following the work in \eqref{eq:Bk_f_2}- \eqref{eq:T1_bound1_pkf} and using the uniformly (bounded conditionally) Lipschitz property of $f_p$, we have the upper bound
\begin{align}\label{eq:Bk_e_5}
  T_2a &\leq \P\left(\left|\gamma^\inn_{pk}-\overline{\gamma}^\inn_{pk}\right| \geq \min\{1, \delta\} \right) +\P\left(\left|\gamma^\out_{pk}-\overline{\gamma}^\out_{pk}\right| \geq \min\{1, \delta\}  \right)\\
  &+ \P\left(\left[\left|\gamma^\inn_{pj}-\overline{\gamma}^\inn_{pj}\right|+\left|\gamma^\out_{pj}-\overline{\gamma}^\out_{pj}\right|\right]\frac{L_1}{N}\sum_{i=1}^N \left|f^\inn_p([\pv_j]_i,[w^\inn_p]_i,\overline{\gamma}_{pk})\right|\left(1+\|([\pv_k]_i,[\wv^\inn_p]_i)\|\right)\geq \epsilon\right).\nonumber
\end{align}

The first two terms can be bounded using $\mathbf{Q_k.(f)}$. The last term, in turn, can be upper bounded using Lemma \ref{products_0} along with $\mathbf{Q_k.(f)}$ and $\mathbf{P_k.(b)}$ since by Lemma \ref{lipprods}
\[
\phi([\pv^\inn_0]_i,\ldots,[\pv^\inn_k]_i,[\pv^\out_0]_i,\ldots,[\pv^\out_k]_i,[\wv^p]_i)=\left|f^\inn_p([\pv_j]_i,[\wv^\inn_p]_i,\overline{\gamma}_{pk})\right|\left(1+\|([\pv_k]_i,[\wv^\inn_p]_i)\|\right)
\]
is (bounded conditionally) pseudo-Lipschitz of order $2$. Term $T2b$ can be bounded in exactly the same way as  $T2a$.

 For $T_2c$, we again use the strategy as in \eqref{eq:Bk_f_2}- \eqref{eq:T1_bound1_pkf} to provide the following upper bound for term $T_2c$. We have
 \begin{align}\label{eq:Bk_e_11}
T_2c& \leq  \P\left(\left|\gamma^\inn_{pk}-\overline{\gamma}^\inn_{pk}\right| \geq \min\{1, \delta\} \right) +\P\left(\left|\gamma^\out_{pk}-\overline{\gamma}^\out_{pk}\right| \geq \min\{1, \delta\}  \right) \nonumber \\
&+ \P\left(\left|\gamma^\inn_{pj}-\overline{\gamma}^\inn_{pj}\right| \geq \min\{1, \delta\} \right) +\P\left(\left|\gamma^\out_{pj}-\overline{\gamma}^\out_{pj}\right| \geq \min\{1, \delta\}  \right)  \nonumber \\
  &+ \P\Big(\left[\left|\gamma^\inn_{pj}-\overline{\gamma}^\inn_{pj}\right|+\left|\gamma^\out_{pj}-\overline{\gamma}^\out_{pj}\right|\right]\left[\left|\gamma^\inn_{pk}-\overline{\gamma}^\inn_{pk}\right|+\left|\gamma^\out_{pk}-\overline{\gamma}^\out_{pk}\right|\right]\nonumber\\
    &\hspace{2cm}\times\frac{L_1^2}{N}\sum_{i=1}^N\left(1+\|([\pv_k]_i,[\wv^\inn_p]_i)\|\right)\left(1+\|([\pv_j]_i,[\wv^\inn_p]_i)\|\right)\geq\frac{\epsilon}{2}\Big).
\end{align}

 We know from $\mathbf{Q_k.(f)}$ and $\mathbf{Q_j.(f)}$ that the first four terms concentrate as needed. Then the last term in \eqref{eq:Bk_e_11} can be upper bounded with repeated application of Lemmas \ref{products_0} and \ref{sums} along with $\mathbf{Q_k.(f)}$, $\mathbf{Q_j.(f)}$, $\mathbf{P_k.(b)}$, and the fact that
\[
\phi([\pv^\inn_0]_i,\ldots,[\pv^\inn_k]_i,[\pv^\out_0]_i,\ldots,[\pv^\out_k]_i,[\wv^\inn_p]_i)=\left(1+\|([\pv_k]_i,[\wv^\inn_p]_i)\|\right)\left(1+\|([\pv_j]_i,[\wv^\inn_p]_i)\|\right)
\]
is pseudo-Lipschitz of order $2$ by Lemma \ref{lipprods}.


$\mathbf{P_k.(g).(i)}$

We want to show that $\frac{1}{N}\|[\Bm^\perp_{\Cm_{vk}^\inn}]^T\pv^\inn_k\|^2$ concentrates around $\rho^\inn_{pk}$. To this end, observe
\begin{align}\label{eq:Bk_g_1}
  \|[\Bm^\perp_{\Cm_{vk}^\inn}]^T\pv^\inn_k\|^2 &= [\pv^\inn_k]^T\pv^\inn_k- [\pv^\inn_k]^T\Bm_{\Cm_{vk}^\inn}[\Bm_{\Cm_{vk}^\inn}]^T\pv^\inn_k \nonumber\\
  &= [\uv^\inn_k]^T\uv^\inn_k- [\pv^\inn_k]^T\Cm^\inn_{vk}([\Cm^\inn_{vk}]^T\Cm^\inn_{vk})^{-1}[\Cm^\inn_{vk}]^T\pv^\inn_k,
\end{align}
where the second equality follows because $\pv^\inn_k=\Vm\uv^\inn_k$ along with the fact that $\Bm_{\Cm_{vk}^\inn}[\Bm_{\Cm_{vk}^\inn}]^T$ and $\Cm^\inn_{vk}([\Cm^\inn_{vk}]^T\Cm^\inn_{vk})^{-1}[\Cm^\inn_{vk}]^T$ are both the orthogonal projection onto $\mathrm{range}\left( \Cm^\inn_{vk} \right)$.

We know that $\frac{1}{N}[\uv^\inn_k]^T\uv^\inn_k$ concentrates around $\left(\covm^\inn_{uk}\right)_{k+1,k+1}$ by $\mathbf{Q_k.(e)}$. Next observe that with $\widetilde{\Cm}^\inn_{k} := ([\Cm^\inn_{vk}]^T\Cm^\inn_{vk})^{-1}$, from the definition of $\Cm^\inn_{vk}$ in \eqref{eq:Cmat},  we have

\begin{align}\label{eq:Bk_g_2}
  &[\pv^\inn_k]^T\Cm^\inn_{vk}([\Cm^\inn_{vk}]^T\Cm^\inn_{vk})^{-1}[\Cm^\inn_{vk}]^T\pv^\inn_k \nonumber \\
  &\qquad = \sum_{i,j=0}^{k-1}([\pv^\inn_k]^T\pv^\inn_i)([\pv^\inn_k]^T\pv^\inn_j)[\widetilde{\Cm}^\inn_k]_{i+1,j+1}+ \sum_{i,j=0}^{k-1}([\pv^\inn_k]^T\vv^\inn_i)([\pv^\inn_k]^T\vv^\inn_j)[\widetilde{\Cm}^\inn_k]_{ i+1+k, j+1+k}\nonumber\\
  &\qquad \qquad +2\sum_{i,j=0}^{k-1}([\pv^\inn_k]^T\pv^\inn_i)([\pv^\inn_k]^T\vv^\inn_j)[\widetilde{\Cm}^\inn_k]_{i+1 + k, j+1}.
\end{align}

Let these terms be labeled $T_1-T_3$. For $T_1$, observe that because $\pv^\inn_k=\Vm\uv^\inn_k$,
\begin{equation}
\label{eq:v_conc_g}
  ([\pv^\inn_k]^T\pv^\inn_i)([\pv^\inn_k]^T\pv^\inn_j)=([\uv^\inn_k]^T\uv^\inn_i)([\uv^\inn_k]^T\uv^\inn_j),
\end{equation}
which concentrates around $[\boldsymbol{\Sigma}^{\inn}_{uk}]_{i+1, k+1}[\boldsymbol{\Sigma}^{\inn}_{uk}]_{j+1, k+1} =  \E\left[U^\inn_i U^\inn_k\right]  \E\left[U^\inn_j U^\inn_k\right] = [\bv^\inn_{uk}]_{i+1}[\bv^\inn_{uk}]_{j+1} $ by $\mathbf{Q_{k}.(e)}$ and Lemma \ref{products}, where the equality follows from \eqref{eq:bvecs} and \eqref{eq:cov_def}. Furthermore, by $\mathbf{P_{k-1}.(g).(iv)}$, $[\widetilde{\Cm}^\inn_k]_{i,j} = [([\Cm^\inn_{vk}]^T\Cm^\inn_{vk})^{-1}]_{i,j} $ concentrates around $[\presm^{\inn}_{v(k-1)}]_{ij} = [(\covm^\inn_{u(k-1)})^{-1}]_{i,j}$ (by \eqref{eq:Rk_Sk}).
Thus, again using Lemma \ref{products} and Lemma \ref{sums}, we get that $T_1$ concentrates around
\[
\sum_{i,j=0}^{k-1} [\bv^\inn_{uk}]_{i+1}[\bv^\inn_{uk}]_{j+1}[(\covm^\inn_{u(k-1)})^{-1}]_{i+1,j+1} = [\bv^\inn_{uk}]^T[\covm^\inn_{u(k-1)}]^{-1}\bv^\inn_{uk}.
\]

For $T_2$, we proceed similarly, noting that $([\pv^\inn_k]^T\vv^\inn_i)([\pv^\inn_k]^T\vv^\inn_j)$ concentrates around $0$ by $\mathbf{P_k.(f)}$ and Lemma \ref{products}. Because $[\widetilde{\Cm}^\inn_k]_{i+ 1+k,j+1+k}$ concentrates around $[(\covm^\inn_{v(k-1)})^{-1}]_{i,j}$ by $\mathbf{P_{k-1}.(g).(iv)}$, by Lemma \ref{products_0} and \ref{sums} we find that $T_2$ concentrates around $0$. By analogous reasoning, $T_3$ also concentrates around $0$.

Thus, combining the above with Lemma \ref{sums} and using \eqref{eq:rhos}, we find $\|\Bm^\perp_{\Cm_{vk}^\inn}\pv^\inn_k\|^2$ concentrates on
\[
\left[\covm^\inn_{uk}\right]_{k+1,k+1}-[\bv^\inn_{uk}]^T[\covm^\inn_{u(k-1)}]^{-1}[\bv^\inn_{uk}] = \rho^\inn_{pk}.
\]

$\mathbf{P_k.(g).(ii)}$
First we establish that $[\Cm^\inn_{pk}]^T\Cm^\inn_{pk}$ is invertible with high probability. Before we can do this, we make a minor technical adjustment to simplify notation. First recall  from \eqref{eq:Cmat}  that $\Cm^\inn_{pk} = [\Pm^\inn_k\;\Vm^\inn_{k-1}]$. Next, define the $(2k-1)\times (2k-1)$ matrix
$\Fm = \left[\ev_k\;\ev_1\;\ev_2\;\cdots\;\ev_{k-1}\;\ev_{k+1}\;\cdots\;\ev_{2k-1}\right].$
Observe that multiplying a matrix on the right by $\Fm$ has the effect of making the $k^\mathrm{th}$ column the first column. Also observe that $\Fm$ has an inverse, we call this $\Gm$ and note that $\Fm \Gm = \Gm \Fm = \mathbb{I}$. Now with this notation, let
$\overline{\Cm}^\inn_{pk} := \Cm^\inn_{pk}\Fm = \left[\pv^\inn_k\; \Cm^\inn_{vk}\right]$ where $\Cm^\inn_{vk} = [\Pm^\inn_{k-1}\;\Vm^\inn_{k-1}].$
Furthermore, we see
\begin{equation}
\label{eq:Ginverse}
\Gm^T\left([\overline{\Cm}^\inn_{pk}]^T\overline{\Cm}^\inn_{pk}\right)\Gm = \Gm^T\left([\Cm^\inn_{pk}\Fm]^T\Cm^\inn_{pk}\Fm\right)\Gm = [\Cm^\inn_{pk}]^T\Cm^\inn_{pk},
\end{equation}
which implies that if $[\overline{\Cm}^\inn_{pk}]^T\overline{\Cm}^\inn_{pk}$ is invertible, then so is $[\Cm^\inn_{pk}]^T\Cm^\inn_{pk}$. Symmetrically, we see that $[\Cm^\inn_{pk}]^T\Cm^\inn_{pk}$ is invertible if and only if $[\overline{\Cm}^\inn_{pk}]^T\overline{\Cm}^\inn_{pk}$ is invertible. Moreover, we have from \eqref{eq:Ginverse} that when $[\Cm^\inn_{pk}]^T\Cm^\inn_{pk}$ is invertible,
\[
([\Cm^\inn_{pk}]^T\Cm^\inn_{pk})^{-1} = [\Gm^T\left([\overline{\Cm}^\inn_{pk}]^T\overline{\Cm}^\inn_{pk}\right)\Gm ]^{-1} = \Fm([\overline{\Cm}^\inn_{pk}]^T\overline{\Cm}^\inn_{pk})^{-1}\Fm^T.
\]
Now, the right hand side of the above differs from $([\overline{\Cm}^\inn_{pk}]^T\overline{\Cm}^\inn_{pk})^{-1}$ by rearranging the columns and rows. In particular, the right hand side has the same entries as $([\overline{\Cm}^\inn_{pk}]^T\overline{\Cm}^\inn_{pk})^{-1}$ up to rearrangement. Therefore it suffices to show that the entries of $([\overline{\Cm}^\inn_{pk}]^T\overline{\Cm}^\inn_{pk})^{-1}$ concentrate. 

To this end, we observe that
\begin{equation}\label{eq:Bk_g_5a}
[\overline{\Cm}^\inn_{pk}]^T\overline{\Cm}^\inn_{pk} = \left[\begin{matrix} [\pv^\inn_k]^T\pv^\inn_k & [\pv^\inn_k]^T \Cm^\inn_{vk} \\ [\Cm^\inn_{vk}]^T\pv^\inn_k & [\Cm^\inn_{vk}]^T\Cm^\inn_{vk}\end{matrix}\right].
\end{equation}
Now, by \eqref{eq:Bk_g_1} in $\mathbf{P_k.(g).(i)}$, we have that
\[
\|[\Bm^\perp_{\Cm_{vk}^\inn}]^T\pv^\inn_k\|^2 = [\uv^\inn_k]^T\uv^\inn_k - [\pv^\inn_k]^T\Cm^\inn_{vk}([\Cm^\inn_{vk}]^T\Cm^\inn_{vk})^{-1}[\Cm^\inn_{vk}]^T\pv^\inn_k.
\]
Define $\nuv = ([\Cm^\inn_{vk}]^T\Cm^\inn_{vk})^{-1}[\Cm^\inn_{vk}]^T\pv^\inn_k$. Using \eqref{eq:Bk_g_5a}, the block matrix inversion formula gives us that, if $[\Cm^\inn_{vk}]^T\Cm^\inn_{vk}$ is invertible and $\|[\Bm^\perp_{\Cm_{vk}^\inn}]^T\pv^\inn_k\|\neq 0$, then
\begin{equation}\label{eq:Bk_g_5}
\left([\overline{\Cm}^\inn_{pk}]^T\overline{\Cm}^\inn_{pk}\right)^{-1} = \left[\begin{matrix} 0 & 0 \\ 0 & ([\Cm^\inn_{vk}]^T\Cm^\inn_{vk})^{-1}\end{matrix}\right] + \frac{1}{\|[\Bm^\perp_{\Cm_{vk}^\inn}]^T\pv^\inn_k\|^2}\left[\begin{matrix} 1 & -\nuv^T\\ -\nuv & \nuv\nuv^T\end{matrix}\right].
\end{equation}
Now, by $\mathbf{P_{k}.(g).(i)}$, we know that $\|[\Bm^\perp_{\Cm_{vk}^\inn}]^T\pv^\inn_k\|^2$ concentrates around $\rho^\inn_{pk}>\epsilon^*_1>0$ (where the first inequality is due to the stopping criterion). Then, we have
\begin{align}\label{eq:Bk_g_3}
  \P\left([\Cm^\inn_{pk}]^T\Cm^\inn_{pk}\text{ singular}\right)&=\P\left([\overline{\Cm}^\inn_{pk}]^T\overline{\Cm}^\inn_{pk}\text{ singular}\right)\nonumber\\
  &\leq \P\left([\Cm^\inn_{vk}]^T\Cm^\inn_{vk}\text{ singular}\right) + \P\left(\left|\frac{1}{N}\|[\Bm^\perp_{\Cm_{vk}^\inn}]^T \pv^\inn_k\|^2 - \rho^\inn_{pk}\right|\geq\frac{\epsilon^*_1}{2}\right).
\end{align}
The first term concentrates by $\mathbf{P_{k-1}.(g).(iv)}$, and the second concentrates by $\mathbf{P_k.(g).(i)}$.

Next, we show entrywise concentration of $([\Cm^\inn_{pk}]^T\Cm^\inn_{pk})^{-1}$ around $\presm^\inn_{uk} \stackrel{\cdot}{=} \left[\begin{matrix}[\covm^\inn_{uk}]^{-1} & 0 \\ 0 & [\covm^\inn_{v(k-1)}]^{-1}\end{matrix}\right]$, where the definition is originally given in \eqref{eq:Rk_Sk}. First, define two events $\mathcal{F}^1_k := \left\lbrace [\Cm^\inn_{pk}]^T\Cm^\inn_{pk}\text{ singular}\right\rbrace$ and $\mathcal{F}^2_{k-1} := \left\lbrace [\Cm^\inn_{vk}]^T\Cm^\inn_{vk}\text{ singular}\right\rbrace$. Then, notice that
\begin{align}\label{eq:Bk_g_4}
 & \P\left(\left|\left[([\Cm^\inn_{pk}]^T\Cm^\inn_{pk})^{-1}\right]_{ij} - \left[\presm^\inn_{uk}\right]_{ij}\right|\geq \epsilon\right) \nonumber\\
  &\hspace{1cm} \leq \P(\mathcal{F}^1_k)+\P(\mathcal{F}^2_{k-1}) +\P\left(\left|\left[([\Cm^\inn_{pk}]^T\Cm^\inn_{pk})^{-1}\right]_{ij} - \left[\presm^\inn_{uk}\right]_{ij}\right|\geq \epsilon\;\Bigg\vert\; (\mathcal{F}^1_k)^c\cap (\mathcal{F}^2_{k-1})^c\right),
\end{align}
where $(\mathcal{F}^1_k)^c$ and $ (\mathcal{F}^2_{k-1})^c$ denote the complimentary events. We mention that the first two terms on the right side of \eqref{eq:Bk_g_4} concentrate by \eqref{eq:Bk_g_3} and $\mathbf{P_{k-1}.(g).(iv)}$. Moving forward, we focus on the third term on the right side of \eqref{eq:Bk_g_4}.

As noted above, we can reorder the columns, so it suffices to replace $[\Cm^\inn_{pk}]^T\Cm^\inn_{pk}$ by $[\overline{\Cm}^\inn_{pk}]^T\overline{\Cm}^\inn_{pk}$ and $\presm^\inn_{pk}$ by $\Fm\presm^\inn_{pk}\Fm^T$. Notice that $\Fm A^{-1}\Fm^T = (\Fm A \Fm^T)^{-1}$.  Then, by the block matrix inversion formula and the fact that $\E\left[P_k^{\inn/\out}\right]^2 = \E\left[U_k^{\inn/\out}\right]^2$we find
\begin{equation}
\begin{split}
\Fm\presm^\inn_{uk}\Fm^T = \left( \Fm \left[\begin{matrix} \covm^\inn_{uk} & 0\\ 0 &\covm^\inn_{v(k-1)}\end{matrix}\right] \Fm^T\right)^{-1} &=\left[\begin{matrix}  \mathbb{E}\{[P_k^{\inn}]^2\} & \bv^\inn_{pk} & 0 \\ (\bv^\inn_{pk})^T  & \covm^\inn_{u(k-1)}  & 0\\ 0 & 0 &\covm^\inn_{v(k-1)}\end{matrix}\right]^{-1} \\
&
= \left[\begin{matrix}0 & 0\\ 0 & \presm^\inn_{v(k-1)}\end{matrix}\right] + \frac{1}{\rho^\inn_{pk}}\left[\begin{matrix}1 & -[\betav^\inn_{pk}\;0]\\ -\left[\begin{matrix}\betav^\inn_{pk}\\0\end{matrix}\right] &\left[\begin{matrix}\betav^\inn_{pk}\\0\end{matrix}\right][\betav^\inn_{pk}\;0]\end{matrix}\right].
\end{split}
\end{equation}
In the above, we have used that $\rho^\inn_{pk} = \mathbb{E}\{[U_k^{\inn}]^2\} - (\bv^\inn_{uk})^T  \left[\covm^\inn_{u(k-1)}\right]^{-1}  \bv^\inn_{uk}$ and $\betav^\inn_{pk} =  \left[\covm^\inn_{u(k-1)}\right]^{-1}\bv^\inn_{uk}$  by \eqref{eq:rhos}.
Considering \eqref{eq:Bk_g_4} and using Lemmas \ref{sums} and \ref{products}, it suffices to show that $([\Cm^\inn_{vk}]^T\Cm^\inn_{vk})^{-1}$ concentrates on $\presm^\inn_{v(k-1)}$, $\|[\Bm^\perp_{\Cm_{vk}^\inn}]^T \pv^\inn_k\|^{-2}$ concentrates on $[\rho^\inn_{pk}]^{-1}$, and $\nuv$ concentrates on $[\betav^\inn_{pk}\;0]^T$. But $\mathbf{P_{k-1}.(g).(iv)}$ gives us concentration for $([\Cm^\inn_{vk}]^T\Cm^\inn_{vk})^{-1}$, and $\mathbf{P_k.(g).(i)}$ along with Lemma $A.6$ gives concentration for $\|\Bm^\perp_{\Cm_{vk}^\inn} \pv^\inn_k\|^{-2}$. Finally, observe that if $0\leq i\leq k-1$, we have
\begin{equation}\label{eq:Bk_g_6}
\left|[\nuv]_i - [[\betav^\inn_{pk}]^T\;0]_i\right| \leq \left|\sum_{j=0}^{k-1} [\widetilde{\Cm}^\inn_{vk}]_{ij}([\pv^\inn_j]^T\pv^\inn_k) - [\betav^\inn_{pk}]_j\right|+\left|\sum_{j=0}^{k-1}[\widetilde{\Cm}^\inn_{vk}]_{i,j+(k-1)}([\vv^\inn_i]^T\pv^\inn_k)\right|.
\end{equation}
By $\mathbf{P_{k-1}.(g).(iv)}$, we get concentration for the values of $\widetilde{\Cm}^\inn_{vk}$. In particular, if $0\leq i,j\leq k-1$, then we have $[\widetilde{\Cm}^\inn_{vk}]_{i,j}$ concentrates around $[\covm^\inn_{u(k-1)}]^{-1}_{ij}$. Moreover, by $\mathbf{P_k.(d)}$ and $\mathbf{P_k.(f)}$, we have that $[\pv^\inn_j]^T\pv^\inn_k$ concentrates around $[\bv^\inn_{pk}]_j$ and $[\vv^\inn_i]^T\pv^\inn_k$ concentrates around $0$. Thus by Lemmas \ref{products} and \ref{sums}, we get that the right hand term above concentrates around $0$ and the left hand term concentrates around 
\[
\sum_{j=0}^{k-1} [\covm^\inn_{u(k-1)}]^{-1}_{ij}[\bv^\inn_{pk}]_j = \left[\left[\covm^\inn_{u(k-1)}\right]^{-1}\bv^\inn_{uk}\right]_i = [\betav^\inn_{pk}]_i.
\]

The same analysis applies if $k\leq i\leq 2(k-1)$ except that now for $0\leq i\leq k-1$, we have $[\widetilde{\Cm}^\inn_{vk}]_{ij}$ concentrates around 0 (again by $\mathbf{P_{k-1}.(g).(iv)}$), which gives us that $\nuv_i$ concentrates around $0$, as desired. Thus, we get that the entries of $([\Cm^\inn_{pk}]^T\Cm^\inn_{pk})^{-1}$ concentrate around those of $\presm^\inn_{pk}$, as needed.

$\mathbf{P_k.(g).(iii)}$
We will show that $\frac{1}{N}\|[\Bm^\perp_{\Cm_{pk}^\inn}]^T\vv^\inn_k\|^2$ concentrates to $\rho^\inn_{qk}$. To this end, observe that
\[
\|[\Bm^\perp_{\Cm_{pk}^\inn}]^T\vv^\inn_k\|^2  = [\vv^\inn_k]^T\vv^\inn_k- [\vv^\inn_k]^T\Bm_{\Cm^{\inn}_{pk}}[\Bm_{\Cm^{\inn}_{pk}}]^T\vv^\inn_k 
 = [\vv^\inn_k]^T\vv^\inn_k- [\vv^\inn_k]^T\Cm^\inn_{pk}([\Cm^\inn_{pk}]^T\Cm^\inn_{pk})^{-1}[\Cm^\inn_{pk}]^T\vv^\inn_k,
\]
where the second equality follows because $\Bm_{\Cm^{\inn}_{pk}}[\Bm_{\Cm^{\inn}_{pk}}]^TT$ and $\Cm^\inn_{pk}([\Cm^\inn_{pk}]^T\Cm^\inn_{pk})^{-1}[\Cm^\inn_{pk}]^T$ are both the orthogonal projection onto $\mathrm{range}\left( \Cm^{\inn}_{pk} \right)$.

First observe that $\frac{1}{N}[\vv^\inn_k]^T\vv^\inn_k$ concentrates around $[\covm^\inn_{vk}]_{k+1,k+1}$ by $\mathbf{P_k.(e)}$. Next observe that defining $\widetilde{\Cm}^\inn_{pk} = ([\Cm^\inn_{pk}]^T\Cm^\inn_{pk})^{-1}$ and using from the definition of $\Cm^\inn_{pk}$ in \eqref{eq:Cmat},  we have
 \begin{equation}
\begin{split}
\label{eq:g_expansion}
 & [\vv^\inn_k]^T\Cm^\inn_{pk}([\Cm^\inn_{pk}]^T\Cm^\inn_{pk})^{-1}[\Cm^\inn_{pk}]^T\vv^\inn_k \\
  &\qquad = \sum_{i,j=0}^{k}([\vv^\inn_k]^T\pv^\inn_i)([\vv^\inn_k]^T\pv^\inn_j)[\widetilde{\Cm}^\inn_{pk}]_{i,j} + \sum_{i,j=0}^{k-1}([\vv^\inn_k]^T\vv^\inn_i)([\vv^\inn_k]^T\vv^\inn_j)[\widetilde{\Cm}^\inn_{pk}]_{i+1+k,j+1+k}\nonumber\\
  &\qquad \hspace{2cm}+2\sum_{i=0}^k\sum_{j=0}^{k-1}([\vv^\inn_k]^T\pv^\inn_i)([\vv^\inn_k]^T\vv^\inn_j)[\widetilde{\Cm}^\inn_{pk}]_{i+1+k,j}.
\end{split}
\end{equation}
Let the three terms of \eqref{eq:g_expansion}  be labeled $T_1-T_3$. For $T_1$, observe that $([\vv^\inn_k]^T\pv^\inn_i)([\vv^\inn_k]^T\pv^\inn_j)$ concentrates around $0$ for $0 \leq i,j \leq k$ by $\mathbf{P_k.(f)}$. Because $[\widetilde{\Cm}^\inn_{pk}]_{i,j}$ concentrates around $[\presm^{\inn}_{uk}]_{i,j} = [(\covm^\inn_{pk})^{-1}]_{i,j}$ when  $0 \leq i,j \leq k$ by $\mathbf{P_k.(g).(ii)}$ and \eqref{eq:Rk_Sk}, we get by Lemma \ref{products} and \ref{sums} that $T_1$ concentrates around $0$.

For $T_2$, observe that $([\vv^\inn_k]^T\vv^\inn_i)([\vv^\inn_k]^T\vv^\inn_j)$ concentrates around $[\boldsymbol{\Sigma}^{\inn}_{vk}]_{i+1, k+1}[\boldsymbol{\Sigma}^{\inn}_{vk}]_{j+1, k+1} =  \E\left[V^\inn_i V^\inn_k\right]  \E\left[V^\inn_j V^\inn_k\right] = [\bv^\inn_{vk}]_i[\bv^\inn_{vk}]_j$ by $\mathbf{P_{k}.(e)}$ and Lemma \ref{products}, where the equality follows from \eqref{eq:bvecs} and \eqref{eq:cov_def}.
Furthermore, 
$[\widetilde{\Cm}^\inn_{pk}]_{i+1+k,j+1+k}$ concentrates around $[\presm^{\inn}_{uk}]_{i+1+k,j+1+k} = [(\covm^\inn_{v(k-1)})^{-1}]_{i,j}$ when  $0 \leq i,j \leq k$ by $\mathbf{P_k.(g).(ii)}$ and \eqref{eq:Rk_Sk}.
Thus, again by Lemma \ref{products_0} and Lemma \ref{sums}, we get that $T_2$ concentrates around
$\sum_{i,j=0}^{k-1} [\bv^\inn_{vk}]_i[\bv^\inn_{vk}]_j[(\covm^\inn_{v(k-1)})^{-1}]_{i,j} = [\bv^\inn_{vk}]^T[\covm^\inn_{v(k-1)}]^{-1}\bv^\inn_{vk}.$ Using analogous reasoning, $T_3$ also concentrates around $0$.

Combining the above with Lemma \ref{sums} and using \eqref{eq:rhos}, we get that $\frac{1}{N}\|[\Bm^\perp_{\Cm_{pk}^\inn}]^T\vv^\inn_k\|^2$ concentrates around the desired constant:
$\left[\covm^\inn_{vk}\right]_{k+1,k+1}-[\bv^\inn_{vk}]^T[\covm^\inn_{v(k-1)}]^{-1}\bv^\inn_{vk} = \rho^\inn_{qk}.$

$\mathbf{P_k.(g).(iv)}$ We want to show that $[\Cm^\inn_{v(k+1)}]^T\Cm^\inn_{v(k+1)}$ is invertible with high probability and that, when invertible, the entries of the inverse concentrate around
\[
\presm^\inn_{vk} = \left[\begin{matrix}[\covm^\inn_{uk}]^{-1} & 0 \\ 0 & [\covm^\inn_{vk}]^{-1}\end{matrix}\right].
\]
This can be proved in much the same way as $\mathbf{P_k.(g).(ii)}$. In particular, note that $\Cm^\inn_{v(k+1)} = [\Cm^\inn_{pk} \vv^\inn_k]$. Then applying the block inversion formula to $[\Cm^\inn_{v(k+1)}]^T\Cm^\inn_{v(k+1)}$, we see that it is invertible if $[\Cm^\inn_{pk}]^T\Cm^\inn_{pk}$ is invertible and $\|[\Bm^\perp_{\Cm_{pk}^\inn}]^T\vv^\inn_k\|^2 \neq 0$. But the probabilities of both of these events concentrate by $\mathbf{P_k.(g).(ii)}$ and $\mathbf{P_k.(g).(iii)}$ respectively. 

Now, analogously to $\mathbf{P_k.(g).(ii)}$, to get concentration for the entries of $[\Cm^\inn_{v(k+1)}]^T\Cm^\inn_{v(k+1)}$, it suffices to get concentration for the entries of $([\Cm^\inn_{pk}]^T\Cm^\inn_{pk})^{-1}$, $([\Cm^\inn_{pk}]^T\Cm^\inn_{pk})^{-1}[\Cm^\inn_{pk}]^T\vv^\inn_k$, and of $\|[\Bm^\perp_{\Cm_{pk}^\inn}]^T\vv^\inn_k\|^{-2}$. For these, we can use $\mathbf{P_k.(g).(ii)}$ and $\mathbf{P_k.(g).(iii)}$ and apply the same arguments as above. This gives concentration for the entries of $[\Cm^\inn_{v(k+1)}]^T\Cm^\inn_{v(k+1)}$ around those of $\presm^\inn_{vk}$.

\subsection{Showing $\mathbf{Q_{k+1}}$ Holds}
The arguments for this case are analogous to those in $\mathbf{P_{k}}$, so we omit them here.

\section{Useful Lemmas}

Many of the lemmas used in this section were originally given in \cite{AMP_FS}. We do not include the proofs of these lemmas here, and instead refer the reader to \cite{AMP_FS} for the proofs.

\begin{lem}[Products of Lipschitz Functions]
  \label{lipprods}
  If $f,g:\mathbb{R}^p\to\mathbb{R}$ are Lipschitz continuous, then $h(x)=f(x)g(x):\mathbb{R}^p\to\mathbb{R}$ is pseudo-Lipschitz of order $2$.
\end{lem}

\begin{lem}[Concentration of Sums]
\label{sums}
If random variables $X_1, \ldots, X_M$ satisfy $P(\abs{X_i} \geq \e) \leq e^{-n\kappa_i \e^2}$ for $1 \leq i \leq M$, then 
\ben
P\Big(  \lvert \sum_{i=1}^M X_i  \lvert \geq \e\Big) \leq \sum_{i=1}^M P\left(|X_i| \geq \frac{\e}{M}\right) \leq M e^{-n (\min_i \kappa_i) \e^2/M^2}.
\een
\end{lem}

\begin{lem}[Concentration of Products]
\label{products} 
For random  variables $X,Y$ and non-zero constants $c_X, c_Y$, if
$
P(  | X- c_X |  \geq \e ) \leq K e^{-\kappa n \e^2},
$
and 
$
P(| Y- c_Y  |  \geq \e) \leq K e^{-\kappa n \e^2},
$
then the probability  $P(  | XY - c_Xc_Y  |  \geq \e)$ is bounded by 
\begin{align*}  
&  P\Big(  | X- c_X  |  \geq \min\Big( \sqrt{\frac{\e}{3}}, \frac{\e}{3 c_Y} \Big) \Big)  +  
P\Big(| Y- c_Y |  \geq \min\Big( \sqrt{\frac{\e}{3}}, \frac{\e}{3 c_X} \Big) \Big) \leq 2K e^{-\frac{\kappa n \e^2}{9\max(1, c_X^2, c_Y^2)}}.
\end{align*}
\end{lem}

\begin{lem}[Concentration of Products]
\label{products_0} 
For random variables $X,Y,$ and constant $c_X$, if
\[
P( | X- c_X  |  \geq \e) \leq K e^{- r_X \e^2} \text{ and } P(  | Y  |  \geq \e) \leq K e^{-r_Y \e^2},
\]
then if $c_X\neq 0$, we have
\begin{align*}  
P(  | XY  |  \geq \sqrt{\e}) \leq P( \abs{X - c_X}   \geq \sqrt{\e} ) +  P\Big( \abs{Y}  \geq 
\frac{\e}{2 \max\{1, |c_X|\} } \Big)  
\leq 
2 K \exp\Big\{- \frac{\e^2 \min\{r_Y, r_X \}}{4 \max\{1, c_X^{2}\}}\Big\},
\end{align*}
and if $c_X =0$ and $\epsilon < 1$, we have
\begin{align*}  
P(  | XY  |  \geq \e) \leq 2K \exp\Big\{- \e^2 \min\{r_Y, r_X \}\Big\}.
\end{align*}
\end{lem}

\begin{lem}[Concentration of Square Roots]
\label{sqroots}
Let $c \neq 0$. If $P( \lvert X_n^2 - c^2 \lvert \geq \epsilon ) \leq e^{-\kappa n \epsilon^2},$ then $P (\lvert \abs{X_n} - \abs{c} \lvert \geq \epsilon) \leq e^{-\kappa n \abs{c}^2 \epsilon^2}.$
\end{lem}

\begin{lem}[Concentration of Scalar Inverses]
\label{inverses} Assume $c \neq 0$ and $0<\e <1$. If
\ben
P( \lvert X_n - c \lvert \geq \epsilon ) \leq e^{-\kappa n \epsilon^2}, \text{ then }P( \lvert X_n^{-1} - c^{-1} \lvert \geq \epsilon ) \leq 2 e^{-n \kappa \e^2 c^2 \min\{c^2, 1\}/4}.
\een
\end{lem}

\begin{lem}
\label{lem:normalconc}
For a random variable $Z \sim  \mathcal{N}(0,1)$ and  $\e > 0$,
$P\Big( \abs{Z} \geq \e \Big) \leq 2e^{-\frac{1}{2}\e^2}$.
\end{lem}

\begin{lem}[$\chi^2$-concentration]
For  $Z_i$, $i \in [n]$ that are i.i.d.\ $\sim \mathcal{N}(0,1)$, and  $0 \leq \epsilon \leq 1$,
\[P\Big(\Big \lvert \frac{1}{n}\sum_{i=1}^n Z_i^2 - 1\Big \lvert \geq \epsilon \Big) \leq 2e^{-n \epsilon^2/8}.\]
\label{subexp}
\end{lem}

\begin{lem}
For any scalars $a_1, ..., a_t$ and positive integer $m$, we have  $\left(\abs{a _1} + \ldots + \abs{a_t} \right)^m \leq t^{m-1} \sum_{i=1}^t \abs{a_i}^m$.
Consequently, for any vectors $\mathbf{u}_1, \ldots, \mathbf{u}_t \in \mathbb{R}^N$, $\norm{\sum_{k=1}^t \mathbf{u}_k}^2 \leq t \sum_{k=1}^t \norm{\mathbf{u}_k}^2$.
\label{lem:squaredsums}
\end{lem}

\begin{lem}
\label{lem:PLsubgaussconc}
Let $Z_1, \ldots, Z_t \in \mathbb{R}^N$ be random vectors such that $(Z_{1,i}, \ldots, Z_{t,i})$ are i.i.d.\ across $i \in [n]$, with  $(Z_{1,i}, \ldots, Z_{t,i})$ being jointly Gaussian with zero mean, unit variance and covariance matrix $K \in \mathbb{R}^{t \times t}$. Let $G \in  \mathbb{R}^N$ be a random vector with entries  $G_1, \ldots, G_N$  i.i.d.\   $\sim p_{G}$, where $p_G$ is sub-Gaussian with variance factor $\nu$.   Then for any sequence of pseudo-Lipschitz functions $\{f_i: \mathbb{R}^{t+1} \to \mathbb{R}\}_{i=1}^N$ with common pseudo-Lipschitz constant, non-negative constants $\sigma_1, \ldots, \sigma_t$, and  $0< \e \leq 1$, we have
\begin{align*}
& P\Big(\Big\lvert \frac{1}{N}\sum_{i=1}^N f(\sigma_1 Z_{1,i}, \ldots, \sigma_tZ_{t,i}, G_i)- \frac{1}{N}\sum_{i=1}^N\mathbb{E}[f_i(Z_{1,1}, \ldots, Z_{t,1}, G)] \Big \lvert \geq \e \Big) \\
&\hspace{30mm} \leq 2 \exp \Big\{ \frac{ - N \e^2}{ 128 L^2 (t+1)^2(  \nu + 4 \nu^2 + \sum_{m=1}^t  (\sigma_m^{2} +  4  \sigma_m^{4})) }  \Big\},
\end{align*}
where $L >0$ is an absolute constant. ($L$ can be bounded above by three times the common pseudo-Lipschitz constant of the $f_i$.)
\begin{proof}
  The statement of Lemma \ref{lem:PLsubgaussconc} is slightly more general than Lemma B.4 in \cite{AMP_FS} in that Lemma B.4 requires the functions $f_i$ to all be equal. However, the proof of that lemma proves the above result exactly if we simply substitute the specific functions $f_i$ for $f$ throughout.
\end{proof}
\end{lem}

\section{Concentration is Preserved By Haar Matrices}

\begin{lem}\label{lem:orth_concentration}
  Suppose that $\vv_N\in\mathbb{R}^N$ is a (possibly random) vector-valued sequence such that
  \[
    \mathbb{P}\left( \left| \frac{1}{N}\sum_{i=1}^N\phi_i([\vv_N]_i) - \frac{1}{N}\sum_{i=1}^N\mathbb{E}\phi_i(V)\right|\geq\epsilon\right) \leq C\exp\left( -cN\epsilon^2 \right)
\]
for any sequence $\{\phi_i\}_{i=1}^N\subset PL(2)$ with a common pseudo-Lipschitz constant and some random variable $V$ with finite second moment. If $\Am_N$ is an $N\times N$ Haar-distributed orthogonal matrix independent of $\vv_N$ for each $N$, then
  \[
    \mathbb{P}\left( \frac{1}{N}\left| \sum_{i=1}^N\phi_i([\Am_N\vv_N]_i) - \frac{1}{N}\sum_{i=1}^N\mathbb{E}\phi_i(Z)\right|\geq\epsilon\right) \leq C'\exp\left( -c'N\epsilon^2 \right),
\]
where $Z\sim N(0,\sigma^2)$ for $\sigma^2 = \lim_{N\to\infty}\|\vv_N\|^2$. In particular, the above statement holds if $\vv_N$ is a deterministic sequence.
\end{lem}

\begin{proof}
  Since $\Am_N$ is independent of $\vv_N$ and $\Am_N$ is Haar distributed, we have that $\Bm\Am_N\vv_N\stackrel{d}{=}\Am_N\vv_N$ for any orthogonal matrix $\Bm$. Thus, the distribution of $\Am_N\vv_N$ is rotationally invariant and depends only on the distribution of the norm $\|\Am_N\vv_N\|=\|\vv_N\|$. Consequently, if $\Zv_N\sim N(\boldsymbol{0},\Im_N)$ independent of $\vv_N$, we have that
  \[
    \Am_N\vv_N \stackrel{d}{=}\frac{\|\vv_N\|}{\|\Zv_N\|}\Zv_N.
  \]
  Then we get that
  \begin{align*}
    \left| \frac{1}{N}\sum_{i=1}^n\phi_i([\Am_N\vv_N]_i) - \frac{1}{N}\sum_{i=1}^N\mathbb{E}\phi_i(Z)\right| &\stackrel{d}{=} \left| \frac{1}{N}\sum_{i=1}^N\phi_i\left(\frac{\|\vv_N\|}{\|\Zv_N\|}[\Zv_N]_i\right) - \frac{1}{N}\sum_{i=1}^N\mathbb{E}\phi_i(Z)\right|\\
                                                                                                             &= \left| \frac{1}{N}\sum_{i=1}^N\phi_i\left(\frac{\|\vv_N\|}{\|\Zv_N\|}[\Zv_N]_i\right) - \frac{1}{N}\sum_{i=1}^N\phi_i\left(\sigma[\Zv_N]_i\right) \right| \\
    &\hspace{1cm} +\left| \frac{1}{N}\sum_{i=1}^N\phi_i\left(\sigma[\Zv_N]_i\right) - \frac{1}{N}\sum_{i=1}^N\mathbb{E}\phi_i(Z)\right|.\\
  \end{align*}
  Now by Lemma \ref{sums}, it suffices to show concentration for these two terms separately. For the first we use the triangle inequality, the pseudo-Lipschitz property of the $\phi_i$, and Cauchy-Schwarz to get
  \begin{align*}
    \left| \frac{1}{N}\sum_{i=1}^N\phi_i\left(\frac{\|\vv_N\|}{\|\Zv_N\|}[\Zv_N]_i\right) - \frac{1}{N}\sum_{i=1}^n\phi_i\left(\sigma[\Zv_N]_i\right) \right|&\leq \frac{1}{N}\sum_{i=1}^N\left| \phi_i\left( \frac{\|\vv_N\|}{\|\Zv_N\|}[\Zv_N]_i \right)-\phi_i\left( \sigma[\Zv_N]_i \right) \right|\\
                                                                                                                                                         &\leq \frac{L}{N}\sum_{i=1}^N\left| 1+\left( \frac{\|\vv_N\|}{\|\Zv_N\|}+\sigma \right)[\Zv_N]_i \right|\left|\frac{\|\vv_N\|}{\|\Zv_N\|}-\sigma \right||[\Zv_N]_i|\\
    &\leq L\left|\frac{\|\vv_N\|}{\|\Zv_N\|}-\sigma \right|\left( \frac{\|\Zv_N\|}{\sqrt{N}} + \left( \frac{\|\vv_N\|}{\|\Zv_N\|}+\sigma \right)\frac{{\|\Zv_N\|^2}}N \right).
  \end{align*}
  Now we know that $\frac{\|\vv_N\|}{\|\Zv_N\|}$ concentrates around $\sigma$ using our assumptions on $\vv_N$, Lemma \ref{subexp}, and Lemma \ref{quotients}. Then using this fact along with Lemmas \ref{sqroots}, \ref{products}, and \ref{subexp}, we get that $\left( \frac{\|\Zv_N\|}{\sqrt{N}} + \left( \frac{\|\vv_N\|}{\|\Zv_N\|}+\sigma \right)\frac{{\|\Zv_N\|^2}}N \right)$ concentrates around a finite value (specifically $1+2\sigma$). Combining these with Lemma \ref{products}, we get that the the above concentrates with the desired rates (since each of our intermediate lemmas preserve this $N\epsilon^2$ rate). Finally concentration for
  \[
\left| \frac{1}{N}\sum_{i=1}^N\phi_i\left(\sigma[\Zv_N]_i\right) - \frac{1}{N}\sum_{i=1}^N\mathbb{E}\phi_i(Z)\right|
\]
follows directly from Lemma \ref{lem:PLsubgaussconc} (again with the desired rate) since we assume that the $\phi_i$ share a common pseudo-Lipschitz constant. This completes the proof of the lemma.
\end{proof}

\section{Extension of Concentration Analysis to Matrix General Recursion}\label{app:matrix_case}
In this appendix, we outline the adjustments that need to made to the main concentration arguments in order to get the analogous concentration result for the general recursion in Algorithm \ref{alg:general_gvamp} with $d>1$ (i.e. with matrix rather than vector iterates). The main differences lie in accounting for the distributions of the asymptotic variables $P_k^{\inn/\out}$ and $Q_k^{\inn/\out}$, which are now zero-mean Gaussian vectors rather than simply Gaussian random variables. As we will see, this adjustment follows directly from a straightforward extension of Lemma \ref{lem:cond_dist}, which in turn relies on the following characterization of orthogonally invariant random matrices.

\begin{lem}\label{lem:ortho_dist_general}
  Suppose that $\xv = \Um\yv\in\RR^{n\times d}$ for $\Um\in\RR^{n\times n}$ a Haar distributed orthogonal matrix and $\yv\in\RR^{n\times d}$ an independent random matrix which has full rank with probability $1$. Furthermore, let $\gram_{\xv} = \xv^T\xv\in\RR^{d\times d}$ be the Gram matrix of $\xv$, and let $\zv\in\RR^{n\times d}$ be a random matrix with independent standard Gaussian entries which is independent of $\Um$ and $\xv$. Then we have that
  \begin{equation}
    \label{eq:ortho_dist_eq}
    \xv \stackrel{d}{=}\zv \gram^{-1/2}_{\zv}\gram^{1/2}_{\yv}
  \end{equation}
\end{lem}

\begin{proof} (\textit{informal})
  
  The following proof is only informal because its claims rest on statements about probabilities of measure zero events. The difficulty of making these informal arguments more formal arises from the fact that the relevant random quantities take values in the space of $n\times d$ matrices with orthonormal columns. As this space is non-Euclidean, probability densities on this space do not satisfy the usual change of variables formula (since the normal definition of the Jacobian does not make sense). I believe this may be fixable using a change of variables formula from geometric measure theory that employs a suitable generalization of the Jacobian.

  Let $\Phi(\cv) = \cv\gram_{\cv}^{-1/2}$ for any matrix $\cv$. Then note that following facts:
  \begin{enumerate}
  \item $\Phi(\xv)\in\RR^{n\times d}$ has orthonormal columns.
  \item $\gram_{\Vm\cv}=\gram_{\cv}$ for any orthogonal matrix $\Vm$.
  \item $\Vm\xv \stackrel{d}{=} \xv$ for any fixed orthogonal matrix $\Vm$ since $\Vm\Um\stackrel{d}{=}\Um$ and $\yv$ and $\Um$ are independent.
  \end{enumerate}

  Now let $\pi:\RR^{n\times n}\to\RR^{n\times d}$ be the map which simply drops the last $n-d$ columns of its input. Then let $\Vm'\in\RR^{n\times d}$ be any matrix with orthonormal columns and let $\Vm$ be an orthogonal matrix such that $\pi(\Vm)=\Vm'$. Then by the above facts, we have
  \[
    \P\left( \Phi(\xv) = \Vm' \right) = \P\left( \Phi(\Vm\xv) = \Vm' \right) = \P\left( \Vm\xv\gram_{\xv}^{-1/2} = \Vm'  \right) = \P\left(\xv\gram^{-1/2}_{\xv} \pi(\mathbf{I})  \right) = \P\left( \Phi(\xv) = \rho(\mathbf{I}) \right).
  \]
  Since this is true for any $\Vm'$, we conclude that $\Phi(\xv)$ is uniform over ${d\times n}$ matrices with orthonormal columns.

  Now since $\zv$ has independent standard Gaussian entries, we have that $\Vm\zv\stackrel{d}{=}\zv$, and thus
  \[
    \P\left( \zv\gram_{\zv}^{-1/2}=\Vm' \right) = \P\left( \Vm\zv\gram_{\Vm\zv}^{-1/2}=\Vm' \right) = \P\left( \zv\gram_{\zv}^{-1/2}= \rho(\mathbf{I}) \right).
  \]

  This implies that $\zv\gram^{-1/2}_{\zv} = \Phi(\xv)$.

  Finally, we note that the entirety of the above argument still goes through with $\yv$ fixed, which implies that $\Phi(\xv)$ is uniformly distributed conditional of any particular of $\yv$ and is thus independent of $\yv$. Therefore,
  \[
    \xv = \xv\gram_{\xv}^{-1/2}\gram_{\xv}^{1/2} = \Phi(\xv)\gram_{\yv}^{1/2} \stackrel{d}{=} \zv\gram^{-1/2}_{\zv}\gram^{1/2}_{\yv},
  \]
  as claimed.
\end{proof}

Using this, we obtain a matrix version of Lemma \ref{lem:cond_dist}. To state this lemma, we first note the necessary adjustments to our notation. The matrices $\Cm^{\inn/\out}_{qk},$ $\Cm^{\inn/\out}_{pk}$, $\Cm^{\inn/\out}_{uk}$, and $\Cm^{\inn/\out}_{vk}$ are defined in exactly the same way as before, where now we understand the iterates (e.g. $\pv^{\inn/\out}_j$) as $(N/M)\times d$ matrices, so that, e.g., $\Cm^{\inn}_{pk}$ is now a $N\times d(2k+1)$ matrix. Likewise, the sigma algebras $\mathcal{P}_j$ and $\mathcal{Q}_j$ retain their definitions in terms of these matrix iterates. With $\Cm$ equal to any of these matrices, the definitions of the corresponding matrices $\Bm_{\Cm}$ and $\Bm_{\Cm}^\perp$ are unchanged. The Gaussian quantities $\Zv^{\inn/\out}_{pk}$, $\breve{\Zv}^{\inn/\out}_{pk}$, and $\overline{\Zv}^{\inn/\out}_{pk}$ have i.i.d. standard Gaussian entries before, but are expanded to have $d$ columns so that, e.g. $\overline{\Zv}^{\inn/\out}_{pk}\in\RR^{(N/M)\times d}$. Finally, the scalar/vector parameters $\rho_{(p/q)k}^{\inn/\out}$ and $\beta_{(p/q)k}^{\inn/\out}$ are replaced by matrix parameters $\Pm_{(p/q)k}^{\inn/\out}\in\RR^{d\times d}$ and $\Bm_{(p/q)k}^{\inn/\out}\in\RR^{kd\times d}$ respectively. We also introduce an additional indexing notation. For each $1\leq i\leq k$ letting
\[
  [\Bm_{(p/q)k}^{\inn/\out}]_{(i)}\in\RR^{d\times d}
\]
be the $d\times d$ submatrix with top-left entry $[\Bm_{(p/q)k}^{\inn/\out}]_{d(i-1)+1,1}$. In other words, if we view $\Bm_{(p/q)k}^{\inn/\out}$ as a $k$-dimensional vector with $d\times d$ matrix entries, then this is just accessing the index $i$ component of this vector.

With this notation, we can now state our generalized lemma. To save space, we state this only for the $\pv$ iterates, but the extension to the $\qv$ iterates is completely symmetric.

\begin{lem}
  If $[\Cm^{\inn}_{vk}]^T\Cm^{\inn}_{vk}$ has full rank for $0\leq k\leq K$, then we have for all such $k$ that
\begin{equation}
\pv^{\inn}_0 \lvert_{\mathcal{P}_0} \stackrel{d}{=} \Om^\inn_{p0} \, \overline{\Zv}^\inn_{p0}\,[\Pm^{\inn}_{p0}]^{1/2} + \Delm^\inn_{p0}, \qquad \text{ and  } \qquad \pv^\inn_k\lvert_{\mathcal{P}_k} \stackrel{d}{=} \sum_{\ell=0}^{k-1} \, \pv^\inn_{\ell}\,[\Bm^\inn_{pk}]_{(\ell +1)} + \Om^\inn_{pk} \, \overline{\Zv}^\inn_{pk}\, [\Pm^{\inn}_{pk}]^{1/2} + \Delm^\inn_{pk},
\label{eq:p_conds}
\end{equation}
where $\Om^\inn_{pk}$ is defined in \eqref{eq:Up_Uq_matrices}, and 
\begin{align*}
\Delm^\inn_{p0} & \text{ given in \textbf{Condition 0}},\\
  \Delm^\inn_{pk} &=  \Cm^\inn_{vk}\left(([\Cm^\inn_{qk}]^T\Cm^\inn_{qk})^{-1}[\Cm_{qk}^\inn]^T\uv^\inn_k-\left[\begin{matrix}\Bm^\inn_{pk}\\ \boldsymbol{0}_{dk\times d}\end{matrix}\right]\right) +\Bm^\perp_{\Cm_{vk}^\inn}\Zv^\inn_{pk}\left[\gram_{\Zv^\inn_{pk}}^{-1/2}\gram_{[\Bm^\perp_{\Cm_{qk}^\inn}]^T\uv_k^{\inn}}^{1/2} - [\Pm_{pk}^{\inn}]^{1/2} \right]\\
  &\hspace{1cm}-  \Bm_{\Cm_{vk}^\inn} \, \breve{\Zv}^\inn_{pk}[\Pm_{pk}^{\inn}]^{1/2}.
\end{align*}
Furthermore, we have that $\pv_k^\inn$ and $\pv_k^\out$ are conditionally independent given $\mathcal{P}_k$ for all $k\geq 0$.
\label{lem:cond_dists_general}
\end{lem}

\begin{proof}
  The proof of this statement is essentially identical to that of Lemma \ref{lem:cond_dist}, merely replacing applications of Lemma \ref{lem:isotropic_inv} with Lemma \ref{lem:ortho_dist_general} and substituting our matrix parameters $\Pm_{(p/q)k}^{\inn/\out}\in\RR^{d\times d}$ and $\Bm_{(p/q)k}^{\inn/\out}\in\RR^{kd\times d}$ for $\rho_{(p/q)k}^{\inn/\out}$ and $\beta_{(p/q)k}^{\inn/\out}$.
\end{proof}

Next we can state an analog of Lemma \ref{lem:joint_dists}. This Lemma depends on a matrix version of the Gaussian equivalent recursion in Algorithm \ref{alg:gaussian}. We omit the explicit statement here, as it is defined exactly to make statement $(2)$ in the following lemma true (and hence is entirely characterized by the statement of this lemma).
\begin{lem}
First define
\begin{equation}
\begin{split}
\ide{\pv}^{\inn}_0 = \widetilde{\Om}_{p0}^\inn \, \overline{\Zv}^{\inn}_{p0}\,[\Pm^{\inn}_{p0}]^{1/2}, &\qquad \text{ and } \qquad  \ide{\pv}^\inn_k = \sum_{r=0}^{k-1} \ide{\pv}_r^\inn[\Bm^\inn_{pk}]_{(r+1)} + \, \widetilde{\Om}^\inn_{pk} \, \overline{\Zv}^\inn_{pk}[\Pm^{\inn}_{pk}]^{1/2}, \\
\ide{\pv}^{\out}_0 = \widetilde{\Om}^\out_{p0} \, \overline{\Zv}^{\out}_{p0}\,[\Pm^{\out}_{p0}]^{1/2}, & \qquad \text{ and } \qquad  \ide{\pv}_k^\out = \sum_{r=0}^{k-1} \ide{\pv}_r^\out[\Bm^\out_{pk}]_{(r+1)} + \widetilde{\Om}^\out_{pk}\, \overline{\Zv}^\out_{pk}\, [\Pm^{\out}_{pk}]^{1/2}.
\label{eq:lemma_res0}
\end{split}
\end{equation}
Similarly define $ \ide{\qv}^\inn_k$ and $\ide{\qv}^\out_k$. With these definitions, we have
\begin{enumerate}
\item 
\begin{equation}
\ide{\pv}_k^\inn = \sum_{r=0}^k \widetilde{\Om}^\inn_{pr}\, \overline{\Zv}^\inn_{pr}\,[\Pm^{\inn}_{pk}]^{1/2}\, [\cv^\inn_{pk}]_r,\qquad \ide{\pv}_k^\out = \sum_{r=0}^k \widetilde{\Om}^\out_{pr}\, \overline{\Zv}^\out_{pr}\,[\Pm^{\out}_{pk}]^{1/2}\, [\cv^\out_{pk}]_r
\label{eq:lemma_res1}
\end{equation}
where for $k\geq 1$ and $0 \leq r\leq k-1$, $[\cv^\inn_{pk}]_r\in\RR^{d\times d}$ and $[\cv^\out_{pk}]_r\in\RR^{d\times d}$ are defined recursively as
\begin{equation}
[\cv^\inn_{pk}]_r = \sum_{i=r}^{k-1}[\Bm^\inn_{pk}]_{(i+1)}[\cv^\inn_{pi}]_r, \qquad \text{ and } \qquad [\cv^\out_{pk}]_r = \sum_{i=r}^{k-1}[\Bm^\out_{pk}]_{(i+1)}[\cv^\out_{pi}]_r,
\label{eq:lemma_c_def}
\end{equation}
with $[\cv^\inn_{pk}]_k=[\cv^\out_{pk}]_k = \mathsf{I}\in\RR^{d\times d}$ for all $k\geq 0$. 
\item
  Let $\underline{\ide{\rv}}_k = \left( \ide{\rv}_0,\ldots,\ide{\rv}_k \right)$ where $\ide{\rv}_j = (\ide{\pv}^{\inn}_j, \ide{\pv}^{\out}_j, \ide{\qv}^{\inn}_j, \ide{\qv}^{\out}_j)$ are defined as above for the ideal variables $\ide{\pv}$ and $\ide{\qv}$. Then we have that
  \begin{equation}
    \label{eq:lemma_res2}
    \widetilde{\underline{\rv}}_k - \underline{\ide{\rv}}_k = \left( \dv_0,\ldots,\dv_k \right),
  \end{equation}
 where $ \underline{\widetilde{\rv}}_k$ is defined in \eqref{eq:under_r} with
  \be
  \label{eq:dv_vec}
    \dv_k = \left(\sum_{r=0}^k\widetilde{\Delm}^\inn_{pr}\, [\cv^\inn_{pk}]_r, \, \sum_{r=0}^k\widetilde{\Delm}^\out_{pr}\,[\cv^\out_{pk}]_r, \, \sum_{r=0}^k\widetilde{\Delm}^\inn_{qr}\, [\cv^\inn_{qk}]_r, \, \sum_{r=0}^k\widetilde{\Delm}^\out_{qr}\,[\cv^\out_{qk}]_r \right).
  \ee
  Furthermore, for all $k\geq 0$, we have that
  \begin{equation}
    \widetilde{\underline{\rv}}_k \stackrel{d}{=} \underline{\rv}_k.
    \label{eq:lemma_eq_dist}
  \end{equation}
\item For all $i\geq 1$, we have $\left( [\ide{\pv}_0^\inn]_i,\ldots,[\ide{\pv}_k^\inn]_i,[\ide{\pv}_0^\out]_i,\ldots,[\ide{\pv}_k^\out]_i \right) \stackrel{d}{=}\left( P^\inn_0,\ldots,P^\inn_k ,P_0^\out,\ldots,P_k^\out\right)$ where $(P_0^\inn,\ldots,P_k^\inn)\in\RR^{d(k+1)}$ and $(P_0^{\out},\ldots,P_k^\out)\in\RR^{d(k+1)}$ are independent, zero-mean, jointly Gaussian vectors (and hence the right hand side is itself jointly Gaussian). 
\end{enumerate}
Analogous statements for (1.)-(3.) hold for the $\ide{\qv}$ variables.
  \label{lem:joint_dist_general}
\end{lem}

\begin{proof}
  The proof of this statement is again essentially identical to that of Lemma \ref{lem:joint_dists} once all vector quantities are replaced by their matrix analogs.
\end{proof}

Finally, we can define limiting quantities for the matrix form of the general recursion. Specifically, if we set
\begin{align*}
  \overline{U}^{\inn/\out}_k &= \left( [U^{\inn/\out}_0]^T,\ldots,[U^{\inn/\out}_{k}]^T \right)^T\\
  &= \left( \left[U^{\inn/\out}_0  \right]_1,\ldots,\left[U^{\inn/\out}_0\right]_d,\ldots,\left[U^{\inn/\out}_k  \right]_1,\ldots,\left[U^{\inn/\out}_k\right]_d \right)^T\in\RR^{d\times(k+1)},
\end{align*}
then we can define
\begin{equation}
  \label{eq:limiting_general}
  \covm^{\inn/\out}_{pk} = \E \left\lbrace\overline{U}^{\inn/\out}_k[\overline{U}^{\inn/\out}_k]^T\right\rbrace \in \RR^{d(k+1)\times d(k+1)},\quad \bv_{uk} = \E\left\lbrace \overline{U}^{\inn/\out}_{k-1}[U^{\inn/\out}_k]^T\right\rbrace\in\RR^{kd\times d}.
\end{equation}
In terms of these, we can then define
\begin{equation}
  \label{eq:limiting_general2}
  \Bm^{\inn/\out}_{pk} = \left[ \covm^{\inn/\out}_{p(k-1)} \right]^{-1}\bv_{uk}, \quad \Pm^{\inn/\out}_{pk}=\left[ \covm_{pk}^{\inn/\out} \right]_{(k,k)} - [\bv_{uk}^{\inn/\out}]^T\left[ \covm_{p(k-1)}^{\inn/\out} \right]^{-1}\bv_{uk}^{\inn/\out},
\end{equation}
where the $(i,j)$ indexing notation is used to extract the $d\times d$ submatrix with top left index $((i-1)d + 1, (j-1)d+1)$.

With these results, it is easy to see how the proof of Lemma \ref{lem:main_general_long} extend to the matrix case. For parts $(g)$ and $(a)$, which are concerned with the broad components of the $\Delm$ terms, the only differences arise from the presence of matrix operations replacing scalar operations (in particular due to the introduction of the Gram matrices). Concentration for matrix multiplications can be established as we have done for other matrix products in the proof of Lemma \ref{lem:main_general_long} (specifically using our concentration results for scalar products and sums). Concentration for matrix square roots is more complicated, but for $d\leq 4$, explicit formulas in terms of radicals are possible, to which are existing suite of concentration tools can be applied. In each case, the matrix operations which replace scalar operations are all of dimension $d$, which is fixed and independent of $N$ and $k$. Hence, whatever penalty is incurred in the concentration rate we expect to be proportional to $d^2$. In particular, we can view this as being absorbed into the universal constants $c$ and $C$.

Part $(b)$ of the proof is almost entirely analogous as the same arguments now simply need to be applied to pseudo-Lipschitz functions with $2kd$ rather than $2k$ inputs. The remaining parts of the proof are essentially just applications of $(b)$ to establish concentration for various inner products, which can also be established in exactly the same way.

\bibliographystyle{Alpha}
\bibliography{VAMP_bib}

\end{document}